\documentclass[review,3p,10pt]{elsarticle}
\biboptions{sort&compress}

\usepackage{amsmath,amssymb,amsthm}
\usepackage{bm,color}
\usepackage{stmaryrd}
\usepackage[T1]{fontenc}
\usepackage{float, caption, subcaption}
\usepackage{booktabs}
\usepackage{mathrsfs}
\usepackage{multirow}
\usepackage[bookmarks=true,
  bookmarksnumbered=true,
  bookmarksopen=true,
  colorlinks,
  pdfborder=001,
  linkcolor=black]{hyperref}

\allowdisplaybreaks
\newtheorem{definition}{Definition}[section]
\newtheorem{theorem}{Theorem}[section]
\newtheorem{lemma}[theorem]{Lemma}
\newtheorem{proposition}{Proposition}[section]
\newtheorem{remark}{Remark}[section]
\newtheorem{example}{Example}[section]

\numberwithin{equation}{section}
\numberwithin{figure}{section}
\numberwithin{table}{section}
\newcommand\bbR{\mathbb{R}}

\newcommand\bx{\bm{x}}
\newcommand\bV{\bm{V}}
\newcommand\bU{\bm{U}}
\newcommand\bF{\bm{F}}

\newcommand\pd[2]{\dfrac{\partial {#1}}{\partial {#2}}}

\newcommand\abs[1]{\lvert #1 \rvert}
\newcommand\mean[1]{\left\{\!\!\left\{ #1 \right\}\!\!\right\}}
\newcommand\jump[1]{\llbracket #1 \rrbracket}
\newcommand\jumpangle[1]{\langle\!\langle #1 \rangle\!\rangle}

\begin{document}

\begin{frontmatter}



\title{High-order accurate well-balanced energy stable finite difference schemes for multi-layer shallow water equations on fixed and adaptive moving meshes}


\author[inst1]{Zhihao Zhang}
\affiliation[inst1]{organization={Center for Applied Physics and Technology, HEDPS and LMAM,
	School of Mathematical Sciences, Peking University},
            city={Beijing},
            postcode={100871},
            country={P.R. China}}
\ead{zhihaozhang@pku.edu.cn}

\author[inst2,inst1]{Huazhong Tang}
\affiliation[inst2]{organization={Nanchang Hangkong University},
            city={Nanchang},
            postcode={330000},
            state={Jiangxi Province},
            country={P.R. China}}
\ead{hztang@math.pku.edu.cn}

\author[inst3]{Junming Duan\corref{cor1}}
\affiliation[inst3]{organization={\ Fakult{\"a}t f{\"u}r Mathematik und Informatik, Universit{\"a}t W{\"u}rzburg, Emil-Fischer-Stra{\ss}e 40},
            city={W\"urzburg},
            postcode={97074},
            country={Germany}}
\cortext[cor1]{Corresponding author}
\ead{junming.duan@uni-wuerzburg.de}

\begin{abstract}
This paper develops high-order accurate well-balanced (WB) energy stable (ES) finite difference schemes for multi-layer (the number of layers $M\geqslant 2$) shallow water equations (SWEs) with non-flat bottom topography on both fixed and adaptive moving meshes, extending our previous work on the single-layer shallow water magnetohydrodynamics \cite{Duan2021_High} and single-layer SWEs on adaptive moving meshes \cite{Zhang2023High}.
To obtain an energy inequality, the convexity of an energy function for an arbitrary $M$ is proved by finding recurrence relations of the leading principal minors or the quadratic forms of the Hessian matrix of the energy function with respect to the conservative variables, which is more involved than the single-layer case due to the coupling between the layers in the energy function.
An important ingredient in developing high-order semi-discrete ES schemes is the construction of a two-point energy conservative (EC) numerical flux.
In pursuit of the WB property, a sufficient condition for such EC fluxes is given with compatible discretizations of the source terms similar to the single-layer case.
It can be decoupled into $M$ identities individually for each layer, making it convenient to construct a two-point EC flux for the multi-layer system.
To suppress possible oscillations near discontinuities, WENO-based dissipation terms are added to the high-order WB EC fluxes, which gives semi-discrete high-order WB ES schemes.
Fully-discrete schemes are obtained by employing high-order explicit strong stability preserving Runge-Kutta methods and proved to preserve the lake at rest.
The schemes are further extended to moving meshes based on a modified energy function for a reformulated system, relying on the techniques proposed in \cite{Zhang2023High}.
Numerical experiments are conducted for some two- and three-layer cases to validate the high-order accuracy, WB and ES properties, and high efficiency of the schemes,
with a suitable amount of dissipation chosen by estimating the maximal wave speed due to the lack of an analytical expression for the eigenstructure of the multi-layer system.

\end{abstract}


\end{frontmatter}



\section{Introduction}\label{section:Intro}

Multi-layer shallow water equations (ML-SWEs) \cite{Schijf1953Theoretical, Vreugdenhil1994Numerical, Frings2012Adaptive} serve as a mathematical model that describes the intricate dynamics of multi-layer immiscible free surface flow in the presence of gravity and bottom topography.
These equations encompass the essential considerations of vertical variations in water density, depth and velocity, omitted by single-layer SWEs, which play a pivotal role in the examination of stratified flows, as evident in oceanography and coastal engineering \cite{Dalziel1991Two,Schijf1953Theoretical}, such as tidal dynamics and coastal hydrodynamics.
The two-dimensional (2D) ML-SWEs defined in a time-space physical domain $(t,\bm{x} = (x_1,x_2))\in\bbR^+\times\Omega_p,~\Omega_p\subset \bbR^2$ with totally $M\geqslant 2$ layers and non-flat bottom topography $b(\bm{x})$ can be written as the following hyperbolic balance laws 
\begin{equation}\label{eq:ML-SWEs-2D}
	\left\{
	\begin{aligned}
		&~\pd{}{t}h_m+\pd{}{x_1}\left(h_mu_m\right)+\pd{}{x_2}\left(h_mv_m\right)=0, 
		\\
		&\pd{}{t}\left(h_m u_m\right)+\pd{}{x_1}\left(h_m u_m^2+\frac{1}{2} g h_m^2\right) + \pd{}{x_2}\left(h_m u_m v_m\right)=-g h_m \pd{}{x_1}\left(b+\sum\limits_{k > m} h_{k} + \sum\limits_{k < m} 
		r_{km}
		h_{k}\right), 
		\\
		&\pd{}{t}\left(h_m v_m\right)+\pd{}{x_1}\left(h_m u_m v_m\right) + \pd{}{x_2}\left(h_m v_m^2+\frac{1}{2} g h_m^2\right)=-g h_m \pd{}{x_2}\left(b+\sum\limits_{k > m} h_{k} + \sum\limits_{k < m} 
		r_{km}
		h_{k}\right),
	\end{aligned}
	\right.
\end{equation}
where $h_m(t,\bm{x})\geqslant 0$ is the water depth of the $m$th-layer for $m=1,\dots,M$,
$u_m$ and $v_m$ are the $x_1$- and $x_2$-component of the $m$th-layer fluid velocity, and $g$ is the gravitational constant.
The density ratio between the $k$th- and $m$th-layer is denoted by $r_{km} = \rho_{k}/\rho_m$, and the constant densities satisfy $0 < \rho_1<\cdots< \rho_M$. 
In this paper, the dry water cases are not considered, i.e., $h_m(t,\bm{x})$ is always positive. 
An illustration of such a multi-layer system is given in Figure \ref{fig:ML-SWEs}.
\begin{figure}[h]
	\centering
	\includegraphics[width=0.5\linewidth]{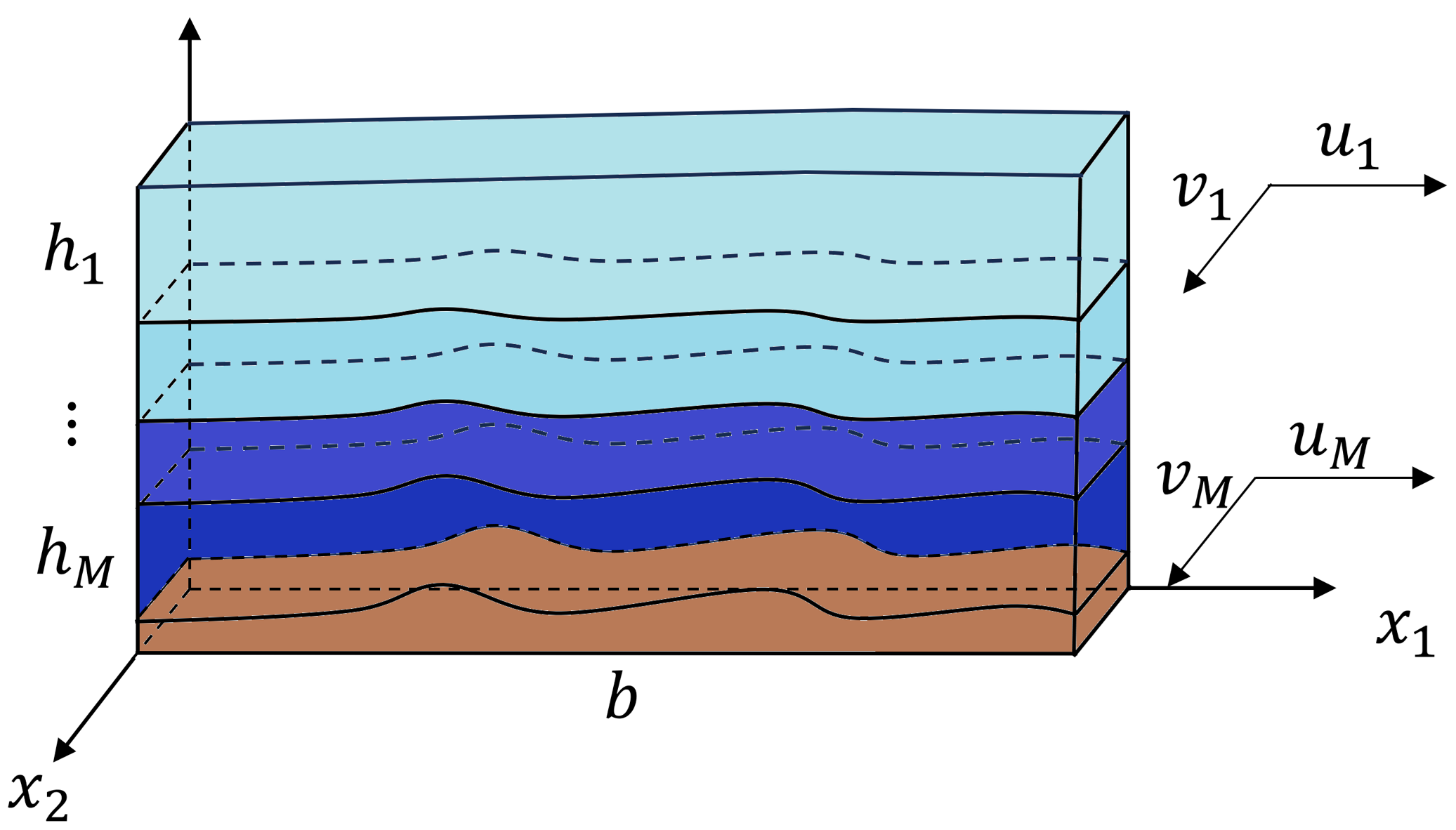}
	\caption{Illustration of the multi-layer shallow water equations.}\label{fig:ML-SWEs}
\end{figure}

In real-world applications of oceanography and coastal engineering, numerically solving the ML-SWEs is of great importance.
One of the difficulties is that there is no analytical expression for the eigenvalues and eigenvectors of the system \eqref{eq:ML-SWEs-2D},
so that some numerical methods, such as upwind schemes, cannot be directly extended to the multi-layer cases.
For the two-layer SWEs \eqref{eq:TL-SWEs-2D}, when $\abs{u_2-u_1}$
and $1-r_{12}$ are small, an approximation of the eigenvalues is given in \cite{Schijf1953Theoretical,Kim2008Two},
and one needs
$\left(u_1-u_2\right)^{2}<g(1-r_{12})(h_1+h_2)$
to ensure real eigenvalues.
But this condition is only valid when $r_{12}\approx 1$ and $u_1\approx u_2$, as stated in \cite{Kurganov2009Central}. When the number of layers  $M\geqslant3$, it is difficult to get an estimation for the eigenvalues \cite{Frings2012Adaptive}.
One more practical way is to estimate the upper or lower bounds of the eigenvalues.

Another important issue that should be addressed is that the ML-SWEs \eqref{eq:ML-SWEs-2D} possess the following lake at rest state
\begin{equation*}
	\left(h_m + b+
	\sum\limits_{k > m} h_{k} 
	+ \sum\limits_{k < m} r_{km} h_{k}\right) \equiv C_m,
	~u_m= v_m = 0,~ m =1,\cdots,M,
\end{equation*}
or equivalently,
\begin{equation*}
	h_m = \widetilde{C}_m, ~m < M,~
        h_M + b = \widetilde{C}_M,~
	u_m= v_m = 0, ~m =1,\cdots,M,
\end{equation*}
where $C_m$ and $\widetilde{C}_m$ are constant.
Many physical phenomena such as waves on a lake or tsunamis in the deep ocean \cite{Castro2012ADER} can be seen as small perturbations of the steady states, and it is crucial to capture these perturbations accurately.
This inspired researchers to develop well-balanced (WB) schemes to preserve the steady states in the discrete sense, enabling accurate capture of small perturbations even on coarse meshes.
The idea of the WB property or the ``C-property'' was first illustrated in \cite{Bermudez1994Upwind},
then various WB numerical methods for the SWEs were studied \cite{Vukovic2002ENO,Xing2005High,AuDusse2004fast,Li2012Hybrid,Noelle2007High,Xing2014Exactly}.
For more details, the readers are referred to \cite{Kurganov2018Finite} and the references therein.

Over the past few years, there have been limited numerical schemes introduced for solving the ML-SWEs. The existing schemes have predominantly concentrated on the two-layer scenarios,  with minimal attention given to the challenges posed by adaptive moving meshes, not to mention high-order WB and ES adaptive moving mesh schemes for $M\geqslant3$. A relaxation approach \cite{Abgrall2009Two} was proposed for the 1D two-layer SWEs, which gives an unconditionally hyperbolic system and an explicit eigenstructure. The similar idea was extended to the 2D two-layer case in \cite{Chiapolino2018Models} and \cite{Liu2021Anew}, where the latter presented a second-order scheme on unstructured meshes.
A central-upwind scheme was proposed in \cite{Kurganov2009Central} based on an estimation of the maximum wave speed and the path-conservative method was employed to improve the robustness of the central-upwind scheme in \cite{Castro2019Path}.
A central discontinuous Galerkin (DG) method was constructed for the two-layer SWEs in \cite{Cheng2020Ahigh} to achieve high-order accuracy.
There are also other works \cite{Bouchut2008An,Cao2023Flux,Dudzinski2013Well,Bouchut2010Arobust,Castro2005Numerical}.

In the design of numerical methods, it is desirable to satisfy semi-discrete or fully-discrete stability conditions to improve the robustness of the numerical simulations.
For hyperbolic conservation laws, the schemes satisfying entropy stability have been extensively studied, e.g., \cite{Tadmor1987The, Tadmor2003Entropy,Lefloch2002Fully, Fjordholm2012Arbitrarily,Carpenter2014Entropy,Gassner2013A,Duan2020RMHD}.
For the single-layer SWEs, energy stable (ES) schemes that satisfy some discrete energy inequalities were proposed in \cite{Berthon2016Afully,Fjordholm2011Well,Bouchut2008An,Zhang2023High},
and recently, high-order WB ES DG schemes were proposed for the two-layer SWEs \cite{Ersing2023Entropy} on fixed curvilinear meshes.
The complex structures such as shocks and sharp transitions that arise in solutions to the SWEs often require fine meshes for better resolution.
To improve both the efficiency and quality of numerical solutions, the adaptive moving mesh method has been adopted in the numerical simulations for the SWEs.
An adaptive moving mesh kinetic flux-vector splitting method was introduced in \cite{Tang2004Solution}.
High-order adaptive moving mesh DG schemes were proposed in \cite{Zhang2021High,Zhang2022AWell}, where the preservation of the positivity of the water height and the lake at rest was also discussed.
In \cite{Zhang2023Structure}, a WB and positivity-preserving finite volume WENO arbitrary
Lagrangian-Eulerian scheme was developed.

This paper aims to develop high-order accurate finite difference schemes for the ML-SWEs that are both ES and WB on fixed and adaptive moving meshes, extending our previous works \cite{Duan2021_High,Zhang2023High}.
To obtain an energy inequality, the convexity of an energy function is proved.
Compared to the single-layer case, the coupling terms between the multiple layers in the energy function make the corresponding Hessian matrix more complicated, thus the proof is more involved, which is based on finding the recurrence relations of the leading principal minors or the quadratic forms of the Hessian matrix.
The WB property requires compatible discretizations for the flux gradient and source terms, thus a sufficient condition involving the discretizations of the source terms from all the layers is derived for two-point energy conservative (EC) numerical fluxes similar to the single-layer case.
The condition can be decomposed into $M$ individual equalities for each layer, which makes it convenient to construct $M$ two-point fluxes for each single layer, then they are concatenated to form the two-point EC flux for the multi-layer system.
Based on such a two-point EC flux, the high-order WB EC fluxes are constructed by using the two-point discretizations as a building block. To suppress possible oscillations near discontinuities, the high-order WB ES schemes are obtained by adding suitable WENO-based dissipation terms to the high-order WB EC fluxes with compatible discretizations of the source terms, which are proved to satisfy semi-discrete energy inequalities.
The explicit strong-stability preserving (SSP) third-order Runge-Kutta (RK3) method is employed to obtain the fully-discrete schemes, which are proved to preserve the lake at rest.
The extension to the adaptive moving meshes is achieved based on a modified energy function for a reformulated system by including the bottom topography as an additional conservative variable and carefully designed dissipation terms, extending the techniques in \cite{Zhang2023High}.
In the numerical experiments for the two- and three-layer systems, the estimation of the maximal wave speed \cite{Kurganov2009Central} is adopted to choose a proper amount of dissipation since there is no explicit expression for the eigenstructure.

The remainder of this paper is structured as follows.
Section \ref{section:EntropyCondition} presents the proof of the convexity of the energy function and corresponding energy inequality.
Section \ref{section:2DHOScheme} is devoted to the construction of the schemes on fixed meshes.
To achieve the WB property, a sufficient condition for the two-point EC fluxes is derived with suitable discretizations of the source terms, and a corresponding two-point EC flux is constructed.
Based on such a flux, the high-order WB EC fluxes are obtained, and then the high-order WB ES schemes are developed.
The extension of the schemes to adaptive moving meshes is detailed in Section \ref{Sec:MM_Case} with a modified energy function for the reformulated system.
Numerical experiments for the two- and three-layer systems conducted in Section \ref{section:Result} serve to validate the high-order accuracy, efficiency, WB and ES properties, and shock-capturing ability.
Section \ref{section:Conc} concludes the paper with final remarks.

\section{Energy inequality for the ML-SWEs}\label{section:EntropyCondition}

The ML-SWEs \eqref{eq:ML-SWEs-2D} can be written in a more compact form
\begin{equation}\label{eq:SWE0}
    \frac{\partial{\bm{U}}}{\partial t}+\sum_{\ell=1}^2 \frac{\partial {\bm{F}}_{\ell}({\bm{U}})}{\partial x_{\ell}}
    =-g \sum_{\ell=1}^2 \sum\limits_{m=1}^Mh_m \frac{\partial {\bm{B}}_{\ell,m}(\bU, b)}{\partial x_{\ell}},
\end{equation}
where the vector of the conservative variables $\bm{U} = (\bm{U}_{1}^{\mathrm{T}},\dots,\bm{U}_M^{\mathrm{T}})^{\mathrm{T}}$ and the $x_\ell$-directional physical flux $\bm{F}_{\ell} = \left(\left({\bm{F}}_{\ell}\right)_1^{\mathrm{T}},\dots,\left({\bm{F}}_{\ell}\right)_M^{\mathrm{T}}\right)^{\mathrm{T}}$ are given by
\begin{equation*}
\begin{aligned}
&{\bm{U}_{m}}  =\left(h_m, h_m u_m, h_m v_m\right)^{\mathrm{T}}, \\
&\left({\bm{F}}_{1}\right)_{m}  =\left(h_mu_m, h u_m^2+\frac{g}{2} h_m^2, h_m u_m v_m\right)^{\mathrm{T}}, \\
&\left({\bm{F}}_{2}\right)_{m}  =\left(h_m v_m, h_m u_m v_m, h_m v_m^2+\frac{g}{2} h_m^2\right)^{\mathrm{T}},
\end{aligned}
\end{equation*}
with the source terms
\begin{equation*}
\begin{aligned}
\bm{{B}}_{1,m} &= \left(\bm{0}_{3m-2},z_m,0,\bm{0}_{3M-3m}\right)^\mathrm{T},\\
\bm{{B}}_{2,m} &= \left(\bm{0}_{3m-2},0,z_m,\bm{0}_{3M-3m}\right)^\mathrm{T}.
\end{aligned}
\end{equation*}
Here $z_m := b+
		\sum\limits_{k > m} h_{k} 
		+ \sum\limits_{k < m} r_{km} h_{k}$ and $\bm{0}_d$ is the row vector with $d$ zeros.
One can define an energy and energy flux pair for the system \eqref{eq:SWE0} as
\begin{equation}\label{eq:Entropy_Pair_ML_2D}
    \begin{aligned}
	{\eta}({\bm{U}})&:= \sum\limits_{m=1}^{M}\dfrac{\rho_m}{2}\left(h_m u_m^2 + h_m v_m^2 + g h_m^2\right) + g\sum\limits_{m=1}^{M} \rho_m h_m b + g\sum\limits_{m=1}^{M}\sum\limits_{k<m} \rho_{k} h_{k} h_m ,\\
	{q}_1({\bm{U}})&:=\sum\limits_{m=1}^{M}\dfrac{\rho_m}{2}u_m\left(h_m u_m^2 + h_m v_m^2\right)+g\sum\limits_{m=1}^{M}\rho_{m}h_m^2u_m+g\sum\limits_{m=1}^{M}\rho_m h_m u_m z_m,\\
	{q}_2({\bm{U}})&:=\sum\limits_{m=1}^{M}\dfrac{\rho_m}{2}v_m\left(h_m u_m^2 + h_m v_m^2\right)+g\sum\limits_{m=1}^{M}\rho_{m}h_m^2v_m+g\sum\limits_{m=1}^{M}\rho_m h_m v_m z_m,
    \end{aligned}
\end{equation}
which satisfies the following consistent condition
\begin{equation}\label{eq:EntropyPairConsistent}
    \pd{{q}_{\ell}}{{\bm{U}}}
    =\pd{\eta}{\bU} \left(\pd{{\bF}_{\ell}}{{\bm{U}}}+\bm{N}_{\ell}\right), ~\ell=1,2,
\end{equation}
where $\bm{N}_{\ell}\left(\bm{U}, b\right)$ comes from the non-conservative source terms
\begin{equation*}
	\bm{N}_{\ell} = g\sum_{m=1}^{M}h_m\pd{\bm{B}_{{\ell},m}}{\bm{U}}.
\end{equation*}
The convexity of the energy $\eta$ with respect to $\bU$ is proved in Proposition \ref{prop:eta_convex_prove}, based on the structure of the Hessian matrix $\partial^2{\eta}/\partial {\bm{U}}^2$ given in Lemma \ref{lemma:Hessian}.

\begin{lemma}\label{lemma:Hessian}\rm
    Denote the Hessian matrix for the $\mathcal{N}$-layer case as $\mathcal{M}_\mathcal{N} \in \mathbb{R}^{3\mathcal{N}\times3\mathcal{N}}$, then the Hessian matrix for the $(\mathcal{N}+1)$-layer case can be expressed as
    \begin{equation*}
    \mathcal{M}_{\mathcal{N}+1} = 
    \left(
        \begin{array}{cc}
             \mathcal{M}_{\mathcal{N}} & A_1^{\mathrm{T}} \\
            A_1 & A_2 
        \end{array}
        \right) \in \mathbb{R}^{(3\mathcal{N}+3)\times(3\mathcal{N}+3)},
    \end{equation*}
    where
    \begin{align*}
        &A_1 = 
        \left(
        \begin{array}{cccccccccccccc}
             g\rho_{1}&0&0  &g\rho_{2}&0&0&  \cdots&g\rho_{\mathcal{N}-1}&0&0&g\rho_{\mathcal{N}}&0&0\\
              0&0&0  &0&0&0& \cdots&0&0&0&0&0&0\\
              0&0&0  &0&0&0 &\cdots&0&0&0&0&0&0
        \end{array}
        \right) \in \mathbb{R}^{3\times3\mathcal{N}},
        \\~
            &A_2 = 
            \left(
            \begin{array}{ccc}
        \frac{\rho _{\mathcal{N}+1}\,\left(u_{\mathcal{N}+1}^2+v_{\mathcal{N}+1}^2+g\,h_{\mathcal{N}+1}\right)}{h_{\mathcal{N}+1}} & -\frac{\rho _{\mathcal{N}+1}\,u_{\mathcal{N}+1}}{h_{\mathcal{N}+1}} & -\frac{\rho _{\mathcal{N}+1}\,v_{\mathcal{N}+1}}{h_{\mathcal{N}+1}}\\
       -\frac{\rho _{\mathcal{N}+1}\,u_{\mathcal{N}+1}}{h_{\mathcal{N}+1}} & \frac{\rho _{\mathcal{N}+1}}{h_{\mathcal{N}+1}} & 0\\
         -\frac{\rho _{\mathcal{N}+1}\,v_{\mathcal{N}+1}}{h_{\mathcal{N}+1}} & 0 & \frac{\rho _{\mathcal{N}+1}}{h_{\mathcal{N}+1}}
            \end{array}
         \right)\in \mathbb{R}^{3\times3}.
    \end{align*}
\end{lemma}

\begin{proof}
    It can be completed because the Hessian matrices $\mathcal{M}_{\mathcal{N}}$ and $\mathcal{M}_{\mathcal{N}+1}$ can be calculated directly.
\end{proof}

\begin{proposition}\rm\label{prop:eta_convex_prove}
Under the condition $0<\rho_1<\rho_2<\dots<\rho_{\mathcal{N}}<\dots$,
the energy $\eta$ is convex with respect to $\bU$,
i.e., the Hessian matrix $\mathcal{M}_\mathcal{N} = \partial^2{\eta}/\partial {\bm{U}}^2$ is positive-definite.
Moreover,  the determinant of $\mathcal{M}_\mathcal{N}$ is
\begin{equation}\label{eq:detM}
    \det(\mathcal{M}_\mathcal{N}) = \frac{g^\mathcal{N}\rho_1^3\prod_{m=2}^{\mathcal{N}}\rho_m^2(\rho_m-\rho_{m-1})}{\prod_{m=1}^{\mathcal{N}}h_m^2}.
\end{equation}
\end{proposition}

\begin{proof}
The proof is given by induction on $\mathcal{N}$.
When $\mathcal{N}=2$, one observes that the Hessian matrix $\mathcal{M}_{2}$ is
\begin{equation}\label{eq:Hessian_M2}
    \mathcal{M}_{2} = \left(\begin{array}{cccccc} \frac{\rho _{1}\,\left(u_{1}^2+v_{1}^2+g\,h_{1}\right)}{h_{1}} & -\frac{\rho _{1}\,u_{1}}{h_{1}} & -\frac{\rho _{1}\,v_{1}}{h_{1}} & g\,\rho _{1} & 0 & 0\\ -\frac{\rho _{1}\,u_{1}}{h_{1}} & \frac{\rho _{1}}{h_{1}} & 0 & 0 & 0 & 0\\ -\frac{\rho _{1}\,v_{1}}{h_{1}} & 0 & \frac{\rho _{1}}{h_{1}} & 0 & 0 & 0\\ g\,\rho _{1} & 0 & 0 & \frac{\rho _{2}\,\left(u_{2}^2+v_{2}^2+g\,h_{2}\right)}{h_{2}} & -\frac{\rho _{2}\,u_{2}}{h_{2}} & -\frac{\rho _{2}\,v_{2}}{h_{2}}\\ 0 & 0 & 0 & -\frac{\rho _{2}\,u_{2}}{h_{2}} & \frac{\rho _{2}}{h_{2}} & 0\\ 0 & 0 & 0 & -\frac{\rho _{2}\,v_{2}}{h_{2}} & 0 & \frac{\rho _{2}}{h_{2}} \end{array}\right).
\end{equation}
Denote the leading principal minors of order $m$ as $\mathcal{D}_m$,
then the leading principal minors of $\mathcal{M}_{2}$ are
\begin{align*}
    &\mathcal{D}_1 = \frac{\rho _{1}
\left(u_{1}^2+v_{1}^2
+gh_{1}
\right)
}{h_{1}},~\mathcal{D}_2 = \frac{\rho_{1}^2
v_{1}^2
+gh_{1}\rho _{1}^2}
{h_{1}^2},~\mathcal{D}_3 = \frac{g\rho_{1}^3}
{h_{1}^2},\\
&\mathcal{D}_4 = \frac{g\rho_{1}^3\left(
gh_{2}\left(\rho_{2}-\rho _{1}\right)+
\rho _{2}
\left(
u_{2}^2+v_{2}^2\right)
\right)}{h_{1}^2\,h_{2}},~
\mathcal{D}_5 = \frac{
g\rho _{1}^3
\rho _{2}
\left(gh_2\left(\rho_{2}-\rho _{1}\right)+
\rho _{2}
v_{2}^2\right)
}
{h_{1}^2 h_{2}^2},\\
&\mathcal{D}_6 = \frac{g^2\rho_{1}^3\,\rho_{2}^2\,\left(\rho_{2}-\rho_{1}\right)}{h_{1}^2\,h_{2}^2},
\end{align*}
which are all positive as $0<\rho_1<\rho_2$, thus the matrix $\mathcal{M}_{2}$ is positive-definite.

Now it suffices to show that if $\mathcal{M}_{\mathcal{N}}$ is positive-definite and its determinant is \eqref{eq:detM},
then the conclusion also holds for the $(\mathcal{N}+1)$-layer case.
The first $3\mathcal{N}$ leading principal minors of $\mathcal{M}_{\mathcal{N}+1}$ are also those of $\mathcal{M}_{\mathcal{N}}$, thus they are positive.
According to the matrix structure in Lemma \ref{lemma:Hessian}, one can calculate $\mathcal{D}_{3\mathcal{N}+1}$ by expanding along the $(3\mathcal{N}+1)$th column to obtain
\begin{align*}
    \mathcal{D}_{3\mathcal{N}+1} =&\  \frac{\rho _{\mathcal{N}+1}\,\left(u_{\mathcal{N}+1}^2+v_{\mathcal{N}+1}^2+g\,h_{\mathcal{N}+1}\right)}{h_{\mathcal{N}+1}} \det(\mathcal{M}_{\mathcal{N}}) \\
    & +\sum_{m=1}^{\mathcal{N}}(-1)^{3\mathcal{N}+3m-1}g\rho_m\det
    \left(\widetilde{\mathcal{\mathcal{N}}}[3m-2;3\mathcal{N}+1]
    \right),
\end{align*}
where $\widetilde{\mathcal{M}}[3m-2;3\mathcal{N}+1] \in \mathbb{R}^{3\mathcal{N}\times3\mathcal{N}}$ represents the submatrix obtained by removing the $(3m-2)$th row and $(3\mathcal{N}+1)$th column in the upper left $(3\mathcal{N}+1)\times(3\mathcal{N}+1)$ submatrix of $\mathcal{M}_{\mathcal{N}+1}$.
When $m<\mathcal{N}$, the last four rows in $\widetilde{\mathcal{M}}[3m-2;3\mathcal{N}+1]$ are
\begin{align*}
&\mathcal{I}_1= \left(
         g\rho_1,~
         0,~
         0,~
         g\rho_2,~
         0,~
         0,~
         \cdots,~
         g\rho_{\mathcal{N}-1},~ 
         0,~
         0,~
         \frac{\rho_\mathcal{N}(u_\mathcal{N}^2+v_\mathcal{N}^2+gh_\mathcal{N})}{h_\mathcal{N}},~  
         -\frac{\rho_\mathcal{N}u_\mathcal{N}}{h_\mathcal{N}},~
         -\frac{\rho_\mathcal{N}v_\mathcal{N}}{h_\mathcal{N}}
    \right),~
    \\
    &\mathcal{I}_2=\left(
         0,~
         0,~
         0,~
         0,~
         0,~
         0,~
         \cdots,~
         0,~
         0,~
         0,~
         -\frac{\rho_\mathcal{N}u_\mathcal{N}}{h_\mathcal{N}},~
         \frac{\rho_\mathcal{N}}{h_\mathcal{N}},~
        0
    \right),\\
       &\mathcal{I}_3=\left(
         0,~
         0,~
         0,~
         0,~
         0,~
         0,~
         \cdots,~
         0,~
         0,~
         0,~
         -\frac{\rho_\mathcal{N}v_\mathcal{N}}{h_\mathcal{N}} ,~
         0,~
        \frac{\rho_\mathcal{N}}{h_\mathcal{N}}
    \right),~\\
    &\mathcal{I}_4=\left(
        g\rho_1  ,~
         0,~
         0,~
         g\rho_2  ,~
         0,~
         0,~
         \cdots,~
         g\rho_{\mathcal{N}-1}  ,~
         0,~
         0,~
       g\rho_{\mathcal{N}},~
         0,~
        0
        \right),
\end{align*}
which are linear dependent as $\mathcal{I}_1 = \mathcal{I}_4-u_\mathcal{N}\mathcal{I}_2-v_\mathcal{N}\mathcal{I}_3$, thus $\det\left(\widetilde{\mathcal{M}}[3m-2;3\mathcal{N}+1]\right) = 0$ for $ m < \mathcal{N}$.
When $m=\mathcal{N}$, the last three rows in $\widetilde{\mathcal{M}}[3m-2;3\mathcal{N}+1]$ are
\begin{align*}
    &\left(
         0,~
         0,~
         0,~
         0,~
         0,~
         0,~
         \cdots,~
         0,~
         0,~
         0,~
         -\frac{\rho_\mathcal{N}u_\mathcal{N}}{h_\mathcal{N}}  ,~
         \frac{\rho_\mathcal{N}}{h_\mathcal{N}},~
        0
    \right),~\\
      &\left(
         0,~
         0,~
         0,~
         0,~
         0,~
         0,~
         \cdots,~
         0,~
         0,~
         0,~
         -\frac{\rho_\mathcal{N}v_\mathcal{N}}{h_\mathcal{N}}  ,~
         0,~
        \frac{\rho_\mathcal{N}}{h_\mathcal{N}}~
    \right),~\\
    &\left(
        g\rho_1  ,~
         0,~
         0,~
         g\rho_2  ,~
         0,~
         0,~
         \cdots,~
         g\rho_{\mathcal{N}-1}  ,~
         0,~
         0,~
       g\rho_{\mathcal{N}}  ,~
         0,~
        0
        \right).
\end{align*}
After some matrix operations, they become
\begin{align*}
&\left(
         g\rho_1,~
         0,~
         0,~
         g\rho_2,~
         0,~
         0,~
         \cdots,~
         g\rho_{\mathcal{N}-1},~ 
         0,~
         0,~
         \frac{\rho_\mathcal{N}(u_\mathcal{N}^2+v_\mathcal{N}^2+gh_\mathcal{N})}{h_\mathcal{N}},~  
         -\frac{\rho_\mathcal{N}u_\mathcal{N}}{h_\mathcal{N}},~
         -\frac{\rho_\mathcal{N}v_\mathcal{N}}{h_\mathcal{N}}
    \right),~
    \\
    &\left(
         0,~
         0,~
         0,~
         0,~
         0,~
         0,~
         \cdots,~
         0,~
         0,~
         0,~
         -\frac{\rho_\mathcal{N}u_\mathcal{N}}{h_\mathcal{N}},~
         \frac{\rho_\mathcal{N}}{h_\mathcal{N}},~
        0
    \right),\\
       &\left(
         0,~
         0,~
         0,~
         0,~
         0,~
         0,~
         \cdots,~
         0,~
         0,~
         0,~
         -\frac{\rho_\mathcal{N}v_\mathcal{N}}{h_\mathcal{N}} ,~
         0,~
        \frac{\rho_\mathcal{N}}{h_\mathcal{N}}
    \right),
    \end{align*}
which recover the last three rows in $\mathcal{M}_{\mathcal{N}}$,
and the determinant of $\widetilde{\mathcal{M}}[3m-2;3\mathcal{N}+1]$ remains the same,
so that
\begin{equation*}
g\rho_\mathcal{N} \det\left(\widetilde{\mathcal{M}}[3m-2;3\mathcal{N}+1]\right) =  g\rho_\mathcal{N}  \det\left(\mathcal{M}_\mathcal{N}\right),
\end{equation*}
and
\begin{align*}
    &\mathcal{D}_{3\mathcal{N}+1} = \frac{\rho _{\mathcal{N}+1}\,\left(u_{\mathcal{N}+1}^2+v_{\mathcal{N}+1}^2+gh_{\mathcal{N}+1}\right)}{h_{\mathcal{N}+1}} \det(\mathcal{M}_{\mathcal{N}})-g\rho_\mathcal{N}\det(\mathcal{M}_{\mathcal{N}}).
    \end{align*}
It is straightforward to obtain the expressions of the last two leading principal minors,
    \begin{align*}
    &\mathcal{D}_{3\mathcal{N}+2} = \frac{\rho_{\mathcal{N}+1}}{h_{\mathcal{N}+1}} \mathcal{D}_{3\mathcal{N}+1}
    -\frac{\rho_{\mathcal{N}+1}^2u_{\mathcal{N}+1}^2}{h_{\mathcal{N}+1}^2}\det(\mathcal{M}_{\mathcal{N}}),\\
    &\mathcal{D}_{3\mathcal{N}+3} = \frac{\rho _{\mathcal{N}+1}}{h_{\mathcal{N}+1}} \mathcal{D}_{3\mathcal{N}+2}
    -\frac{\rho_{\mathcal{N}+1}^3v_{\mathcal{N}+1}^2}{h_{\mathcal{N}+1}^3}
    \det(\mathcal{M}_{\mathcal{N}}).
\end{align*}
Using \eqref{eq:detM} yields
\begin{align*}
    \mathcal{D}_{3\mathcal{N}+1} = &\frac{\left(\rho_{\mathcal{N}+1}(u_{\mathcal{N}+1}^2+v_{\mathcal{N}+1}^2)+gh_{\mathcal{N}+1}(\rho_{\mathcal{N}+1}-\rho_{\mathcal{N}})\right)
    g^{\mathcal{N}}\rho_1^3\prod_{m=2}^{\mathcal{N}}\rho_m^2(\rho_m-\rho_{m-1})}
    {h_{\mathcal{N}+1}\prod_{m=1}^{\mathcal{N}}h_m^2},\\
    \mathcal{D}_{3\mathcal{N}+2} = &\frac{\rho_{\mathcal{N}+1}(\rho_{\mathcal{N}+1}v_{\mathcal{N}+1}^2+gh_{\mathcal{N}+1}(\rho_{\mathcal{N}+1}-\rho_{\mathcal{N}}))g^{\mathcal{N}}\rho_1^3\prod_{m=2}^{\mathcal{N}}\rho_m^2(\rho_m-\rho_{m-1})}
    {\prod_{m=1}^{\mathcal{N}+1}h_m^2},\\
    \mathcal{D}_{3\mathcal{N}+3} = &\frac{g^{\mathcal{N}+1}\rho_1^3\prod_{m=2}^{\mathcal{N}+1}\rho_m^2(\rho_m-\rho_{m-1})}{\prod_{m=1}^{\mathcal{N}+1}h_m^2},
\end{align*}
which are all positive, and \eqref{eq:detM} also holds for the $(\mathcal{N}+1)$-layer case when $\mathcal{N} $ is replaced by $\mathcal{N}+1$.
\end{proof}

\begin{remark}\rm\label{prop:eta_convex_prove_Qua}
One can also prove the convexity of the energy function by finding the recurrence relations of the quadratic forms of the Hessian matrix, with more details in \ref{Sec:Proof_of_Remark2.1}.
\end{remark}

Based on the convexity of $\eta$ and the consistent condition \eqref{eq:EntropyPairConsistent}, the following energy inequality is satisfied
\begin{equation*}
    \frac{\partial {\eta}({\bm{U}})}{\partial t}+\sum_{\ell=1}^{2} \frac{\partial {q}_{\ell}({\bm{U}})}{\partial x_{\ell}} \leqslant 0,
\end{equation*}
where the equality holds only when the solutions are smooth and the inequality holds in the distribution sense.
For clarity of notations, define ${\bV}:=\left(\bV_1,\dots,\bV_M\right)^{\mathrm{T}} = (\partial{\eta}/\partial{\bm{U}})^{\mathrm{T}}$ with
\begin{equation}\label{eq:Entropy_Variables_ML}
	\bm{V}_{m} = \left(
		g\rho_m\left(h_m+z_m\right) - \dfrac{\rho_m}{2}\left(u_m^2 + v_m^2\right),
		~\rho_mu_m,
		~\rho_mv_m\right).
\end{equation}
Several auxiliary variables are introduced,
\begin{equation}\label{eq:Entropy_Potential_ML_2D}
		\begin{aligned}
		&\phi = \bm{{V}}^{T}\bm{{U}} - {\eta} =
		\sum\limits_{m=1}^{M}\dfrac{g\rho_{m}h_m}{2}\left(h_m+2\sum_{k> m}h_{k}\right),\\
		&\psi_1 = \bm{{V}}^{T}\bm{{F}}_1 - {q}_1 = \sum\limits_{m=1}^{M}\dfrac{g}{2}\rho_mh_m^2u_m,\\
		&\psi_2 = \bm{{V}}^{T}\bm{{F}}_2 - {q}_2 = \sum\limits_{m=1}^{M}\dfrac{g}{2}\rho_mh_m^2v_m,
	\end{aligned}
\end{equation}
which play important roles in the construction of so-called two-point EC flux in Section \ref{section:2DHOScheme}.

\begin{remark}\rm
    When $M=2$, the system
    \eqref{eq:ML-SWEs-2D} degenerates to the two-layer shallow water equations 
\begin{equation}\label{eq:TL-SWEs-2D}
	\left\{
	\begin{aligned}
		&~\pd{h_1}{t}+\pd{}{x_1}\left(h_1u_1\right)+\pd{}{x_2}\left(h_1v_1\right)=0, 
		\\
		&\pd{}{t}\left(h_1 u_1\right)+\pd{}{x_1}\left(h_1 u_1^2+\frac{1}{2} g h_1^2\right) + \pd{}{x_2}\left(h_1 u_1 v_1\right)=-g h_1 \pd{h_2}{x_1}-g h_1 \pd{b}{x_1}, 
		\\
		&\pd{}{t}\left(h_1 v_1\right)+\pd{}{x_1}\left(h_1 u_1 v_1\right) + \pd{}{x_2}\left(h_1 v_1^2+\frac{1}{2} g h_1^2\right)=-g h_1 \pd{h_2}{x_2}-g h_1 \pd{b}{x_2}, 
		\\
		&~\pd{h_2}{t}+\pd{}{x_1}\left(h_2u_2\right)+\pd{}{x_2}\left(h_2v_2\right)=0, 
		\\
		&\pd{}{t}\left(h_2 u_2\right)+\pd{}{x_1}\left(h_2 u_2^2+\frac{1}{2} g h_2^2\right) + \pd{}{x_2}\left(h_2 u_2 v_2\right)=-g r_{12} h_2 \pd{h_1}{x_1}-g h_2 \pd{b}{x_1}, 
		\\
		&\pd{}{t}\left(h_2 v_2\right)+\pd{}{x_1}\left(h_2 u_2 v_2\right) + \pd{}{x_2}\left(h_2 v_2^2+\frac{1}{2} g h_2^2\right)=
  -g r_{12} h_2 \pd{h_1}{x_2}-g h_2 \pd{b}{x_2},
	\end{aligned}
	\right.
\end{equation}
with $r_{12} = \rho_{1}/\rho_{2}$.
The energy and energy flux in \eqref{eq:Entropy_Pair_ML_2D} reduce to
 \begin{equation*}
		\begin{aligned}
			{\eta}({\bm{U}})&= \sum\limits_{m=1}^{2}\dfrac{\rho_m}{2}
			\left(h_m u_m^2 + h_m v_m^2 + g h_m^2 \right)  
			+  \rho_{1} g h_{1} h_{2}  
			+ \sum\limits_{m=1}^{2} \rho_m g h_m b,\\
			{q}_1({\bm{U}})&=
			\sum\limits_{m=1}^{2}\dfrac{\rho_m}{2} 
			u_m\left(h_m u_m^2 + h_m v_m^2\right) + g \sum\limits_{m=1}^{2}\rho_m u_mh_m^2+g\sum\limits_{m=1}^{2}\rho_m u_m h_m  z_m ,\\
			{q}_2({\bm{U}})&=\sum\limits_{m=1}^{2}\dfrac{\rho_m}{2} v_m\left(h_m u_m^2 
			+ h_m v_m^2\right)+
   g \sum\limits_{m=1}^{2}\rho_m v_mh_m^2+g\sum\limits_{m=1}^{2}\rho_m v_m h_m  z_m ,
		\end{aligned}
	\end{equation*}
which are similar to those 1D versions in \cite{Bouchut2008An}.
\end{remark}

\section{Numerical schemes on fixed meshes}\label{section:2DHOScheme}
This section constructs high-order accurate WB ES schemes for the 2D ML-SWEs \eqref{eq:SWE0} on fixed meshes, extending our previous work in \cite{Duan2021_High} for the shallow water magnetohydrodynamical equations, which reduce to the single-layer SWEs when the magnetic fields are zero.
The construction of the schemes on fixed meshes is the foundation of those on adaptive moving meshes in Section \ref{Sec:MM_Case}, and is not available in the literature. 

Denote the physical domain as $[a_1,b_1]\times[a_2,b_2]$, and corresponding uniform computational mesh
$(x_1)_i = a_1 + i\Delta x_1,~i=0,1,\cdots,N_1-1,~\Delta x_1 = (b_1-a_1)/(N_1-1)$,
$(x_2)_j = a_2 + j\Delta x_2,~j=0,1,\cdots,N_2-1,~\Delta x_2 = (b_2-a_2)/(N_2-1)$.
Consider the following semi-discrete conservative finite difference schemes for \eqref{eq:SWE0}
\begin{align}
  \dfrac{\mathrm{d}}{\mathrm{d} t} \bm{{U}}_{i, j}=
  & -\dfrac{1}{\Delta x_1}
  \left(\left(\bm{\widetilde{{F}}}_{1}\right)_{i+\frac{1}{2}, j}
  -\left(\bm{\widetilde{{F}}}_{1}\right)_{i-\frac{1}{2}, j}\right)
  -\frac{1}{\Delta x_2}
  \left(\left(\bm{\widetilde{{F}}}_{2}\right)_{i, j
  	+\frac{1}{2}}-
  \left(\bm{\widetilde{{F}}}_{2}\right)_{i, j-\frac{1}{2}}
  \right)- \left(\bm{S}_1\right)_{i,j} - \left(\bm{S}_2\right)_{i,j},\label{eq:2D_semi_Discrete}
\end{align}
where $\left(\bm{\widetilde{{F}}}_{1}\right)_{i\pm\frac{1}{2}, j}$ and $\left(\bm{\widetilde{{F}}}_{2}\right)_{i, j\pm\frac{1}{2}}$ are the $x_1$- and $x_2$-directional numerical fluxes, $\left(\bm{S}_1\right)_{i,j}$ and $\left(\bm{S}_2\right)_{i,j}$ are suitable discretizations of the source terms.

\begin{definition}\rm\label{def:ECES_condition}
    The semi-discrete scheme \eqref{eq:2D_semi_Discrete} or its  flux is ES, if its solution satisfies the semi-discrete numerical energy inequality 
     \begin{equation*}
	\frac{\mathrm{d}}
 {\mathrm{d} t}\eta_{i, j}+\frac{1}{\Delta x_{1}}
 \left(\left({\widetilde{{{Q}}}}_1\right)_{i+\frac{1}{2},j}
 -\left({\widetilde{{{Q}}}}_1\right)_{i-\frac{1}{2},j}\right)
 +\frac{1}{\Delta x_{2}}
 \left(\left({\widetilde{{{Q}}}}_2\right)_{i,j+\frac{1}{2}}
 -\left({\widetilde{{{Q}}}}_2\right)_{i,j-\frac{1}{2}}
 \right) \leqslant 0,
	\end{equation*}
 where $\left({\widetilde{{{Q}}}}_1\right)_{i\pm\frac{1}{2},j}$ and $\left({\widetilde{{{Q}}}}_2\right)_{i,j\pm\frac{1}{2}}$ are numerical energy fluxes consistent with $q_1$ and $q_2$, respectively.
 If the equality holds, the semi-discrete scheme or its flux is called EC. 
\end{definition}

\begin{definition}\rm\label{def:WB_condition}
   The fully-discrete scheme for the ML-SWEs \eqref{eq:SWE0} obtained by the semi-discrete scheme \eqref{eq:2D_semi_Discrete} using forward Euler or SSP-RK3 time discretization is WB, if the lake at rest state is preserved,
   i.e., if the numerical solution satisfies the lake at rest at the initial time
  \begin{equation*}
    \left(h_m + z_m\right)_{i,j}^{0} \equiv C_m,
    ~(u_m)_{i,j}^{0} = (v_m)_{i,j}^{0} = 0,~ m =1,\cdots,M,
    ~\forall i,j,
  \end{equation*}
  then it also satisfies the lake at rest at $t^n$,
  \begin{equation}\label{eq:2pnd_WB_tn}
      \left(h_m + z_m\right)_{i,j}^{n} \equiv C_m,
    ~(u_m)_{i,j}^{n}=(v_m)_{i,j}^{n} = 0,~ m =1,\cdots,M,
    ~\forall i,j,
  \end{equation}
  where $C_m$ is constant.
  In such cases, the corresponding numerical flux is also called WB.
\end{definition}

In this paper, high-order WB ES schemes are obtained by adding suitable dissipation terms to high-order WB EC fluxes, and the first step is to construct the so-called two-point EC flux.

\subsection{Construction of a two-point EC flux}
In pursuit of the WB property, a sufficient condition for the two-point EC fluxes is proposed, which can be decoupled into $M$ identities individually for each layer, and makes it convenient to construct the two-point EC flux.

\begin{proposition}\rm\label{def:sufficient_condition_ML_2D}
In the semi-discrete schemes \eqref{eq:2D_semi_Discrete} with the second-order central difference discretization for the first-order derivatives in the source terms,
 if consistent two-point numerical fluxes $\widetilde{\bm{{F}}}_1\left(\bm{U}_{L},\bm{U}_{R}\right)$ and  $\widetilde{\bm{{F}}}_2\left(\bm{U}_{L},\bm{U}_{R}\right)$ satisfy
	\begin{align}
		\left(\bm{{V}}\left(\bU_{R}\right)-\bm{{V}}\left(\bU_{L}\right)\right)^{\mathrm{T}}\widetilde{\bm{{F}}}_1
		= \ &\left(\psi_1\left(\bU_R\right)-\psi_1\left(\bU_L\right)\right)\nonumber
		 + \sum\limits_{m=1}^{M}
		\rho_m 
		g\left[ \left(h_m z_m u_m\right)_R  -  
            \left(h_m z_m u_m\right)_L\right] \nonumber\\
		& - \sum\limits_{m=1}^{M}
		 \dfrac{\rho_m}{2} g\left[ 
		 \left(h_m u_m\right)_R 
		 -  \left(h_m u_m\right)_L\right]
		 \left(\left(z_m\right)_{L}+\left(z_m\right)_{R}\right),
		\label{eq:EC_condition_2D_ML}
\\
		\left(\bm{{V}}\left(\bU_{R}\right)-\bm{{V}}\left(\bU_{L}\right)\right)^{\mathrm{T}}\widetilde{\bm{{F}}}_2
		= \ &\left(\psi_2\left(\bU_R\right)-\psi_2\left(\bU_L\right)\right)\nonumber
		 + \sum\limits_{m=1}^{M}
		\rho_m 
		g\left[ \left(h_m z_m v_m\right)_R  -  
            \left(h_m z_m v_m\right)_L\right] \nonumber\\
		& - \sum\limits_{m=1}^{M}
		 \dfrac{\rho_m}{2} g\left[ 
		 \left(h_m v_m\right)_R 
		 -  \left(h_m v_m\right)_L\right]
		 \left(\left(z_m\right)_{L}+\left(z_m\right)_{R}\right),\nonumber
	\end{align}
	 with $\psi_{\ell}$ given in \eqref{eq:Entropy_Potential_ML_2D}, they are two-point EC fluxes.
\end{proposition}

\begin{proof}
    The proof is omitted here as it can be obtained by reducing the proof of Proposition \ref{prop:2D_HO_sufficient_condition} to the two-point case and without considering the dissipation terms.
\end{proof}
Such a condition involves the compatible discretizations of the source terms, which is important for preserving the WB property of the EC schemes, similar to those in \cite{Duan2021_High,Zhang2023High}, but considering the coupling terms between the multiple layers.

\begin{proposition}\rm\label{Prop:Two-point_EC}
For the 2D ML-SWEs \eqref{eq:SWE0}, two-point EC fluxes $\bm{\widetilde{F}}_1 = \left((\widetilde{\bm{F}}_1)_{1}^{\mathrm{T}},\dots
(\widetilde{\bm{F}}_1)_{M}^{\mathrm{T}}\right)^{\mathrm{T}}$ and  $\bm{\widetilde{F}}_2 = \left((\widetilde{\bm{F}}_2)_{1}^{\mathrm{T}},\dots,(\widetilde{\bm{F}}_2)_{M}^{\mathrm{T}}\right)^{\mathrm{T}}$ can be chosen as
	\begin{equation*}
		\left(\widetilde{\bm{F}}_{1}\right)_{m}=\left(\begin{array}{c}
			\mean{h_m}\mean{u_m}\\
			\mean{h_m}\mean{u_m}^2 + \frac{g}{2}\mean{h_m^2}+g\left(\mean{h_m z_m} - \mean{h_m}\mean{z_m}\right) \\
			\mean{h_m}\mean{u_m}\mean{v_m} \\
		\end{array}\right),
	\end{equation*}
	\begin{equation*}
 \left(\widetilde{\bm{F}}_{2}\right)_{m}=\left(\begin{array}{c}
		\mean{h_m}\mean{v_m}\\
		\mean{h_m}\mean{u_m}\mean{v_m} \\
		\mean{h_m}\mean{v_m}^2 + \frac{g}{2}\mean{h_m^2}+g\left(\mean{h_m z_m} - \mean{h_m}\mean{z_m}\right)
	\end{array}\right),
\end{equation*}
where $m=1,\dots,M$, $\mean{a}= (a_{L}+a_{R})/2$
and $\jump{a} = (a_{R}-a_{L})$
denote the average and jump of $a$, respectively.
\end{proposition}

\begin{proof}
    The sufficient condition \eqref{eq:EC_condition_2D_ML} can be decoupled into $M$ identities individually for each layer as 
\begin{align}
    \left(
    \bm{V}_{m}
    (\bm{U}_R) 
    - \bm{V}_{m}
    (\bm{U}_L) 
    \right)
    \left(
    \bm{\widetilde{F}}_1
    \right)_m
    = ~
    &
    (\psi_1)_m(\bm{U}_{R}) 
    - 
    (\psi_1)_m(\bm{U}_L)
    +
    \rho_m 
    g
    \left[
    \left(h_m z_m u_m\right)_R  
    -  
    \left(h_m z_m u_m\right)_L\right]
    \nonumber\\
    &+
    \dfrac{\rho_m}{2} 
    g\left[ 
    \left(h_m u_m\right)_R 
-  \left(h_m u_m\right)_L
\right]
\left(
\left(z_m\right)_{L}
+\left(z_m\right)_{R}
\right),
\label{eq:EC_condition_2D_ML_each_layer}
\end{align}
with $m=1,\cdots, M$ and $(\psi_1)_m = \dfrac{g}{2}\rho_mh_m^2u_m$.
One can derive a two-point flux for each layer, and concatenate $M$ two-point fluxes to form a two-point EC flux for the multi-layer system.
To be specific, denote the three components of $\bV_{m}$ in \eqref{eq:Entropy_Variables_ML} as
  \begin{align*}
  		\bV_{m,1}:= g\rho_m\left(h_m+z_m\right) - \dfrac{\rho_m}{2}\left(u_m^2 + v_m^2\right),~
  		\bV_{m,2}:= \rho_mu_m,~
  		\bV_{m,3}:= \rho_mv_m,
  \end{align*}
  and one can utilize $\jump{ab} = \jump{a}\mean{b}+\mean{a}\jump{b}$ to obtain the following expansions
   \begin{align*}
  	\jump{\bV_{m,1}}= &\ \jump{g\rho_m\left(h_m+z_m\right) - \dfrac{\rho_m}{2}\left(u_m^2 + v_m^2\right)}\\
  = &\  g\rho_{m}\jump{h_m}+g\rho_{m}\jump{z_m} - \rho_m\mean{u_m}\jump{u_m} - \rho_m\mean{v_m}\jump{v_m},\\
  	\jump{\bV_{m,2}} = &\ \rho_m\jump{u_m},\\
  	\jump{\bV_{m,3}} = &\ \rho_m\jump{v_m}.
  \end{align*}
The jump of $(\psi_1)_m$ is
\begin{equation*}
	\jump{(\psi_1)_m} = \dfrac{g}{2}\rho_m\left(\mean{h_m^{2}}\jump{u_m}+2\mean{h_m}\mean{u_m}\jump{h_m}\right),
\end{equation*}
and one also has
\begin{align*}
 \rho_m 
	g\jump{ h_m
	z_m
	u_m} 
	-
	\rho_m g\jump{ 
	h_mu_m}
	\mean{z_m}
 =&~ \rho_m g\mean{h_m
 	z_m}\jump{u_m}-\rho_m g\mean{h_m}\mean{z_m}\jump{u_m}\\
 &+\rho_m g\mean{h_m}\mean{u_m}\jump{z_m}.
\end{align*}
Substituting the above expansions into \eqref{eq:EC_condition_2D_ML_each_layer} and equating the coefficients of the corresponding jump terms on both sides results in
\begin{equation}\label{eq:Fix_Flux_Derive}
  \begin{aligned}
  	{\left(\bm{\widetilde{F}}_1\right)}_{m,1}\rho_m\jump{h_m}  =&\ \rho_m \mean{h_m}\mean{u_m}\jump{h_m},
  	\\{\left(\bm{\widetilde{F}}_1\right)}_{m,1}\rho_m\jump{z_m} =&\ \rho_m\mean{h_m}\mean{u_m}\jump{z_m},
  	\\
   {\left(\bm{\widetilde{F}}_1\right)}_{m,2}\rho_m\jump{u_m}-{\left(\bm{\widetilde{F}}_1\right)}_{m,1}\rho_m\mean{u_m}\jump{u_m} 
  	=&\ \dfrac{g}{2}\rho_m\mean{h_m^{2}}\jump{u_m}+\rho_mg\mean{h_m
  		z_m}\jump{u_m} \\
    &\ -\rho_mg\mean{h_m}\mean{z_m}\jump{u_m},\\
    {\left(\bm{\widetilde{F}}_1\right)}_{m,3}\rho_m\jump{v_m}-{\left(\bm{\widetilde{F}}_1\right)}_{m,1}\rho_m\mean{v_m}\jump{v_m} =&\ 0,
  \end{aligned}
  \end{equation}
where $\left(\widetilde{\bm{F}}_1\right)_{m,1}$, $\left(\widetilde{\bm{F}}_1\right)_{m,2}$, and $\left(\widetilde{\bm{F}}_1\right)_{m,3}$ represent the three components of $\left(\widetilde{\bm{F}}_1\right)_{m}$.
Solving the linear system \eqref{eq:Fix_Flux_Derive} gives the explicit expressions of $\left(\widetilde{\bm{F}}_1\right)_{m}$ as follows
  	\begin{equation*}
  	\left(\widetilde{\bm{F}}_1\right)_{m}=\left(\begin{array}{c}
  		\mean{h_m}\mean{u_m}\\
  		\mean{h_m}\mean{u_m}^2 + \frac{g}{2}\mean{h_m^2}+g\left(\mean{h_m z_m} - \mean{h_m}\mean{z_m}\right) \\
  		\mean{h_m}\mean{u_m}\mean{v_m} \\
  	\end{array}\right).
  \end{equation*}
For the $x_2$-direction, $\widetilde{\bm{F}}_2$ can be derived similarly and it is straightforward to verify the consistency of $\widetilde{\bm{F}}_\ell$ with $\bm{F}_\ell$.
\end{proof}
\begin{remark}\rm
    When $M=2$, the two-point EC fluxes $\bm{\widetilde{F}}_1$ and $\bm{\widetilde{F}}_2$ in Proposition \ref{Prop:Two-point_EC} reduce to
	\begin{equation*}
		\widetilde{\bm{F}}_1=\left(\begin{array}{c}
			\mean{h_1}\mean{u_1} \\
			\mean{h_1}\mean{u_1}^2 + \frac{g}{2}\mean{h_1^2}+g\left(\mean{h_1(h_2+b)} - \mean{h_1}\mean{h_2+b}\right) \\
			\mean{h_1}\mean{u_1}\mean{v_1} \\
			\mean{h_2}\mean{u_2}\\
			\mean{h_2}\mean{u_2}^2 + \frac{g}{2}\mean{h_2^2}+g\left(\mean{h_2(rh_1+b)} - \mean{h_2}\mean{rh_1+b}\right) \\
			\mean{h_2}\mean{u_2}\mean{v_2} \\
		\end{array}\right),
	\end{equation*}
	\begin{equation*}
		\widetilde{\bm{F}}_2=\left(\begin{array}{c}
			\mean{h_1}\mean{v_1} \\
			\mean{h_1}\mean{u_1}\mean{v_1} \\
			\mean{h_1}\mean{v_1}^2 + \frac{g}{2}\mean{h_1^2}+g\left(\mean{h_1(h_2+b)} - \mean{h_1}\mean{h_2+b}\right) \\
			\mean{h_2}\mean{v_2}\\
			\mean{h_2}\mean{u_2}\mean{v_2} \\
			\mean{h_2}\mean{v_2}^2 + \frac{g}{2}\mean{h_2^2}+g\left(\mean{h_2(rh_1+b)} - \mean{h_2}\mean{rh_1+b}\right) \\
		\end{array}\right).
	\end{equation*}
    When $M=1$, the two-point fluxes recover those in \cite{Zhang2023High}, or \cite{Duan2021_High} with zero magnetic fields.
\end{remark}

\subsection{High-order WB ES schemes}
Our high-order WB ES schemes can be constructed by using high-order WB ES fluxes and high-order source term discretizations in \eqref{eq:2D_semi_Discrete}, to be specific
\begin{align}
  \dfrac{\mathrm{d}}{\mathrm{d} t} \bm{{U}}_{i, j}=
  & -\dfrac{1}{\Delta x_1}
  \left(\left(\bm{\widetilde{{F}}}_{1}\right)_{i+\frac{1}{2}, j}^{\mathtt{ES}}
  -\left(\bm{\widetilde{{F}}}_{1}\right)_{i-\frac{1}{2}, j}^{\mathtt{ES}}\right)
  -\frac{1}{\Delta x_2}
  \left(\left(\bm{\widetilde{{F}}}_{2}\right)_{i, j
  	+\frac{1}{2}}^{\mathtt{ES}}-
  \left(\bm{\widetilde{{F}}}_{2}\right)_{i, j-\frac{1}{2}}^{\mathtt{ES}}
  \right)\nonumber
  \\
  &- \sum\limits_{m=1}^{M}
  \dfrac{g(h_m)_{i,j}}
  {\Delta x_1}
  \left(\left(\bm{\widetilde{{B}}}_{1,m}\right)_{i+\frac{1}{2}, j}^{{\mathtt{2pth}}}
  -\left(\bm{\widetilde{{B}}}_{1,m}\right)_{i-\frac{1}{2}, j}^{{\mathtt{2pth}}}\right)
 	\nonumber\\
   &- \sum\limits_{m=1}^{M}\dfrac{g(h_m)_{i,j}}
   {\Delta x_2}
   \left(
   \left(\bm{\widetilde{{B}}}_{2,m}\right)_{i, j
   	+\frac{1}{2}}^{{\mathtt{2pth}}}-
   \left(\bm{\widetilde{{B}}}_{2,m}\right)_{i, j-\frac{1}{2}}^{{\mathtt{2pth}}}\right).\label{eq:2D_HighOrder_EC_Discrete}
\end{align}
Here the high-order WB ES fluxes are
\begin{equation}\label{eq:HighOrder_EC_Flux}
  \begin{aligned}
    \left(\bm{\widetilde{{F}}}_{1}\right)_{i+\frac{1}{2}, j}^{\mathtt{ES}}=\left(\bm{\widetilde{{F}}}_{1}\right)_{i+\frac{1}{2}, j}^{\mathtt{2pth}} -
\bm{D}_{i+\frac{1}{2},j},
    \\
    \left(\bm{\widetilde{{F}}}_{2}\right)_{i, j+\frac{1}{2}}^{\mathtt{ES}}=\left(\bm{\widetilde{{F}}}_{2}\right)_{i, j+\frac{1}{2}}^{\mathtt{2pth}}-
   \bm{D}_{i,j+\frac{1}{2}},
  \end{aligned}
\end{equation}
where 
\begin{equation}
    \begin{aligned}
    \left(\bm{\widetilde{{F}}}_{1}\right)_{i+\frac{1}{2}, j}^{\mathtt{2pth}}=\sum_{q=1}^{p} \alpha_{p,q} \sum_{s=0}^{q-1} 
    \widetilde{\bm{F}}_{1}\left(\bm{U}_{i-s, j}, \bm{U}_{i-s+q, j}\right) ,
    \\
    \left(\bm{\widetilde{{F}}}_{2}\right)_{i, j+\frac{1}{2}}^{\mathtt{2pth}}=\sum_{q=1}^{p} \alpha_{p,q} \sum_{s=0}^{q-1}  
    \widetilde{\bm{F}}_{2}\left(\bm{U}_{i, j-s}, \bm{U}_{i, j-s+q}\right),
  \end{aligned}
\end{equation}
are the $2p$th-order WB EC fluxes based on the two-point EC fluxes in Proposition \ref{Prop:Two-point_EC}, and the fluxes for the source terms are
\begin{equation}\label{eq:HighOrder_EC_Flux_B}
  \begin{aligned}
    \left(\bm{\widetilde{{B}}}_{1,m}\right)_{i+\frac{1}{2},j}^{{\mathtt{2pth}}}=\sum_{q=1}^{p} \alpha_{p,q} \sum_{s=0}^{q-1} &\frac{1}{2}\left(\left(\bm{B}_{1,m}\right)_{i-s,j}+\left(\bm{B}_{1,m}\right)_{i-s+q,j}\right),
    \\
    \left(\bm{\widetilde{{B}}}_{2,m}\right)_{i, j+\frac{1}{2}}^{{\mathtt{2pth}}}=\sum_{q=1}^{p} \alpha_{p,q} \sum_{s=0}^{q-1} &\frac{1}{2}\left(\left(\bm{B}_{2,m}\right)_{i, j-s}+\left(\bm{B}_{2,m}\right)_{i, j-s+q}\right).\\
  \end{aligned}
\end{equation}
The dissipation terms $\bm{D}_{i+\frac{1}{2},j}$ and $\bm{D}_{i,j+\frac{1}{2}}$ in \eqref{eq:HighOrder_EC_Flux} will be detailed in Section \ref{section:ES}. The coefficients $\alpha_{p,q}$ satisfy 
\begin{equation*}
  \sum_{q=1}^{p} q \alpha_{p,q} = 1,~\sum_{q=1}^{p} q^{2s-1}\alpha_{p,q} = 0,~s = 2,\dots,p,
\end{equation*}
and in the numerical experiments, the case of $p=3$ is used, they are
\begin{equation*}
    \alpha_{3,1} = \dfrac{3}{2},~\alpha_{3,2}=-\frac{3}{10},~\alpha_{3,3}=\frac{1}{30}.
\end{equation*}
Notice that the numerical fluxes for the flux gradient are linear combinations of the two-point case \cite{Lefloch2002Fully}, and those for the source terms are compatible so that one can prove the semi-discrete stability conditions \cite{Duan2021_High,Wu2020Entropy}.

\subsection{Dissipation terms on fixed meshes}\label{section:ES}
Considering the presence of discontinuities in the solution, the WB ES schemes are developed based on adding suitable dissipation terms to the WB EC fluxes to suppress oscillations.
In pursuit of the WB and ES properties simultaneously, one needs to deal with the dissipation terms carefully.
The construction of the dissipation terms is based on the jumps of $\bm{V}$, which are constant for the lake at rest.
Take the $x_1$-direction as an example, the dissipation terms can be chosen as
\begin{align}
\bm{D}_{i+\frac{1}{2},j}
  = 
\frac{1}{2}{\bm{S}}_{i+\frac{1}{2},j}{\bm{Y}}_{{i+\frac{1}{2},j} }\jumpangle{\bm{\widetilde{V}}}_{i+\frac{1}{2},j}^{\mathtt{WENO}}, \label{eq:diss1}
\end{align}
where
\begin{equation}\label{eq:diss1_S1}
  {\bm{S}}_{i+\frac{1}{2},j} = \left(\alpha_1\right)_{i+\frac12,j}\left(\bm{R}( {\bm{U}})\right)_{i+\frac{1}{2},j}.
\end{equation}
A suitable amount of dissipation depends on the choice of the maximal wave speed $\left(\alpha_1\right)_{i+\frac12,j}$ in \eqref{eq:diss1_S1} evaluated at  $\left(\left(x_1\right)_{i+\frac12},\left(x_2\right)_{j}\right)$.
Due to the lack of an analytical expression for the eigenstructure, an estimation will be used in the numerical experiments for the two- and three-layer cases, see Section \ref{subsec:Upper/Lower Bound}.
The matrix $\bm{R}$ satisfies the Cholesky decomposition
\begin{equation*}
  \pd{{\bm{U}}}{\bm{{V}}}=\bm{R} \bm{R}^{\mathrm{T}},
\end{equation*}
based on the convexity of ${\eta}$ w.r.t ${\bU}$,
with $\bm{{U}}$ and $\bm{{V}}$ defined in Section \ref{section:EntropyCondition}.
For $M=2$, the specific expression of $\bm{R}$ is
\begin{equation}\label{eq:R_2L}
  \bm{R} = 
   \left(\begin{array}{cccccc} \sqrt{\frac{\rho_2}{g(\rho_2-\rho_1)\rho_1}} & 0 & 0 & 0 & 0 & 0
   	\\ 
   	\sqrt{\frac{\rho_2}{g(\rho_2-\rho_1)\rho_1}}u_1 & \sqrt{\frac{h_{1}}{\rho _{1}}} & 0 & 0 & 0 & 0
   	\\
   	 \sqrt{\frac{\rho_2}{g(\rho_2-\rho_1)\rho_1}}v_1 & 0 & \sqrt{\frac{h_{1}}{\rho _{1}}} & 0 & 0 & 0
   	 \\ 
   	-\sqrt{\frac{\rho_1}{g(\rho_2-\rho_1)\rho_2}} & 0 & 0 & \frac{1}{\sqrt{g\rho _{2}}} & 0 & 0
   	\\
   		-\sqrt{\frac{\rho_1}{g(\rho_2-\rho_1)\rho_2}}u_2 & 0 & 0 & \frac{u_{2}}{\sqrt{g\rho _{2}}} & \sqrt{\frac{h_{2}}{\rho _{2}}} & 0
   		\\ 
   		-\sqrt{\frac{\rho_1}{g(\rho_2-\rho_1)\rho_2}}v_2 & 0 & 0 & \frac{v_{2}}{\sqrt{g\rho _{2}}} & 0 & \sqrt{\frac{h_{2}}{\rho _{2}}} \end{array}\right),
\end{equation}
while for $M=3$,
\begin{footnotesize}
\begin{equation}\label{eq:R_3L}
	\bm{R} = \left(\begin{array}{ccccccccc} \sqrt{\frac{{\rho _{2}}}{g\left(\rho _{2}-\rho _{1}\right)\rho _{1}}} & 0 & 0 & 0 & 0 & 0 & 0 & 0 & 0\\ \sqrt{\frac{{\rho _{2}}}{g\left(\rho _{2}-\rho _{1}\right)\rho _{1}}} u_1 & \sqrt{\frac{h_{1}}{\rho _{1}}} & 0 & 0 & 0 & 0 & 0 & 0 & 0\\  \sqrt{\frac{{\rho _{2}}}{g\left(\rho _{2}-\rho _{1}\right)\rho _{1}}} v_1 & 0 & \sqrt{\frac{h_{1}}{\rho _{1}}} & 0 & 0 & 0 & 0 & 0 & 0\\ -\sqrt{\frac{\rho_1}{g(\rho_2-\rho_1)\rho_2}} & 0 & 0 & \sqrt{\frac{{\rho _{3}}}{{g}\left(\rho _{3}-\rho _{2}\right){\rho _{2}}}} & 0 & 0 & 0 & 0 & 0\\ -\sqrt{\frac{\rho_1}{g(\rho_2-\rho_1)\rho_2}} u_2 & 0 & 0 & \sqrt{\frac{{\rho _{3}}}{{g}\left(\rho _{3}-\rho _{2}\right){\rho _{2}}}} u_2 & \sqrt{\frac{h_{2}}{\rho _{2}}} & 0 & 0 & 0 & 0\\ 
 -\sqrt{\frac{\rho_1}{g(\rho_2-\rho_1)\rho_2}} v_2 & 0 & 0 & \sqrt{\frac{{\rho _{3}}}{{g}\left(\rho _{3}-\rho _{2}\right){\rho _{2}}}} v_2 & 0 & \sqrt{\frac{h_{2}}{\rho _{2}}} & 0 & 0 & 0\\ 0 & 0 & 0 & -\sqrt{\frac{{\rho _{2}}}{{g}\left(\rho _{3}-\rho _{2}\right){\rho _{3}}}} & 0 & 0 & \frac{1}{\sqrt{g \rho _{3}}} & 0 & 0\\ 0 & 0 & 0 & -\sqrt{\frac{{\rho _{2}}}{{g}\left(\rho _{3}-\rho _{2}\right){\rho _{3}}}} u_3 & 0 & 0 & \frac{u_{3}}{\sqrt{g \rho _{3}}} & \sqrt{\frac{h_{3}}{\rho_{3}}} & 0\\ 0 & 0 & 0 & -\sqrt{\frac{{\rho _{2}}}{{g}\left(\rho _{3}-\rho _{2}\right){\rho _{3}}}} v_3 & 0 & 0 & \frac{v_{3}}{\sqrt{g\rho_{3}}} & 0 & \sqrt{\frac{h_{3}}{\rho_{3}}} \end{array}\right).
\end{equation}
\end{footnotesize}%
When $M>3$, the non-zero entries of $\bm{R}$ lie on the diagonal and $(3m-2)$th columns for $m=1,\dots,M$.
The diagonal entries are 
\begin{align*}
&\sqrt{
\dfrac
{\rho_2}
{g
\left(
\rho_2-\rho_1
\right)\rho_1
}}, 
\sqrt{
\frac{h_1}
{\rho_1}
}
,
\sqrt{
\dfrac{h_1}
{\rho_1}
},
\dots
,
\sqrt{
\dfrac
{\rho_{m+1}}
{g
\left(
\rho_{m+1}-\rho_{m}
\right)\rho_{m}
}}, 
\sqrt{
\dfrac{h_{m}}
{\rho_{m}}
}
,
\sqrt{
\frac{h_{m}}
{\rho_{m}}}
,
\dots
\\
&\sqrt{
\dfrac
{\rho_{M}}
{g
\left(
\rho_{M}-\rho_{M-1}
\right)\rho_{M-1}
}}, 
\sqrt{
\dfrac{h_{M-1}}
{\rho_{M-1}}
}
,
\sqrt{
\frac{h_{M-1}}
{\rho_{M-1}}},
\sqrt{
\dfrac
{1}
{g
\rho_{M}
}}, 
\sqrt{
\dfrac{h_{M}}
{\rho_{M}}
}
,
\sqrt{
\frac{h_{M}}
{\rho_{M}}},
\end{align*}
and the $(3m-2)$th column ($m=1,\dots,M-1$) is
\begin{align*}
    \Bigg(&\bm{0}_{3m-3},\sqrt{\frac{\rho_{m+1}}{g\left(\rho_{m+1}-\rho_{m}\right)\rho_{m}}}, 
    \sqrt{\frac{\rho_{m+1}}{g\left(\rho_{m+1}-\rho_{m}\right)\rho_{m}}}u_{m},
    \sqrt{\frac{\rho_{m+1}}{g\left(\rho_{m+1}-\rho_{m}\right)\rho_{m}}}v_{m},
    \\ &-\sqrt{\frac{\rho_{m}}{g\left(\rho_{m+1}-\rho_{m}\right)\rho_{m+1}}},
    -\sqrt{\frac{\rho_{m}}{g\left(\rho_{m+1}-\rho_{m}\right)\rho_{m+1}}}u_{m+1}
    ,
    -\sqrt{\frac{\rho_{m}}{g\left(\rho_{m+1}-\rho_{m}\right)\rho_{m+1}}}v_{m+1},\bm{0}_{3M-3m-3}
    \Bigg)^{\mathrm{T}},
\end{align*}
with the $(3M-2)$th column
\begin{equation*}
\left(\bm{0}_{3M-3}, \frac{1}{\sqrt{g\rho_m}}, \frac{u_m}{\sqrt{g\rho_m}}, \frac{v_m}{\sqrt{g\rho_m}}\right)^\mathrm{T}.
\end{equation*}
High-order accuracy can be achieved by employing the WENO reconstruction \cite{Borges2008An} to obtain high-order jumps, similar to the works \cite{Duan2020RHD,Duan2021_High}.
To be specific, the scaled variables
$\Big\{\widetilde{\bm{{V}}}_{i+r,j} := \big(\bm{R}^{\mathrm{T}}({\bU})\big)_{i+\frac12,j}\bm{{V}}_{i+r,j}\Big\}$
are reconstructed to obtain the left and right limit values $\widetilde{\bm{{V}}}_{i+\frac{1}{2},j}^{{\mathtt{WENO}},-}$ (using stencil $r=-p+1,\cdots,p-1$) and $\widetilde{\bm{{V}}}_{i+\frac{1}{2},j}^{{\mathtt{WENO}},+}$ (using stencil $r=-p+2,\cdots,p$),
then the corresponding high-order jumps are defined as
\begin{equation*}
  \jumpangle{\widetilde{\bm{{V}}}}_{i+\frac{1}{2},j}^{{\mathtt{WENO}}} = \widetilde{\bm{{V}}}_{i+\frac{1}{2},j}^{{\mathtt{WENO}},+} - \widetilde{\bm{{V}}}_{i+\frac{1}{2},j}^{{\mathtt{WENO}},-}.
\end{equation*}
The diagonal matrix ${\bm{Y}}_{i+\frac{1}{2},j}$ is used to preserve the ``sign'' property \cite{Biswas2018Low},
with the diagonal elements
\begin{equation*}
  \left({\bm{Y}}_{i+\frac{1}{2},j}\right)_{\iota,\iota} = \begin{cases}
    1,\quad
    &\text{if}\quad \text{sign}\left(\jumpangle{\widetilde{\bm{{V}}}_\iota}^{{\mathtt{WENO}}}_{{{i+\frac{1}{2},j} }}\right) ~\text{sign}\left(\jump{\bm{\widetilde{V}}_\iota}_{{{i+\frac{1}{2},j} }}\right)\geqslant 0, \\
    0, &\text{otherwise}, \\
  \end{cases}
\end{equation*}
where $\iota = 1,\dots,3M$.
The lake at rest state
\begin{equation*}
	\left(h_m + z_m\right) \equiv C_m,
	~u_m = v_m = 0,
\end{equation*}
is not affected by the dissipation terms,
because the jump of the scaled variables $\bm{\widetilde{V}}$ is zero in this case.

\begin{proposition}\rm\label{prop:2D_HO_sufficient_condition}
    The schemes \eqref{eq:2D_HighOrder_EC_Discrete}-\eqref{eq:HighOrder_EC_Flux_B} are ES, in the sense of Definition \ref{def:ECES_condition},
with the numerical energy fluxes
 \begin{align*}
		&\left(\widetilde{{Q}}_1\right)_{i+\frac{1}{2},j}^{\mathtt{ES}}=\sum_{q=1}^p \alpha_{p, q} \sum_{s=0}^{q-1} \widetilde{{Q}}_1\left(\bm{U}_{i-s,j}, \bm{U}_{i-s+q,j}\right)  -
\frac{1}{2}
\mean{\bm{V}}_{i+\frac{1}{2},j}^{\mathrm{T}}
{\bm{S}}_{i+\frac{1}{2},j}{\bm{Y}}_{{i+\frac{1}{2},j} }\jumpangle{\bm{\widetilde{V}}}_{i+\frac{1}{2},j}^{\mathtt{WENO}},\\
&\left(\widetilde{{Q}}_2\right)_{i,j+\frac{1}{2}}^{\mathtt{ES}}=\sum_{q=1}^p \alpha_{p, q} \sum_{s=0}^{q-1} \widetilde{{Q}}_2\left(\bm{U}_{i-s,j}, \bm{U}_{i-s+q,j}\right)  -
\frac{1}{2}
\mean{\bm{V}}_{i,j+\frac{1}{2}}^{\mathrm{T}}
{\bm{S}}_{i,j+\frac{1}{2}}{\bm{Y}}_{{i,j+\frac{1}{2}} }\jumpangle{\bm{\widetilde{V}}}_{i,j+\frac{1}{2}}^{\mathtt{WENO}},
	\end{align*}
where
	\begin{align*}
		\widetilde{{Q}}_1\left(\bU_L, \bU_R\right) =
		\ &\frac{1}{2}\left(\bm{{V}}(\bU_{L})+\bm{{V}}(\bU_{R})\right)^{\mathrm{T}}\widetilde{\bm{{F}}}_{1}\left(\bU_L, \bU_R\right)
	- \frac{1}{2}
 \left(\psi_{1}
 \left(\bU_{L}\right)
 +\psi_{1}
 \left(\bU_{R}\right)\right)
 \nonumber\\
		& 
  + \sum\limits_{m=1}^{M}\frac{g\rho_m}{4} 
  \left(\left(h_m u_m\right)_{L}
  +\left(h_m u_m\right)_{R}\right)
  \left(		\left(z_m\right)_{L}+		\left(z_m\right)_{R}\right)
		\\
		& 
  -\sum\limits_{m=1}^{M} 
  \frac{g\rho_m}{2} 
  \left(	\left(h_m u_m	z_m\right)_{L}
  +	\left(h_m u_m z_m\right)_{R}\right),
  \\
 	\widetilde{{Q}}_2\left(\bU_L, \bU_R\right) =
		\ &\frac{1}{2}\left(\bm{{V}}(\bU_{L})+\bm{{V}}(\bU_{R})\right)^{\mathrm{T}}
  \widetilde{\bm{{F}}}_{2}\left(\bU_L, \bU_R\right)
	- \frac{1}{2}
 \left(\psi_{2}\left(\bU_{L}\right)
 +\psi_{2}
 \left(\bU_{R}\right)\right) \nonumber\\
		& + \sum\limits_{m=1}^{M}
  \frac{g\rho_m}{4} 
  \left(\left(h_m v_m\right)_{L}
  +\left(h_m v_m\right)_{R}\right)
  \left(		\left(z_m\right)_{L}+		\left(z_m\right)_{R}\right)
		\\
		& -\sum\limits_{m=1}^{M} \frac{g\rho_m}{2} 
  \left(	\left(h_m v_m	z_m\right)_{L}
  +	\left(h_m v_m z_m\right)_{R}\right).
 \end{align*}
\end{proposition}

\begin{proof}\rm
Left multiply the schemes \eqref{eq:2D_HighOrder_EC_Discrete} with $\bm{{V}}_{i,j}^\mathrm{T}:=\bm{{V}}^\mathrm{T}(\bU_{i,j})$, then one has
	\begin{align*}
		\frac{\mathrm{d}}{\mathrm{d} t}\eta_{i,j}
		=
		& -\dfrac{1}{\Delta x_1}\bm{{V}}_{i,j}^{\mathrm{T}}
		\left(
		\left(\bm{\widetilde{{F}}}_{1}\right)_{i+\frac{1}{2}, j}^{\mathtt{ES}}-\left(\bm{\widetilde{{F}}}_{1}\right)_{i-\frac{1}{2}, j}^{\mathtt{ES}}\right)
		-\frac{1}{\Delta x_2}\bm{{V}}_{i,j}^{\mathrm{T}}
		\left(\left(\bm{\widetilde{{F}}}_{2}\right)_{i, j+\frac{1}{2}}^{\mathtt{2pth}}
		-\left(\bm{\widetilde{{F}}}_{2}\right)_{i, j-\frac{1}{2}}^{\mathtt{2pth}}\right)\nonumber
		\\
 &- \sum\limits_{m=1}^{M}
 \bm{{V}}_{i,j}^{\mathrm{T}}
 \dfrac{g(h_m)_{i,j}}{\Delta x_1}
 \left(\left(\bm{\widetilde{{B}}}_{1,m}\right)_{i+\frac{1}{2}, j}^{{\mathtt{2pth}}}-\left(\bm{\widetilde{{B}}}_{1,m}\right)_{i-\frac{1}{2}, j}^{{\mathtt{2pth}}}\right)
\\
&- \sum\limits_{m=1}^{M}
\bm{{V}}_{i,j}^{\mathrm{T}}
\dfrac{g(h_m)_{i,j}}{\Delta x_2}
\left(
\left(\bm{\widetilde{{B}}}_{2,m}\right)_{i, j+\frac{1}{2}}^{{\mathtt{2pth}}}-\left(\bm{\widetilde{{B}}}_{2,m}\right)_{i, j-\frac{1}{2}}^{{\mathtt{2pth}}}\right).
	\end{align*}
	For clarity, the right-hand side can be split as
	\begin{equation*}
		\frac{\mathrm{d}}{\mathrm{d} t}\eta_{i,j}
		= -\frac{1}{\Delta x_1}\left[\sum_{q=1}^{p}\alpha_{p,q}\left(I_1+I_2\right)+I_3\right]-\frac{1}{\Delta x_2}\left[\sum_{q=1}^{p}\alpha_{p,q}\left(I_4+I_5\right)+I_6\right],
	\end{equation*}
	with
	\begin{align*}
		&I_1 = \bm{{V}}_{i,j}^{\mathrm{T}}
		\left[\bm{\widetilde{{F}}}_1
		\left(\bm{U}_{i,j},\bm{U}_{i+q,j}\right)
		-\bm{\widetilde{{F}}}_1
		\left(\bm{U}_{i,j},\bm{U}_{i-q,j}\right)\right],\nonumber\\
		&I_2 =\sum\limits_{m = 1}^{M}
		\dfrac{ g\rho_m\left(h_mu_m\right)_{i,j}}{2}
		\left[\left({\left(z_m\right)}_{i,j}+\left(z_m\right)_{i+q,j}\right)
		-\left({\left(z_m\right)}_{i,j}
		+\left(z_m\right)_{i-q,j}\right)\right],\\
              &I_3 = -\frac{1}{2}\bm{{V}}_{i,j}^{\mathrm{T}}\left({\bm{S}}_{i+\frac{1}{2},j}{\bm{Y}}_{{i+\frac{1}{2},j} }\jumpangle{\bm{\widetilde{V}}}_{i+\frac{1}{2},j}^{\mathtt{WENO}}-{\bm{S}}_{i-\frac{1}{2},j}{\bm{Y}}_{{i-\frac{1}{2},j} }\jumpangle{\bm{\widetilde{V}}}_{i-\frac{1}{2},j}^{\mathtt{WENO}}\right),
            \\
		&I_4 = \bm{{V}}_{i,j}^{\mathrm{T}}
		\left[\bm{\widetilde{{F}}}_2
		\left(\bm{U}_{i,j},\bm{U}_{i,j+q}\right)
		-\bm{\widetilde{{F}}}_2\left(\bm{U}_{i,j},\bm{U}_{i,j-q}\right)\right],\nonumber\\
		&I_5 =\sum\limits_{m = 1}^{M}
		\dfrac{ g\rho_m\left(h_m v_m\right)_{i,j}}{2}
		\left[\left({\left(z_m\right)}_{i,j}
		+\left(z_m\right)_{i,j+q}\right)
		-\left({\left(z_m\right)}_{i,j}
		+\left(z_m\right)_{i,j-q}\right)\right],
  \\
            &I_6 = -\frac{1}{2}\bm{{V}}_{i,j}^{\mathrm{T}}\left({\bm{S}}_{i,j+\frac{1}{2}}{\bm{Y}}_{{i,j+\frac{1}{2}} }\jumpangle{\bm{\widetilde{V}}}_{i,j+\frac{1}{2}}^{\mathtt{WENO}}-{\bm{S}}_{i,j-\frac{1}{2}}{\bm{Y}}_{{i,j-\frac{1}{2}} }\jumpangle{\bm{\widetilde{V}}}_{i,j-\frac{1}{2}}^{\mathtt{WENO}}\right).
	\end{align*}
	Due to the fact $a_{i,j}= \frac{1}{2}\left(a_{i,j}+a_{i+q,j}\right)-\frac{1}{2}\left(a_{i+q,j}-a_{i,j}\right)$ and $a_{i,j} = \frac{1}{2}\left(a_{i,j}+a_{i-q,j}\right)+\frac{1}{2}\left(a_{i,j}-a_{i-q,j}\right)$,
	the term $I_1$ can be simplified as
	\begin{align*}
		I_1 = &+\frac{1}{2}\left(\bm{{V}}_{i,j}+\bm{{V}}_{i+q,j}\right)^{\mathrm{T}}
		\bm{\widetilde{{F}}}_1\left(\bm{U}_{i,j},\bm{U}_{i+q,j}\right)-\frac{1}{2}\left(\bm{{V}}_{i+q,j}-\bm{{V}}_{i,j}\right)^{\mathrm{T}}
		\bm{\widetilde{{F}}}_1\left(\bm{U}_{i,j},\bm{U}_{i+q,j}\right)\\
		&-\frac{1}{2}\left(\bm{{V}}_{i,j}+\bm{{V}}_{i-q,j}\right)^{\mathrm{T}}\bm{\widetilde{{F}}}_1\left(\bm{U}_{i,j},\bm{U}_{i-q,j}\right) -\frac{1}{2}\left(\bm{{V}}_{i,j}-\bm{{V}}_{i-q,j}\right)^{\mathrm{T}}\bm{\widetilde{{F}}}_1\left(\bm{U}_{i,j},\bm{U}_{i-q,j}\right).
	\end{align*}
Similarly, one can obtain
	\begin{align*}
		I_2 =\sum_{m=1}^M g\rho_m
  \Bigg[
  &\frac{1}{4}\left(
  \left(h_m u_m
  \right)_{i+q,j}+\left(h_mu_m\right)_{i,j}
  \right)
  \left(
  {\left(z_m\right)}_{i,j}
  +\left(z_m\right)_{i+q,j}
  \right) 
  \\
  -&\frac{1}{4}\left(\left(h_m u_m\right)_{i+q,j}-\left(h_mu_m\right)_{i,j}\right)
		\left({\left(z_m\right)}_{i,j}+\left(z_m\right)_{i+q,j}\right)
  \\
		-&\frac{1}{4}\left(\left(h_mu_m\right)_{i-q,j}+\left(h_mu_m\right)_{i,j}\right)
	\left({\left(z_m\right)}_{i,j}+\left(z_m\right)_{i-q,j}\right)
 \\
 -&\frac{1}{4}\left(\left(h_m u_m\right)_{i,j}-\left(h_m u_m\right)_{i-q,j}\right)
		\left({\left(z_m\right)}_{i,j}+\left(z_m\right)_{i-q,j}\right)
		\Bigg].
	\end{align*}
	Based on the condition \eqref{eq:EC_condition_2D_ML}, 
 and
 \begin{align*}
 &\dfrac{1}{2}\left(
 \rho_mg
 \left(
 \left({h_{m}z_mu_m}
 \right)_{i+q,j}-
 \left({h_{m}z_mu_m}
 \right)_{i,j}\right)
 +
 \rho_mg
 \left(
 \left({h_{m}z_mu_m}
 \right)_{i,j}-
 \left({h_{m}z_mu_m}
 \right)_{i-q,j}\right)
 \right)
 \\=& 
\dfrac{1}{2}\left(
 \rho_mg
 \left(
 \left({h_{m}z_mu_m}
 \right)_{i+q,j}+
 \left({h_{m}z_mu_m}
 \right)_{i,j}\right)
 -
 \rho_mg
 \left(
 \left({h_{m}z_mu_m}
 \right)_{i,j}+
 \left({h_{m}z_mu_m}
 \right)_{i-q,j}\right)
 \right),
 \end{align*}
 one has
	\begin{equation*}
		I_1+I_2 = \widetilde{{Q}}_1\left(\bm{U}_{i,j}, \bm{U}_{i+q,j}\right)-\widetilde{{Q}}_1\left(\bm{U}_{i,j}, \bm{U}_{i-q,j}\right).
	\end{equation*}
As for the term $I_3$, one deduces that
\begin{align*}
    I_3 = &-\frac{1}{2}\bm{{V}}_{i,j}^{\mathrm{T}}\left({\bm{S}}_{i+\frac{1}{2},j}{\bm{Y}}_{{i+\frac{1}{2},j} }\jumpangle{\bm{\widetilde{V}}}_{i+\frac{1}{2},j}^{\mathtt{WENO}}-{\bm{S}}_{i-\frac{1}{2},j}{\bm{Y}}_{{i-\frac{1}{2},j} }\jumpangle{\bm{\widetilde{V}}}_{i-\frac{1}{2},j}^{\mathtt{WENO}}\right)\\
    =&-\frac{1}{2}
    \left(\mean{\bm{V}}_{i+\frac{1}{2},j}^{\mathrm{T}}
    {\bm{S}}_{i+\frac{1}{2},j}{\bm{Y}}_{{i+\frac{1}{2},j} }\jumpangle{\bm{\widetilde{V}}}_{i+\frac{1}{2},j}^{\mathtt{WENO}}
    -\mean{\bm{V}}_{i-\frac{1}{2},j}^{\mathrm{T}}
    {\bm{S}}_{i-\frac{1}{2},j}
    {\bm{Y}}_{{i-\frac{1}{2},j} }\jumpangle{\bm{\widetilde{V}}}_{i-\frac{1}{2},j}^{\mathtt{WENO}}\right)
    \\&+\frac{1}{4}
    \left(
    \jump{\bm{V}}_{i+\frac{1}{2},j}
    ^{\mathrm{T}}
    {\bm{S}}_{i+\frac{1}{2},j}
    {\bm{Y}}_{{i+\frac{1}{2},j} }
    \jumpangle{\bm{\widetilde{V}}}_{i+\frac{1}{2},j}^{\mathtt{WENO}}
    +
    \jump{\bm{V}}_{i-\frac{1}{2},j}^{\mathrm{T}}
    {\bm{S}}_{i-\frac{1}{2},j}
    {\bm{Y}}_{{i-\frac{1}{2},j} }
    \jumpangle{\bm{\widetilde{V}}}_{i-\frac{1}{2},j}^{\mathtt{WENO}}\right),
\end{align*}
which implies
\begin{align*}
    -\frac{1}{\Delta x_1}\left[\sum_{q=1}^{p}\alpha_{p,q}\left(I_1+I_2\right)+I_3\right] = &-\frac{1}{\Delta x_1}\left(\left({\widetilde{{{Q}}}}_1\right)_{i+\frac{1}{2},j}^{\mathtt{ES}}
 -\left({\widetilde{{{Q}}}}_1\right)_{i-\frac{1}{2},j}^{\mathtt{ES}}\right) 
 \\
 &- \frac{1}{4\Delta x_1} 
 \left(
 \left[
 {\alpha_1}
 \jump{\widetilde{\bm{V}}}^{\mathrm{T}}
 {\bm{Y}}
 \jumpangle{\bm{\widetilde{V}}}^{\mathtt{WENO}}
 \right]
 _{i+\frac{1}{2},j}
    +
    \left[
    {\alpha_1}
    \jump{\widetilde{\bm{V}}}^{\mathrm{T}}
    {\bm{Y}}
    \jumpangle{\bm{\widetilde{V}}}^{\mathtt{WENO}}
    \right]
_{i-\frac{1}{2},j}\right).
\end{align*}
 The terms $I_4$, $I_5$, and $I_6$ can be simplified in the same way. The proof is completed.
\end{proof}

This paper uses the explicit SSP-RK3 time discretization to obtain the fully-discrete schemes, 
\begin{equation*}
  \begin{aligned}
    \bm{{U}}^{*} &= \bm{{U}}^{n}+\Delta t^n \bm{{L}}\left(\bU^{n}\right),\\
    \bm{{U}}^{**} &= \frac{3}{4}\bm{{U}}^{n}+\frac{1}{4}\left(\bm{U}^{*}+\Delta t^n \bm{{L}}\left(\bU^{*}\right)\right),\\
    \bm{{U}}^{n+1} &= \frac{1}{3}\bm{{U}}^{n}+\frac{2}{3}\left(\bm{{U}}^{**}+\Delta t^n \bm{{L}}\left(\bU^{**}\right)\right),
  \end{aligned}
\end{equation*}
where $\bm{{L}}$ is the right-hand side of
\eqref{eq:2D_HighOrder_EC_Discrete}
for the semi-discrete high-order WB ES schemes. The time stepsize $\Delta t^{n}$ of the 2D schemes is
\begin{equation}\label{eq:dt_fixed}
  \Delta t^{n} = \frac{C_{\text{\tiny \tt CFL}}}{\sum\limits_{\ell = 1}^{2}\max_{i,j}\limits\left({\alpha}_{\ell}\right)^{n}_{i,j}/\Delta {x_\ell}},
\end{equation}
with $\alpha_1$ defined in \eqref{eq:alpha_i} for the $x_1$-direction.

\begin{proposition}\rm\label{prop:WB_proof}
 The fully-discrete schemes based on the semi-discrete ES schemes \eqref{eq:2D_HighOrder_EC_Discrete}-\eqref{eq:HighOrder_EC_Flux_B} and the forward Euler or explicit SSP-RK3 time discretizations are WB, in the sense of Definition \ref{def:WB_condition}.
\end{proposition}

\begin{proof}
    The explicit SSP-RK3 scheme can be written as a convex combination of the forward Euler discretization so that it suffices to prove in the case of forward Euler time discretization.
    By induction, assuming that the conditions \eqref{eq:2pnd_WB_tn} are satisfied at $t^{n}$,
  one needs to prove that they still hold at $t^{n+1}$.
  Under the forward Euler time discretization, the update of the water height for the $m$th-layer is
  \begin{align*}
    (h_{m})_{i,j}^{n+1} - (h_{m})_{i,j}^{n} = &~0,
  \end{align*}
then one has
\begin{align*}
	 \left(h_m + b+
	\sum\limits_{k > m} h_{k} 
	+ \sum\limits_{k < m} r_{km} h_{k}\right)_{i,j}^{n+1} = \left(h_m + b+
	\sum\limits_{k > m} h_{k} 
	+ \sum\limits_{k < m} r_{km} h_{k}\right)^{n}_{i,j} = C_m,
\end{align*}
i.e., $\left(h_m+z_m\right)^{n+1} = C_m$.

Notice that the dissipation terms in \eqref{eq:HighOrder_EC_Flux} become zero for the lake at rest, so that only the contributions from the high-order EC flux and discretizations of the source terms are considered.
  The second component of the semi-discrete schemes for the $m$th-layer can be written as
  \begin{align}\label{eq:JHU_original}
    (h_mu_m)^{n+1}_{i,j} - (h_mu_m)^n_{i,j} =
    -\frac{\Delta t}{\Delta x_1}\sum_{q=1}^{p}\alpha_{p,q}\left(H_1-H_2\right),
  \end{align}
  where
  \begin{align}
    H_1 = \ g\Bigg[&\frac{1}{4}\left((h_m)^2_{i,j}+(h_m)^2_{i+q,j}\right) + \frac{1}{2}\left(\left(h_m z_m\right)_{i,j} +\left(h_m z_m\right)_{i+q,j}\right)\nonumber
      \\
    - &\frac{1}{4}\left(\left(h_m\right)_{i,j}+\left(h_m \right)_{i+q,j}\right)\left({\left(z_m\right)}_{i,j}+{\left(z_m\right)}_{i+q,j}\right) + \frac{1}{2}\left(h_m\right)_{i,j} \left({\left(z_m\right)}_{i,j}+{\left(z_m\right)}_{i+q,j}\right) \Bigg]\nonumber,
    \\
    H_2 = \ g\Bigg[&\frac{1}{4}\left((h_m)^2_{i,j}+(h_m)^2_{i-q,j}\right) + \frac{1}{2}\left(\left(h_m z_m \right)_{i,j} + \left(h_m z_m \right)_{i-q,j}\right)\nonumber
    \\
    - &\frac{1}{4}\left(\left(h_m\right)_{i,j}+\left(h_m \right)_{i-q,j}\right)\left({\left(z_m\right)}_{i,j}+{\left(z_m\right)}_{i-q,j}\right) + \frac{1}{2}\left(h_m\right)_{i,j} \left({\left(z_m\right)}_{i,j}+{\left(z_m\right)}_{i-q,j}\right) \Bigg]\nonumber,
  \end{align}
    and the superscript $n$ is omitted in the right-hand side for simplicity.
  By adding $-\dfrac{g}{2}(h_m)^{2}_{i,j}-g\left(h_m z_m\right)_{i,j}$ to $H_1$ and $H_2$,
  one can rewrite \eqref{eq:JHU_original} as
  \begin{align*}
    (h_mu_m)^{n+1}_{i,j} - (h_mu_m)^n_{i,j} =
    -\frac{\Delta t}{\Delta x_1}\sum_{q=1}^{p}\alpha_{p,q}\left(\widetilde{H}_1-\widetilde{H}_2\right),
  \end{align*}
  where
  \begin{align*}
   	\widetilde{H}_1 = \  g\Bigg[&\frac{1}{4}\left((h_m)^2_{i,j}+(h_m)^2_{i+q,j}\right) + \frac{1}{2}\left(\left(h_m z_m \right)_{i,j} +\left(h_m z_m\right)_{i+q,j}\right) -\dfrac{1}{2}(h_m)^{2}_{i,j} -\left(h_m z_m\right)_{i,j}  \nonumber
   	\\
   	- & \frac{1}{4}\left(\left(h_m \right)_{i,j}+\left(h_m \right)_{i+q,j}\right)\left({\left(z_m\right)}_{i,j}+{\left(z_m\right)}_{i+q,j}\right) + \frac{1}{2}\left(h_m\right)_{i,j} \left({\left(z_m\right)}_{i,j}+{\left(z_m\right)}_{i+q,j}\right)  \Bigg]\nonumber,
   	\\
   	\widetilde{H}_2 = \  g\Bigg[&\frac{1}{4}\left((h_m)^2_{i,j}+(h_m)^2_{i-q,j}\right) + \frac{1}{2}\left(\left(h_m z_m\right)_{i,j} +\left(h_m z_m\right)_{i-q,j}\right) -\dfrac{1}{2}(h_m)^{2}_{i,j} -\left(h_m z_m\right)_{i,j} \nonumber
   	\\
   	-&  \frac{1}{4}\left(\left(h_m \right)_{i,j}+\left(h_m \right)_{i-q,j}\right)\left({\left(z_m\right)}_{i,j}+{\left(z_m\right)}_{i-q,j}\right) + \frac{1}{2}\left(h_m\right)_{i,j} \left({\left(z_m\right)}_{i,j}+{\left(z_m\right)}_{i-q,j}\right) \Bigg]\nonumber.
  \end{align*}
  Simplify the two parts in the brackets in $\widetilde{H}_1$ as
  \begin{align*}
    &\frac{1}{4}\left(\left(h_m \right)_{i,j}^2+\left(h_m \right)_{i+q,j}^2\right)  - \frac{1}{2}\left(h_m \right)^2_{i,j} = \frac{1}{4}\left(\left(h_m\right)_{i,j}+\left(h_m \right)_{i+q,j}\right)\left(\left(h_m \right)_{i+q,j}-\left(h_m\right)_{i,j}\right),
    \\
    &\frac{1}{2}\left(\left(h_m z_m\right)_{i,j} + \left(h_m z_m\right)_{i+q,j}\right) -\left(h_m z_m\right)_{i,j} - \frac{1}{4}\left(\left(h_m \right)_{i,j}+\left(h_m \right)_{i+q,j}\right)\left({\left(z_m\right)}_{i,j}+{\left(z_m\right)}_{i+q,j}\right) \nonumber
    \\ + &\frac{1}{2}\left(h_m\right)_{i,j} \left({\left(z_m\right)}_{i,j}+{\left(z_m\right)}_{i+q,j}\right) = \frac{1}{4}\left(\left(h_m\right)_{i,j}+\left(h_m \right)_{i+q,j}\right)\left(\left(z_m\right)_{i+q,j}-\left(z_m\right)_{i,j}\right),
  \end{align*}
  then one has
\begin{align*}
    \widetilde{H}_1= \left\{\left(\left(h_m \right)_{i+q,j}+\left(h_m\right)_{i,j}\right)\left[\left(h_{m}+z_m\right)_{i+q,j} - \left(h_{m}+z_m\right)_{i,j}\right]\right\}=0.
  \end{align*}
Similarly,
  \begin{align*}
    \widetilde{H}_2= \left\{\left(\left(h_m \right)_{i-q,j}+\left(h_m\right)_{i,j}\right)\left[\left(h_{m}+z_m\right)_{i-q,j} - \left(h_{m}+z_m\right)_{i,j}\right]\right\}=0,
  \end{align*}
 so that $(h_mu_m)^{n+1}_{i,j} - (h_mu_m)^n_{i,j} = 0$,
  or $\left(u_m\right)_{i,j}^{n+1} = 0$.
  One can also obtain $(v_m)^{n+1}_{i,j}=0$ and the proof is completed.
\end{proof}

\section{Extension to adaptive moving meshes}\label{Sec:MM_Case}

In this section, the high-order accurate WB ES finite difference schemes are extended to adaptive moving meshes based on our previous work \cite{Zhang2023High}.

\subsection{Reformulated ML-SWEs and corresponding energy inequality}
Generally, there are two ways to obtain the bottom topography on a new mesh during mesh movement: by interpolation using the analytical expression or by regarding $b$ as a time-dependent variable evolved with the original equations simultaneously.
It is challenging to obtain the two-point EC fluxes in the first way, as mentioned in \cite{Zhang2023High}, so the second way is adopted here.
By adding $\partial{b}/ \partial{t} = 0$ to the original system of ML-SWEs \eqref{eq:SWE0}, it is reformulated as
\begin{equation}\label{eq:ML_SWE_Modified}
    \frac{\partial{\widehat{\bm{U}}}}{\partial t}+\sum_{\ell=1}^2 \frac{\partial {\widehat{\bm{F}}}_{\ell}({\widehat{\bm{U}}})}{\partial x_{\ell}}
    =-g \sum_{\ell=1}^2 \sum\limits_{m=1}^Mh_m \frac{\partial \widehat{\bm{B}}_{\ell,m}}{\partial x_{\ell}},
\end{equation}
with modified conservative variables, physical flux, and source terms
\begin{equation*}
\begin{aligned}
{\widehat{\bm{U}}}  =\left(\bm{U}^{\mathrm{T}},b\right)^{\mathrm{T}},~{\widehat{\bm{F}}}_\ell =\left(\bm{F}_{\ell}^{\mathrm{T}},0\right)^{\mathrm{T}}, ~
\widehat{\bm{{B}}}_{\ell,m} = \left(\bm{{B}}_{\ell,m}^{\mathrm{T}},0\right)^\mathrm{T}.
\end{aligned}
\end{equation*}
The corresponding energy and energy flux for \eqref{eq:ML_SWE_Modified} are
\begin{equation*}
	\widehat{\eta}(\widehat{\bm{U}})=\eta(\bU)+\gamma g b^2,~
\widehat{q}_\ell(\widehat{{\bm{U}}})= {q}_{\ell}(\bm{U}).
\end{equation*}
One can verify that the Hessian matrix can be expressed as
\begin{equation*}
    \partial^2 \widehat{\eta}/\partial \widehat{\bU}^2 = 
    \left(
        \begin{array}{cc}
             \mathcal{M}_{M} & \widehat{A}_1^{\mathrm{T}} \\
            \widehat{A}_1 & 2\gamma g
        \end{array}
        \right) \in \mathbb{R}^{(3M+1)\times(3M+1)},
\end{equation*}
where $\widehat{A}_1 = 
        \left(g\rho_{1}, 0, 0, g\rho_{2}, 0, 0,  \cdots, g\rho_{M}, 0, 0
        \right) \in \mathbb{R}^{1\times3M}$.
The first $3M$ leading principle minors of the Hessian matrix equal to those in Proposition \ref{prop:eta_convex_prove} with $\mathcal{N} = M$,
and the determinant is
\begin{equation*}
    \frac{g^{M+1}\rho_1^3(2\gamma - \rho_M)\prod_{m=2}^{M}\rho_m^2(\rho_m-\rho_{m-1})}{\prod_{m=1}^{M}h_m^2}.
\end{equation*}
One can also verify that the quadratic form of the Hessian matrix is 
\begin{equation*}
    \begin{aligned} 
    \widehat{\bm{\beta}}_{M} 
    \dfrac{\partial^2 \widehat{\eta}}{\partial \widehat{\bU}^2}\widehat{\bm{\beta}}_{M}^{\mathrm{T}} = &\ g\rho_1\left(
    \sum_{m=1}^{M}
    \beta_{3m-2}
    +\beta_{3M+1}\right)^2
    +g\sum_{m=2}^{M}(\rho_{m}-\rho_{m-1})
    \left(
    \sum_{l=m}^{M}
    \beta_{3l-2}
    +\beta_{3M+1}\right)^2
    \\ & +\sum_{m=1}^{M}
    \dfrac{\rho_m}{h_m}
    \left(
    \left(
    u_{m}\beta_{3m-2}-\beta_{3m-1}
    \right)^2
    +
    \left(
    v_{m}\beta_{3m-2}-\beta_{3m}
    \right)^2
    \right)   +g(2\gamma - \rho_M)\beta_{3M+1}^2,
\end{aligned}
\end{equation*}
with $\widehat{\bm{\beta}}_{M} = (\bm{\beta}_M,\beta_{3M+1})$.
Thus the choice of the constant $\gamma > \rho_M/2$ ensures the convexity of $\widehat{\eta}(\widehat{\bU})$.
Define
\begin{equation*}\label{eq:Entropy_Variables_ML_MM}
	\widehat{\bm{{V}}} := \partial \widehat{\eta}/\partial \widehat{\bU}^\mathrm{T} = \left(\bm{V}^{T},
  g\sum\limits_{m=1}^{M}\rho_m h_m +2\gamma gb\right)^{\mathrm{T}},
\end{equation*}
and one can derive the following energy inequality
\begin{equation}\label{eq:Energy_inequality_MM}
    \pd{\widehat{\eta}(\widehat{\bU})}{t}
    +\sum_{\ell=1}^{2}
    \pd{\widehat{q_{\ell}}(\widehat{\bU})}{x_{\ell}} \leqslant 0.
\end{equation}

To generate adaptive moving meshes, the adaptive moving mesh strategy in \cite{Duan2021_High} is adopted,
based on a time-dependent coordinate transformation $t = \tau, \bm{x}= \bm{x}\bm(\tau,\bm{\xi})$ from the computational domain $\Omega_c$ to the physical domain $\Omega_p$.
Under such a transformation, the ML-SWEs \eqref{eq:ML_SWE_Modified} and energy inequality \eqref{eq:Energy_inequality_MM} can be written in curvilinear coordinates as
\begin{align}\label{eq:ML_SWE_Curvilinear}
   &\frac{\partial{\bm{\mathcal{U}}}}{\partial t}+\sum_{\ell=1}^2 \frac{\partial \bm{\mathcal{F}}_{\ell}}{\partial \xi_{\ell}}
    =-g \sum_{\ell=1}^2 \sum\limits_{m=1}^Mh_m \frac{\partial \bm{\mathcal{B}}_{\ell,m}}{\partial \xi_{\ell}},\\
    & \pd{\mathcal{E}}{\tau}+\sum\limits_{\ell=1}^{2}\pd{\mathcal{Q}_{\ell}}{\xi_{\ell}}\leqslant 0,
\end{align}
where
\begin{align*}
    &\bm{{\mathcal{U}}}=J \widehat{\bm{U}},~
    \bm{\mathcal{F}}_{\ell}
    =\left(J \dfrac{\partial \xi_\ell}{\partial t} \right)\bm{\widehat{U}}
    +\sum_{k=1}^2\left(J \dfrac{\partial \xi_\ell}{\partial x_k}\right) \bm{\widehat{F}}_{k}, ~\bm{\mathcal{B}}_{\ell,m}=\sum_{k=1}^2\left(J \dfrac{\partial \xi_\ell}{\partial x_k}\right) \bm{\widehat{B}}_{k,m},\\
    &\mathcal{E}=J\widehat{\eta},~
    \mathcal{Q}_{\ell}=\left(J \dfrac{\partial \xi_\ell}{\partial t} \right)\widehat{\eta}
    +\sum_{k=1}^2\left(J \dfrac{\partial \xi_\ell}{\partial x_k}\right) \widehat{q}_{k},~\ell=1,2.
\end{align*}
The Jacobian matrix is
\begin{equation}
    J = \det\left(\frac{\partial (t,\bx)}{\partial (\tau,\bm{\xi})}\right) =
    \left|\begin{matrix}
		1 & 0 &0 \\
		\dfrac{\partial x_1}{\partial \tau} & \dfrac{\partial x_1}{\partial \xi_1} &  \dfrac{\partial x_1}{\partial \xi_2}\\
		\dfrac{\partial x_2}{\partial \tau} & \dfrac{\partial x_2}{\partial \xi_1} &  \dfrac{\partial x_2}{\partial \xi_2}\\
	\end{matrix}\right|,\nonumber
\end{equation}
and the mesh metrics satisfy the geometric conservation laws (GCLs),
which consist of the volume conservation law (VCL) and the surface conservation laws (SCLs)
\begin{equation}\label{eq:GCL_2D}
    \begin{aligned}
	&\text { VCL: } \quad \frac{\partial J}{\partial \tau}+ \sum_{\ell=1}^{2}\frac{\partial}{\partial \xi_\ell}\left(J \frac{\partial \xi_\ell}{\partial t}\right)=0,\\
    &\text { SCLs: } \quad \sum_{\ell = 1}^2\frac{\partial}{\partial \xi_\ell}\left(J \frac{\partial \xi_\ell}{\partial x_k}\right)=0,~k = 1,2.
    \end{aligned}
\end{equation}

\subsection{High-order WB ES schemes on moving meshes}
In this section, high-order accurate WB ES schemes for the 2D ML-SWEs in curvilinear coordinates \eqref{eq:ML_SWE_Curvilinear} are constructed based on corresponding two-point EC fluxes.
According to \cite{Zhang2023High},
the two-point EC fluxes in curvilinear coordinates are
\begin{align}\label{eq:F_MM}
  \widetilde{\bm{\mathcal{F}}}_{\ell}
  = \ &\frac{1}{2}\left(\left(J\frac{\partial \xi_\ell}{\partial t}\right)_{L}+\left(J\frac{\partial \xi_\ell}{\partial t}\right)_{R}\right)\widetilde{\bm{U}}+ \sum_{k=1}^{2} \frac{1}{2}\left(\left(J\frac{\partial \xi_\ell}{\partial x_k}\right)_{L}+\left(J\frac{\partial \xi_\ell}{\partial x_k}\right)_{R}\right)\bm{\widetilde{\widehat{F}}}_{k},
\end{align}
where $\widetilde{\bU} = \left(\widetilde{\bU}_{1},\dots,\widetilde{\bU}_{M}\right)^{\mathrm{T}}$ with
$\widetilde{\bm{U}}_{m}=\left(
        \mean{h_m},
        \mean{h_m}\mean{u_m},
        \mean{h_m}\mean{v_m},
        \mean{b} \right)$,
is consistent with the reformulated conservative variables $\widehat{\bU}$ and satisfies
\begin{align*}
  &\left(\bm{\widehat{V}}_{R}-\bm{\widehat{V}}_{L}\right)^{\mathrm{T}} \widetilde{\bm{U}}=\widehat{\phi}_{R}-\widehat{\phi}_{L},\\
    &\widehat{\phi} = \widehat{\bV}^{\mathrm{T}}\widehat{\bm{U}} - \widehat{\eta} = \phi+\sum\limits_{m=1}^{M}\rho_{m}gh_{m}+\gamma g b^2,\quad
    \widehat{\psi}_\ell  = \psi_\ell, 
\end{align*}
and $\widetilde{\widehat{\bm{F}}}_{k}$ can be obtained through the augmentation of \eqref{def:sufficient_condition_ML_2D} by incorporating an additional zero as the last component.

\begin{remark}\rm
      When $M=1$, $\bm{\widetilde{\mathcal{F}}}_\ell$ in \eqref{eq:F_MM} reduces to the single-layer case in \cite{Zhang2023High}. 
\end{remark}

Considering a uniform Cartesian mesh in the computational domain
$(\xi_1)_i = a_1 + i\Delta\xi_1,~i=0,1,\cdots,N_1-1,~\Delta\xi_1 = (b_1-a_1)/(N_1-1)$,
$(\xi_2)_j = a_2 + j\Delta\xi_2,~j=0,1,\cdots,N_2-1,~\Delta\xi_2 = (b_2-a_2)/(N_2-1)$,
and high-order semi-discrete conservative finite difference schemes for the 2D ML-SWEs in curvilinear coordinates
\begin{align}
  \dfrac{\mathrm{d}}{\mathrm{d} t} \bm{\mathcal{U}}_{i, j}=
  & -\dfrac{1}{\Delta \xi_1}
  \left(
  \left(\bm{\widetilde{\mathcal{F}}}_{1}\right)_{i+\frac{1}{2}, j}^{\mathtt{ES}}-\left(\bm{\widetilde{\mathcal{F}}}_{1}\right)_{i-\frac{1}{2}, j}^{\mathtt{ES}}
  \right)
  -\frac{1}{\Delta \xi_2}
  \left(
  \left(\bm{\widetilde{\mathcal{F}}}_{2}\right)_{i, j+\frac{1}{2}}^{\mathtt{ES}}-\left(\bm{\widetilde{\mathcal{F}}}_{2}\right)_{i, j-\frac{1}{2}}^{\mathtt{ES}}
  \right)\nonumber
  \\
    &- \sum\limits_{m=1}^{M}\dfrac{g(h_m)_{i,j}}
    {\Delta \xi_1}
    \left(
    \left(\bm{\widetilde{\mathcal{B}}}_{1,m}\right)_{i+\frac{1}{2}, j}^{{\mathtt{2pth}}}-\left(\bm{\widetilde{\mathcal{B}}}_{1,m}\right)_{i-\frac{1}{2}, j}^{{\mathtt{2pth}}}
    \right)
 	\nonumber\\
   &- \sum\limits_{m=1}^{M}\dfrac{g(h_m)_{i,j}}
   {\Delta \xi_2}
\left(
\left(\bm{\widetilde{\mathcal{B}}}_{2,m}\right)_{i, j+\frac{1}{2}}^{{\mathtt{2pth}}}-\left(\bm{\widetilde{\mathcal{B}}}_{2,m}\right)_{i, j-\frac{1}{2}}^{{\mathtt{2pth}}}
\right),\label{eq:2D_HighOrder_EC_Discrete_MM}
\end{align}
where $\bm{\mathcal{U}}_{i,j}=J_{i,j}\bU_{i,j}$ approximates the point values of $J\bU$ at $((\xi_1)_i, (\xi_2)_j)$,
and the numerical fluxes are given by
\begin{equation*}
  \begin{aligned}
  \left(
  \bm{\widetilde{\mathcal{F}}}_{1}
  \right)_{i+\frac{1}{2}, j}^{{\mathtt{ES}}}=\left(\bm{\widetilde{\mathcal{F}}}_{1}\right)_{i+\frac{1}{2}, j}^{{\mathtt{2pth}}}
- \begin{pmatrix}
    \widehat{\bm{D}}_{i+\frac12,j} \\ 0 \\
  \end{pmatrix}
  - \mathring{\bm{D}}_{i+\frac12,j},
    \\
    \left(\bm{\widetilde{\mathcal{F}}}_{2}\right)_{i, j+\frac{1}{2}}^{{\mathtt{ES}}}
    =
   \left(\bm{\widetilde{\mathcal{F}}}_{2}\right)_{i, j+\frac{1}{2}}^{{\mathtt{2pth}}}
    - \begin{pmatrix}
    \widehat{\bm{D}}_{i,j+\frac12} \\ 0 \\
  \end{pmatrix}
  - \mathring{\bm{D}}_{i,j+\frac12},
  \end{aligned}
\end{equation*}
with
\begin{equation*}
\begin{aligned}
  \left(
  \bm{\widetilde{\mathcal{F}}}_{1}
  \right)_{i+\frac{1}{2}, j}^{{\mathtt{2pth}}}=
  \sum_{q=1}^{p} \alpha_{p,q} \sum_{s=0}^{q-1} \Bigg[
  &\frac{1}{2}\left(
  \left(
  J \frac{\partial \xi_1}{\partial t}
  \right)_{i-s, j}
  +\left(
  J \frac{\partial \xi_1}{\partial t}
  \right)_{i-s+q, j}
  \right) \widetilde{{\bm{U}}}\left(\widehat{\bm{U}}_{i-s, j}, \widehat{\bm{U}}_{i-s+q, j}\right) \\
      +& \frac{1}{2}
      \left(
      \left(
      J \frac{\partial \xi_1}{\partial x_1}
      \right)_{i-s, j}
      +\left(
      J \frac{\partial \xi_1}{\partial x_1}
      \right)_{i-s+q, j}
      \right) \widetilde{\widehat{\bm{F}}}_{1}\left(\widehat{\bm{U}}_{i-s, j}, \widehat{\bm{U}}_{i-s+q, j}\right) \\
    +& \frac{1}{2}
    \left(
    \left(
    J \frac{\partial \xi_1}{\partial x_2}\right)_{i-s, j}
    +\left(
    J \frac{\partial \xi_1}{\partial x_2}
    \right)_{i-s+q, j}\right)
    \widetilde{\widehat{\bm{F}}}_{2}
    \left(\widehat{\bm{U}}_{i-s, j}, \widehat{\bm{U}}_{i-s+q, j}\right)\Bigg],
    \\
    \left(\bm{\widetilde{\mathcal{F}}}_{2}\right)_{i, j+\frac{1}{2}}^{{\mathtt{2pth}}}
    =\sum_{q=1}^{p} \alpha_{p,q} \sum_{s=0}^{q-1}  \Bigg[
    &\frac{1}{2}
    \left(
    \left(
    J \frac{\partial \xi_2}{\partial t}
    \right)_{i, j-s}
    +\left(
    J \frac{\partial \xi_2}{\partial t}
    \right)_{i, j-s+q}
    \right) \widetilde{{\bm{U}}}\left(\widehat{\bm{U}}_{i, j-s}, \widehat{\bm{U}}_{i, j-s+q}\right) \\
      +& \frac{1}{2}
      \left(
      \left(J \frac{\partial \xi_2}{\partial x_1}
      \right)_{i, j-s}
      +\left(
      J \frac{\partial \xi_2}{\partial x_1}
      \right)_{i, j-s+q}
      \right) \widetilde{\widehat{\bm{F}}}_{1}
      \left(
      \widehat{\bm{U}}_{i, j-s}, \widehat{\bm{U}}_{i, j-s+q}
      \right) \\
    +&\frac{1}{2}
    \left(
    \left(
    J \frac{\partial \xi_2}{\partial x_2}
    \right)_{i, j-s}
    +\left(
    J \frac{\partial \xi_2}{\partial x_2}
    \right)_{i, j-s+q}\right) \widetilde{\widehat{\bm{F}}}_{2}\left(\widehat{\bm{U}}_{i, j-s}, \widehat{\bm{U}}_{i, j-s+q}
    \right)\Bigg],
  \end{aligned}
\end{equation*}
denoting the $2p$th-order WB EC fluxes in curvilinear coordinates.
The dissipation terms $\widehat{\bm{D}}$ and $\mathring{\bm{D}}$ are given in Section \ref{Sec:HOMMES_Scheme} and
\begin{equation*}
  \begin{aligned}
    \left(\bm{\widetilde{\mathcal{B}}}_{1,m}\right)_{i+\frac{1}{2},j}^{{\mathtt{2pth}}}=\sum_{q=1}^{p} \alpha_{p,q} \sum_{s=0}^{q-1} \Bigg[&\frac{1}{4}\left(\left(J \frac{\partial \xi_1}{\partial x_1}\right)_{i-s,j}+\left(J \frac{\partial \xi_1}{\partial x_1}\right)_{i-s+q,j}\right)\left(\left(\widehat{\bm{B}}_{1,m}\right)_{i-s,j}+\left(\widehat{\bm{B}}_{1,m}\right)_{i-s+q,j}\right) \\
    +&\frac{1}{4}\left(\left(J \frac{\partial \xi_1}{\partial x_2}\right)_{i-s,j}+\left(J \frac{\partial \xi_1}{\partial x_2}\right)_{i-s+q,j}\right)\left(\left(\widehat{\bm{B}}_{2,m}\right)_{i-s,j}+\left(\widehat{\bm{B}}_{2,m}\right)_{i-s+q,j}\right)\Bigg],
    \\
    \left(\bm{\widetilde{\mathcal{B}}}_{2,m}\right)_{i, j+\frac{1}{2}}^{{\mathtt{2pth}}}=\sum_{q=1}^{p} \alpha_{p,q} \sum_{s=0}^{q-1} \Bigg[&\frac{1}{4}\left(\left(J \frac{\partial \xi_2}{\partial x_1}\right)_{i, j-s}+\left(J \frac{\partial \xi_2}{\partial x_1}\right)_{i, j-s+q}\right)\left(\left(\widehat{\bm{B}}_{1,m}\right)_{i, j-s}+\left(\widehat{\bm{B}}_{1,m}\right)_{i, j-s+q}\right) \\
    +&\frac{1}{4}\left(\left(J \frac{\partial \xi_2}{\partial x_2}\right)_{i, j-s}+\left(J \frac{\partial \xi_2}{\partial x_2}\right)_{i, j-s+q}\right)\left(\left(\widehat{\bm{B}}_{2,m}\right)_{i, j-s}+\left(\widehat{\bm{B}}_{2,m}\right)_{i, j-s+q}\right)\Bigg].
  \end{aligned}
\end{equation*}
The discretized GCLs \eqref{eq:GCL_2D} are the same as those in \cite{Duan2022High,Zhang2023High}, which can be written as
\begin{equation}\label{eq:VCL_Dis_2D}
\begin{aligned}
  &\frac{\mathrm{d}}{\mathrm{d} t} J_{i, j}=
  -\frac{1}{\Delta \xi_1}\left(
  \left(
  \widetilde{J \frac{\partial \xi_1}{\partial t}}
  \right)_{i+\frac{1}{2}, j}^{\mathtt{2pth}}
  -\left(
  \widetilde{J \frac{\partial \xi_1}{\partial t}}
  \right)_{i-\frac{1}{2}, j}^{\mathtt{2pth}}
  \right)-\frac{1}{\Delta \xi_2}\left(\left(\widetilde{J \frac{\partial \xi_2}{\partial t}}\right)_{i, j+\frac{1}{2}}^{\mathtt{2pth}}-\left(\widetilde{J \frac{\partial \xi_2}{\partial t}}\right)_{i, j-\frac{1}{2}}^{\mathtt{2pth}}\right),
  \\
  &\frac{1}{\Delta \xi_1}
  \left(
  \left(
  \widetilde{J \frac{\partial \xi_1}{\partial x_1}}
  \right)_{i+\frac{1}{2},j}^{{\mathtt{2pth}}}
  -\left(
  \widetilde{J \frac{\partial \xi_1}{\partial x_1}}
  \right)_{i-\frac{1}{2},j}^{{\mathtt{2pth}}}
  \right) + \frac{1}{\Delta \xi_2}\left(\left(\widetilde{J \frac{\partial \xi_2}{\partial x_1}}\right)_{i, j+\frac{1}{2}}^{{\mathtt{2pth}}}-\left(\widetilde{J \frac{\partial \xi_2}{\partial x_1}}\right)_{i,j-\frac{1}{2}}^{{\mathtt{2pth}}}\right)=0,
  \\
  &\frac{1}{\Delta \xi_1}\left(\left(\widetilde{J \frac{\partial \xi_1}{\partial x_2}}\right)_{i+\frac{1}{2},j}^{{\mathtt{2pth}}}-\left(\widetilde{J \frac{\partial \xi_1}{\partial x_2}}\right)_{i-\frac{1}{2},j}^{{\mathtt{2pth}}}\right) + \frac{1}{\Delta \xi_2}\left(\left(\widetilde{J \frac{\partial \xi_2}{\partial x_2}}\right)_{i, j+\frac{1}{2}}^{{\mathtt{2pth}}}-\left(\widetilde{J \frac{\partial \xi_2}{\partial x_2}}\right)_{i,j-\frac{1}{2}}^{{\mathtt{2pth}}}\right)=0,
\end{aligned}
\end{equation}
with the discrete mesh metrics
\begin{equation}\label{eq:Mesh_Dis_Metric}
\begin{aligned}
  & \left(
  \widetilde{J \frac{\partial \xi_1}{\partial t}}
  \right)_{i+\frac{1}{2}, j}^{\mathtt{2pth}}=
  \sum_{q=1}^p \alpha_{p,q} \sum_{s=0}^{q-1} 
  \frac{1}{2}
  \left(
  \left(J \frac{\partial \xi_1}{\partial t}
  \right)_{i-s, j}
  +\left(
  J \frac{\partial \xi_1}{\partial t}
  \right)_{i-s+q, j}\right),
  \\
  & \left(\widetilde{J \frac{\partial \xi_2}{\partial t}}\right)_{i, j+\frac{1}{2}}^{\mathtt{2pth}}=\sum_{q=1}^p \alpha_{p,q} \sum_{s=0}^{q-1} \frac{1}{2}\left(\left(J \frac{\partial \xi_2}{\partial t}\right)_{i, j-s}+\left(J \frac{\partial \xi_2}{\partial t}\right)_{i, j-s+q}\right),
  \\
  &\left(\widetilde{J \frac{\partial \xi_1}{\partial x_k}}\right)_{i+\frac{1}{2},j}^{{\mathtt{2pth}}} =
  \sum_{q=1}^{p} \alpha_{p,q} \sum_{s=0}^{q-1} \frac{1}{2}\left(\left(J \frac{\partial \xi_1}{\partial x_k}\right)_{i-s,j}+\left(J \frac{\partial \xi_1}{\partial x_k}\right)_{i-s+q,j}\right),
 \\
  &\left(\widetilde{J \frac{\partial \xi_2}{\partial x_k}}\right)_{i,j+\frac{1}{2}}^{{\mathtt{2pth}}} =
  \sum_{q=1}^{p} \alpha_{p,q} \sum_{s=0}^{q-1} \frac{1}{2}\left(\left(J \frac{\partial \xi_2}{\partial x_k}\right)_{i,j-s}+\left(J \frac{\partial \xi_2}{\partial x_k}\right)_{i,j-s+q}\right),
\end{aligned}
\end{equation}
where
\begin{equation}
  \begin{aligned}\label{eq:Grid_Matrix_Coeff}
    & \left(J \frac{\partial \xi_1}{\partial t}\right)_{i, j}=-(\dot{x}_1)_{i, j}\left(J \frac{\partial \xi_1}{\partial x_1}\right)_{i, j}-(\dot{x}_2)_{i, j}\left(J \frac{\partial \xi_1}{\partial x_2}\right)_{i, j},
    \\
    & \left(J \frac{\partial \xi_2}{\partial t}\right)_{i, j}=-(\dot{x}_1)_{i, j}\left(J \frac{\partial \xi_2}{\partial x_1}\right)_{i, j}-(\dot{x}_2)_{i, j}\left(J \frac{\partial \xi_2}{\partial x_2}\right)_{i, j},
    \\
    & \left(J \frac{\partial \xi_1}{\partial x_1}\right)_{i, j}=+\sum_{q=1}^p \frac{\alpha_{p,q}}{2
    \Delta\xi_2}
    \left(\left(x_2\right)_{i, j+q}-\left(x_2\right)_{i, j-q}\right),
    \\
    & \left(J \frac{\partial \xi_1}{\partial x_2}\right)_{i, j}=-\sum_{q=1}^p \frac{\alpha_{p,q}}{2
    \Delta\xi_2}
    \left(\left(x_1\right)_{i, j+q}-\left(x_1\right)_{i, j-q}\right),
    \\
    & \left(J \frac{\partial \xi_2}{\partial x_1}\right)_{i, j}=-\sum_{q=1}^p \frac{\alpha_{p,q}}{2
    \Delta\xi_1}
    \left(\left(x_2\right)_{i+q, j}-\left(x_2\right)_{i-q, j}\right),
    \\
    & \left(J \frac{\partial \xi_2}{\partial x_2}\right)_{i, j}=+\sum_{q=1}^p \frac{\alpha_{p,q}}{2
    \Delta\xi_1}
    \left(\left(x_1\right)_{i+q, j}-\left(x_1\right)_{i-q, j}\right).
  \end{aligned}
\end{equation}
The choice of the mesh velocities $\dot{x}_1,\dot{x}_2$ follows the mesh adaptation strategy in \cite{Duan2022High,Li2022High}.

\subsection{The dissipation terms on moving meshes}\label{Sec:HOMMES_Scheme}
This section presents suitable dissipation terms added to the high-order WB EC fluxes.
The dissipation terms are designed to maintain both the WB and ES properties of the schemes.
Especially, to suppress possible oscillations due to the transportation of the discontinuous bottom topography on moving meshes, the approach in \cite{Zhang2023High} is adopted.
For the $\xi_1$-direction, the dissipation terms are
\begin{align}
  &\widehat{\bm{D}}_{i+\frac12,j} 
  = 
  \frac{1}{2}
  \widehat{\bm{S}}_{i+\frac{1}{2},j}
  \widehat{\bm{Y}}_{{i+\frac{1}{2},j} }
  \jumpangle{\bm{\widetilde{V}}}_{i+\frac{1}{2},j}^{\mathtt{WENO}},\label{eq:diss1_MM}\\
  &\mathring{\bm{D}}_{i+\frac12,j} 
  = \frac{1}{2}
  \mathring{\bm{S}}_{i+\frac{1}{2},j}
  \mathring{\bm{Y}}_{{i+\frac{1}{2},j} }
  \jumpangle{\widehat{\bm{U}}}_{i+\frac{1}{2},j}^{\mathtt{WENO}},\label{eq:diss2_MM}
\end{align}
where $\widehat{\bm{D}}_{i+\frac{1}{2},j}\in\bbR^{3M}$ is the high-order dissipation terms in curvilinear coordinates similar to \cite{Duan2022High},
and $\mathring{\bm{D}}_{i+\frac{1}{2},j}\in\bbR^{3M+1}$ stabilizes the schemes for discontinuous $b$.

In the first part \eqref{eq:diss1_MM}, $\widehat{\bm{S}}_{i+\frac{1}{2},j}$ is defined as
\begin{equation}\label{eq:diss1_S1_MM}
  \widehat{\bm{S}}_{i+\frac{1}{2},j} = \left(\widehat{\alpha}_1\right)_{i+\frac12,j}\left(\bm{T}^{-1} \bm{R}(\bm{T} {\bm{U}})\right)_{i+\frac{1}{2},j},
\end{equation}
with the rotational matrix
\begin{equation*}
\bm{T}=\begin{pmatrix}
    \bm{T}_1 & 0 & 0  \\
    0 &  \ddots & 0 \\
    0 & 0 & \bm{T}_{M} \\
  \end{pmatrix},~
  \bm{T}_m=\begin{pmatrix}
    1 & 0 & 0  \\
    0 &  \cos \varphi &  \sin \varphi  \\
    0 & -\sin \varphi & \cos \varphi \\
  \end{pmatrix},~
  \varphi=\arctan \left(\left(J \frac{\partial \xi_1}{\partial x_2}\right) \Big/\left(J \frac{\partial \xi_1}{\partial x_1}\right)\right).
\end{equation*}
The matrix $\bm{R}$ is the same as that in \eqref{eq:R_2L} and \eqref{eq:R_3L} for $M=2$ and $M=3$, respectively,
and $\bm{{U}}$ is the vector of the conservative variables for the original ML-SWEs \eqref{eq:SWE0}.
The maximal wave speed $\widehat{\alpha}_1$ in \eqref{eq:diss1_S1_MM} is
\begin{equation}\label{eq:eigens_MM}
  \widehat{\alpha}_1:=\left|J\frac{\partial \xi_1}{\partial t}+\widetilde{L}_1{\alpha_1}(\bm{T} {\bm{U}})\right|,
\end{equation}
with
$\widetilde{L}_1=\sqrt{\sum\limits_{k=1}^{2}\left(J \pd{\xi_1}{x_{k}}\right)^{2}}$, and $\alpha_1$ defined in \eqref{eq:alpha_i}.
The high-order WENO-based jumps and the diagonal matrix $\widehat{\bm{Y}}_{i+\frac{1}{2},j}$ are the same as those in Section \ref{section:ES}.

In the second part \eqref{eq:diss2_MM},
$\mathring{\bm{S}}_{i+\frac{1}{2},j}$ and $\mathring{\bm{Y}}_{i+\frac{1}{2},j}$ are both diagonal matrices,
\begin{align*}
  &\mathring{\bm{S}}_{i+\frac{1}{2},j} = \left|J\frac{\partial \xi_1}{\partial t}\right|_{i+\frac{1}{2},j}\bm{I}_{3M+1},\\
  &(\mathring{\bm{Y}}_{i+\frac{1}{2},j})_{\nu,\nu} =
  \begin{cases}
    1,~ &\text{if}~ \nu =3M-2,~3M+1, ~        \text{sign}\left(\jumpangle{\widehat{\bm{U}}_{3M-2}}_{i+\frac{1}{2},j}^{{\mathtt{WENO}}}\right)\text{sign}\left(\jump{\widehat{\bm{V}}_{3M-2}}_{i+\frac{1}{2},j} \right)\geqslant 0,
    \\
&\text{sign}\left(\jumpangle{\widehat{\bm{U}}_{3M+1}}_{i+\frac{1}{2},j}^{{\mathtt{WENO}}}\right)\text{sign}\left(\jump{\widehat{\bm{V}}_{3M+1}}_{i+\frac{1}{2},j} \right)\geqslant 0,\\
    1,~ &\text{if}~ \nu \neq 3M-2,~3M+1, ~       \text{sign}\left(\jumpangle{\widehat{\bm{U}}_\nu}_{i+\frac{1}{2},j}^{{\mathtt{WENO}}}\right)\text{sign}\left(\jump{\widehat{\bm{V}}_\nu}_{i+\frac{1}{2},j} \right)\geqslant 0,\\
    0,&\text{otherwise},
  \end{cases}
\end{align*}
and the nonlinear weights in the WENO reconstruction of $h_M$ and $b$ should be the same to maintain the lake at rest, as illustrated in  \cite{Zhang2023High}.
The high-order dissipation terms in the $\xi_2$-direction in $\left(\bm{\widetilde{\mathcal{F}}}_2\right)_{i,j+\frac{1}{2}}^{\mathtt{ES}}$ can be established similarly.

\begin{proposition}\rm\label{Prop:ES_MM}
    The high-order schemes \eqref{eq:2D_HighOrder_EC_Discrete_MM}-\eqref{eq:Grid_Matrix_Coeff} are ES,
    in the sense that the numerical solutions satisfy the semi-discrete energy inequality
    \begin{equation*}
    \frac{\mathrm{d}}{\mathrm{d} t}{\mathcal{E}}_{i, j}+\frac{1}{\Delta \xi_{1}}\left(\left({\widetilde{{\mathcal{Q}}}}_1\right)_{i+\frac{1}{2},j}^{\tt ES}-\left({\widetilde{{\mathcal{Q}}}}_1\right)_{i-\frac{1}{2},j}^{\tt ES}\right)+\frac{1}{\Delta \xi_{2}}\left(\left({\widetilde{{\mathcal{Q}}}}_2\right)_{i,j+\frac{1}{2}}^{\tt ES}-\left({\widetilde{{\mathcal{Q}}}}_2\right)_{i,j-\frac{1}{2}}^{\tt ES}\right)\leqslant 0,
  \end{equation*}
  where the numerical energy fluxes are
  \begin{align*}
    &\left({\widetilde{{\mathcal{Q}}}}_1\right)_{i+\frac{1}{2},j}^{\tt ES} = \left({\widetilde{{\mathcal{Q}}}}_1\right)_{i+\frac{1}{2},j}^{{\tt 2pth}}
    -\frac{1}{2}\left[\widehat{\alpha}_1\mean{\bm{\widetilde{V}}}^{\mathrm{T}}\widehat{\bm{Y}}\jumpangle{\bm{\widetilde{V}}}^{\tt WENO}\right]_{i+\frac{1}{2},j}
    -\frac{1}{2}\left[\left|J\frac{\partial \xi_1}{\partial t}\right|\mean{\widehat{\bm{V}}}^{\mathrm{T}}\mathring{\bm{Y}}\jumpangle{\widehat{\bm{U}}}^{\tt WENO}\right]_{i+\frac{1}{2},j},\nonumber\\
    &\left({\widetilde{{\mathcal{Q}}}}_2\right)_{i,j+\frac{1}{2}}^{\tt ES} = \left({\widetilde{{\mathcal{Q}}}}_2\right)_{i,j+\frac{1}{2}}^{\tt 2pth}
    -\frac{1}{2}\left[\widehat{\alpha}_2\mean{\bm{\widetilde{V}}}^{\mathrm{T}}\widehat{\bm{Y}}\jumpangle{\bm{\widetilde{V}}}^{\tt WENO}\right]_{i,j+\frac{1}{2}}-\frac{1}{2}\left[\left|J\frac{\partial \xi_2}{\partial t}\right|\mean{\widehat{\bm{V}}}^{\mathrm{T}}\mathring{\bm{Y}}\jumpangle{\widehat{\bm{U}}}^{\tt WENO}\right]_{i,j+\frac{1}{2}}.\nonumber
  \end{align*}
  The first parts come from the EC fluxes on moving meshes,
  \begin{align*}
    \left(\widetilde{\mathcal{Q}}_1\right)_{i+\frac{1}{2},j}^{\tt 2pth }=\sum_{q=1}^p \alpha_{p, q} \sum_{s=0}^{q-1} \widetilde{\mathcal{Q}}_1\left(\widehat{\bm{U}}_{i-s,j}, \widehat{\bm{U}}_{i-s+q,j},\left(J \frac{\partial \xi_1}{\partial \zeta}\right)_{i-s,j},\left(J \frac{\partial \xi_1}{\partial \zeta}\right)_{i-s+q,j}\right),
    \\
    \left(\widetilde{\mathcal{Q}}_2\right)_{i,j+\frac{1}{2}}^{\tt 2pth }=\sum_{q=1}^p \alpha_{p, q} \sum_{s=0}^{q-1} \widetilde{\mathcal{Q}}_2\left(\widehat{\bm{U}}_{i,j-s}, \widehat{\bm{U}}_{i,j-s+q},\left(J \frac{\partial \xi_2}{\partial \zeta}\right)_{i,j-s},\left(J \frac{\partial \xi_2}{\partial \zeta}\right)_{i,j-s+q}\right),
  \end{align*}
  with $\zeta = t,~x_1, ~x_2$,
  and
  \begin{align*}
    \widetilde{\mathcal{Q}}_1 =
    \ &\frac{1}{2}\left(\widehat{\bm{V}}(\widehat{\bU}_{L})+\widehat{\bm{V}}(\widehat{\bU}_{R})\right)^{\mathrm{T}}\widetilde{\bm{\mathcal{F}}}
    _{1}
    - \frac{1}{2}\left(
    \left(
    J \frac{\partial \xi_{1}}{\partial t}
    \right)_L+
    \left(
    J \frac{\partial \xi_{1}}{\partial t}
    \right)_R
    \right)
    \left(
    \widehat{\phi}\left(\widehat{\bm{U}}_{L}\right)+\widehat{\phi}\left(\widehat{\bm{U}}_{R}\right)
    \right) \nonumber\\
    &- \sum_{k=1}^{2}\frac{1}{4}\left(\left(J \frac{\partial \xi_{1}}{\partial x_k}\right)_L+\left(J \frac{\partial \xi_{1}}{\partial x_k}\right)_R\right)\left(\widehat{\psi}_{k}\left(\widehat{\bm{U}}_{L}\right)+\widehat{\psi}_{k}\left(\widehat{\bm{U}}_{R}\right)\right) \nonumber\\
    & +\sum_{m=1}^{M}\sum_{k=1}^{2} \frac{g\rho_m}{8} \left(\left(J \frac{\partial \xi_1}{\partial x_k}\right)_L+\left(J \frac{\partial \xi_1}{\partial x_k}\right)_R\right)\left(\left(hu_m\right)_{L}+\left(hu_m\right)_{R}\right)\left(\left(z_m
    \right)_{L}+
    \left(z_m
    \right)_{R}\right)\\
    &-\sum_{m=1}^{M}\sum_{k=1}^{2} \frac{g\rho_m}{4} \left(\left(J \frac{\partial \xi_1}{\partial x_k}\right)_L+\left(J \frac{\partial \xi_1}{\partial x_k}\right)_R\right)\left(\left(hu_mz_m\right)_{L}+\left(hu_mz_m\right)_{R}\right),
    \\
      \widetilde{\mathcal{Q}}_2 =
    \ &\frac{1}{2}\left(\widehat{\bm{V}}(\widehat{\bU}_{L})+\widehat{\bm{V}}(\widehat{\bU}_{R})\right)^{\mathrm{T}}\widetilde{\bm{\mathcal{F}}}
    _2
    - \frac{1}{2}\left(
    \left(
    J \frac{\partial \xi_2}{\partial t}
    \right)_L+
    \left(
    J \frac{\partial \xi_2}{\partial t}
    \right)_R
    \right)
    \left(
    \widehat{\phi}\left(\widehat{\bm{U}}_{L}\right)+\widehat{\phi}\left(\widehat{\bm{U}}_{R}\right)
    \right) \nonumber \nonumber\\
    &- \sum_{k=1}^{2}\frac{1}{4}\left(\left(J \frac{\partial \xi_2}{\partial x_k}\right)_L+\left(J \frac{\partial \xi_2}{\partial x_k}\right)_R\right)\left(\widehat{\psi}_{k}\left(\widehat{\bm{U}}_{L}\right)+\widehat{\psi}_{k}\left(\widehat{\bm{U}}_{R}\right)\right) \nonumber\\
    & +\sum_{m=1}^{M}\sum_{k=1}^{2} \frac{g\rho_m}{8} \left(\left(J \frac{\partial \xi_2}{\partial x_k}\right)_L+\left(J \frac{\partial \xi_2}{\partial x_k}\right)_R\right)\left(\left(hv_m\right)_{L}+\left(hv_m\right)_{R}\right)\left(\left(z_m
    \right)_{L}+
    \left(z_m
    \right)_{R}\right)
    \\
    &-\sum_{m=1}^{M}\sum_{k=1}^{2} \frac{g\rho_m}{4} \left(\left(J \frac{\partial \xi_2}{\partial x_k}\right)_L+\left(J \frac{\partial \xi_2}{\partial x_k}\right)_R\right)\left(\left(hv_mz_m\right)_{L}+\left(hv_mz_m\right)_{R}\right).
  \end{align*}
\end{proposition}

To get the fully-discrete schemes on moving meshes, this paper uses the explicit SSP-RK3 time discretization 
\begin{equation*}
  \begin{aligned}
    &\bm{\mathcal{U}}^{*} = \ \bm{\mathcal{U}}^{n}+\Delta t^n \bm{\mathcal{L}}\left(\bU^{n}, \bx^n\right),~
    J^{*} = J^{n}+\Delta t^n {\mathcal{L}}\left(
    \bx^n\right),\\
    &\bx^{*} = \bx^{n}+\Delta t^n \dot{\bx}^n,
    \\
    &\bm{\mathcal{U}}^{**} = \ \frac{3}{4}\bm{\mathcal{U}}^{n}+\frac{1}{4}\left(\bm{\mathcal{U}}^{*}+\Delta t^n \bm{\mathcal{L}}\left(\bU^{*},
    {\bx}^*\right)\right),~
    J^{**} = \frac{3}{4}J^{n}+\frac{1}{4}\left(J^{*}+\Delta t^n {\mathcal{L}}\left(
    {\bx}^*\right)\right),\\
    &\bx^{**} = \frac{3}{4}\bx^{n}+\frac{1}{4}\left(\bx^{*} + \Delta t^n\dot{\bx}^n\right),
    \\
    &\bm{\mathcal{U}}^{n+1} = \ \frac{1}{3}\bm{\mathcal{U}}^{n}+\frac{2}{3}\left(\bm{\mathcal{U}}^{**}+\Delta t^n \bm{\mathcal{L}}\left(\bU^{**},
    {\bx}^{**}\right)\right),~
    J^{n+1} = \frac{1}{3}J^{n}+\frac{2}{3}\left(J^{**}+\Delta t^n {\mathcal{L}}\left(
    {\bx}^{**}\right)\right),\\
    &\bx^{n+1} = \frac{1}{3}\bx^{n}+\frac{2}{3}\left(\bx^{**} + \Delta t^n{\bx}^n\right),
  \end{aligned}
\end{equation*}
where $\dot{\bx}^n$ is the mesh velocity determined by using the adaptive moving mesh strategy in \cite{Duan2022High,Zhang2023High},
$\bm{\mathcal{L}}$ is the right-hand side of \eqref{eq:2D_HighOrder_EC_Discrete_MM}
for the semi-discrete high-order WB ES schemes on moving meshes,
and ${\mathcal{L}}$ is the right-hand side of the semi-discrete VCL, i.e., the first equation in \eqref{eq:VCL_Dis_2D}.
The time stepsize is given by 
\begin{equation}\label{eq:dt_MM}
	\Delta t^{n} = \frac{C_{\tt CFL}}{\sum\limits_{\ell = 1}^{2}\max\limits_{i,j} \left(\widehat{\alpha}_{\ell}\right)_{i,j}/\Delta_{{\xi}_{\ell}}},
\end{equation}
with $\widehat{\alpha}_{1}$ defined in \eqref{eq:eigens_MM} for the $\xi_1$-direction.

\begin{proposition}\rm\label{Prop:WB_MM}
    The fully-discrete high-order schemes obtained by using the semi-discrete schemes \eqref{eq:2D_HighOrder_EC_Discrete_MM}-\eqref{eq:Grid_Matrix_Coeff} with the forward Euler or explicit SSP-RK3 time discretization preserve the lake at rest, in the sense that, if the mesh does not interleave during the computation, then for the initial data
     \begin{equation*}
    \left(h_m + z_m\right)_{i,j}^{0} \equiv C_m,
    ~(u_m)_{i,j}^{0} = (v_m)_{i,j}^{0} = 0,~ m =1,\cdots,M,
    ~\forall i,j,
  \end{equation*}
  the lake at rest is also satisfied at $t^n$,
  \begin{equation*}
      \left(h_m + z_m\right)_{i,j}^{n} \equiv C_m,
    ~(u_m)_{i,j}^{n}=(v_m)_{i,j}^{n} = 0,~ m =1,\cdots,M,
    ~\forall i,j,
  \end{equation*}
  where $C_m$ is constant.
\end{proposition}
The proofs of Propositions \ref{Prop:ES_MM} and \ref{Prop:WB_MM} are similar to those in \cite{Zhang2023High}, so they are skipped here for brevity.

\section{Numerical results}\label{section:Result}
In this section, several 1D and 2D numerical experiments are conducted to demonstrate the performance of our high-order accurate WB ES schemes for the two- and three-layer SWEs,
in which the amount of dissipation depends on an estimation of the maximal wave speed, given in Section \ref{subsec:Upper/Lower Bound}.
The two-layer test problems mainly come from literature, while the three-layer examples are newly adapted from the two-layer cases.
The reconstruction in the dissipation terms adopts the $5$th-order WENO-Z reconstruction \cite{Borges2008An}.
The time stepsize is defined in \eqref{eq:dt_fixed} and \eqref{eq:dt_MM} for the fixed and moving meshes, respectively. 
In the accuracy tests, the time stepsize is taken as $C_{\text{\tiny \tt CFL}} (\min_{\ell}\Delta x_{\ell})^{5/3}$ (for fixed mesh) or $C_{\text{\tiny \tt CFL}} (\min_{\ell}\Delta \xi_{\ell})^{5/3}$ (for moving mesh) to make the spatial errors dominant. Unless otherwise specified, the CFL number is chosen as $C_{\mathtt{CFL}} = 0.4$.
For the mesh adaptation, the total number of the iterations for the mesh equations is $10$,
and the monitor function $\omega$ will be given in the examples, see \cite{Duan2022High,Zhang2023High} for more details on the adaptive moving mesh strategy.
Our WB ES schemes on the fixed uniform mesh and adaptive moving mesh will be denoted by 
``\texttt{UM-ES}''
and ``\texttt{MM-ES}''.
In 1D cases, $x$ and $\xi$ will replace $x_1$ and $\xi_1$, respectively.

\subsection{Estimation of the maximal wave speed used in the dissipation terms}\label{subsec:Upper/Lower Bound}

Similar to \cite{Kurganov2009Central}, based on the estimation of polynomial roots by Lagrange \cite{Mignotte2002Estimation}, it is possible to derive upper and lower bounds for the roots of the characteristic polynomial, which can be used as an estimation of the maximal wave speed.
To be specific, when $M=2$ the characteristic polynomial of the Jacobian matrix along the $x_1$-direction for the two-layer SWEs equals $(\lambda - u_1)(\lambda-u_2) \mathcal{P}_{2}(\lambda)$, with 
	\begin{equation*}
		\mathcal{P}_{2}(\lambda) = \lambda^{4}+c_1 \lambda^{3}+c_2 \lambda^{2}+c_3 \lambda^{1}+c_4,
	\end{equation*}
	where 
 \begin{equation}\label{eq:Two_layer_eigen_eq}
	\begin{aligned}
		c_1 = ~&-2(u_1+u_2),\\
		c_2 = ~&(u_1+u_2)^2+2u_1u_2 -g(h_1+h_2),\\
		c_3 = ~&- 2u_1^2u_2 - 2u_1u_2^2 + 2gh_2u_1 + 2gh_1u_2,\\
		c_4 = ~&u_1^2u_2^2 - gh_1u_2^2 - gh_2u_1^2 + g^2h_1h_2 - g^2h_1h_2r_{12}.
	\end{aligned}
 \end{equation}
    For $M=3$, the characteristic polynomial is $(\lambda - u_1)(\lambda-u_2) (\lambda-u_3)\mathcal{P}_{3}(\lambda)$, with 
	\begin{equation*}
		\mathcal{P}_{3}(\lambda) = \lambda^{6} +c_1 \lambda^{5}+c_2 \lambda^{4}+c_3 \lambda^{3}+c_4 \lambda^{2}+c_5 \lambda^{1}+c_6,
	\end{equation*}
    where 
\begin{equation}\label{eq:Three_layer_eigen_eq}
    \begin{aligned}
    	c_1 =~& -2(u_1+u_2+u_3),\\
    	c_2 =~& (u_1+u_2)^2+(u_2+u_3)^{2}+2u_1u_2+2u_2u_3 -g(h_1+h_2+h_3),\\
    	c_3 =~& 2 g (h_1 (u_2+u_3)+h_2 (u_1+u_3)+h_3 (u_1+u_2))\\
     &-2 \left(u_1^2 (u_2+u_3)+u_1 \left(u_2^2+4 u_2 u_3+u_3^2\right)+u_2 u_3 (u_2+u_3)\right), \\
    	c_4 = ~& g^2 (h_1 (h_2-h_3 {r_{13}}+h_3)-h_2 h_3 ({r_{23}}-h_2 {r_{12}}-1))-g \big(h_1 \left(u_2^2+4 u_2 u_3+u_3^2\right)
    	\\&+h_2 \left(u_1^2+4 u_1 u_3+u_3^2\right)+h_3 \left(u_1^2+4 u_1 u_2+u_2^2\right)\big)
    	\\& +u_1^2 \left(u_2^2+4 u_2 u_3+u_3^2\right)+4 u_1 u_2 u_3 (u_2+u_3)+u_2^2 u_3^2,\\
    	c_5 =~& 2 \big(g^2 (h_1 h_2 ({r_{12}}-1) u_3+h_1 h_3 ({r_{13}}-1) u_2+h_2 h_3 ({r_{23}}-1) u_1)\\
    	&+g (h_1 u_2 u_3 (u_2+u_3)+h_2 u_1 u_3 (u_1+u_3)+h_3 u_1 u_2 (u_1+u_2))
     \\&-u_1 u_2 u_3 (u_1 (u_2+u_3)+u_2 u_3)\big), \\
    	c_6 =~&  g^3 h_1 h_2 h_3 ({r_{12}}+{r_{23}}-{r_{12}} {r_{23}}-1)+g^2 \big(-h_1 h_2 ({r_{12}}-1) u_3^2
     \\&-h_1 h_3 ({r_{13}}-1) u_2^2-h_2 h_3 ({r_{23}}-1) u_1^2\big)\\
    	&-g \left(u_3^2 \left(h_1 u_2^2+h_2 u_1^2\right)+h_3 u_1^2 u_2^2\right)+u_1^2 u_2^2 u_3^2.
        \end{aligned}
\end{equation}
According to the results in \cite{Mignotte2002Estimation}, the largest positive root of $\mathcal{P}_{2}$ or $\mathcal{P}_{3}$ is smaller than the sum of the first two largest values in the set $\{\sqrt[j]{\abs{c_{j}}},~j\in J_{\max}\}$, where $J_{\max}$ is the set of the negative coefficients of 	\eqref{eq:Two_layer_eigen_eq} or \eqref{eq:Three_layer_eigen_eq}. Meanwhile, the smallest nonpositive root is larger than the sum of the first two smallest numbers in the set $\{-\sqrt[j]{\abs{d_{j}}},~j\in J_{\min}\}$, where $J_{\min}$ is the set of the negative coefficients of the polynomial
\begin{equation*}
	  	\lambda^{4} +d_1 \lambda^{3}+d_2 \lambda^{2}+d_3 \lambda^{1}+d_4
	\end{equation*}
 for $\mathcal{P}_2$, or
	\begin{equation*}
	  	\lambda^{6} +d_1 \lambda^{5}+d_2 \lambda^{4}+d_3 \lambda^{3}+d_4 \lambda^{2}+d_5 \lambda^{1}+d_6
	\end{equation*}
    for $\mathcal{P}_3$ with $d_{j} = (-1)^{j}c_{j}$.
    Thus one can obtain the upper and lower bounds of the eigenvalues, denoted by $\lambda_{\max}(h_m,u_m)$ and $\lambda_{\min}(h_m,u_m)$,
    and the maximal wave speed in the $x_1$-direction is taken as
    \begin{equation}\label{eq:alpha_i}
    	\alpha_1 = \max\left\{\lambda_{\max}(h_m,u_m),~\abs{\lambda_{\min}(h_m,u_m)},~\abs{u_m}\right\}, ~ m=1,\dots,M.
    \end{equation}
    For the $x_2$-direction, $\alpha_2$ can be obtained by replacing $u_m$ with $v_m$ in \eqref{eq:alpha_i}.

\begin{remark}\rm
    For $M>3$, the characteristic polynomial along $x_1$-direction can be written as $\prod_{m=1}^{M}(\lambda-u_m)\mathcal{P}_{M}(\lambda)$, with $\mathcal{P}_{M}(\lambda) = \sum_{m=0}^{2M}c_{2M-m}\lambda^{m}$.
    Once the coefficients $c_{m}$ are obtained, the estimation proceeds similarly.
    However, it is difficult to give uniform expressions of those coefficients, so one needs to calculate the coefficients for each case separately.
\end{remark}

\subsection{1D tests}
\begin{example}[Accuracy test with manufactured solutions]\label{ex:1DSmooth}\rm
  This example is first used to test the convergence rates of our schemes.
  The manufactured solutions are constructed in the physical domain $[0,2]$ with periodic boundary conditions for the ML-SWEs with extra source terms
  \begin{equation*}
    \pd{\bU}{t} + \pd{\bF}{x} = -
    \sum\limits_{m=1}^{M}gh_m\pd{\bm{B}_m}{x} + \bm{S}.
  \end{equation*}
 For the two-layer case, the exact solutions are chosen as
  \begin{align*}
    &h_1(x,t) = \cos\left(\pi t\right)\cos\left(\pi x\right)+6,~u_1(x,t) = \frac{\sin\left(\pi t\right)\sin\left(\pi x\right)}{\cos\left(\pi t\right)\cos\left(\pi x\right)+6}, \\
    &h_2(x,t) = \cos\left(\pi t\right)\cos\left(\pi x\right)+4,~u_2(x,t) = \frac{\sin\left(\pi t\right)\sin\left(\pi x\right)}{\cos\left(\pi t\right)\cos\left(\pi x\right)+4},\\
    &b(x) = \sin\left(\pi x\right)+\frac{3}{2},
  \end{align*}
  while for the three-layer case, they are
  \begin{align*}
  	  &h_1(x,t) = \cos\left(\pi t\right)\cos\left(\pi x\right)+8,~u_1(x,t) = \frac{\sin\left(\pi t\right)\sin\left(\pi x\right)}{\cos\left(\pi t\right)\cos\left(\pi x\right)+8}, \\
  	&h_2(x,t) = \cos\left(\pi t\right)\cos\left(\pi x\right)+6,~u_2(x,t) = \frac{\sin\left(\pi t\right)\sin\left(\pi x\right)}{\cos\left(\pi t\right)\cos\left(\pi x\right)+6},\\
  	 &h_3(x,t) = \cos\left(\pi t\right)\cos\left(\pi x\right)+4,~u_3(x,t) = \frac{\sin\left(\pi t\right)\sin\left(\pi x\right)}{\cos\left(\pi t\right)\cos\left(\pi x\right)+4},\\
  	&b(x) = \sin\left(\pi x\right)+\frac{3}{2}.
  \end{align*}
  The expressions of the source terms $\bm{S}$ can be obtained by some algebraic manipulations, with those for the two-layer case given in  \ref{Sec:Exact_Solution}.
  In both cases, the gravitational acceleration constant is $g=1$, and the tests are solved up to $t = 0.1$.
  The density ratio for the two-layer case is $r_{12} = 7/10$, $\rho_2=1$, while for the three-layer case, $r_{12} = 7/10$, $r_{13} = 7/13$, $r_{23} = 10/13,$ $ \rho_3 = 13/10$.
  The monitor function used in the mesh adaptation is
  \begin{equation*}
      \omega = \left(1+\theta\left(\frac{\left|\nabla_{\xi}\sigma\right|}{\max\left|\nabla_{\xi}\sigma\right|}\right)^2\right)^{\frac{1}{2}},
  \end{equation*}
  with $\theta = 1$ and $\sigma = h_2+b$.
\end{example}

Figure \ref{1D_Smooth_Order} shows the errors and convergence rates in velocity $u_2$ at $t =0.1$, which verify that our schemes can achieve the designed $5$th-order accuracy for both two-layer and three-layer cases.

 \begin{figure}[!htb]
 	\centering
         \begin{subfigure}[b]{0.4\textwidth}
         \centering
         \includegraphics[width=1.0\textwidth,  clip]{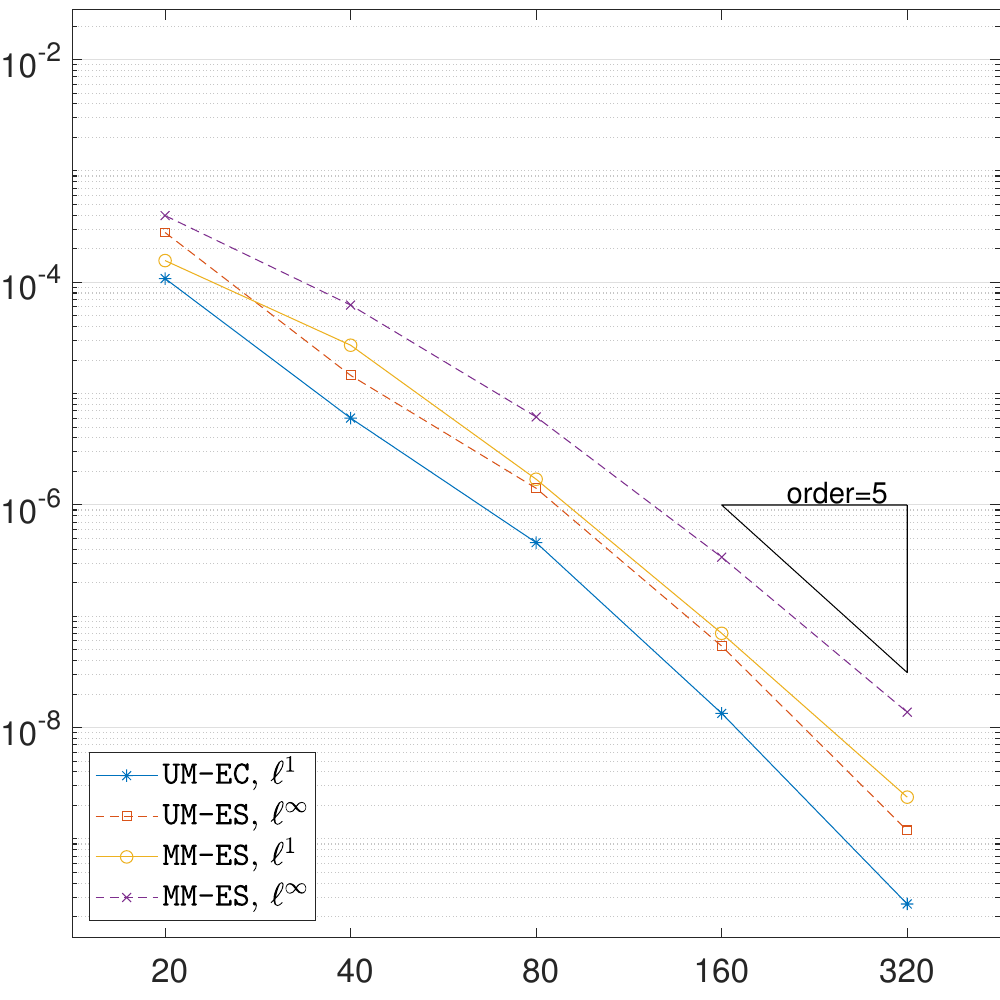}
     \end{subfigure}
     \begin{subfigure}[b]{0.4\textwidth}
 	\centering
         \includegraphics[width=1.0\textwidth]{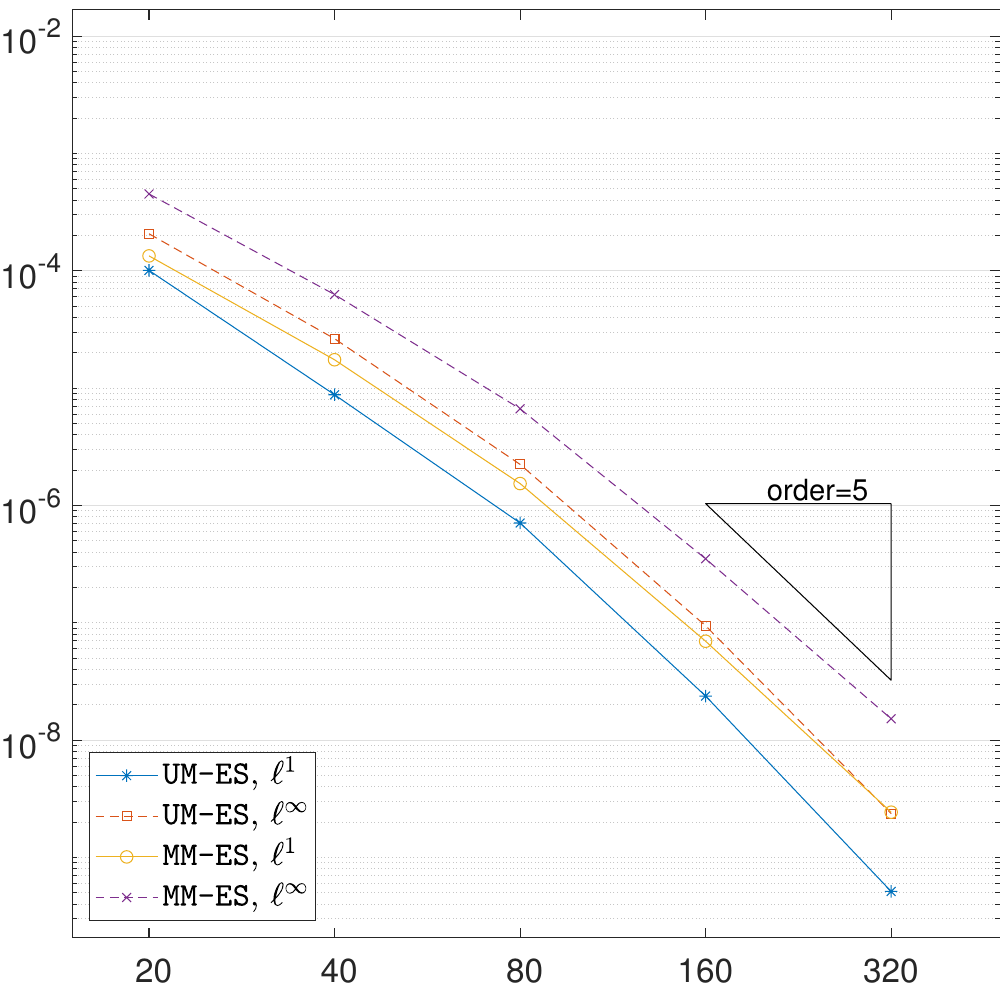}
      \end{subfigure}
      \caption{Example \ref{ex:1DSmooth}. The errors and convergence orders in velocity $u_2$ at $t =0.1$ obtained by our schemes. left: the two-layer case, right: the three-layer case.}\label{1D_Smooth_Order}
 \end{figure}

\begin{example}[1D WB test]\label{ex:1D_WB_Test}\rm
  This example is used to verify the WB properties of our schemes for the 1D ML-SWEs.
  The physical domain is $[0, 20]$ with  outflow boundary conditions,
  and the bottom topography is a smooth Gaussian profile
  \begin{equation}\label{eq:b_Smooth}
    b(x) = 2 e^{-(x-9)^2/2} + 3e^{-(x-11.5)^2},
  \end{equation}
  or a discontinuous square step
  \begin{equation}\label{eq:b_dis}
    b(x)= \begin{cases}2,~&\text{if}\quad x \in[9,13], \\ 0,~&\text{otherwise}.\end{cases}
  \end{equation}
  The initial data are $h_1 = 6-h_2 - b$, $h_2 = 4 - b$ for the two-layer case with the density ratio $r_{12} = 4/5$, $\rho_2=1$,
  while $h_1 = 8-h_2-h_3 - b$, $h_2 = 6-h_3 - b$, $h_3 = 4 - b$ for the three-layer case with $r_{12} = 4/5$, $r_{13} = 2/3$, $r_{23} = 5/6$, $\rho_{3} = 6/5$.
  The velocities are zero, and the problem is simulated up to $t = 0.2$ with the gravitational acceleration constant taken as $g = 1$.
  The monitor function is the same as that in Example \ref{ex:1DSmooth}, but with $\theta = 100$, and $\sigma = h_2$, $h_3$ for the two-layer and three-layer cases, respectively. 
\end{example}

Tables \ref{1D Well_Balance}-\ref{tb:1D Well_Balance_3L} give the $\ell^{1}$ and $\ell^{\infty}$ errors in $h_1+h_2+b$ and $h_2+b$ for the two-layer case, and the errors in $h_1+h_2+h_3+b$, $h_2+h_3+b$, and $h_3+b$ for the three-layer case with 50 mesh points, respectively.
The errors are all at the level of rounding error in double precision, showing that our schemes are WB.
The water surface levels, bottom topography $b$, and the mesh trajectories are plotted in Figures \ref{1D_Well_balance}-\ref{1D_Well_balance_3L}.
One can observe that our schemes can preserve the lake at rest well, and the mesh points concentrate near the sharp transitions of the bottom topography.

\begin{table}[!htb]
  \centering
  \begin{tabular}{cccccc}
   \hline\hline
    \multicolumn{2}{c}{\multirow{2}{*}{}} & \multicolumn{2}{c}{\texttt{UM-ES}} & \multicolumn{2}{c}{\texttt{MM-ES}} \\ 
    \cmidrule(lr){3-4}\cmidrule(lr){5-6}
    \multicolumn{2}{c}{} & \multicolumn{1}{c}{$\ell^{1}$~error}  & \multicolumn{1}{c}{$\ell^{\infty}$~error} &  \multicolumn{1}{c}{$\ell^{1}$~error}  & \multicolumn{1}{c}{$\ell^{\infty}$~error}  \\
    \hline
    \multirow{2}{*}{$b$ in \eqref{eq:b_Smooth}} &
 $h_1+h_2+b$ & 9.95e-15 	 & 2.66e-15 	 & 3.02e-14 	 & 5.33e-15  \\ 
 	 &  $h_2+b$ & 1.03e-14 	 & 3.55e-15 	 & 2.50e-14 	 & 3.55e-15  \\ 
    \hline
    \multirow{2}{*}{$b$ in \eqref{eq:b_dis}} &
 $h_1+h_2+b$ & 2.13e-15 	 & 1.78e-15 	 & 3.91e-14 	 & 4.44e-15  \\ 
 	&   $h_2+b$ & 1.07e-15 	 & 1.78e-15 	 & 3.80e-14 	 & 5.33e-15  \\ 
  \hline\hline
  \end{tabular}
  \caption{Example \ref{ex:1D_WB_Test} for the two-layer case.
  Errors in the water surface levels obtained by using our schemes with $50$ mesh points at $t=0.2$, with the bottom topography \eqref{eq:b_Smooth} and \eqref{eq:b_dis}.}\label{1D Well_Balance}
\end{table}

\begin{table}[!htb]
	\centering
	\begin{tabular}{cccccc}
    \hline\hline
    \multicolumn{2}{c}{\multirow{2}{*}{}} & \multicolumn{2}{c}{\texttt{UM-ES}} & \multicolumn{2}{c}{\texttt{MM-ES}} \\ 
    \cmidrule(lr){3-4}\cmidrule(lr){5-6}
    \multicolumn{2}{c}{} & \multicolumn{1}{c}{$\ell^{1}$~error}  & \multicolumn{1}{c}{$\ell^{\infty}$~error} &  \multicolumn{1}{c}{$\ell^{1}$~error}  & \multicolumn{1}{c}{$\ell^{\infty}$~error}  \\
    \hline
		\multirow{3}{*}{$b$ in \eqref{eq:b_Smooth}} &
	 $h_1+h_2+h_3+b$ & 1.99e-14 	 & 3.55e-15 	 & 7.03e-14 	 & 1.07e-14  \\ 
 	 &  $h_2+h_3+b$ & 2.31e-14 	 & 3.55e-15 	 & 5.97e-14 	 & 1.07e-14  \\ 
 	   &$h_3+b$ & 1.76e-14 	 & 3.55e-15 	 & 5.05e-14 	 & 8.88e-15  \\ 
		\hline
		\multirow{3}{*}{$b$ in \eqref{eq:b_dis}} &
		 $h_1+h_2+h_3+b$ & 4.62e-15 	 & 1.78e-15 	 & 4.90e-14 	 & 8.88e-15  \\ 
 	 &  $h_2+h_3+b$ & 5.33e-15 	 & 2.67e-15 	 & 5.29e-14 	 & 7.11e-15  \\ 
 	  & $h_3+b$ & 8.88e-15 	 & 3.55e-15 	 & 4.67e-14 	 & 7.11e-15  \\ 
	\hline\hline
	\end{tabular}
	\caption{Example \ref{ex:1D_WB_Test} for the three-layer case.
 Errors in the water surface levels obtained by using our schemes with $50$ mesh points at $t=0.2$, with the bottom topography \eqref{eq:b_Smooth} and \eqref{eq:b_dis}.}\label{tb:1D Well_Balance_3L}
\end{table}

\begin{figure}[!htb]
  \centering
  \begin{subfigure}[b]{0.35\textwidth}
    \centering
    \includegraphics[width=1.0\textwidth]{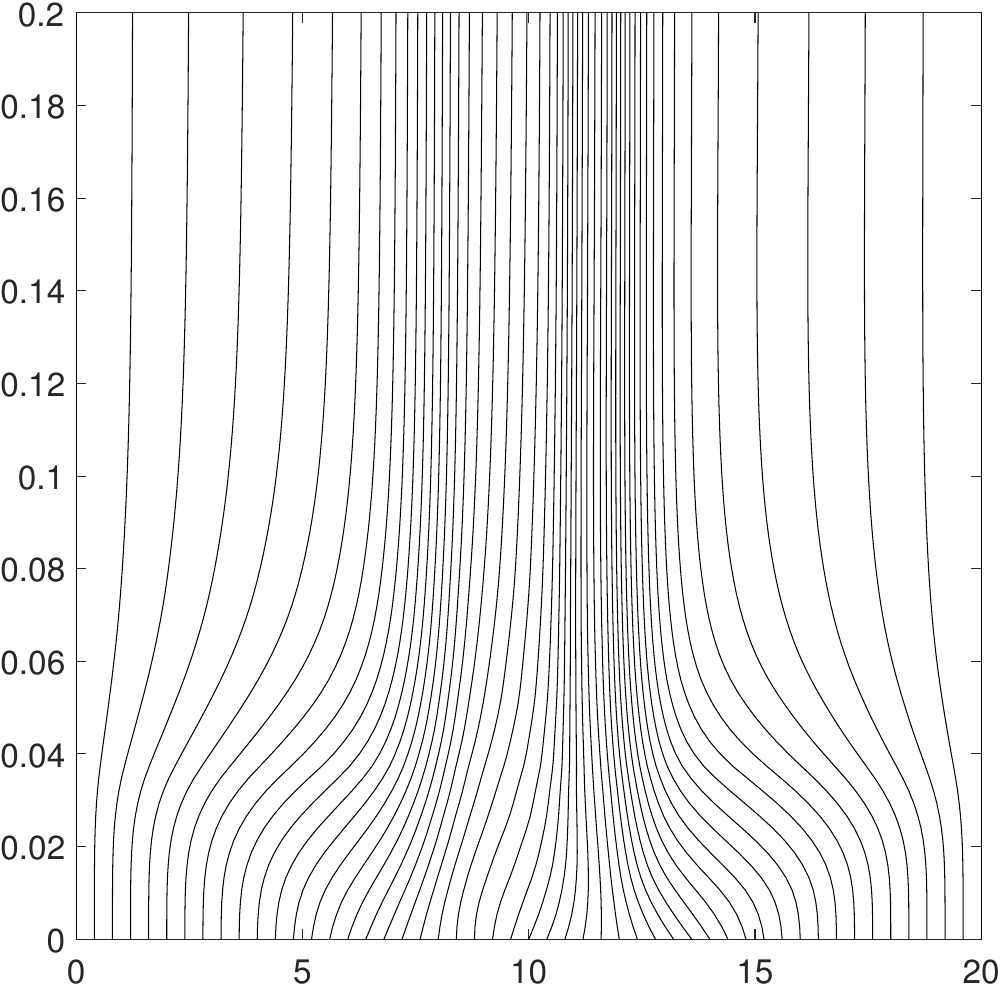}
  \end{subfigure}
  \begin{subfigure}[b]{0.35\textwidth}
    \centering
    \includegraphics[width=1.0\textwidth,  clip]{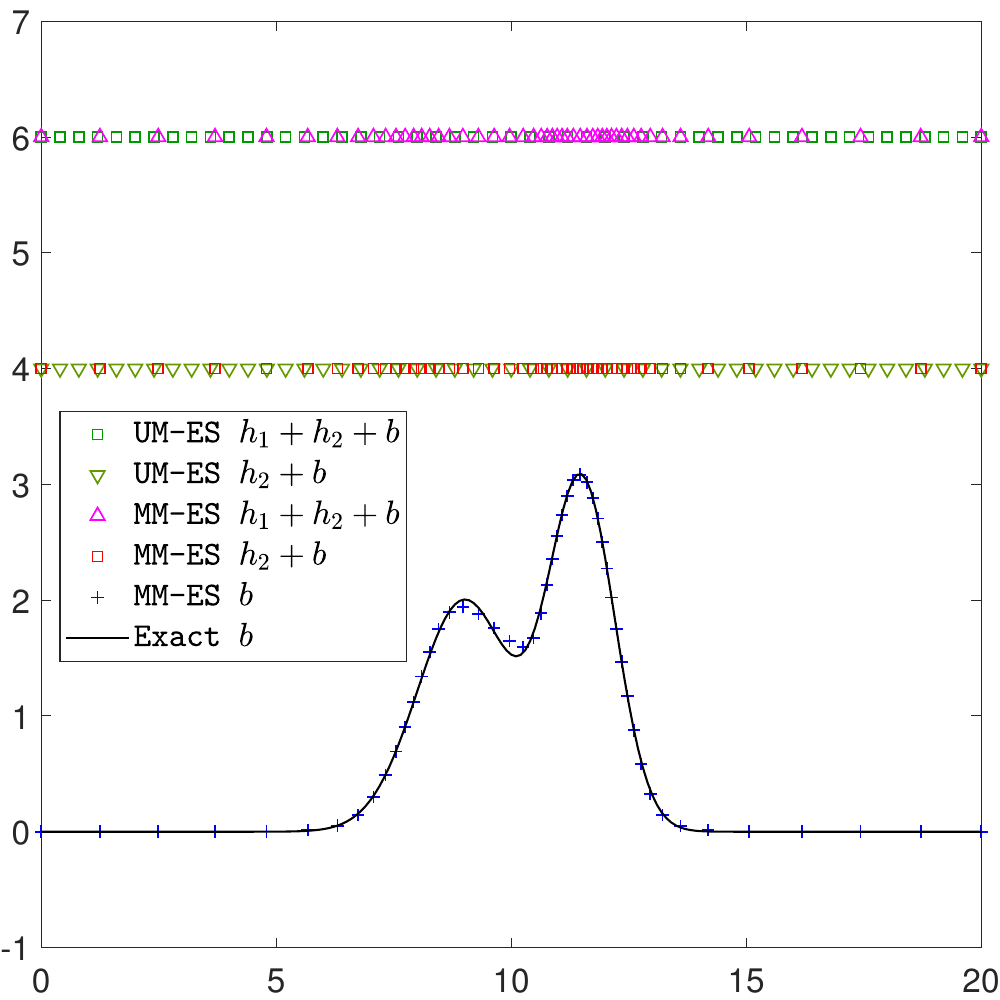}
  \end{subfigure}
  \\
  \begin{subfigure}[b]{0.35\textwidth}
	\centering
	\includegraphics[width=1.0\textwidth]{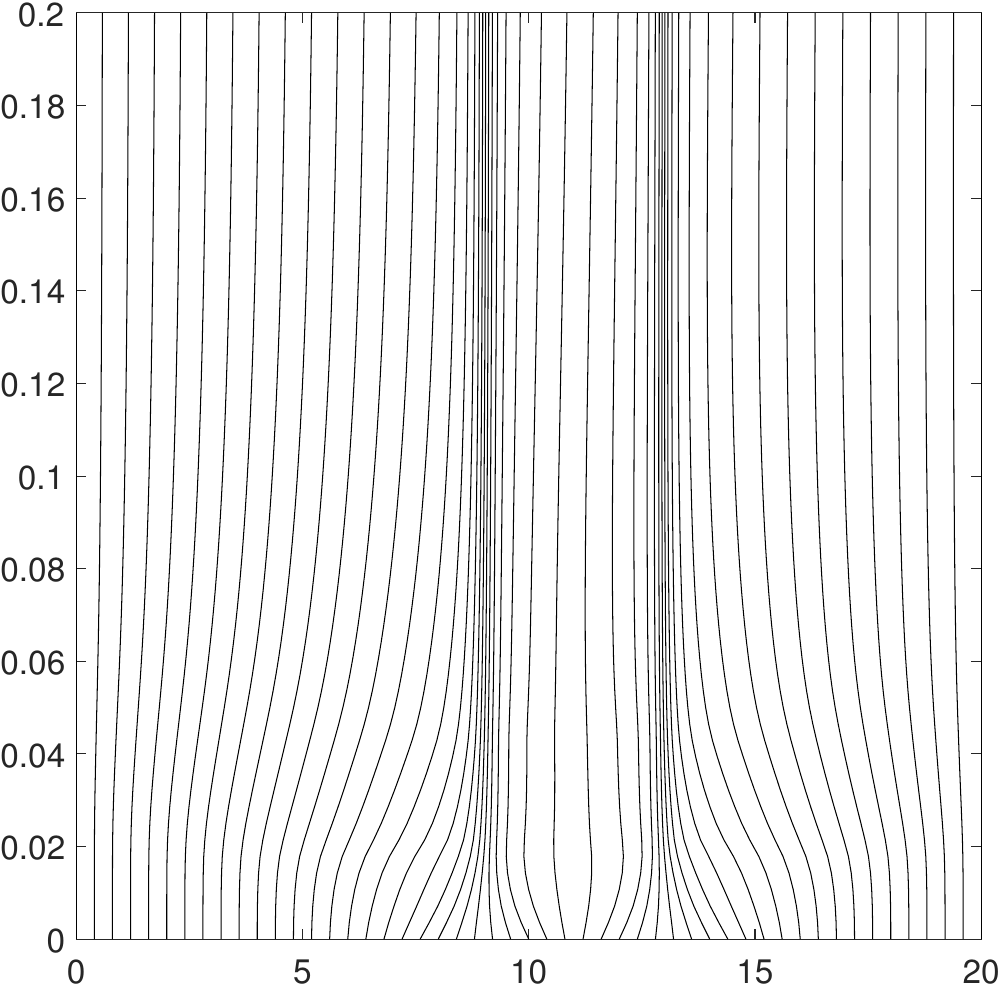}
\end{subfigure}
\begin{subfigure}[b]{0.35\textwidth}
	\centering
	\includegraphics[width=1.0\textwidth,  clip]{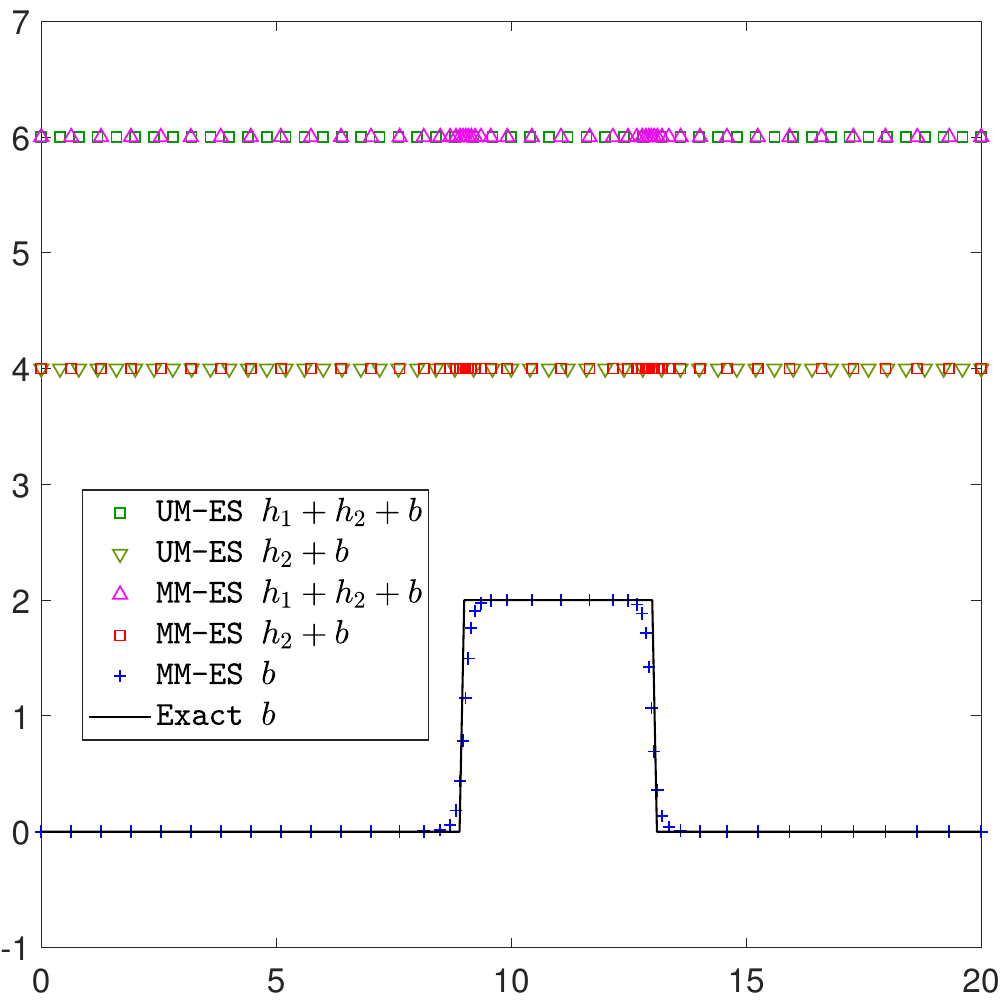}
\end{subfigure}
  \caption{Example \ref{ex:1D_WB_Test} for the two-layer case.
  Left: the mesh trajectories obtained by the \texttt{MM-ES} scheme,
  right: the bottom topography $b$ and water surface levels.
    Top: with the bottom topography \eqref{eq:b_Smooth},
    bottom: with the bottom topography \eqref{eq:b_dis}.
    The results are obtained with $50$ mesh points at $t=0.2$.
  }
  \label{1D_Well_balance}
\end{figure}

\begin{figure}[!htb]
  \centering
  \begin{subfigure}[b]{0.35\textwidth}
    \centering
    \includegraphics[width=1.0\textwidth]{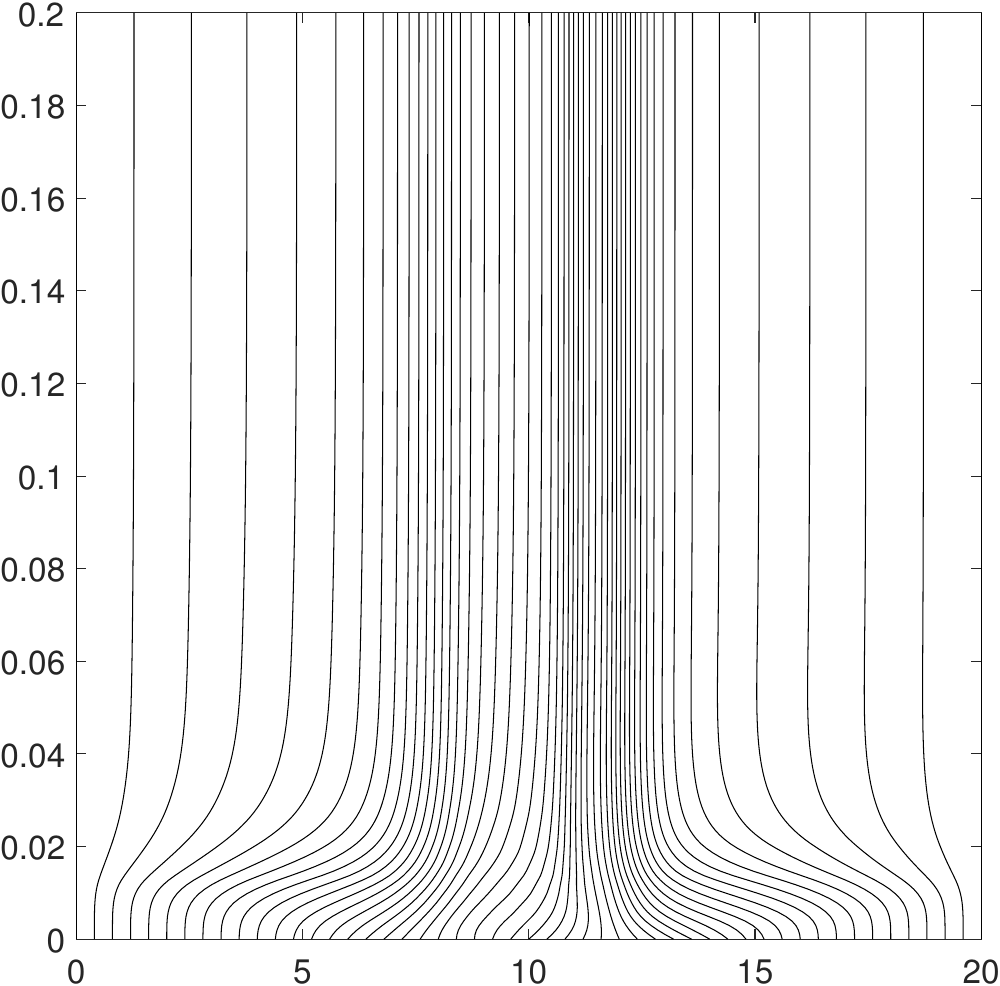}
  \end{subfigure}
  \begin{subfigure}[b]{0.35\textwidth}
    \centering
    \includegraphics[width=1.0\textwidth,  clip]{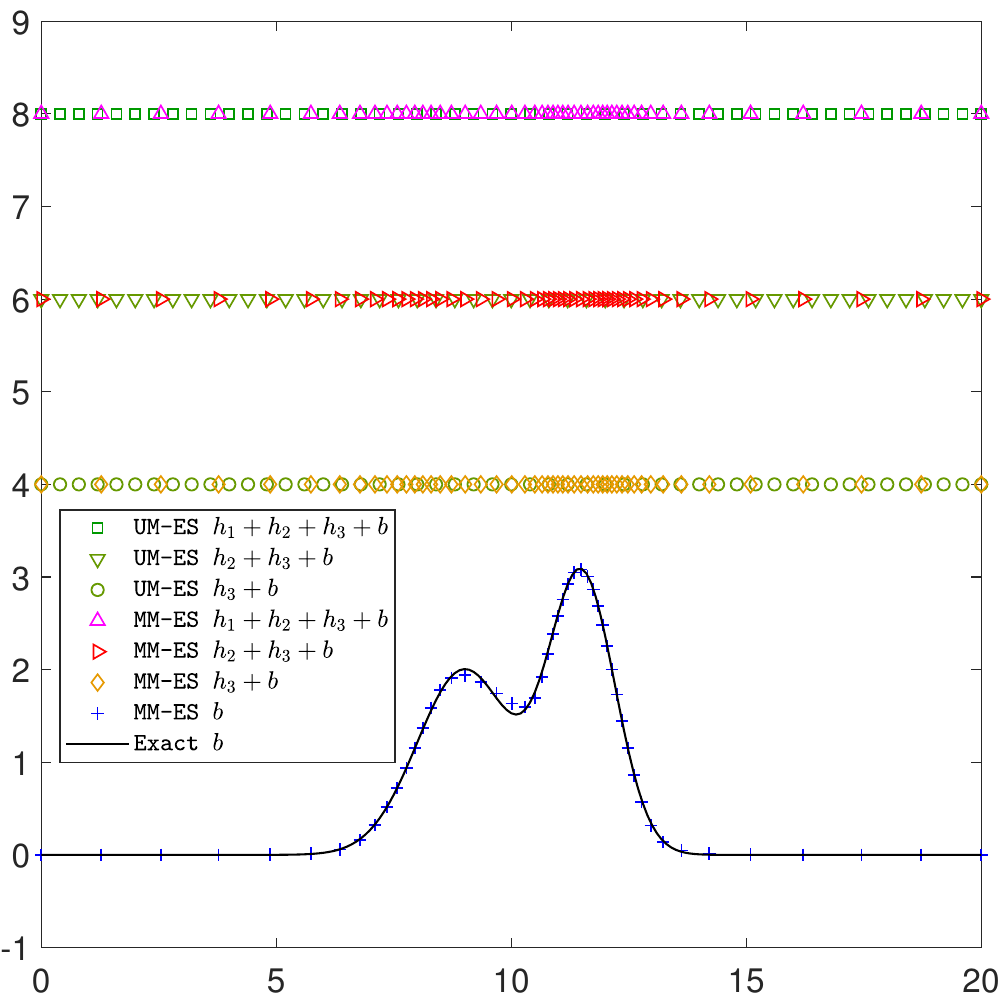}
  \end{subfigure}
  \\
\begin{subfigure}[b]{0.35\textwidth}
	\centering
	\includegraphics[width=1.0\textwidth]{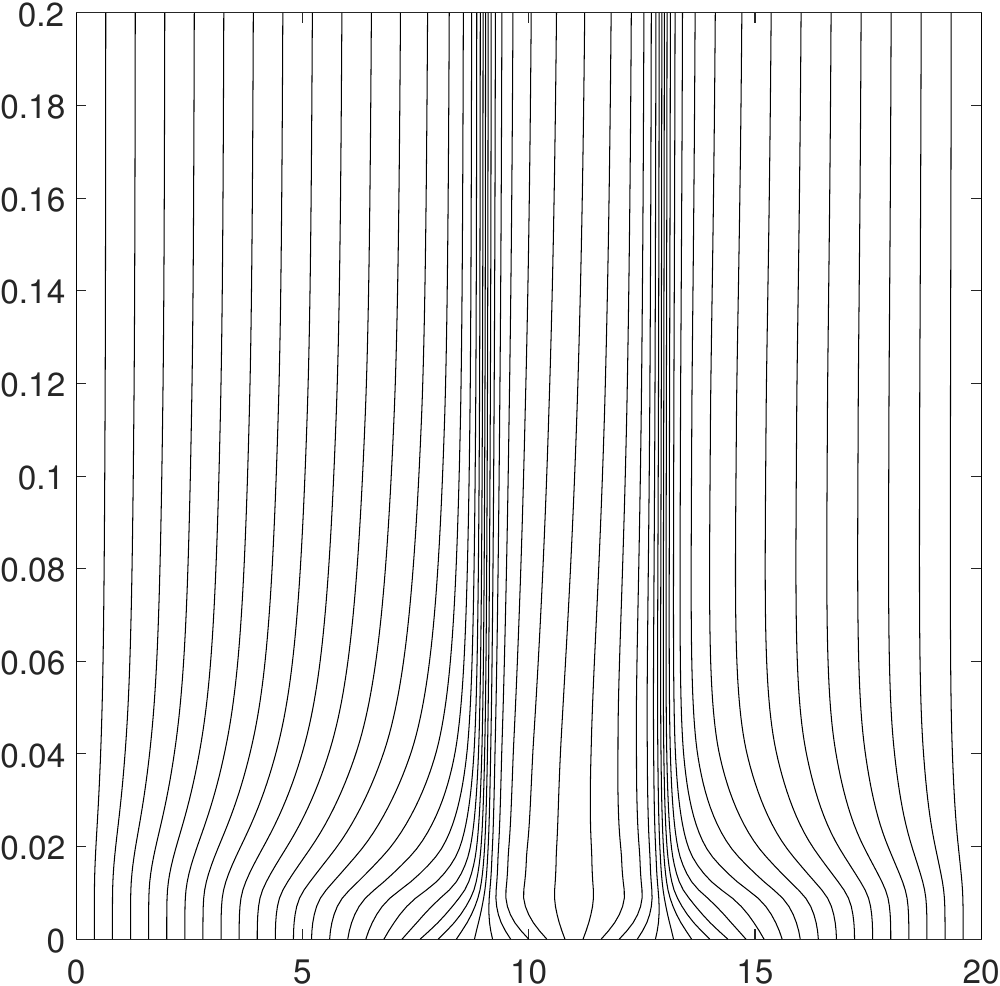}
\end{subfigure}
\begin{subfigure}[b]{0.35\textwidth}
	\centering
	\includegraphics[width=1.0\textwidth,  clip]{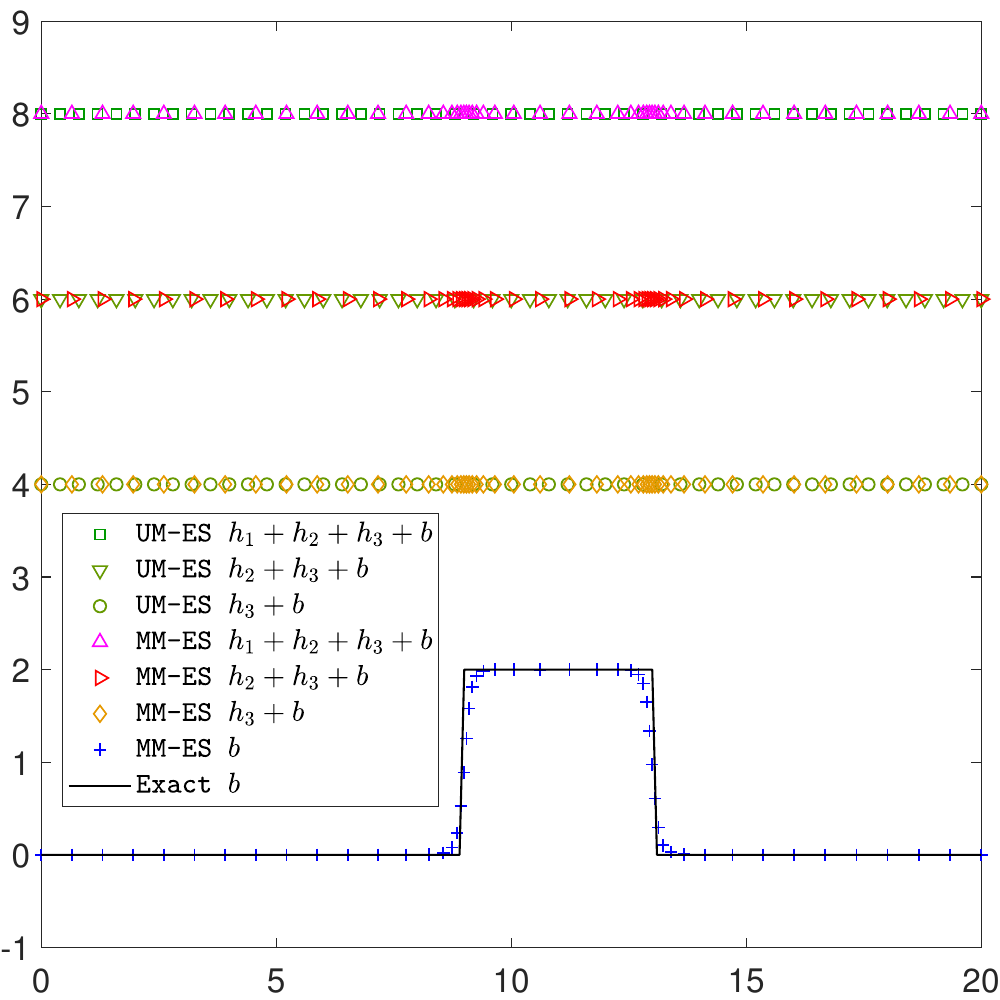}
\end{subfigure}
  \caption{Example \ref{ex:1D_WB_Test} for the three-layer case.
  Left: the mesh trajectories obtained by the \texttt{MM-ES} scheme,
  right: the bottom topography $b$ and water surface levels.
    Top: with the bottom topography \eqref{eq:b_Smooth},
    bottom: with the bottom topography \eqref{eq:b_dis}.
    The results are obtained with $50$ mesh points at $t=0.2$.
  }
  \label{1D_Well_balance_3L}
\end{figure}

\begin{example}[1D dam break over a flat bottom]\label{ex:1D_IDB_Test}\rm
	This test investigates the ability of the proposed schemes to capture the small-scale wave structures.
    The initial data for the two-layer case are
    \begin{align*}
    	h_2(x,0)&=\begin{cases}
    		0.6, &\text{if} \quad x \leqslant 5,\\
    		0.4,&\text{otherwise}, \\
    	\end{cases} \\
            h_1(x,0) &= 1 - h_2(x,0),
    \end{align*}
    with zero velocities in the physical domain $[0,10]$ and $r_{12} = 0.8$, $\rho_2=1$.
 The output time is $t = 1.25$ and outflow boundary conditions are used.
 For the three-layer case, the initial data for the water depth are
  \begin{align*}
    	h_3(x,0) &= \begin{cases}
    		0.6, &\text{if} \quad x \leqslant 5,\\
    		0.4,&\text{otherwise}, \\
    	\end{cases}\\
     h_2(x,0) &= 1 - h_3(x,0),\\
     h_1(x,0) &= 2-h_2(x,0)-h_3(x,0),
    \end{align*}
    with the density ratios $r_{12} = 0.8$, $r_{13} = 0.64$, $r_{23} = 0.8$, $\rho_3=1.0$, and the output time $t=0.8$. The gravitational acceleration constant is $g = 9.812$ in both cases.
 The monitor function is chosen as that in Example \ref{ex:1D_WB_Test}, but with
$\theta = 100$, and $\sigma = h_1+h_2+b$ and $h_1+h_2+h_3+b$ for the two-layer and three-layer cases, respectively.
\end{example}

Figure \ref{1D_IDB_Test} shows the mesh trajectories and water surface levels $h_1+h_2+b$ and $h_2+b$ obtained by the \texttt{UM-ES} and \texttt{MM-ES} schemes for the two-layer case at $t=1.25$,
and corresponding results for the three-layer case at $t=0.8$.
The reference solutions are given by using the \texttt{UM-ES} schemes with $3000$ mesh points.
One can see that the wave structures for the three-layer case are more complicated, and our schemes can well capture the discontinuities without obvious oscillations.
Moreover, the \texttt{MM-ES} scheme using $400$ mesh points gives similar local resolution compared to the \texttt{UM-ES} scheme using $1200$ mesh points, which indicates the high efficiency of our schemes on adaptive moving meshes.
Figure \ref{1D_IDB_Test_Energy} gives the evolution of the discrete total energy 
$\sum\limits_i J_{i}\eta(\bU_i)\Delta \xi$ obtained by the \texttt{MM-ES} scheme with $400$ mesh points in both cases, which decay as expected.

\begin{figure}[!htb]
	\centering
 \begin{subfigure}[b]{0.3\textwidth}
	\centering
	\includegraphics[width=1.0\textwidth]{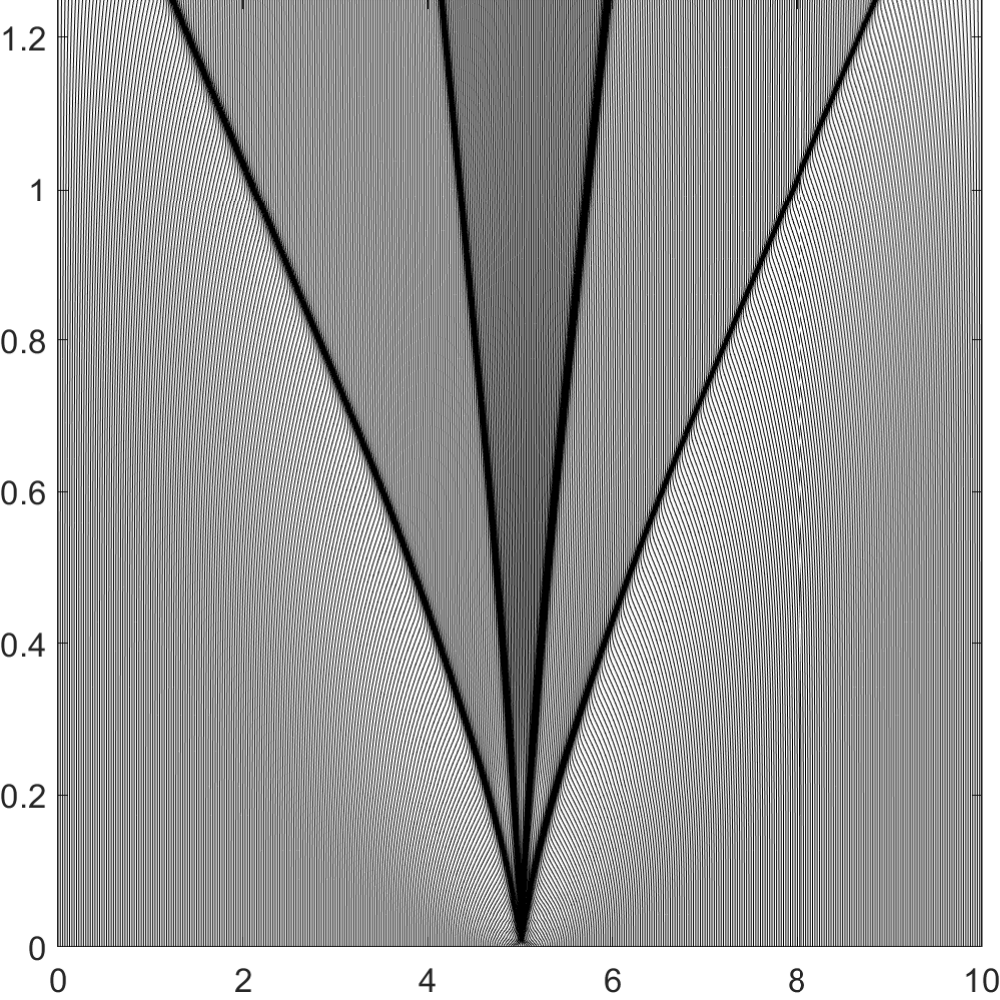}
	\caption{mesh trajectory}
\end{subfigure}
	\begin{subfigure}[b]{0.31\textwidth}
		\centering
		\includegraphics[width=1.0\textwidth]{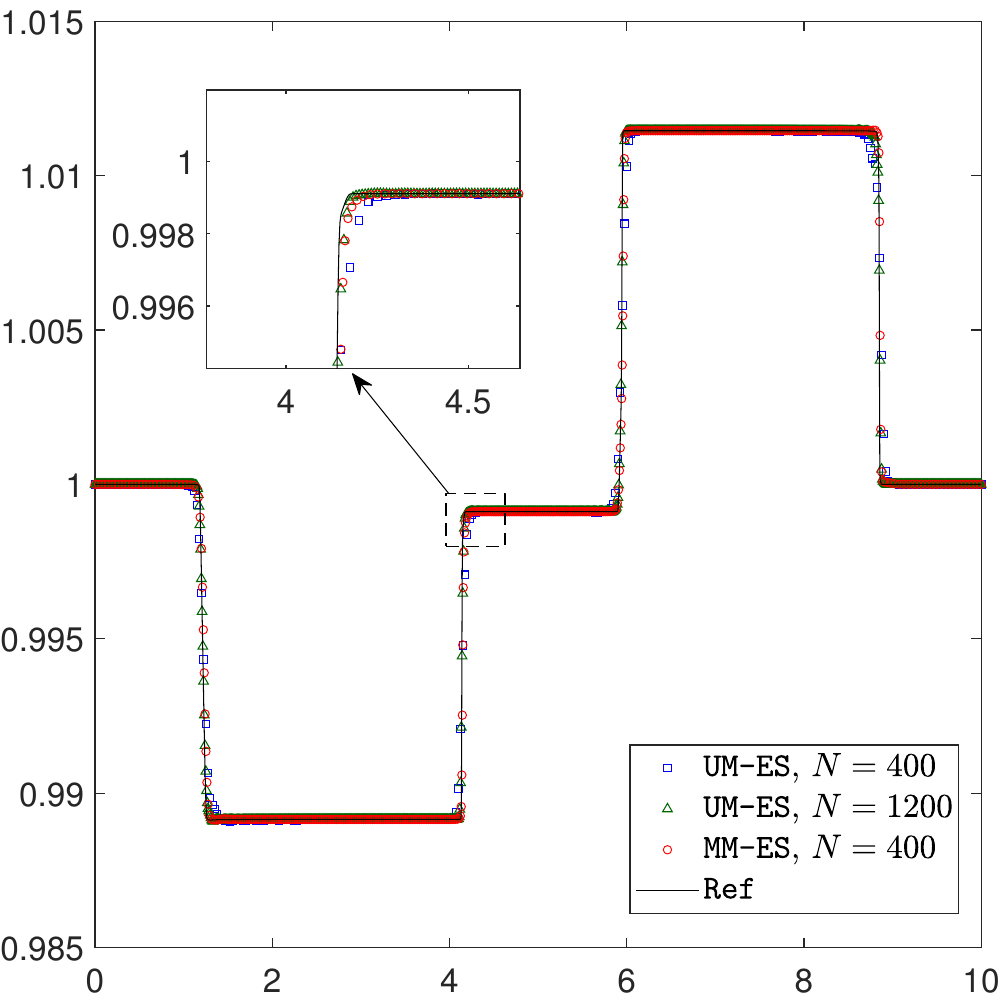}
		\caption{ $h_1+h_2+b$}
	\end{subfigure}
	\begin{subfigure}[b]{0.31\textwidth}
		\centering
		\includegraphics[width=1.0\textwidth]{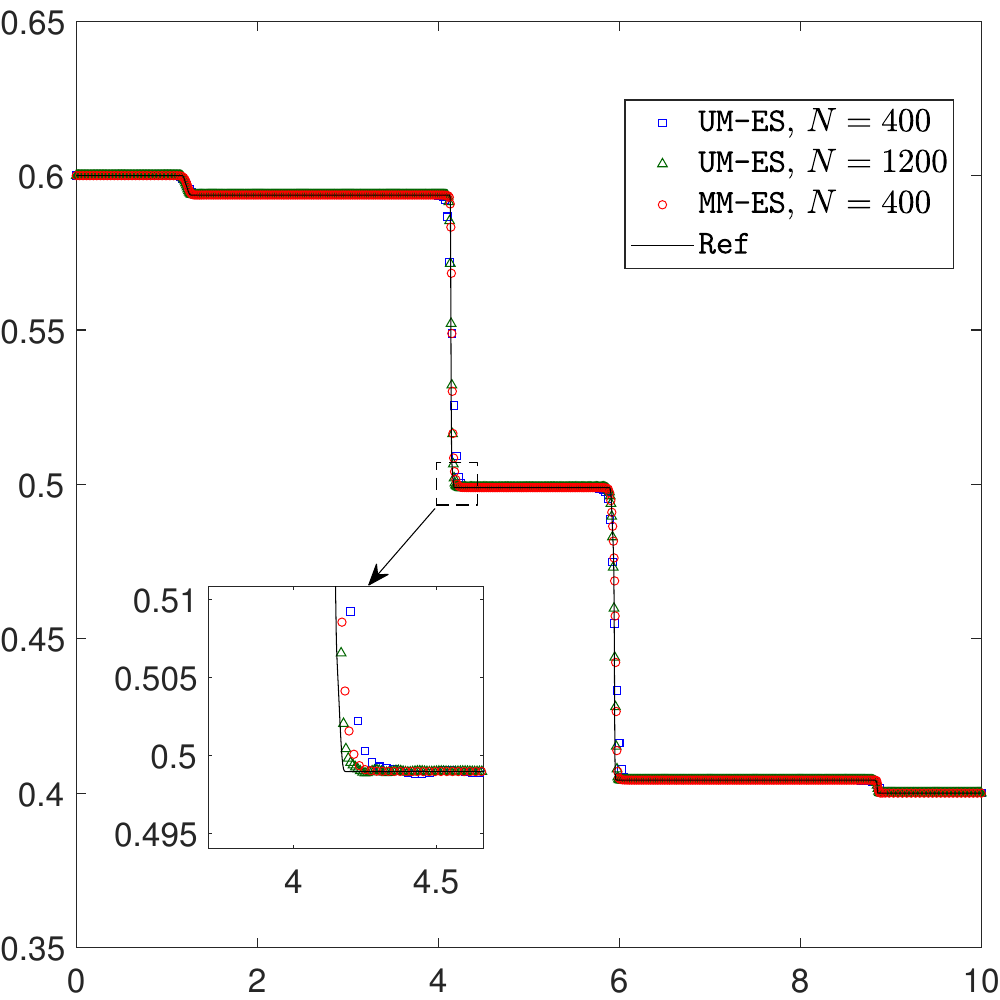}
		\caption{$h_2+b$}
	\end{subfigure}
 
\begin{subfigure}[b]{0.3\textwidth}
	\centering
	\includegraphics[width=1.0\textwidth]{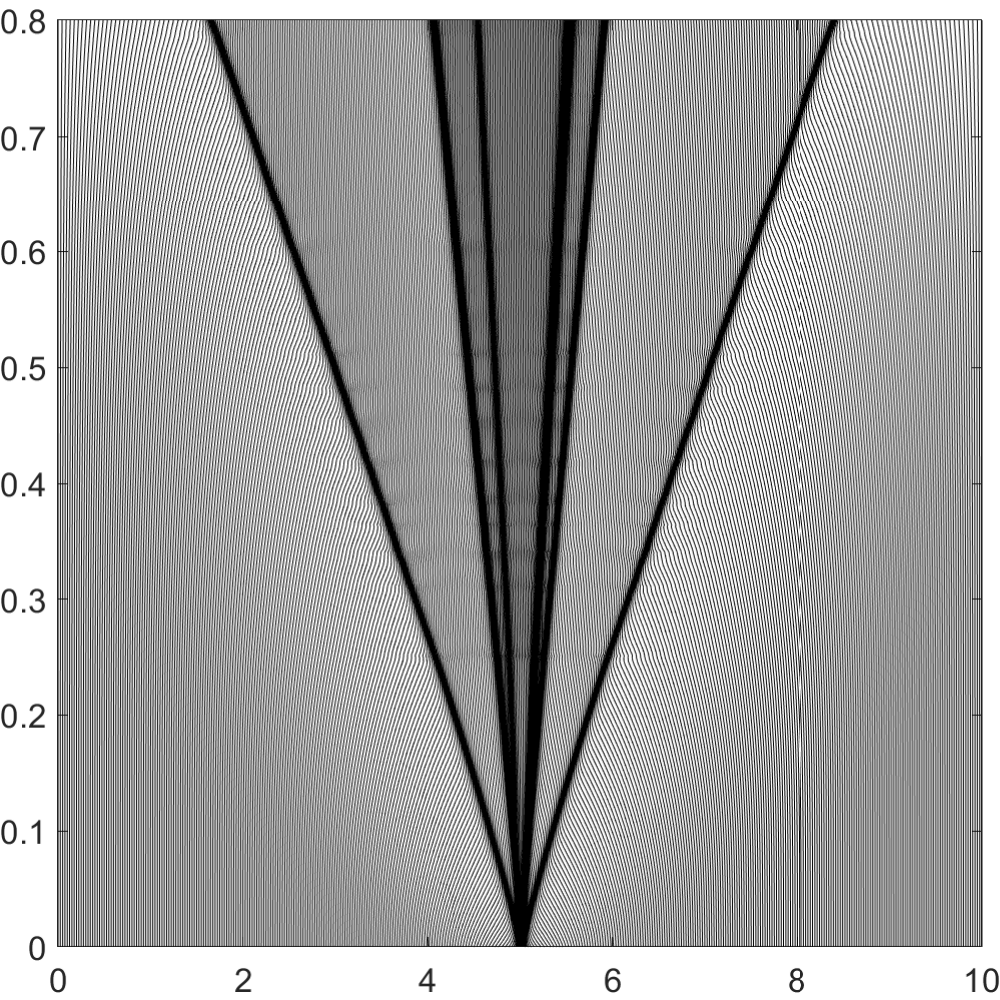}
	\caption{mesh trajectory}
	\end{subfigure}
\begin{subfigure}[b]{0.3\textwidth}
	\centering
	\includegraphics[width=1.0\textwidth]{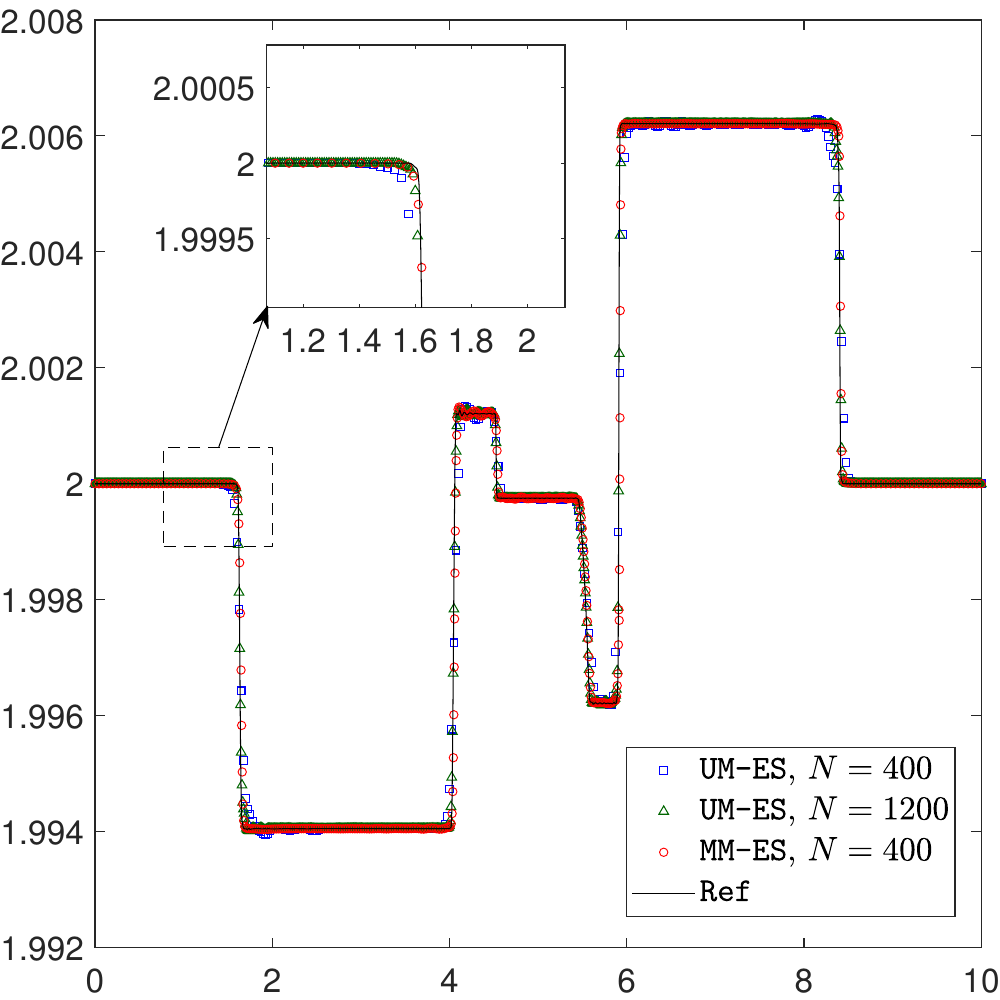}
	\caption{ $h_1+h_2+h_3+b$}
\end{subfigure}
\begin{subfigure}[b]{0.3\textwidth}
	\centering
	\includegraphics[width=1.0\textwidth]{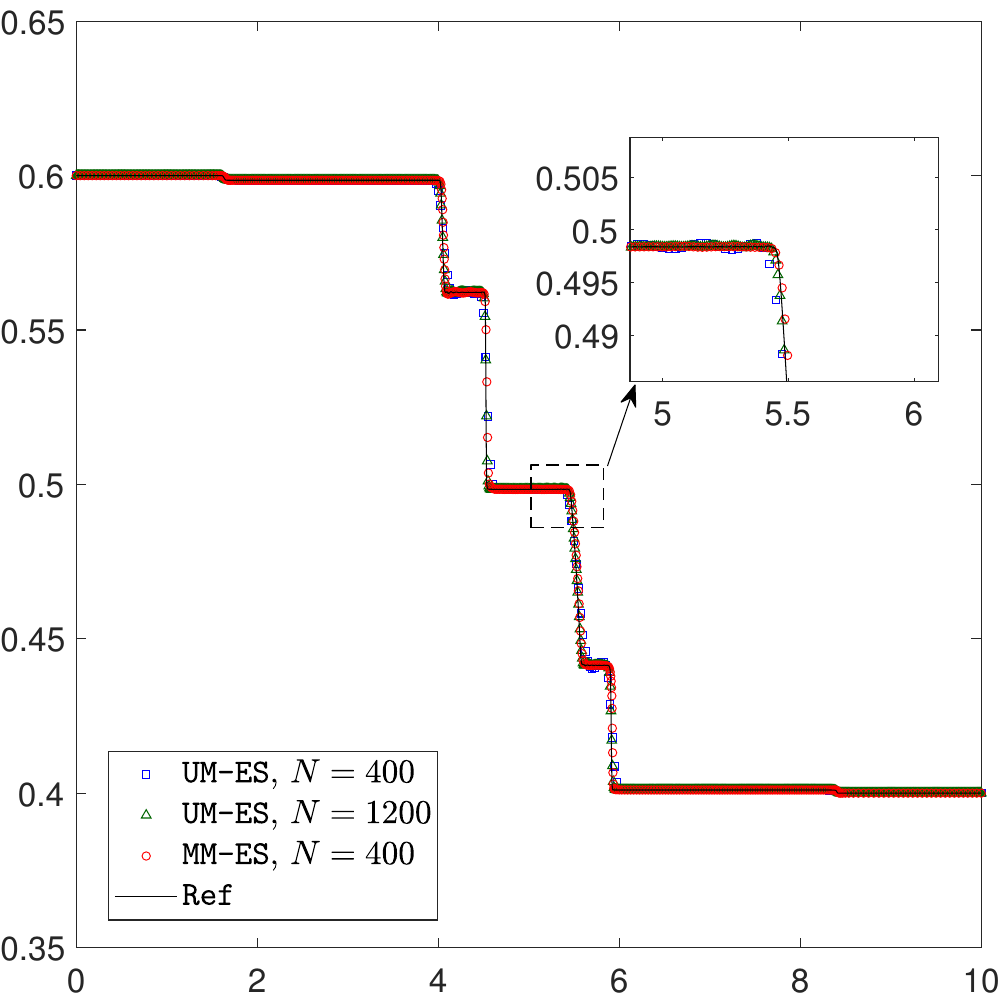}
	\caption{$h_3+b$}
\end{subfigure}
	\caption{Example \ref{ex:1D_IDB_Test}. The numerical solutions obtained by the \texttt{UM-ES} and \texttt{MM-ES}  schemes, the reference solutions are obtained by using the \texttt{UM-ES}  schemes with $3000$ mesh points. Top: the two-layer case, bottom: the three-layer case.}\label{1D_IDB_Test}
\end{figure}

\begin{figure}[!htb]
\centering
\begin{subfigure}[b]{0.3\textwidth}
	\centering
	\includegraphics[width=1.0\textwidth]{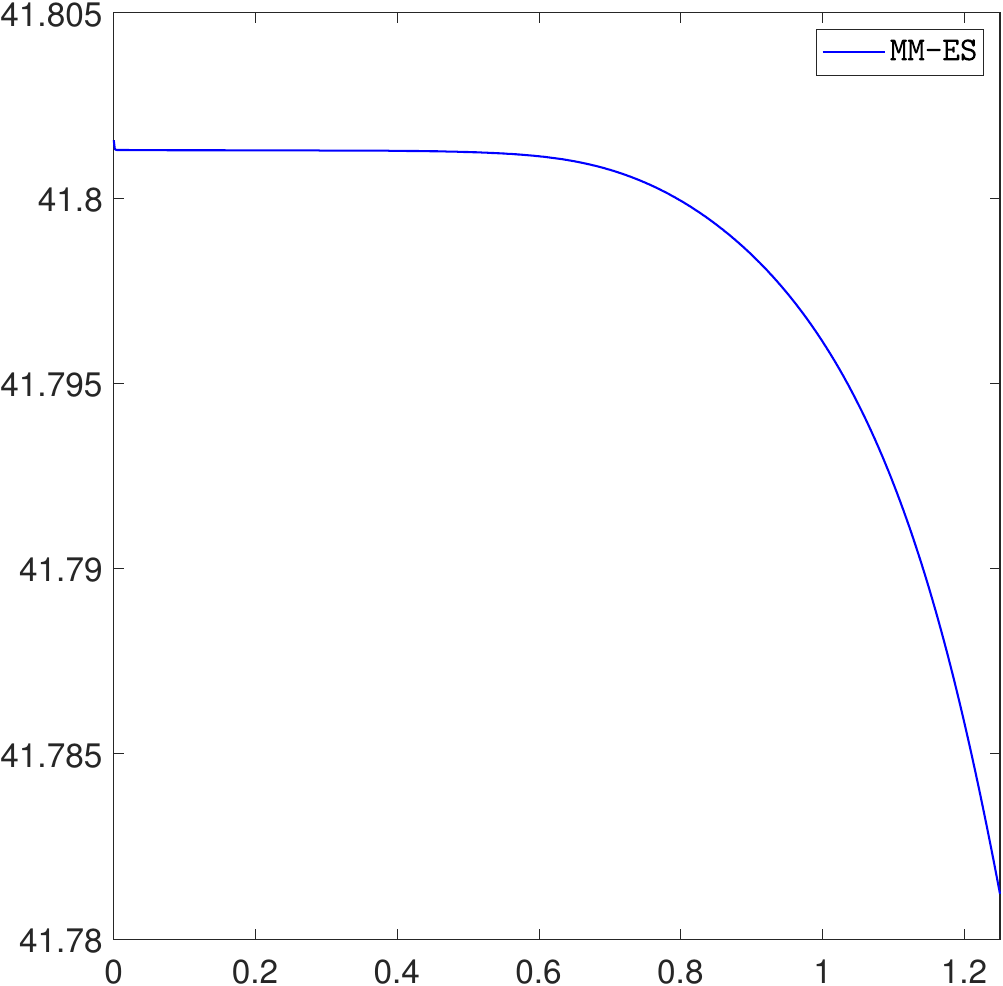}
	\caption{two-layer case}
\end{subfigure}
\begin{subfigure}[b]{0.3\textwidth}
	\centering
	\includegraphics[width=1.0\textwidth]{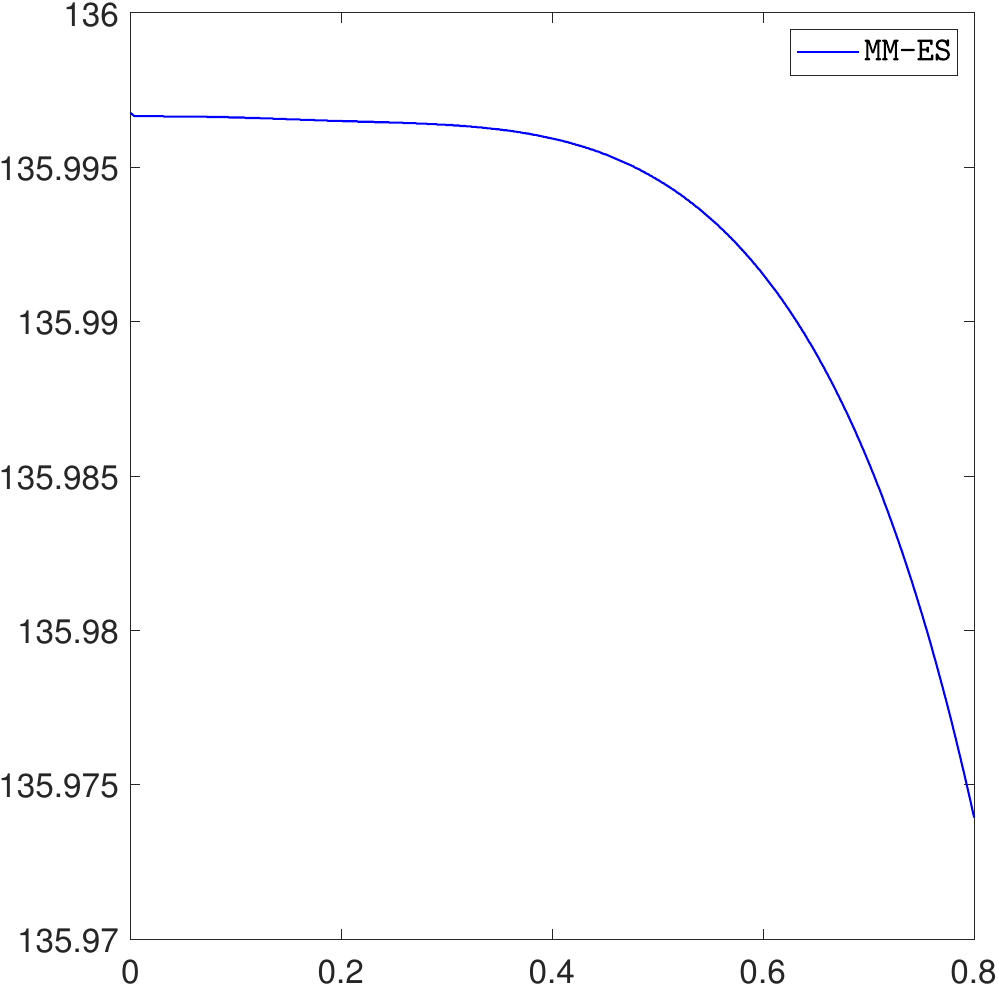}
	\caption{three-layer case}
\end{subfigure}
	\caption{Example \ref{ex:1D_IDB_Test}. The evolution of the discrete total energy by using the \texttt{MM-ES} scheme with 400 mesh points.}\label{1D_IDB_Test_Energy}
\end{figure}

\begin{example}[Small perturbation test]\label{ex:1D_Pertubation_Test}\rm
  To test the ability of our schemes to capture small perturbations in the steady-state flow, the bottom topography consisting of a ``hump'' \cite{Leveque1998Balancing} is considered here
  \begin{equation*}
    b(x) = \begin{cases}
      0.25(\cos(10\pi(x-0.5))+1)-2, &\text{if}\quad 0.4 \leqslant x \leqslant 0.6,\\
      0,   &\text{otherwise},
    \end{cases}
  \end{equation*}
  and the initial data for the two-layer case are
  \begin{align*}
        h_1 &= \begin{cases}
      1.00001, &\text{if}\quad 0.1\leqslant x\leqslant 0.2,\\
      1, &\text{otherwise},\\
    \end{cases}\\
    h_2 &= -1-b,
  \end{align*}
  with initial zero velocities and $g=9.812$, $r_{12}=0.98$, $\rho_2 = 1$.
  The physical domain is taken as $[-1, 1]$ with outflow boundary conditions.
  For the three-layer case, the initial data are
  \begin{align*}
    h_1 &= 1,\\
    h_2 &= \begin{cases}
      1.00001, &\text{if}\quad 0.1\leqslant x\leqslant 0.2,\\
      1, &\text{otherwise},\\
    \end{cases}\\
    h_3 &= -1-b,
  \end{align*}
  with zero velocities, and $r_{12} = 0.9898$, $r_{23} = 0.98$,  $r_{13} = 0.97$, $\rho_3 = 1.0$.
  The monitor function is the same as Example \ref{ex:1D_IDB_Test}.
\end{example}

Figure \ref{1D_ES_Pertubation_WB} presents the numerical solutions obtained by the \texttt{UM-ES} and \texttt{MM-ES} schemes with $200$ and $600$ mesh points at $t=0.15$ for the two-layer case, and at $t=0.1$ for the three-layer case, respectively. 
The reference solutions are obtained by the \texttt{UM-ES} scheme with $3000$ mesh points.
It can be seen that the structures in the solution can be captured well and there is no apparent numerical oscillation.
The results obtained by using the adaptive moving mesh are better than those on the fixed uniform mesh with the same number of mesh points,
and comparable to those on the fixed uniform mesh with three times the number of mesh points.
Figure \ref{1D_ES_Pertubation_WB_Entropy} gives the discrete total energy evolution versus time obtained by the \texttt{MM-ES} scheme with $200$ mesh points, which decays as expected.

\begin{figure}[!htb]
  \centering
    \begin{subfigure}[b]{0.31\textwidth}
	\centering
	\includegraphics[width=1.0\textwidth]{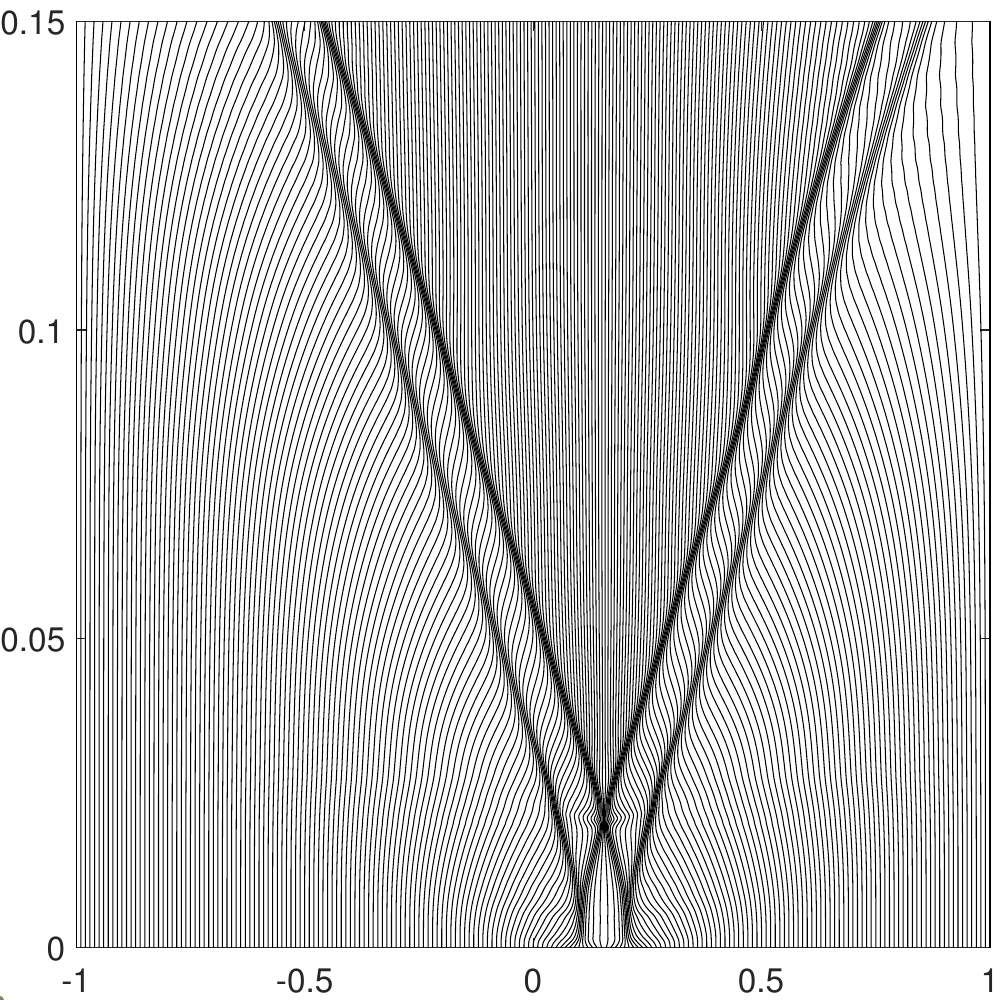}
	\caption{mesh trajectory}
\end{subfigure}
  \begin{subfigure}[b]{0.32\textwidth}
    \centering
    \includegraphics[width=1.0\textwidth]{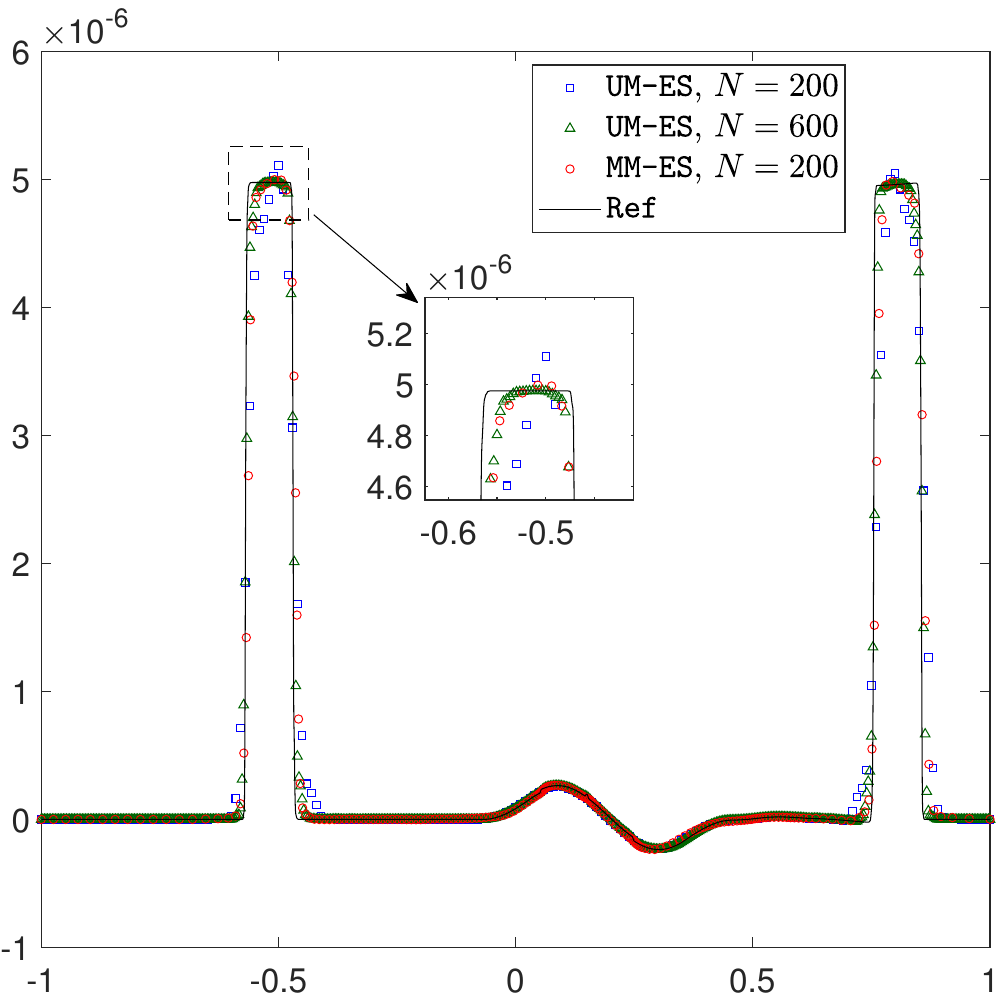}
    \caption{$h_1+h_2+b$}
  \end{subfigure}
  \begin{subfigure}[b]{0.32\textwidth}
    \centering
    \includegraphics[width=1.0\textwidth]{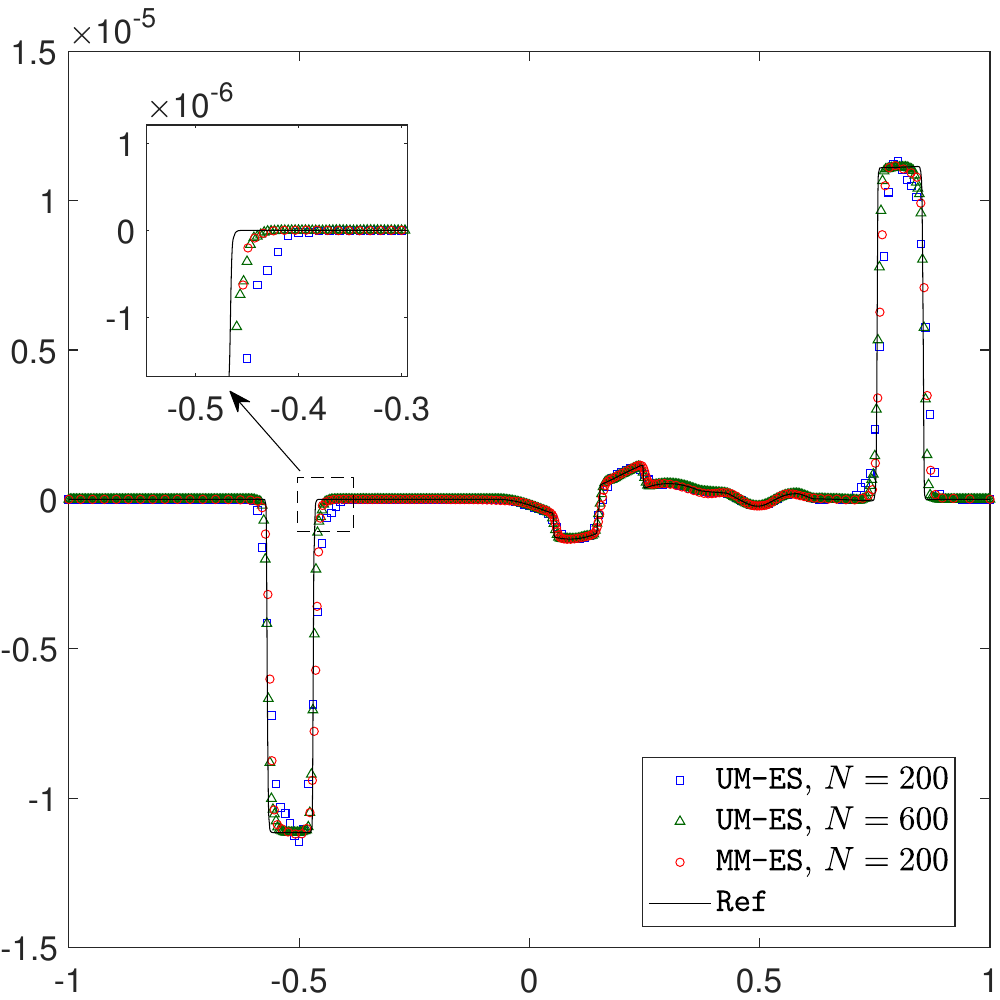}
    \caption{$u_1$}
  \end{subfigure}
\begin{subfigure}[b]{0.31\textwidth}
	\centering
	\includegraphics[width=1.0\textwidth]{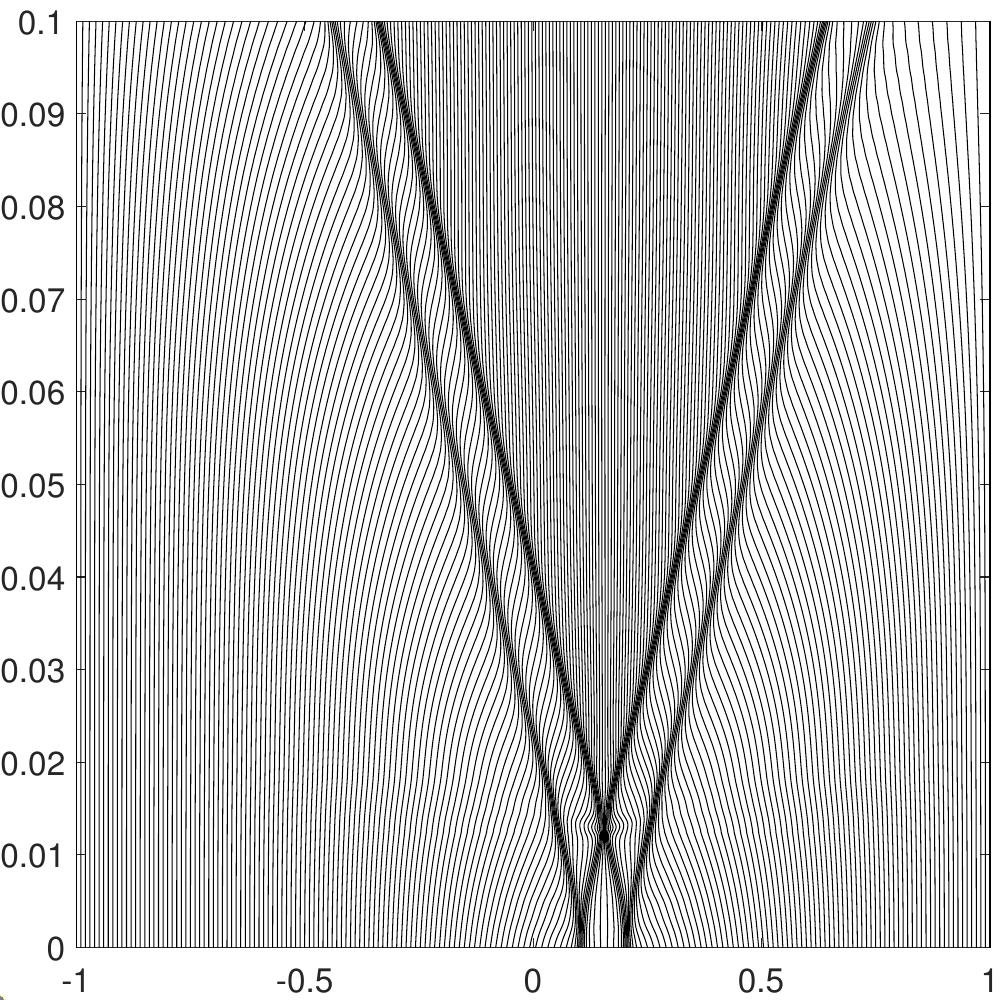}
	\caption{mesh trajectory}
 \end{subfigure}
 \begin{subfigure}[b]{0.32\textwidth}
	\centering
	\includegraphics[width=1.0\textwidth,height = 0.95\textwidth]{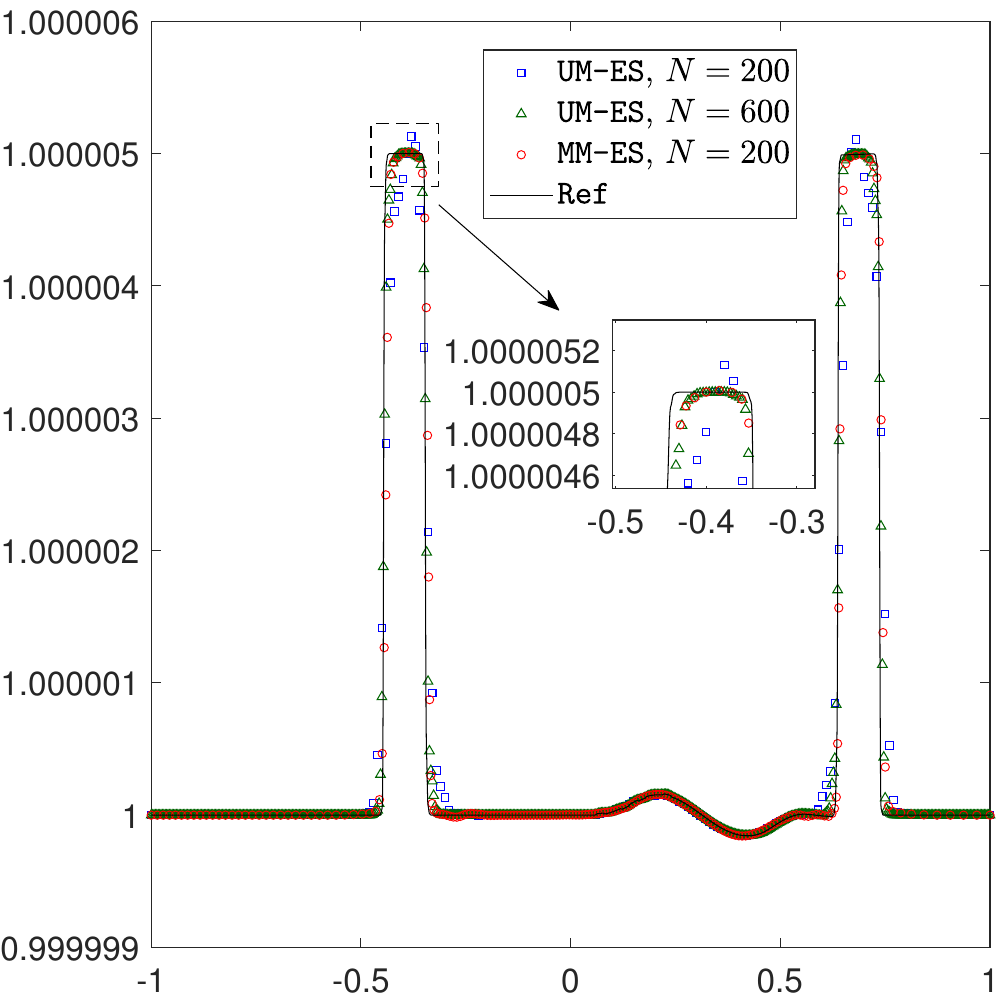}
	\caption{$h_1+h_2+b$}
\end{subfigure}
\begin{subfigure}[b]{0.32\textwidth}
	\centering
	\includegraphics[width=1.0\textwidth]{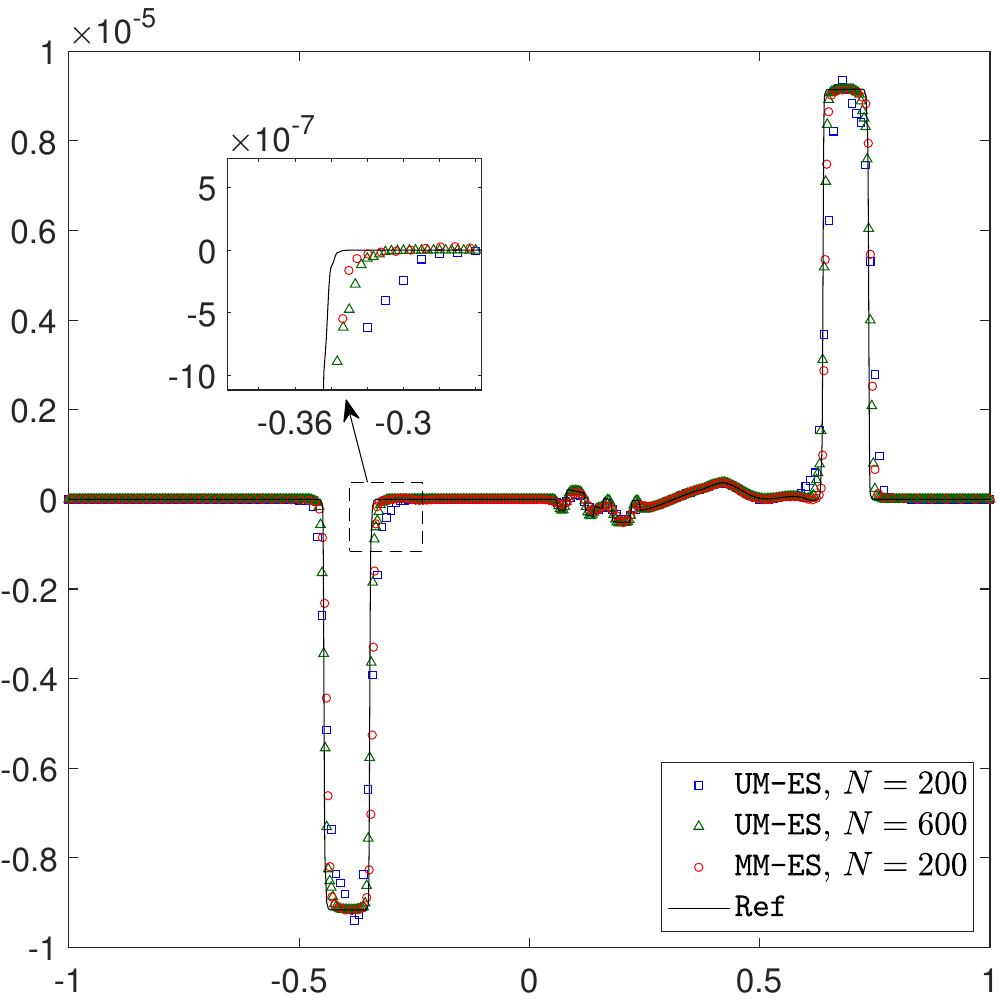}
	\caption{$u_1$}
\end{subfigure}
\caption{Example \ref{ex:1D_Pertubation_Test}.
The mesh trajectories and numerical solutions obtained by using the \texttt{UM-ES} and \texttt{MM-ES} schemes at $t=0.15$ for the two-layer case (top), and at $t=0.1$ for the three-layer case (bottom).
The reference solutions are obtained by using the \texttt{UM-ES} scheme with $3000$ mesh points.}\label{1D_ES_Pertubation_WB}
\end{figure}

\begin{figure}[!htb]
  \centering
  \begin{subfigure}[b]{0.255\textwidth}
  	\centering
  	\includegraphics[width=1.0\textwidth]{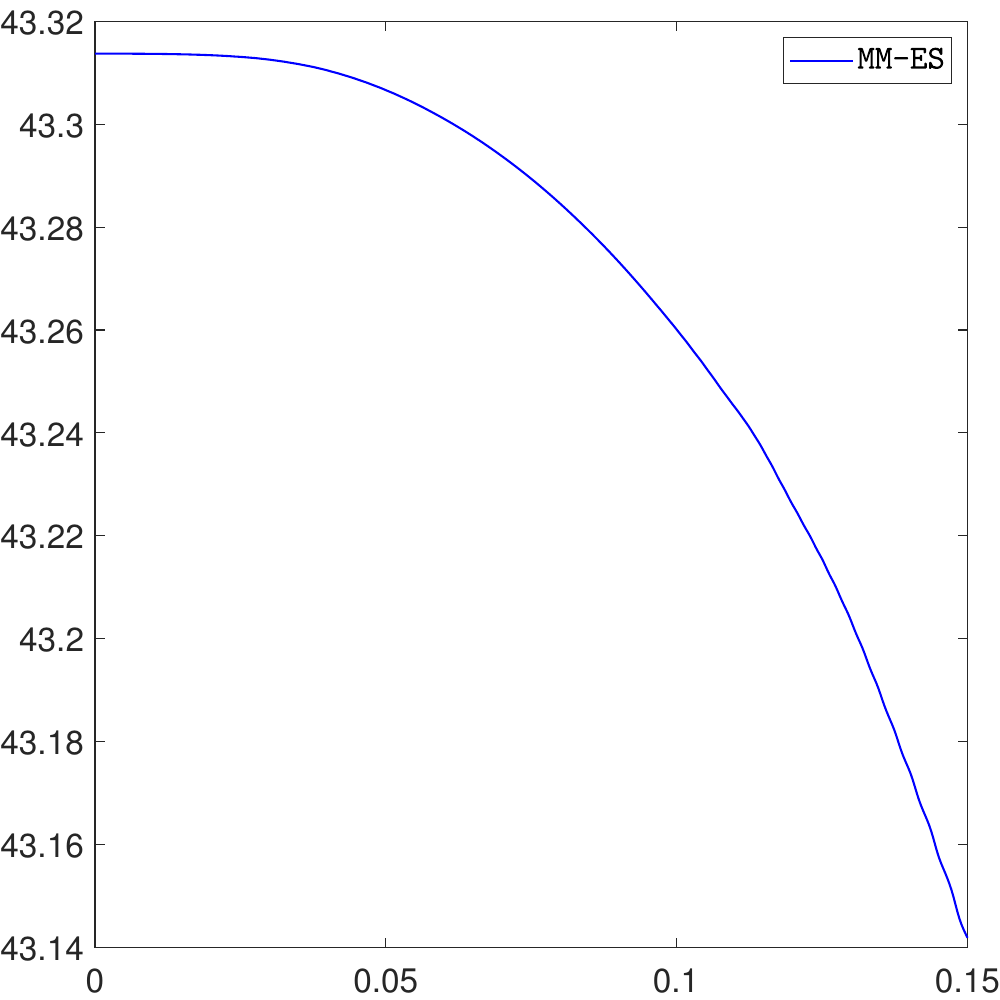}
   \caption{two-layer case}
  \end{subfigure}
  \begin{subfigure}[b]{0.25\textwidth}
  	\centering
  	\includegraphics[width=1.0\textwidth]{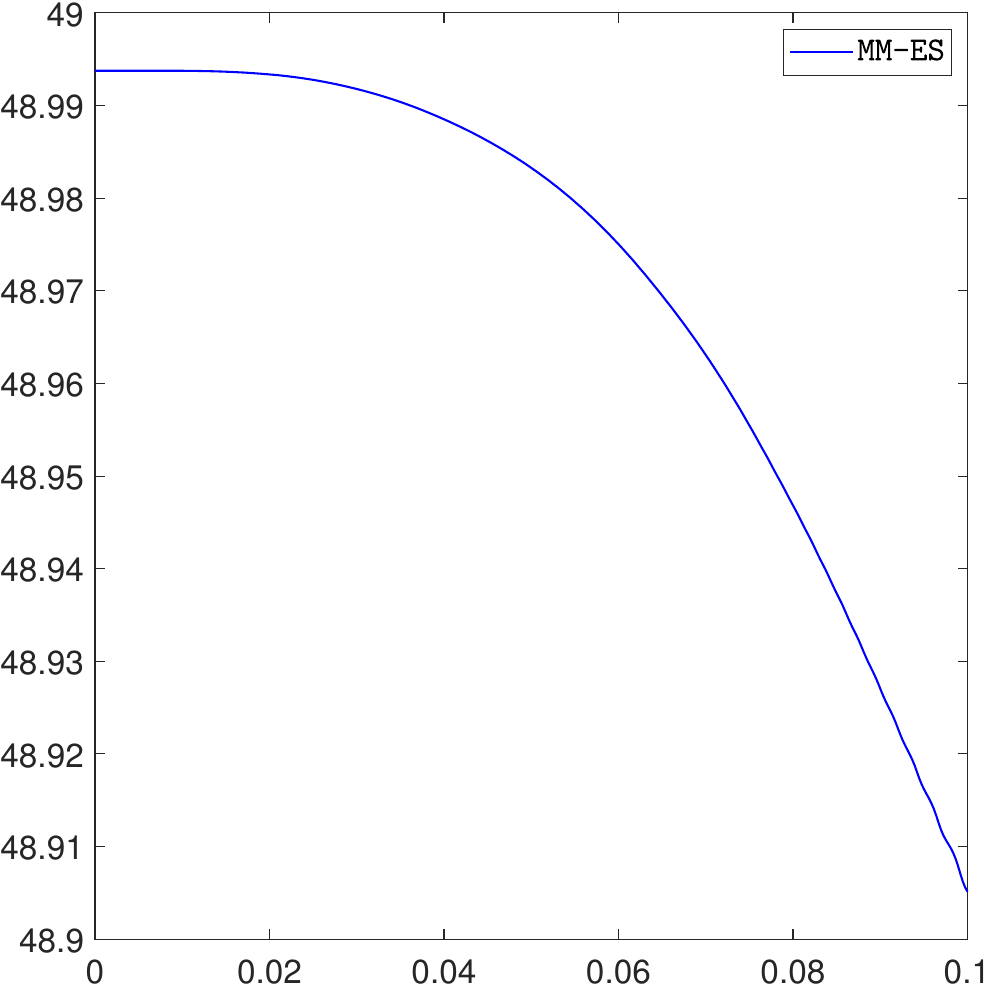}
\caption{three-layer case}
  \end{subfigure}
  \caption{Example \ref{ex:1D_Pertubation_Test}. The evolution of the discrete total energy over time by using the \texttt{MM-ES} schemes with $200$ mesh points.}\label{1D_ES_Pertubation_WB_Entropy}
\end{figure}

\subsection{2D tests}
\begin{example}[Accuracy test with manufactured solutions]\label{eq:Smooth_2D}\rm
  This test is taken to verify the accuracy of the 2D schemes.
  The manufactured solutions are constructed for the two-layer and three-layer SWEs with extra source terms.
For the two-layer case, the exact solutions are
\begin{align*}
	&h_1(x_1,x_2,t) = 
	\cos\left(\pi t\right)\cos\left(\pi x_1\right)+\cos\left(\pi t\right)\cos\left(\pi x_2\right)+6,\\
	&u_1(x_1,x_2,t) = \frac{\sin\left(\pi t\right)\sin\left(\pi x_1\right)}{h_1},~v_1(x_1,x_2,t) = \frac{\sin\left(\pi t\right)\sin\left(\pi x_2\right)}{h_1}, \\
&h_2(x_1,x_2,t) = 
\cos\left(\pi t\right)\cos\left(\pi x_1\right)+\cos\left(\pi t\right)\cos\left(\pi x_2\right)+4,\\
&u_2(x_1,x_2,t) = \frac{\sin\left(\pi t\right)\sin\left(\pi x_1\right)}{h_2},~v_2(x_1,x_2,t) = \frac{\sin\left(\pi t\right)\sin\left(\pi x_2\right)}{h_2}, \\
	&b(x_1,x_2) = \sin\left(\pi x_1\right)+\sin\left(\pi x_1\right)+\frac{3}{2},
\end{align*}
and the source terms are omitted here.
The physical domain is $[0, 2] \times [0, 2] $ with periodic boundary conditions, and the output time is $t = 0.1$.
The density ratio $r_{12}$ is $7/10$, $\rho_2=1$ with the gravitational acceleration constant $g=1$.
For the three-layer case, the exact solutions can be constructed similarly, presented in  \ref{Sec:Exact_Solution},
with $r_{12} = 7/10$, $r_{13} = 7/13$, $r_{23} = 10/13$, $\rho_3=13/10$.
The monitor function is chosen as
    \begin{equation*}
      \omega = \left(1+\theta\left(\frac{\left|\nabla_{\bm{\xi}}\sigma\right|}{\max\left|\nabla_{\bm{\xi}}\sigma\right|}\right)^2\right)^{\frac{1}{2}},
  \end{equation*}
  where $\theta = 1$ and $\sigma = h_2+b$ for both cases.
\end{example}

Figure \ref{fig:2D_Smooth_Accuracy} gives the $\ell^1$ and $\ell^\infty$ errors and corresponding convergence orders of the \texttt{UM-ES} and \texttt{MM-ES} schemes in velocity $u_2$ at $t =0.1$,
which verify the $5$th-order accuracy. 
Figure \ref{fig:2D_Smooth_Result_Entropy} shows the water surface levels and bottom topography obtained by the \texttt{UM-ES} scheme with $160\times160$ mesh, and the evolution of the discrete total energy $\sum\limits_{i,j} \eta(\bU_{i,j})\Delta x_1 \Delta x_2$,
 which decays on different meshes. 

\begin{figure}[!htb]
  \centering
  \begin{subfigure}[b]{0.4\textwidth}
    \centering
    \includegraphics[width=1.0\textwidth]{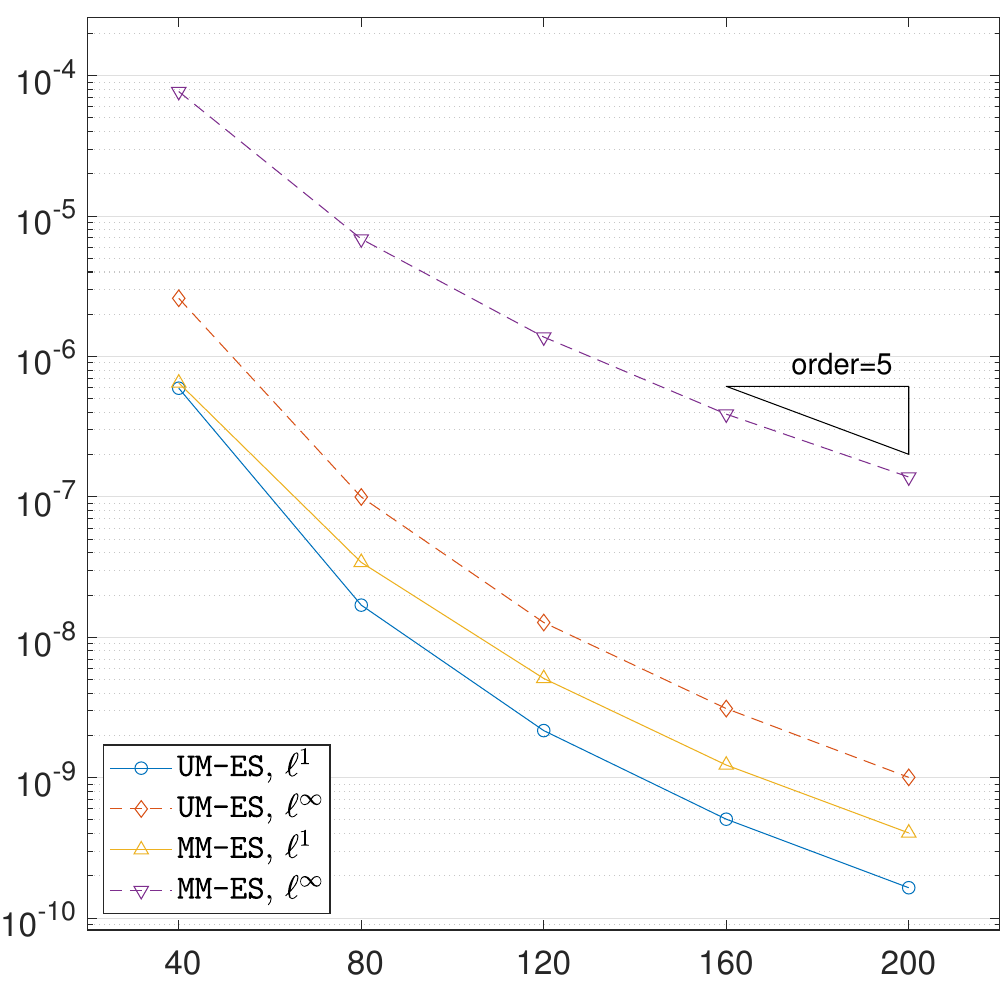}
    \caption{two-layer case}
  \end{subfigure}
  \begin{subfigure}[b]{0.4\textwidth}
    \centering
    \includegraphics[width=1.0\textwidth]{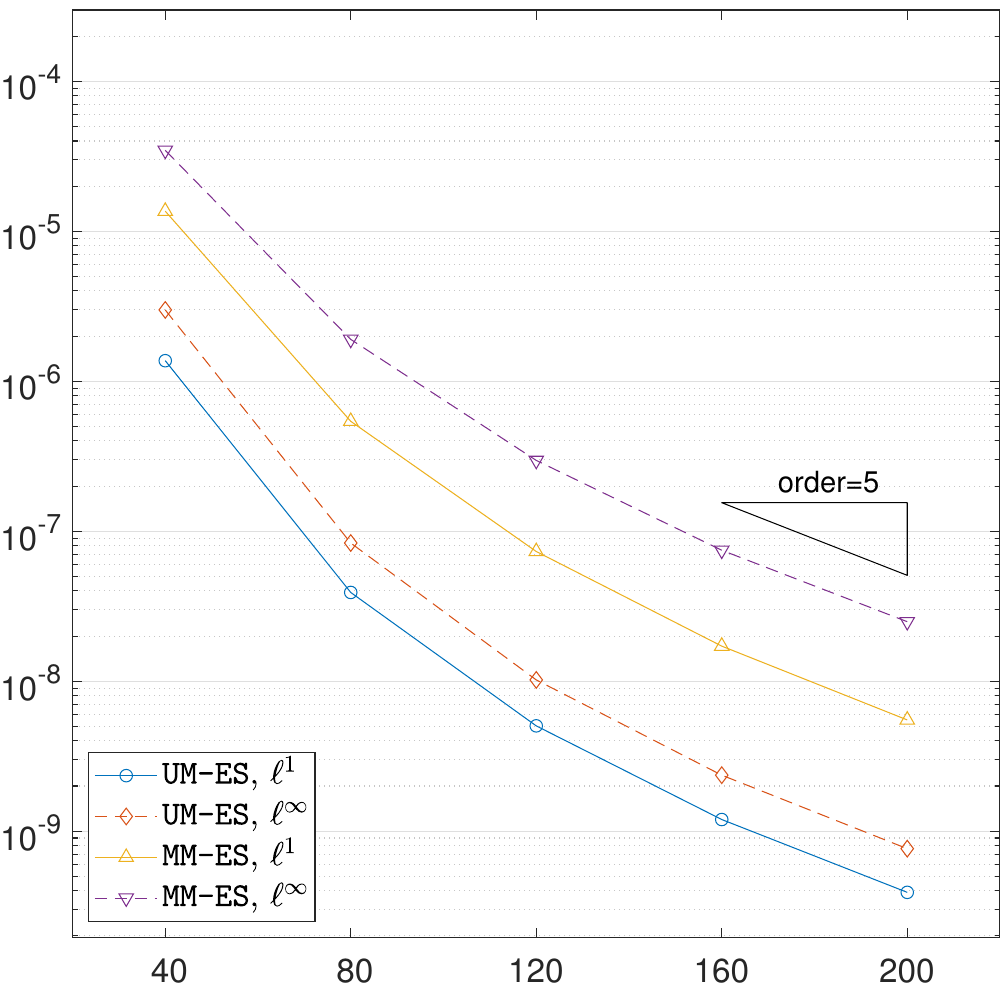}
    \caption{three-layer case}
  \end{subfigure}
  \caption{Example \ref{eq:Smooth_2D}. The errors and convergence rates in $u_2$ at $t=0.1$.}\label{fig:2D_Smooth_Accuracy}
\end{figure}

\begin{figure}[!htb]
  \centering
  \begin{subfigure}[t]{0.45\textwidth}
    \centering
    \includegraphics[width=1.0\textwidth]{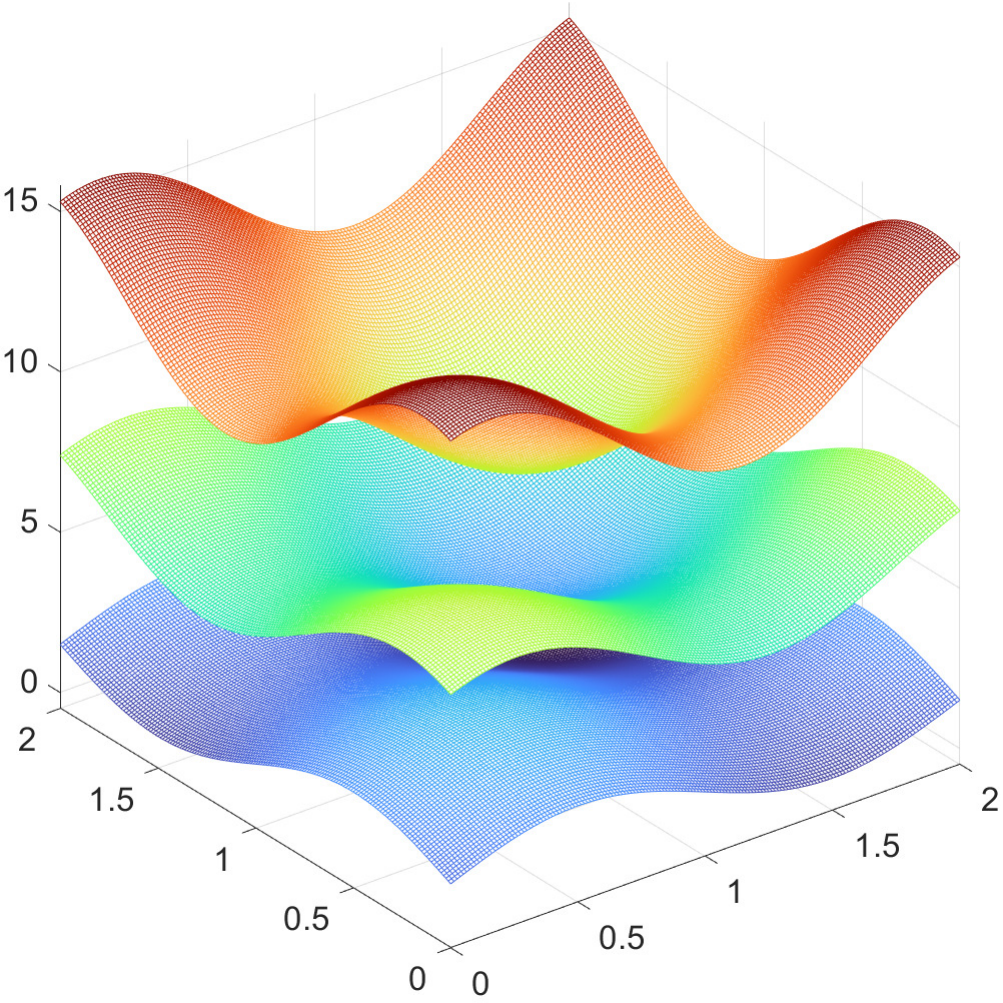}
    \caption{$h_1+h_2+b$, $h_2+b$, and $b$}
  \end{subfigure}
  \begin{subfigure}[t]{0.45\textwidth}
    \centering
    \includegraphics[width=1.0\textwidth]{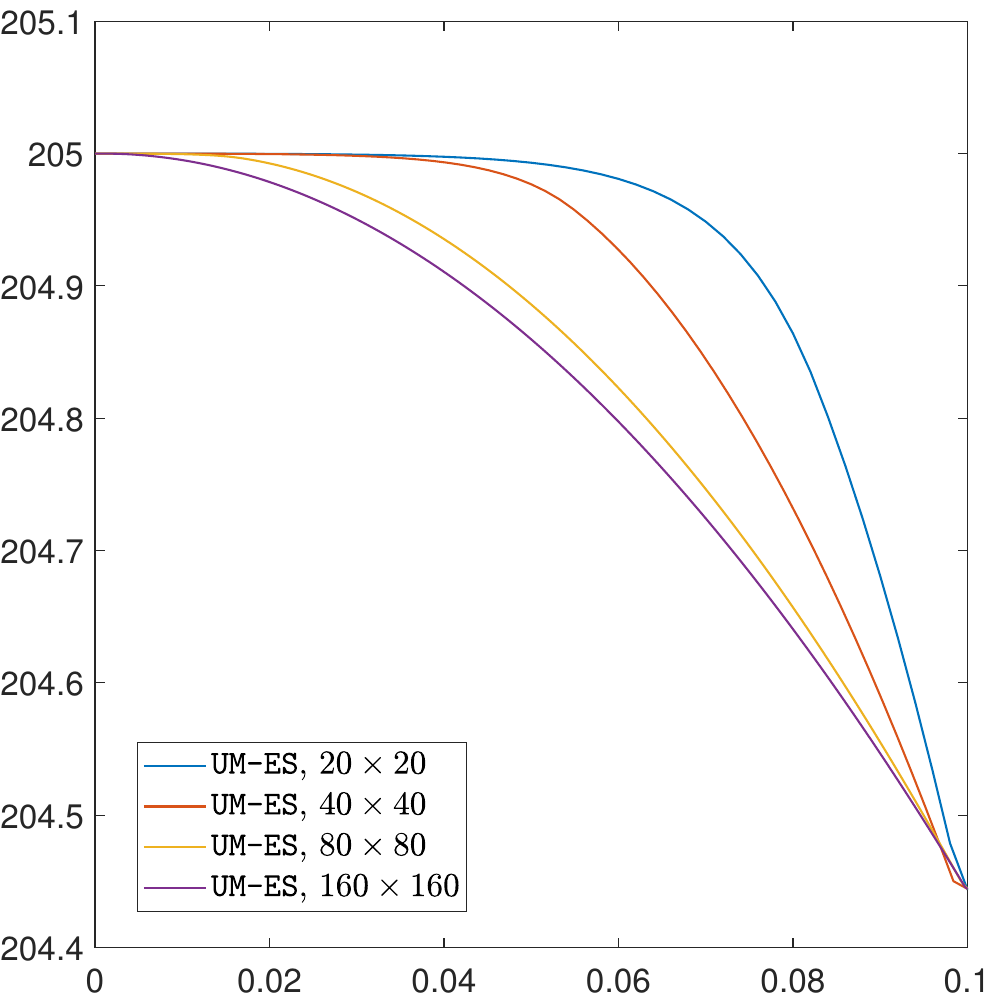}
      \caption{discrete total energy}
  \end{subfigure}\\
\begin{subfigure}[t]{0.45\textwidth}
\centering
\includegraphics[width=1.0\textwidth]{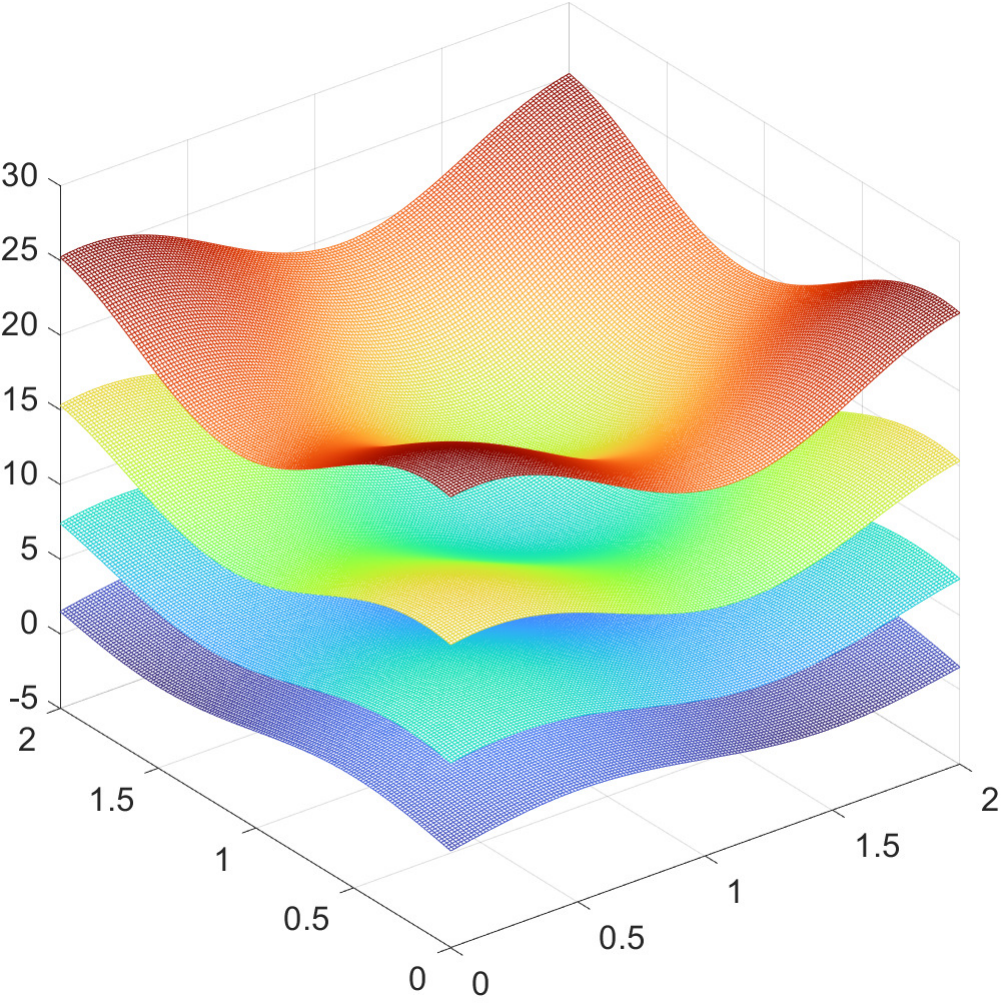}
  \caption{$h_1+h_2+h_3+b$, $h_2+h_3+b$, $h_3+b$, and $b$}
\end{subfigure}
\begin{subfigure}[t]{0.45\textwidth}
\centering
\includegraphics[width=1.0\textwidth]{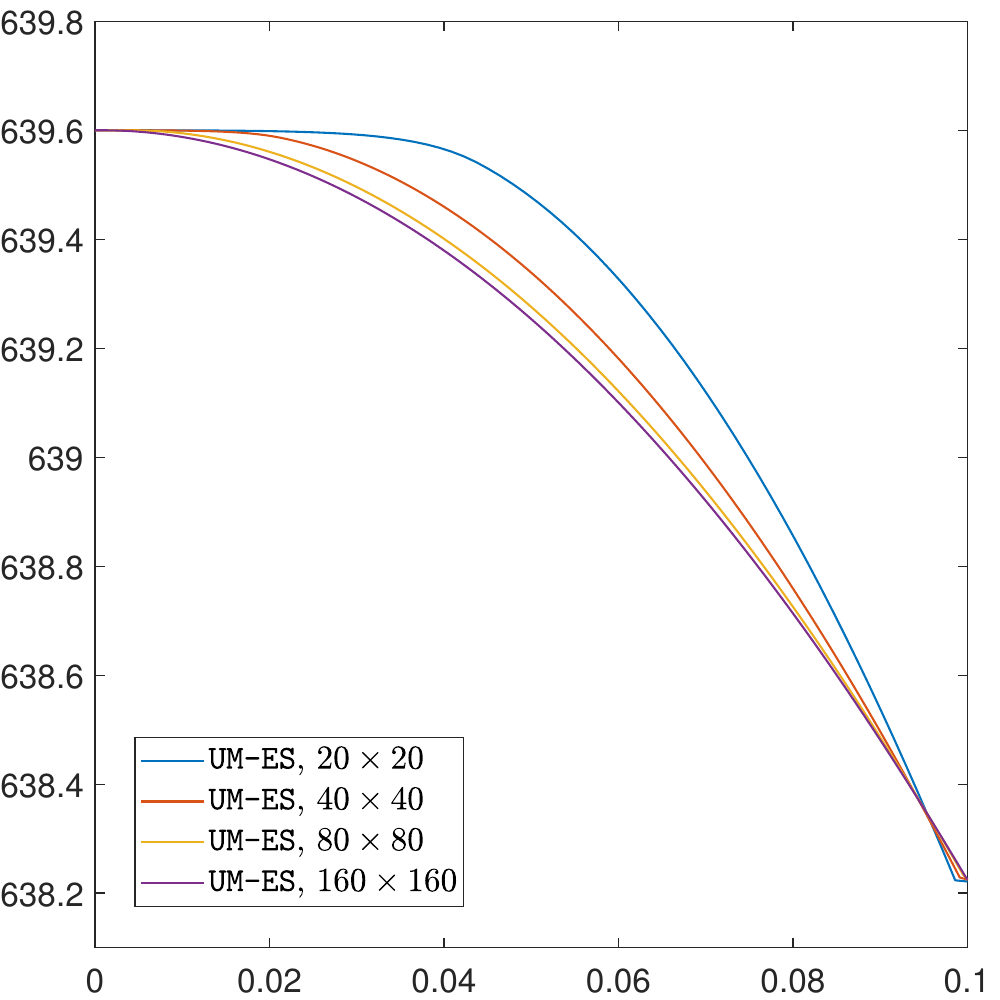}
  \caption{discrete total energy}
\end{subfigure}
  \caption{Example \ref{eq:Smooth_2D}.
  The water surface levels and bottom topography obtained by using the \texttt{UM-ES} scheme with $160\times160$ mesh at $t =0.1$,
  and the discrete total energy evolution over time on different meshes.
  Top: two-layer case, bottom: three-layer case.}\label{fig:2D_Smooth_Result_Entropy}
\end{figure}

\begin{example}[Accuracy test with a moving vortex]\rm\label{ex:2D_Vortex}
This example is used to further test the accuracy of our \texttt{MM-ES} schemes for the two-layer SWEs.
Similar to \cite{Duan2021_High,Zhang2023High}, a moving vortex can be constructed as follows
\begin{equation*}
    \begin{aligned}
        &h_1 = 5,~u_1 = v_1=0,\\
      &h_2=1-u_{\max}^2 e^{1-R^2} /(2 g), \\
      &\left(u_2, v_2\right)=(1,1)+u_{\max} e^{0.5\left(1-R^2\right)}(-(x_2-t), ~x_1-t), \\
    \end{aligned}
  \end{equation*}
  with $u_{\max} = 0.2$, $R = \sqrt{(x_1-t)^2+(x_2-t)^2}$, $g=1$, $r_{12}=0.7$, $\rho_2=1$, flat bottom topography, and the extra source terms
\begin{equation*}
	\bm{S} = (0,~-\frac{\exp(1 - (t - x_2)^2 - (t - x_1)^2)(2t - 2x_1)}{40},~-\frac{\exp(1 - (t - x_2)^2 - (t - x_1)^2)(2t - 2x_2)}{40},~0,~0,~0,~0)^{\mathrm{T}}.
\end{equation*}
  The monitor function is set to be
  \begin{equation*}
    \omega=\left({1+10 \left(\frac{|\nabla_{\bm{\xi}} (h_2+b)|}{\max |\nabla_{\bm{\xi}} (h_2+b)|}\right)^2+10\left(\frac{ |\Delta_{\bm{\xi}} (h_2+b)|}{\max |\Delta_{\bm{\xi}} (h_2+b)|}\right)^2}\right)^{\frac{1}{2}}.
  \end{equation*} 
\end{example}

Figure \ref{fig:2D_vortex} shows the adaptive mesh, the errors, and convergence orders in the water depth $h_2$ at $t =0.5$, and the evolution of the discrete total energy.
It is seen that the $5$th-order accuracy is achieved with the mesh points adaptively concentrating near the moving vortex,
and the discrete total energy $\sum\limits_{i,j} J_{i,j}\eta(\bm{U}_{i,j})\Delta \xi_{1} \Delta {\xi_2}$ decays as expected.

\begin{figure}[!htb]
  \centering
  \begin{subfigure}[b]{0.31\textwidth}
    \centering
    \includegraphics[width=1.0\textwidth]{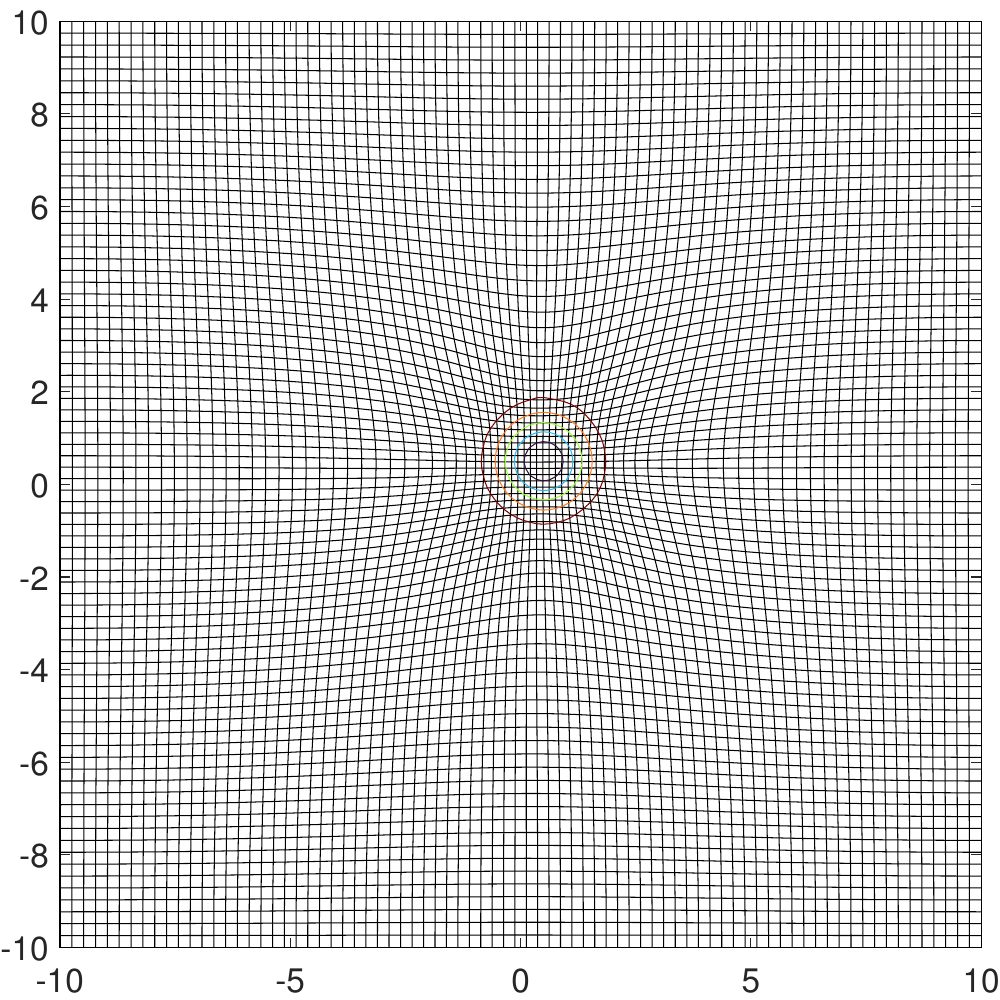}
    \caption{adaptive mesh}
  \end{subfigure}
  \begin{subfigure}[b]{0.35\textwidth}
    \centering
    \includegraphics[width=0.87\textwidth]{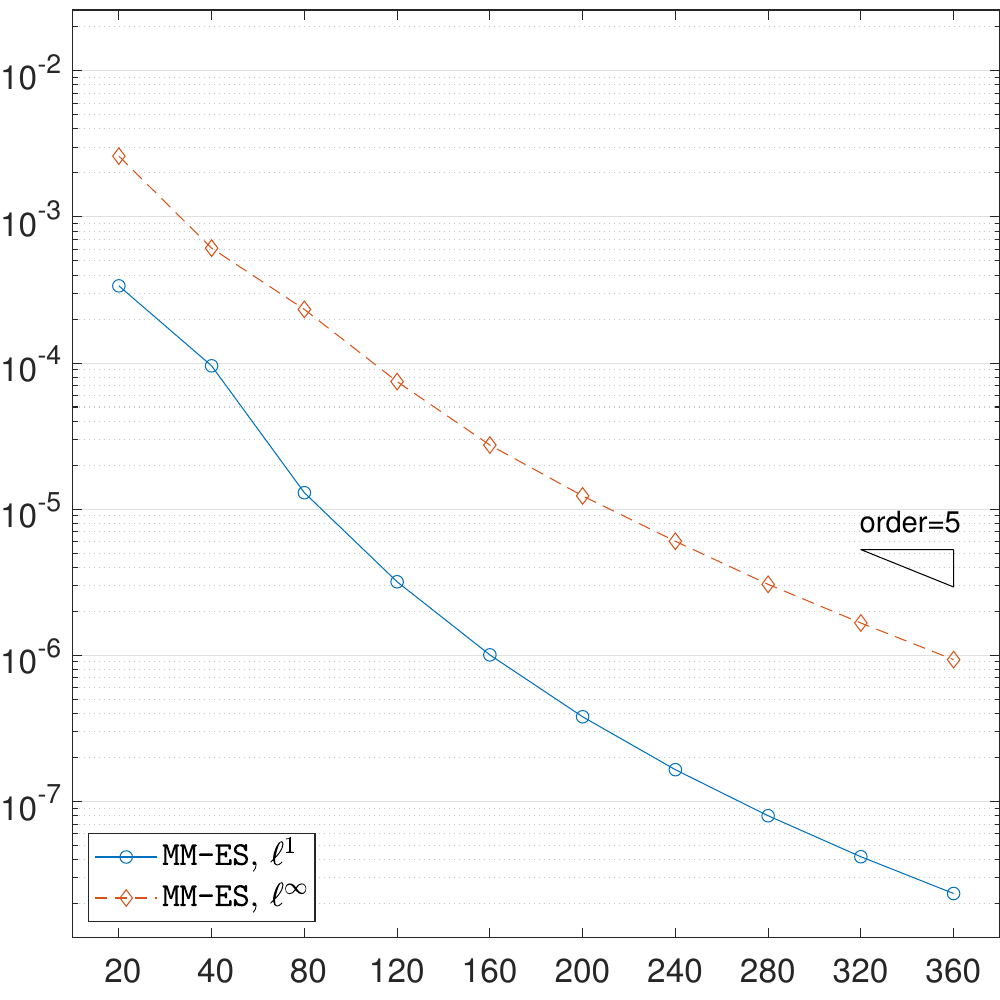}
    \caption{errors and convergence orders}
  \end{subfigure}
    \begin{subfigure}[b]{0.31\textwidth}
    \centering
    \includegraphics[width=1.0\textwidth]{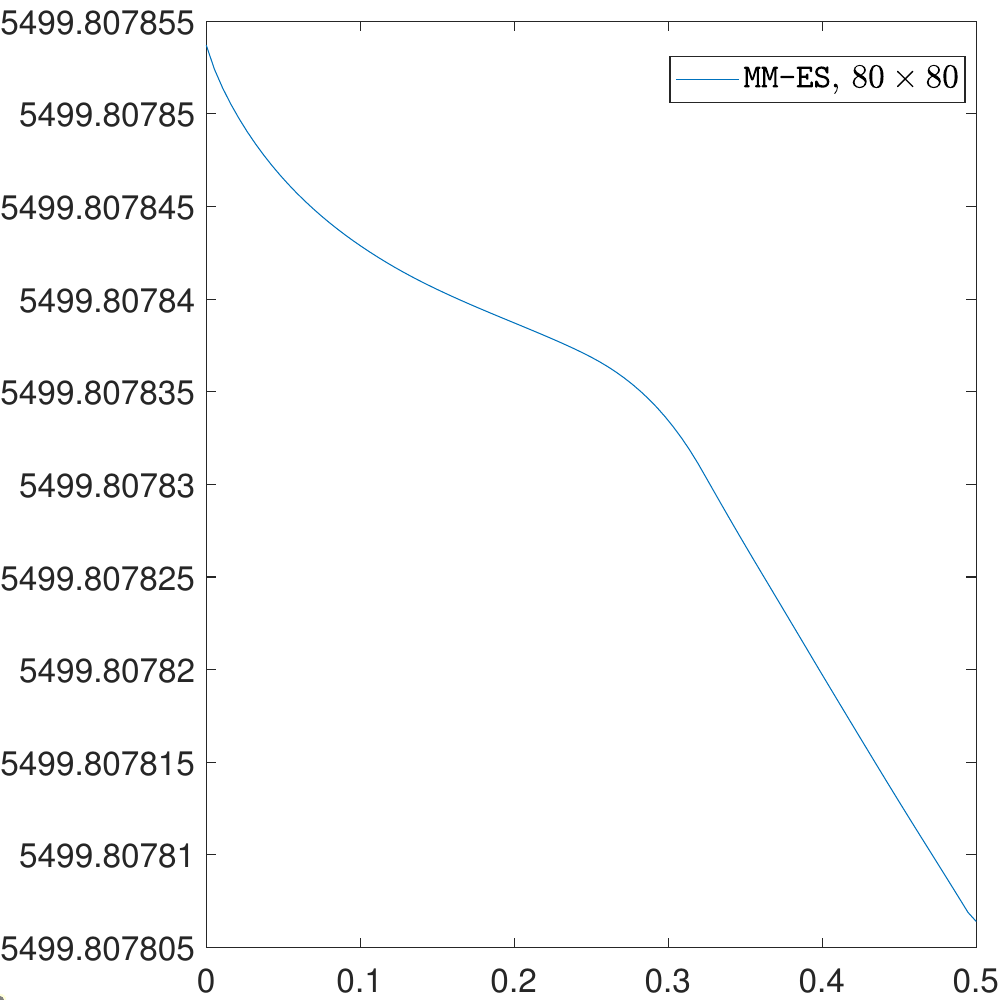}
    \caption{discrete total energy }
  \end{subfigure}
  \caption{Example \ref{ex:2D_Vortex}. 
  The $80\times80$ adaptive mesh with $5$ equally spaced contour lines of the water surface $h_2+b$, the errors and convergence orders in $h_2+b$ at $t=0.5$, and the discrete total energy obtained by the \texttt{MM-ES} scheme with $80\times80$ mesh.}\label{fig:2D_vortex}
\end{figure}

\begin{example}[2D WB test]\label{ex:2D_WB_Test}\rm

  The example is utilized to verify the WB properties of our 2D \texttt{UM-ES} and \texttt{MM-ES} schemes.
  The bottom topography is
  \begin{equation}\label{eq:2D_b_Smooth}
    b(x_1, x_2) = 1.2 \exp\left(-50\left((x_1-0.5)^2+(x_2-0.5)^2\right)\right),~ (x_1,x_2) \in [0,1]\times[0,1],
  \end{equation}
  or
  \begin{equation}\label{eq:2D_b_dis}
    b(x_1, x_2) = \begin{cases}
      1.0 , &\text{if}\quad  (x_1,x_2) \in [0.4,0.5]\times[0.4,0.5],\\
      0.5 , &\text{if}\quad  (x_1,x_2) \in ([0.4,0.6]\times[0.4,0.6]) \cup ([0.3,0.5]\times[0.3,0.5]) \backslash ([0.4,0.5]\times[0.4,0.5]),\\
      0, &\text{otherwise}.
    \end{cases}
  \end{equation}
  For the two-layer case, the initial data are $ h_2 = 2 - b$, $h_1 = 2.2-h_2-b$,
  zero velocities, with $g=1$, $r_{12} = 4/5$, $\rho_2=1$.
  For the three-layer case, the initial data are $ h_3 = 2 - b$, $h_2 = 2.2-h_2-b$, $h_1 = 2.4-h_2-h_1-b$ with $r_{12} = 4/5$, $r_{13} = 2/3$, $r_{23} = 5/6$, $\rho_3=6/5$.
  The monitor function is chosen as
  \begin{equation*}
      \omega = \left(1+\theta\left(\frac{\left|\nabla_{\bm{\xi}}\sigma\right|}{\max\left|\nabla_{\bm{\xi}}\sigma\right|}\right)^2\right)^{\frac{1}{2}},
  \end{equation*}
  with $\theta = 100$, $\sigma = h_2$ and $\sigma = h_3$ for the two-layer and three-layer cases, respectively.
  
\end{example}

The $\ell^1$ and $\ell^{\infty}$ errors in the water surface levels are listed in Tables \ref{tb:2D Well_Balance}-\ref{tb:2D Well_Balance_3L} by using our schemes with $100\times100$ mesh at $t=0.1$,
from which one observes that the lake at rest is maintained up to the rounding error in double precision.
Figures \ref{fig:2D_Well_Balance}-\ref{fig:2D_Well_Balance_3L} show the results of the water surface levels  $h_1+h_2+b$,  $h_2+b$ and $h_1+h_2+h_3+b$, $h_1+h_2+b$,  $h_2+b$ with corresponding bottom topography $b$  obtained by using the \texttt{MM-ES} scheme with $100\times100$ mesh.
Those results confirm that our schemes are WB on both fixed and adaptive moving meshes.

\begin{table}[!htb]
	\centering
	\begin{tabular}{cccccc}
		\hline\hline
		\multicolumn{2}{c}{\multirow{2}{*}{}} & \multicolumn{2}{c}{\texttt{UM-ES}} & \multicolumn{2}{c}{\texttt{MM-ES}}  \\ 
		\cmidrule(lr){3-4}\cmidrule(lr){5-6}
		\multicolumn{2}{c}{} & \multicolumn{1}{c}{$\ell^{1}$~error}  & \multicolumn{1}{c}{$\ell^{\infty}$~error} &  \multicolumn{1}{c}{$\ell^{1}$~error}  & \multicolumn{1}{c}{$\ell^{\infty}$~error}   \\
		\hline
		\multirow{2}{*}{$b$ in \eqref{eq:2D_b_Smooth}} &
	 $h_1+h_2+b$ & 6.32e-16 	 & 3.11e-15 	 & 1.02e-15 	 & 1.73e-14  \\ 
 	 &  $h_2+b$ & 6.40e-16 	 & 3.55e-15 	 & 9.24e-16 	 & 1.42e-14  \\ 
		\hline
		\multirow{2}{*}{$b$ in \eqref{eq:2D_b_dis}} &
	 $h_1+h_2+b$ & 1.95e-18 	 & 8.88e-16 	 & 1.99e-15 	 & 2.27e-14  \\ 
 	 &  $h_2+b$ & 2.13e-18 	 & 8.88e-16 	 & 1.83e-15 	 & 1.84e-14  \\ 
		\hline\hline
	\end{tabular}
	\caption{Example \ref{ex:2D_WB_Test} for the two-layer case. Errors in $h_1+h_2+b$ and $h_2+b$ obtained by our schemes using $100\times100$ mesh at $t=0.1$, with the bottom topography \eqref{eq:2D_b_Smooth} and \eqref{eq:2D_b_dis}.}\label{tb:2D Well_Balance}
\end{table}

\begin{table}[!htb]
	\centering
	\begin{tabular}{cccccc}
		\hline\hline
		\multicolumn{2}{c}{\multirow{2}{*}{}} & \multicolumn{2}{c}{\texttt{UM-ES}} & \multicolumn{2}{c}{\texttt{MM-ES}}  \\ 
			\cmidrule(lr){3-4}\cmidrule(lr){5-6}
		\multicolumn{2}{c}{} & \multicolumn{1}{c}{$\ell^{1}$~error}  & \multicolumn{1}{c}{$\ell^{\infty}$~error} &  \multicolumn{1}{c}{$\ell^{1}$~error}  & \multicolumn{1}{c}{$\ell^{\infty}$~error}   \\
		\hline
		\multirow{3}{*}{$b$ in \eqref{eq:2D_b_Smooth}} &
$h_1+h_2+h_3+b$ & 2.81e-16 	 & 2.67e-15 	 & 1.17e-15 	 & 1.24e-14  \\ 
 	 &  $h_2+h_3+b$ & 3.86e-16 	 & 3.11e-15 	 & 1.13e-15 	 & 1.07e-14   \\ 
 	   &$h_3+b$ & 3.11e-16 	 & 2.44e-15 	 & 1.02e-15 	 & 9.33e-15   \\ 
		\hline
		\multirow{3}{*}{$b$ in \eqref{eq:2D_b_dis}} &
	 $h_1+h_2+h_3+b$ & 2.66e-16 	 & 8.88e-16 	 & 1.57e-15 	 & 9.77e-15  \\ 
 	  & $h_2+h_3+b$ & 2.63e-16 	 & 1.33e-15 	 & 1.51e-15 	 & 1.11e-14   \\ 
 	  & $h_3+b$ & 1.99e-16 	 & 2.22e-15 	 & 1.41e-15 	 & 1.29e-14   \\ 
		\hline\hline
	\end{tabular}
		\caption{Example \ref{ex:2D_WB_Test} for the three-layer case. Errors in $h_1+h_2+h_3+b$, $h_2+h_3+b$ and $h_3+b$ obtained by our schemes using $100\times100$ mesh at $t=0.1$, with the bottom topography \eqref{eq:2D_b_Smooth} and \eqref{eq:2D_b_dis}.}\label{tb:2D Well_Balance_3L}
\end{table}

\begin{figure}[!htb]
  \centering
    \begin{subfigure}[b]{0.35\textwidth}
    \centering
    \includegraphics[width=1.0\textwidth]{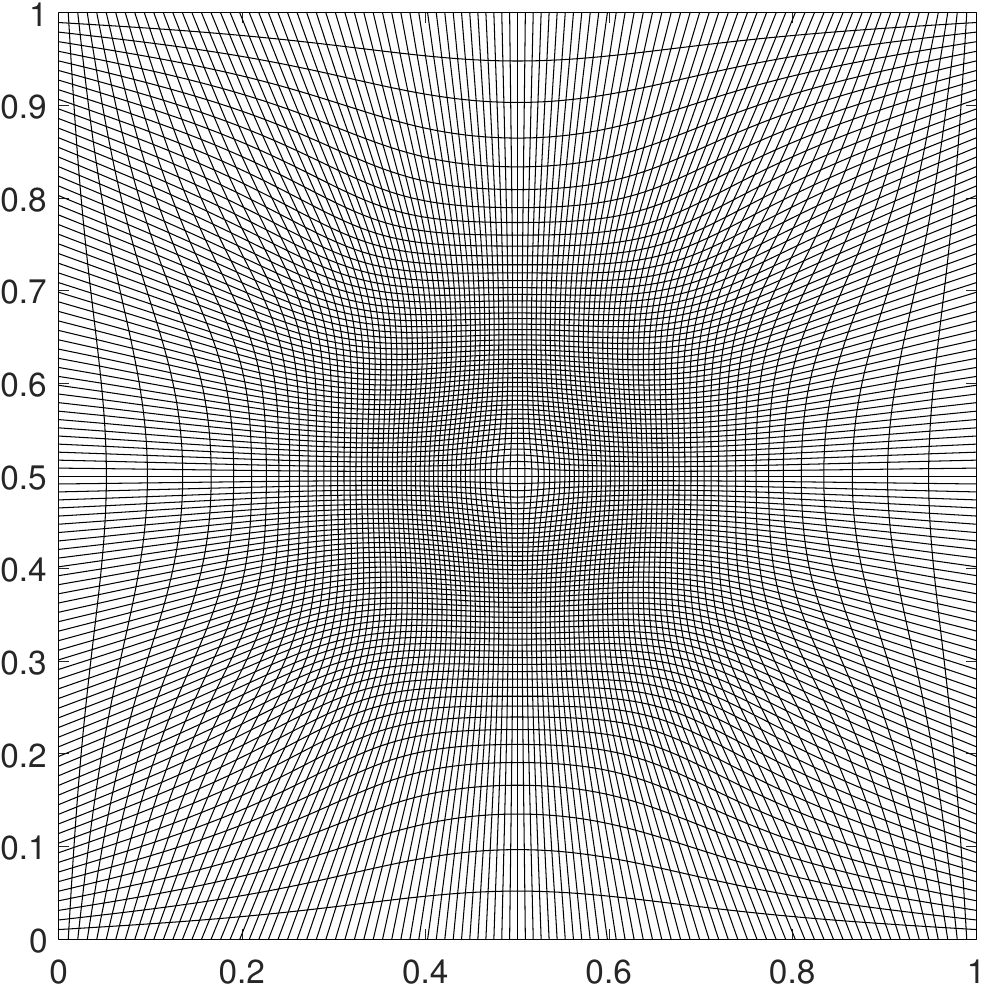}
    \caption{adaptive mesh}
  \end{subfigure}
  \begin{subfigure}[b]{0.35\textwidth}
    \centering
    \includegraphics[width=1.0\textwidth]{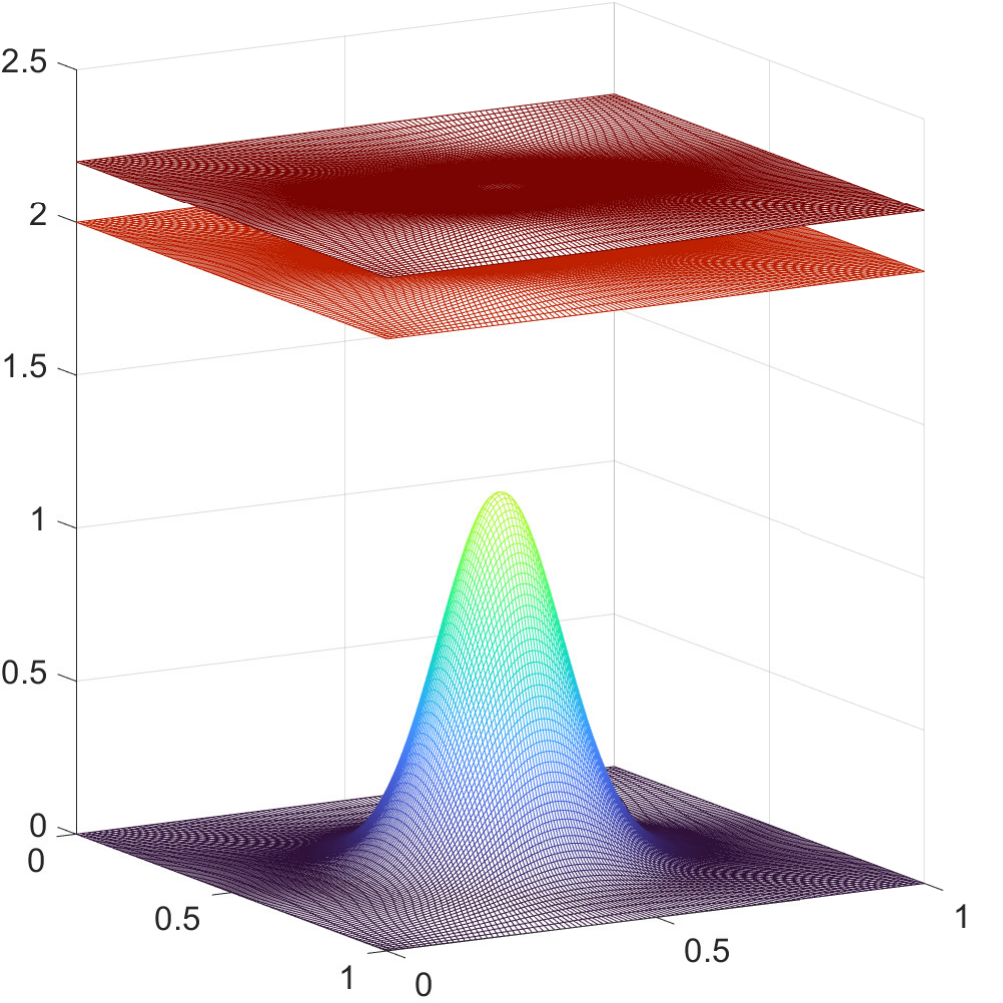}
       \caption{water surface levels}
  \end{subfigure}
\\
\begin{subfigure}[b]{0.35\textwidth}
	\centering
	\includegraphics[width=1.0\textwidth]{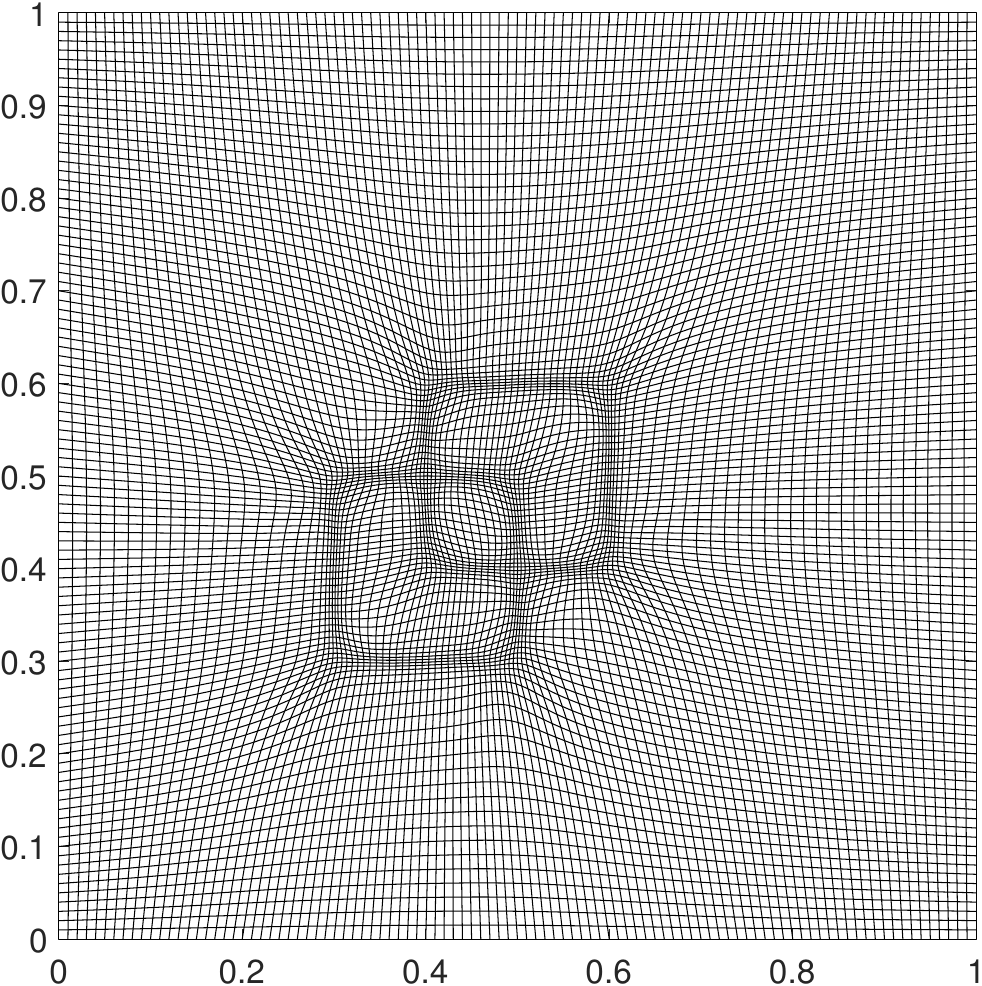}
 \caption{adaptive mesh}
\end{subfigure}
\begin{subfigure}[b]{0.35\textwidth}
	\centering
	\includegraphics[width=1.0\textwidth]{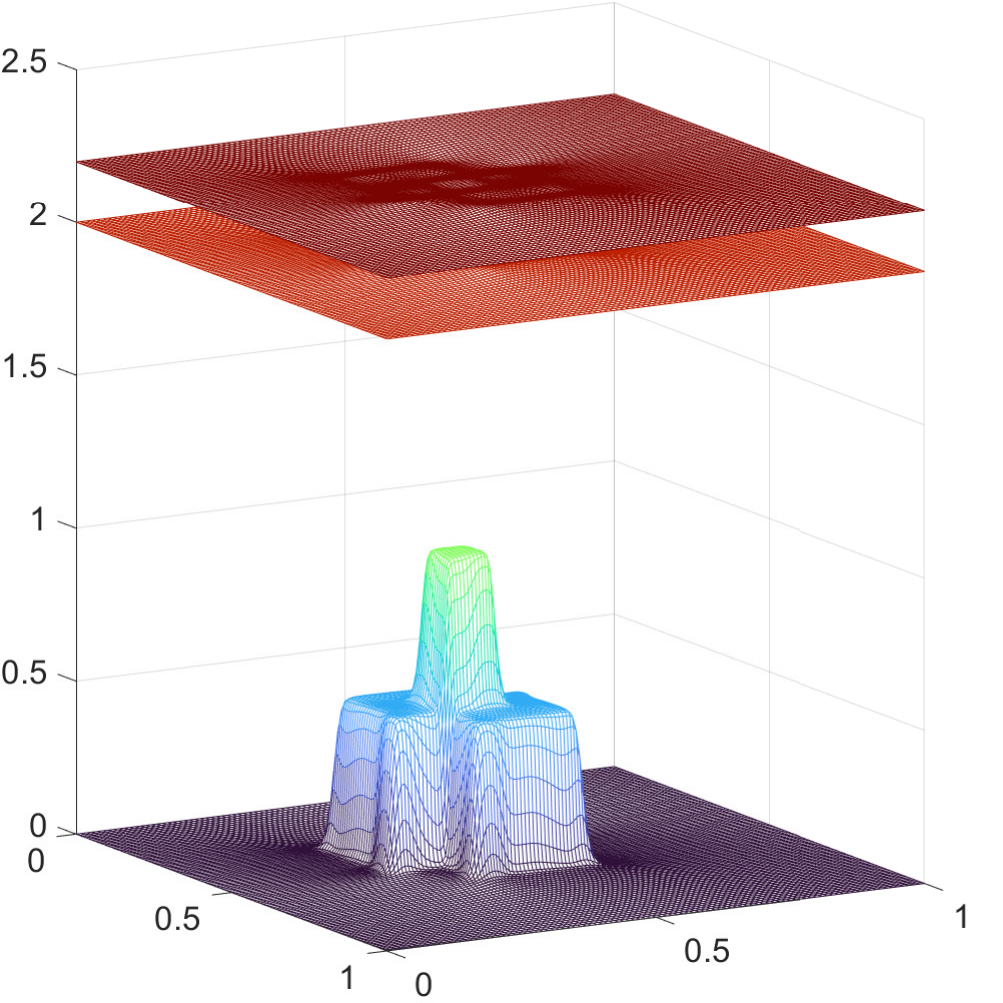}
        \caption{water surface levels}
\end{subfigure}
  \caption{Example \ref{ex:2D_WB_Test} for the two-layer case. The water surface levels $h_1+h_2+b$ and $h_2+b$  with the bottom topography $b$ and adaptive meshes obtained by using the \texttt{MM-ES} scheme with $100 \times 100$ mesh at $t=0.1$. Top:
  	with the bottom topography \eqref{eq:2D_b_Smooth}, bottom: with the bottom topography \eqref{eq:2D_b_dis}.}\label{fig:2D_Well_Balance}
\end{figure}

\begin{figure}[!htb]
  \centering
  \begin{subfigure}[b]{0.35\textwidth}
    \centering
    \includegraphics[width=1.0\textwidth]{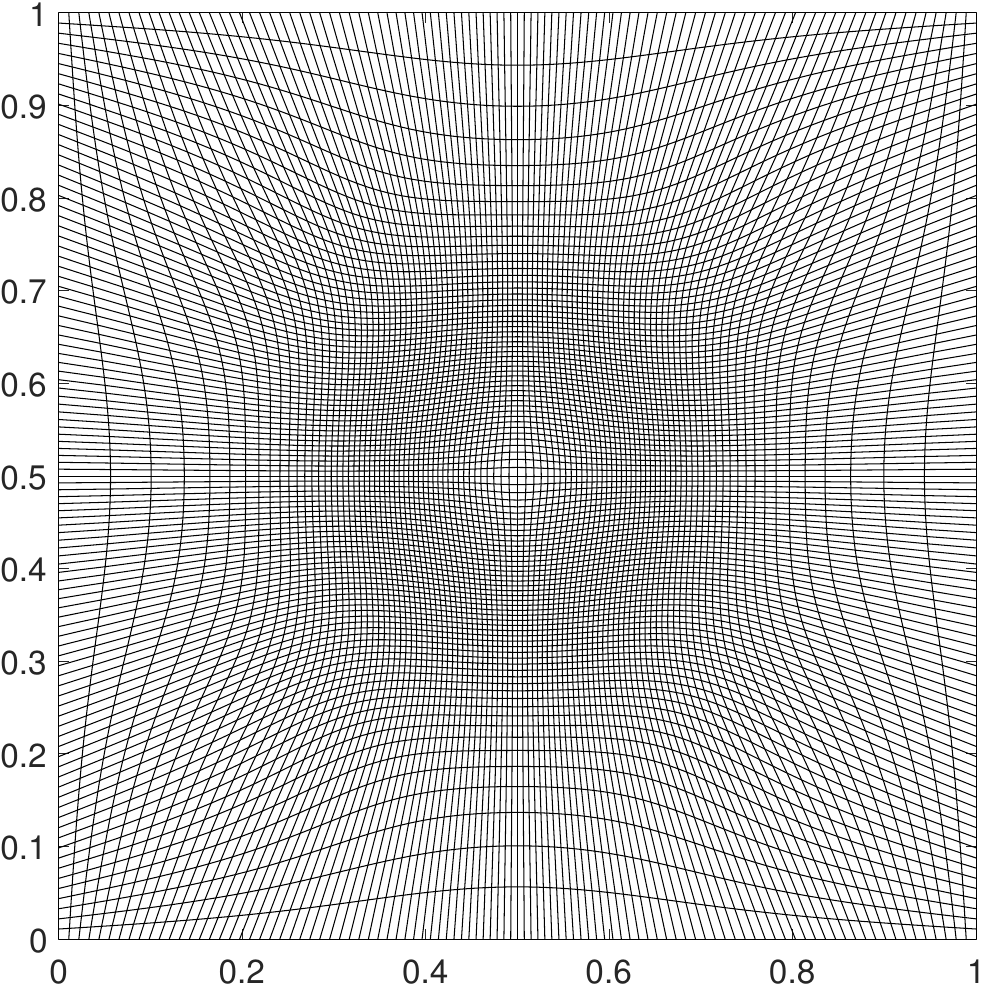}
    \caption{adaptive mesh}
  \end{subfigure}
  \begin{subfigure}[b]{0.35\textwidth}
    \centering
    \includegraphics[width=1.0\textwidth]{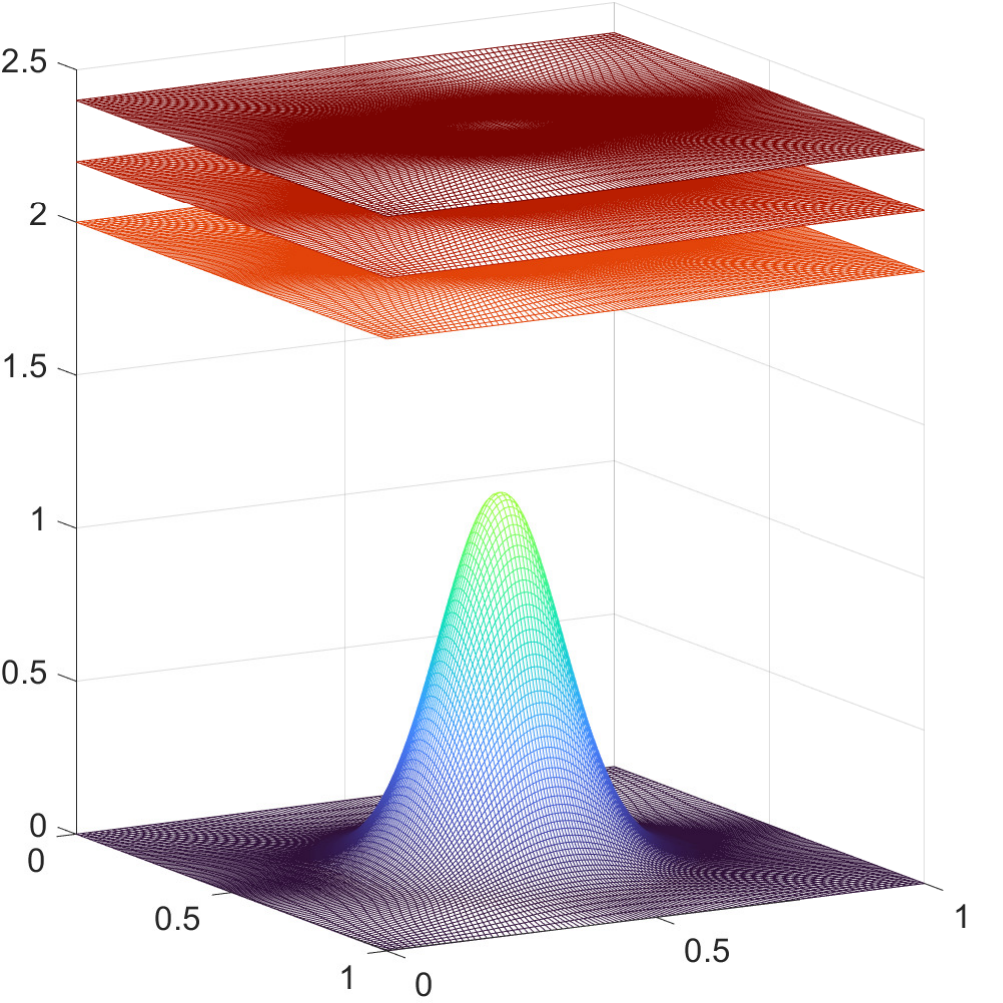}
    \caption{water surface levels}
  \end{subfigure}
  \\
  \begin{subfigure}[b]{0.35\textwidth}
	\centering
	\includegraphics[width=1.0\textwidth]{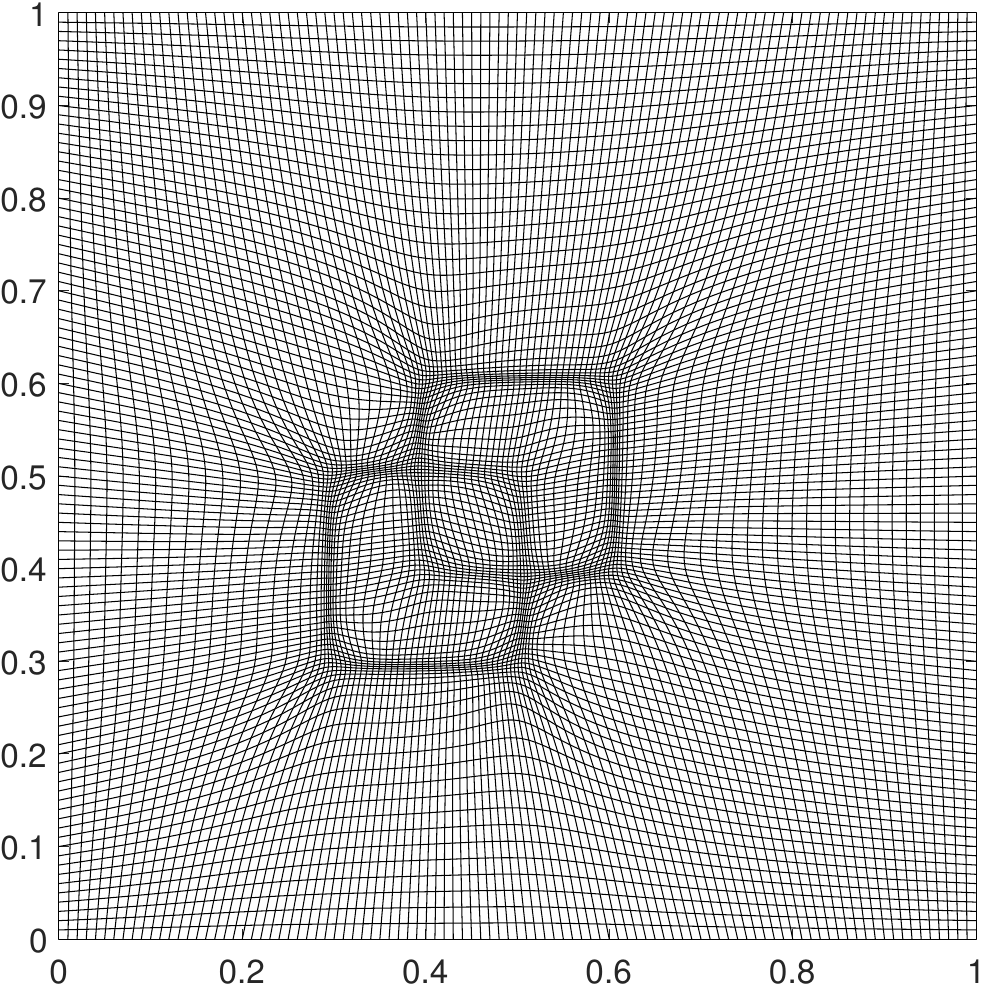}
 \caption{adaptive mesh}
\end{subfigure}
\begin{subfigure}[b]{0.35\textwidth}
	\centering
	\includegraphics[width=1.0\textwidth]{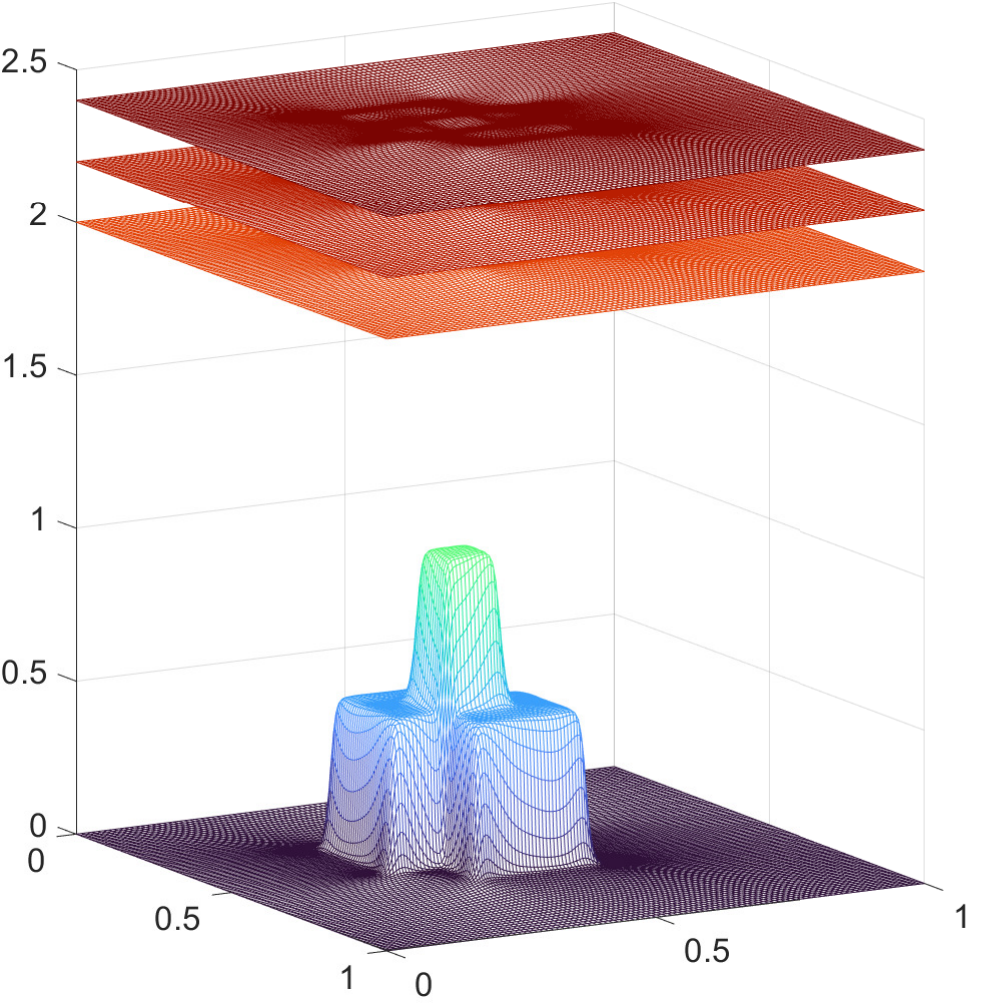}
 \caption{water surface levels}
\end{subfigure}
  \caption{Example \ref{ex:2D_WB_Test} for the three-layer case. The water surface levels $h_1+h_2+h_3+b$, $h_2+h_3+b$, $h_3+b$, with the bottom topography $b$ and adaptive meshes obtained by using the \texttt{MM-ES} scheme with $100 \times 100$ mesh at $t=0.1$. Top: with the bottom topography \eqref{eq:2D_b_Smooth}, bottom: with the bottom topography \eqref{eq:2D_b_dis}.}\label{fig:2D_Well_Balance_3L}
\end{figure}

\begin{example}[Propagation of an interface over a non-flat bottom]\rm\label{ex:Pro_NonFlat_Test}
This example tests the interface propagation over a non-flat bottom topography. Initially, the bottom topography is a Gaussian-shaped function 
\begin{equation*}
	b(x_1,~x_2) = 0.05e^{-100\left(x_1^2+x_2^2\right)}-1,
\end{equation*}
with the initial data for the two-layer case
\begin{equation*}
\left(h_1,~u_1,~v_1,~h_2,~u_2,~v_2\right) = \begin{cases}
	\left(0.50,~2.5,~2.5,~-0.50-b,~2.5,~2.5\right), & \text{if}\quad \left(x_1,x_2\right) \in \Omega,\\
	\left(0.45,~2.5,~2.5,~-0.45-b,~2.5,~2.5\right), & \text{otherwise},
\end{cases}
\end{equation*}
where
\begin{align*}
	\Omega =& \{(x_1,x_2)\mid x_1<-0.5,~x_2<0\}\cup\{(x_1,x_2)\mid x_1<0,~x_2<-0.5\} \\
 &\cup\{(x_1,x_2)\mid(x_1+0.5)^2+(x_2+0.5)^2<0.5^2\}.
\end{align*}
The physical domain is $[-1,1] \times [-1,1]$ with outflow boundary conditions, and $g=10$, $r_{12}=49/50$, $\rho_2=1$.
The solution is evolved up to $t=0.1$.
For the three-layer case, the initial data are
\begin{align*}
&\left(h_2,~u_2,~v_2,~h_3,~u_3,~v_3\right) = \begin{cases}
	\left(0.50,~2.5,~2.5,~-0.50-b,~2.5,~2.5\right), & \text{if}\quad \left(x_1,~x_2\right) \in \Omega,\\
	\left(0.45,~2.5,~2.5,~-0.45-b,~2.5,~2.5\right), & \text{otherwise},
\end{cases}\\
&\left(h_1,~u_1,~v_1\right) = 
	\left(1.0,~0.0,~0.0\right),
\end{align*}
with $r_{12} = 49/50$, $r_{13} = 49/55$, $r_{23} = 10/11$, and $\rho_3=11/10$.
The monitor function is the same as that in Example \ref{ex:2D_WB_Test}, except for $\sigma = h_1$ in both cases and $\theta = 300$.

\end{example}
\begin{figure}[!htb]
	\centering
	\begin{subfigure}[b]{0.4\textwidth}
		\centering
		\includegraphics[width=1.0\textwidth]{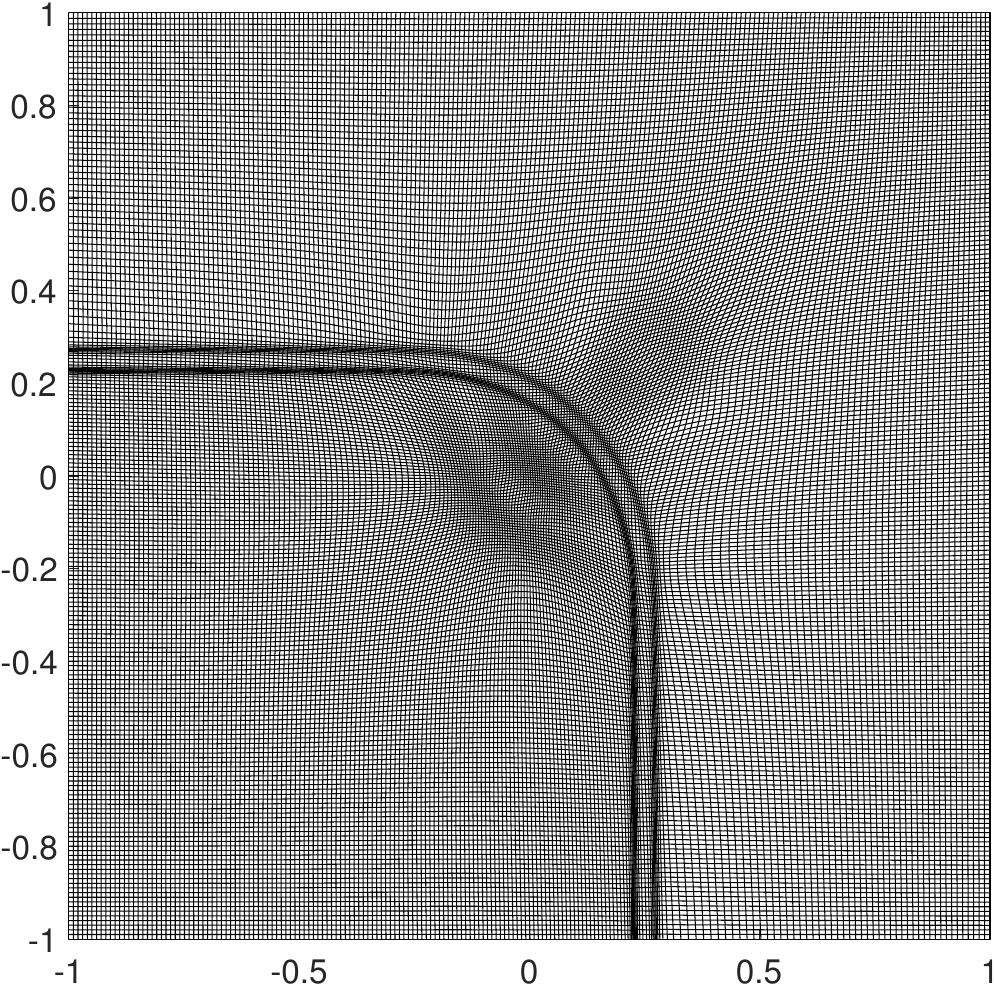}
  \caption{two-layer case}
	\end{subfigure}
	\begin{subfigure}[b]{0.4\textwidth}
		\centering
		\includegraphics[width=1.0\textwidth]{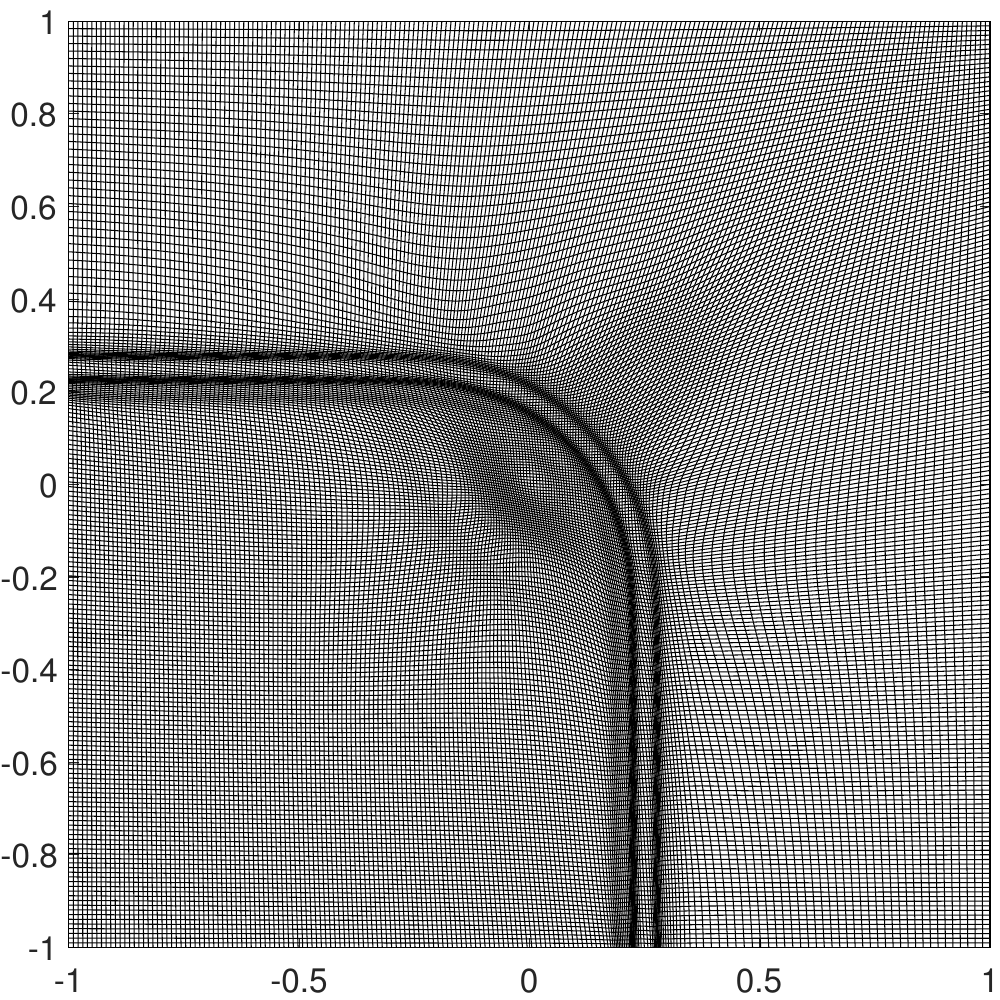}
  \caption{three-layer case}
	\end{subfigure}
 \caption{Example \ref{ex:Pro_NonFlat_Test}. The $200\times200$ adaptive meshes obtained by the \texttt{MM-ES} scheme.
}
\label{fig:2D_Pro_NonFlat_Test_Mesh}
\end{figure}

\begin{figure}[!htb]
  \centering
  \begin{subfigure}[b]{0.3\textwidth}
    \centering
    \includegraphics[width=1.0\textwidth]{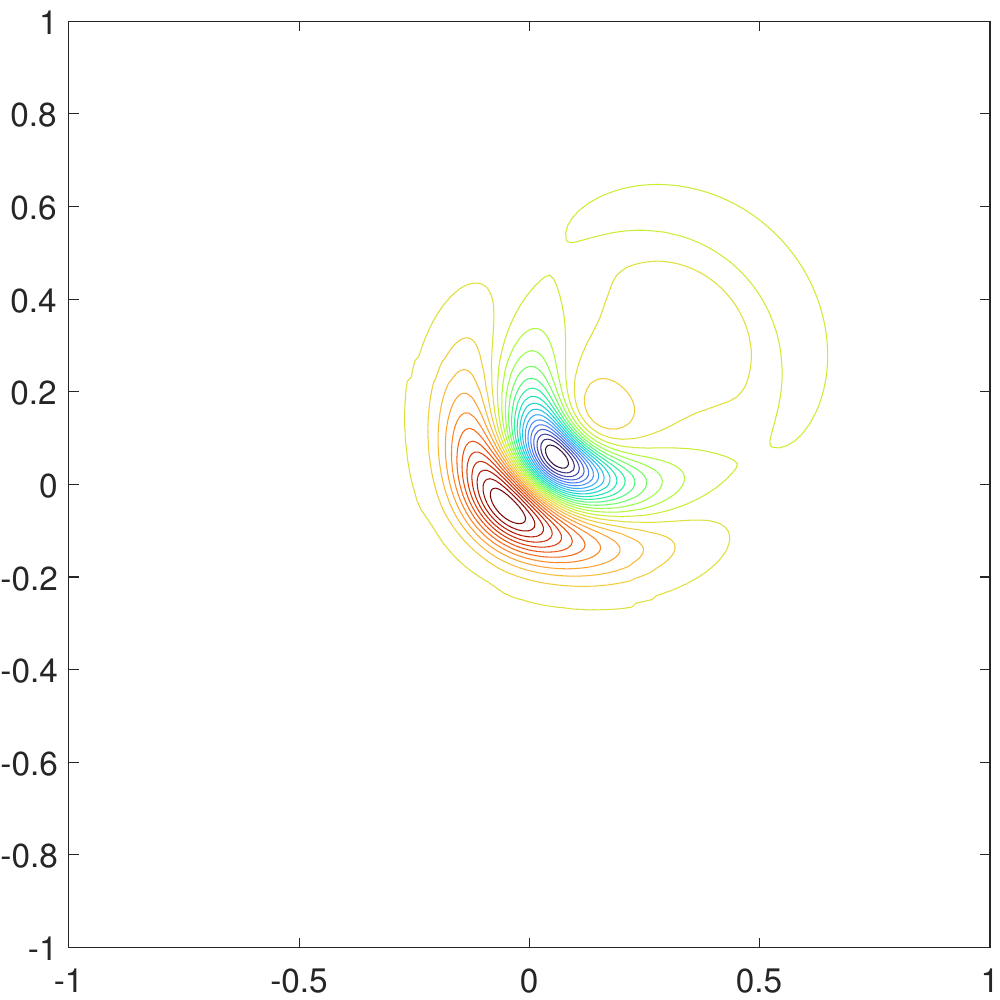}
    \caption{$h_1+h_2+b$}
  \end{subfigure}
  \begin{subfigure}[b]{0.3\textwidth}
    \centering
    \includegraphics[width=1.0\textwidth]{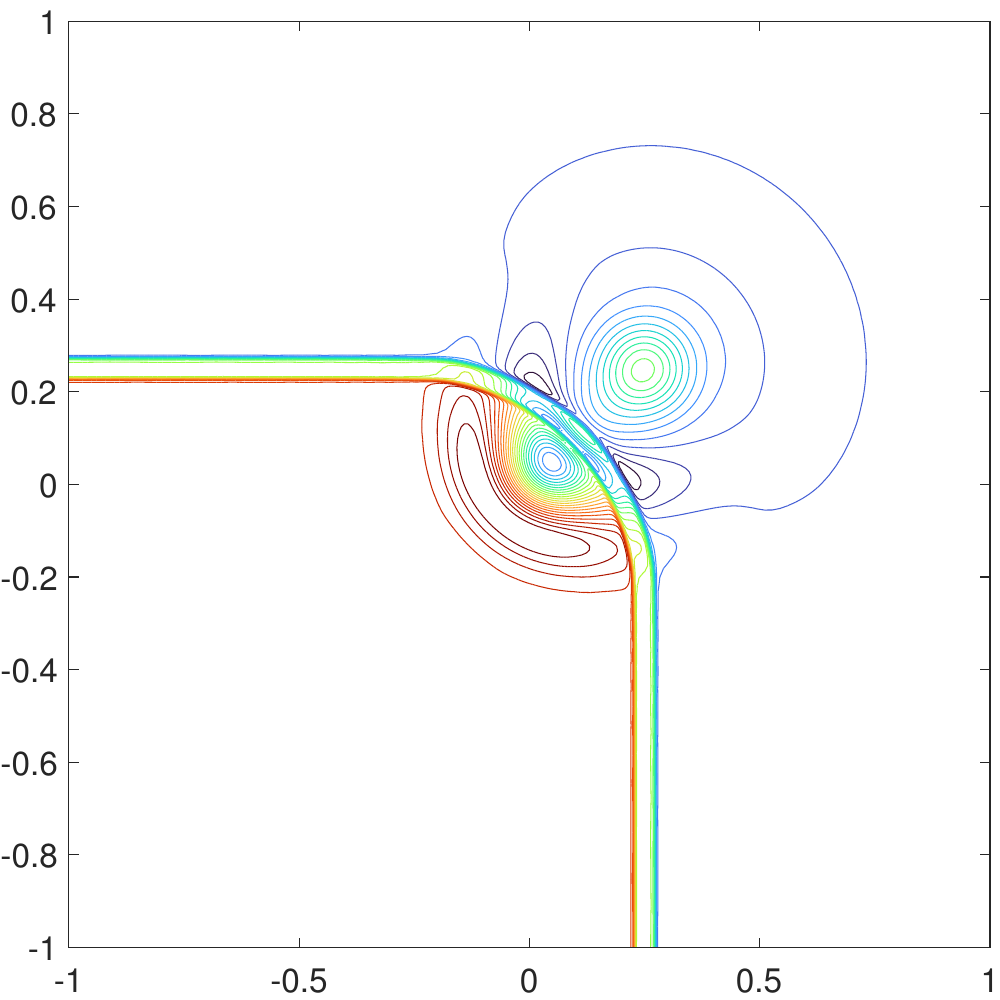}
    \caption{$h_1$}
  \end{subfigure}
\\
\begin{subfigure}[b]{0.3\textwidth}
	\centering
	\includegraphics[width=1.0\textwidth]{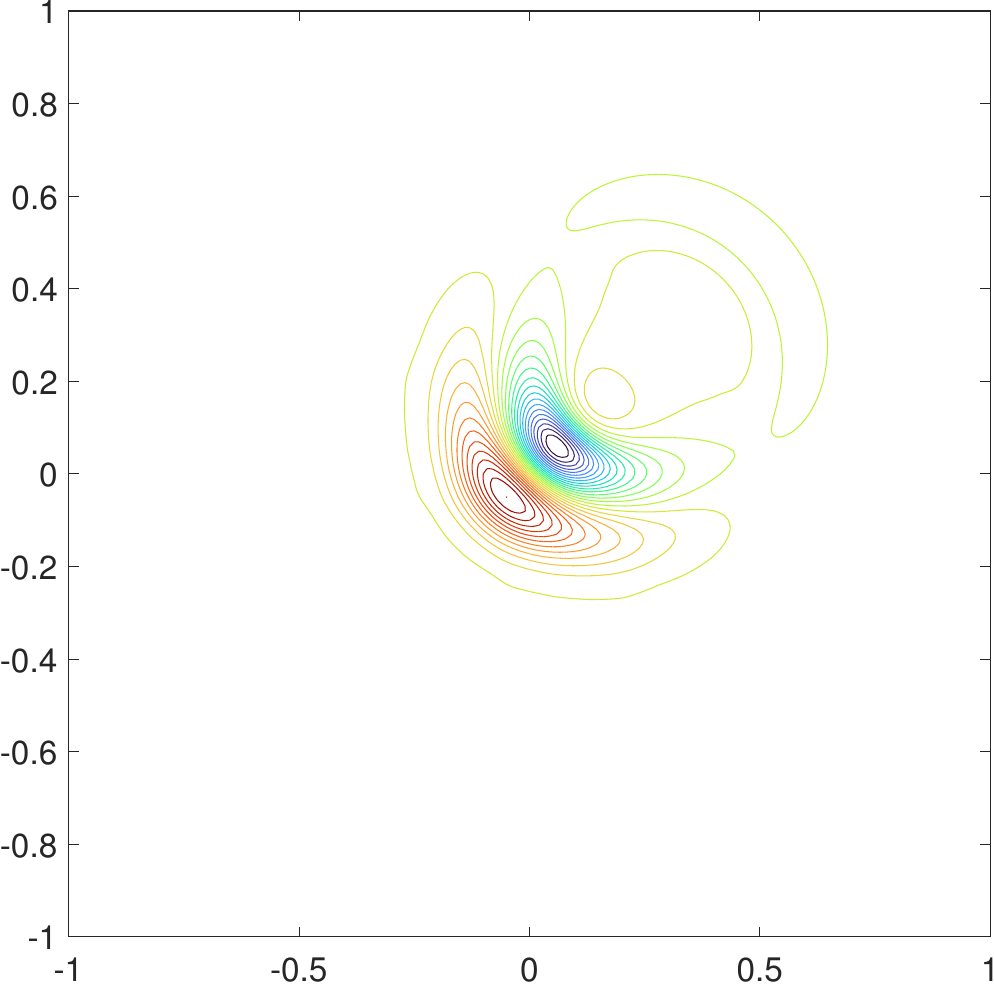}
  \caption{$h_1+h_2+b$}
\end{subfigure}
\begin{subfigure}[b]{0.3\textwidth}
	\centering
	\includegraphics[width=1.0\textwidth]{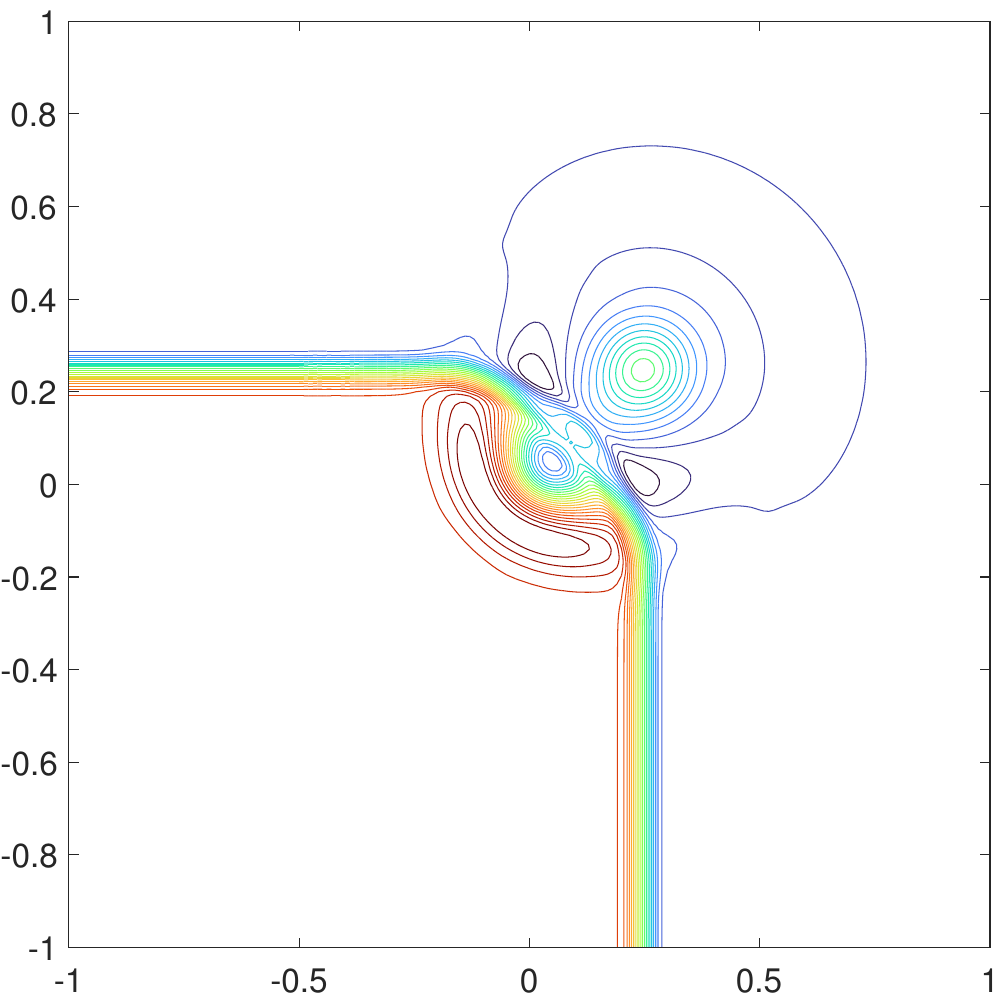}
     \caption{$h_1$}
\end{subfigure}
\\
\begin{subfigure}[b]{0.3\textwidth}
	\centering
	\includegraphics[width=1.0\textwidth]{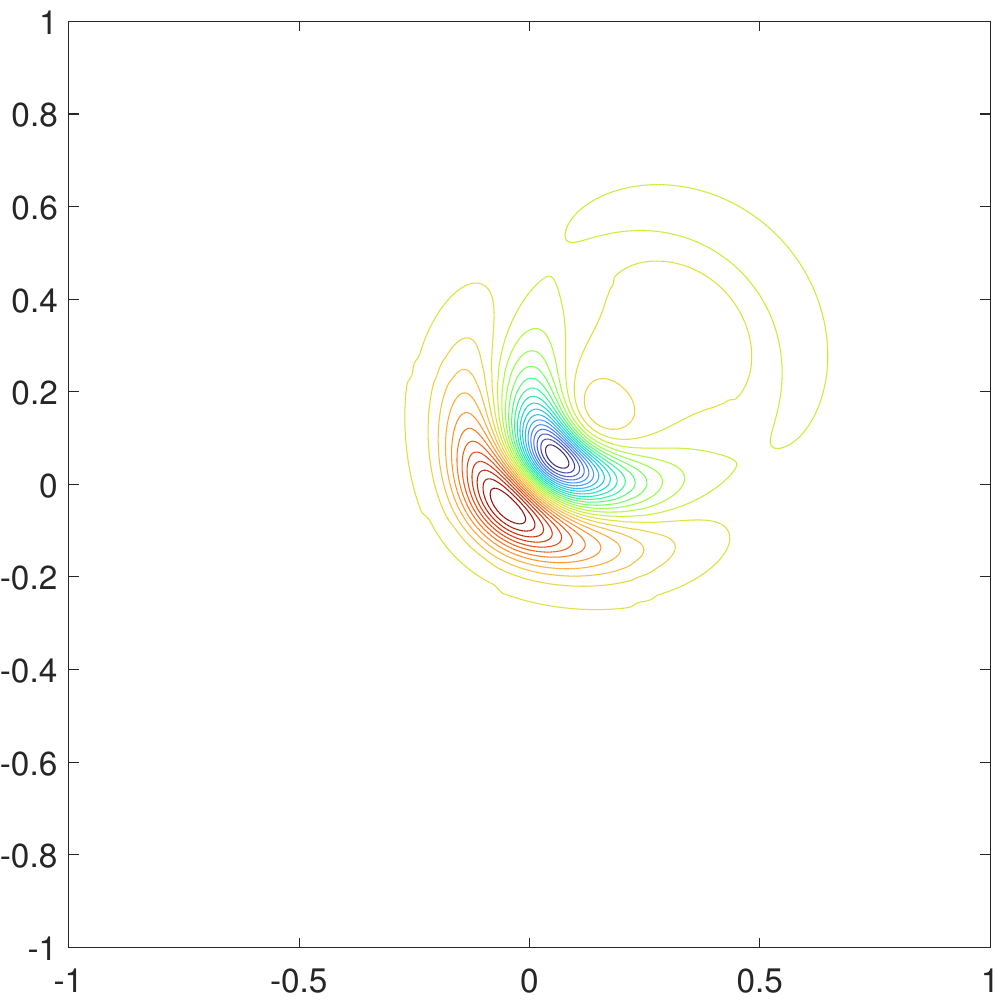}
  \caption{$h_1+h_2+b$}
\end{subfigure}
\begin{subfigure}[b]{0.3\textwidth}
	\centering
	\includegraphics[width=1.0\textwidth]{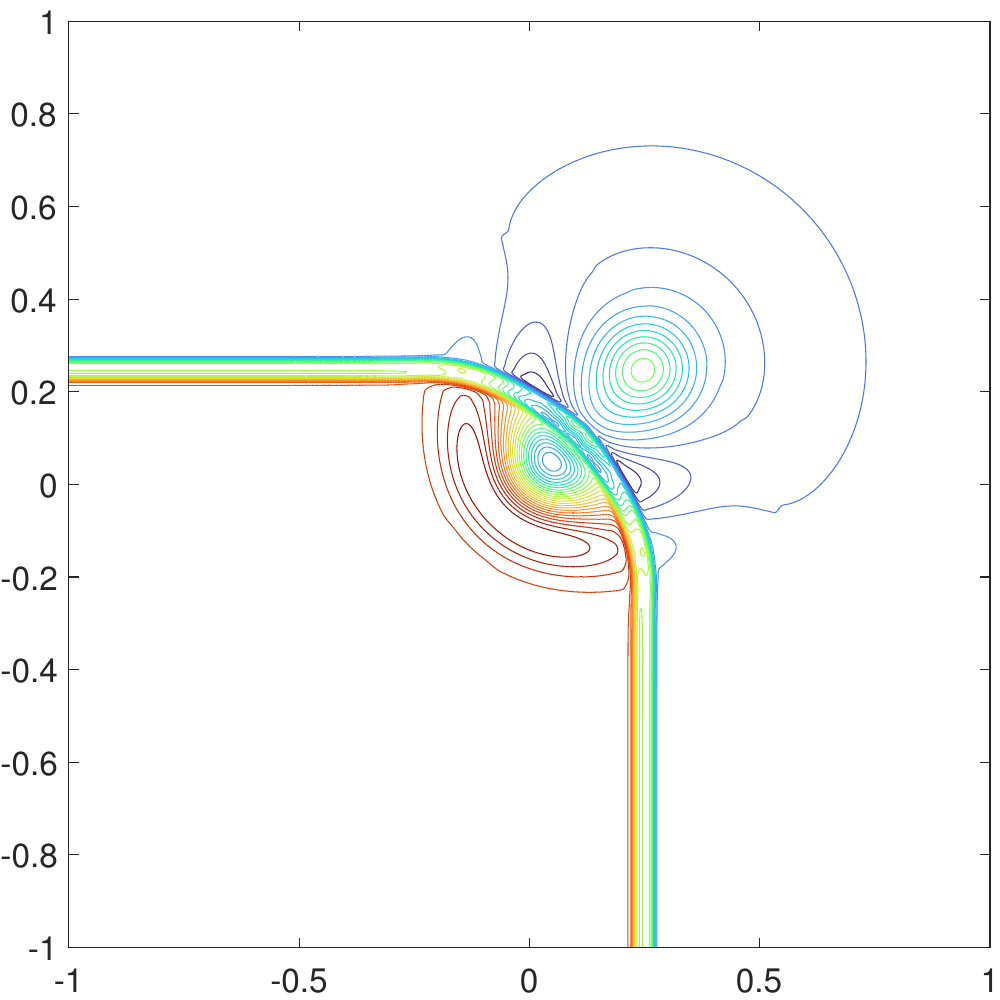}
     \caption{$h_1$}
\end{subfigure}
  \caption{Example \ref{ex:Pro_NonFlat_Test} for the two-layer case.  $30$ equally spaced contours obtained by our schemes.
  From top to bottom: \texttt{MM-ES} $(200\times200)$, \texttt{UM-ES} $(200\times200)$, \texttt{UM-ES} $(600\times600)$.
  }
  \label{fig:2D_Pro_NonFlat_Test}
\end{figure}

\begin{figure}[!htb]
	\centering
 \begin{subfigure}[b]{0.3\textwidth}
		\centering
		\includegraphics[width=1.0\textwidth]{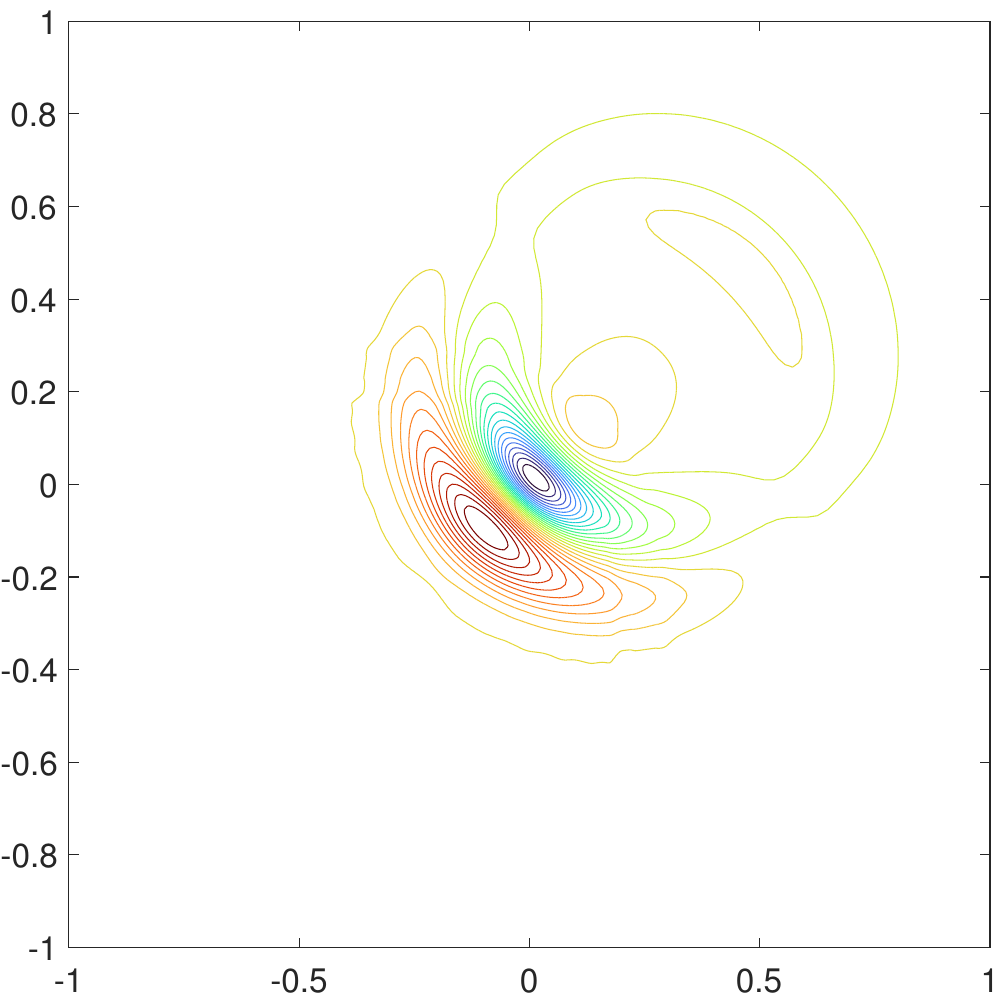}
  \caption{$h_1+h_2+h_3+b$}
	\end{subfigure}
	\begin{subfigure}[b]{0.3\textwidth}
		\centering
		\includegraphics[width=1.0\textwidth]{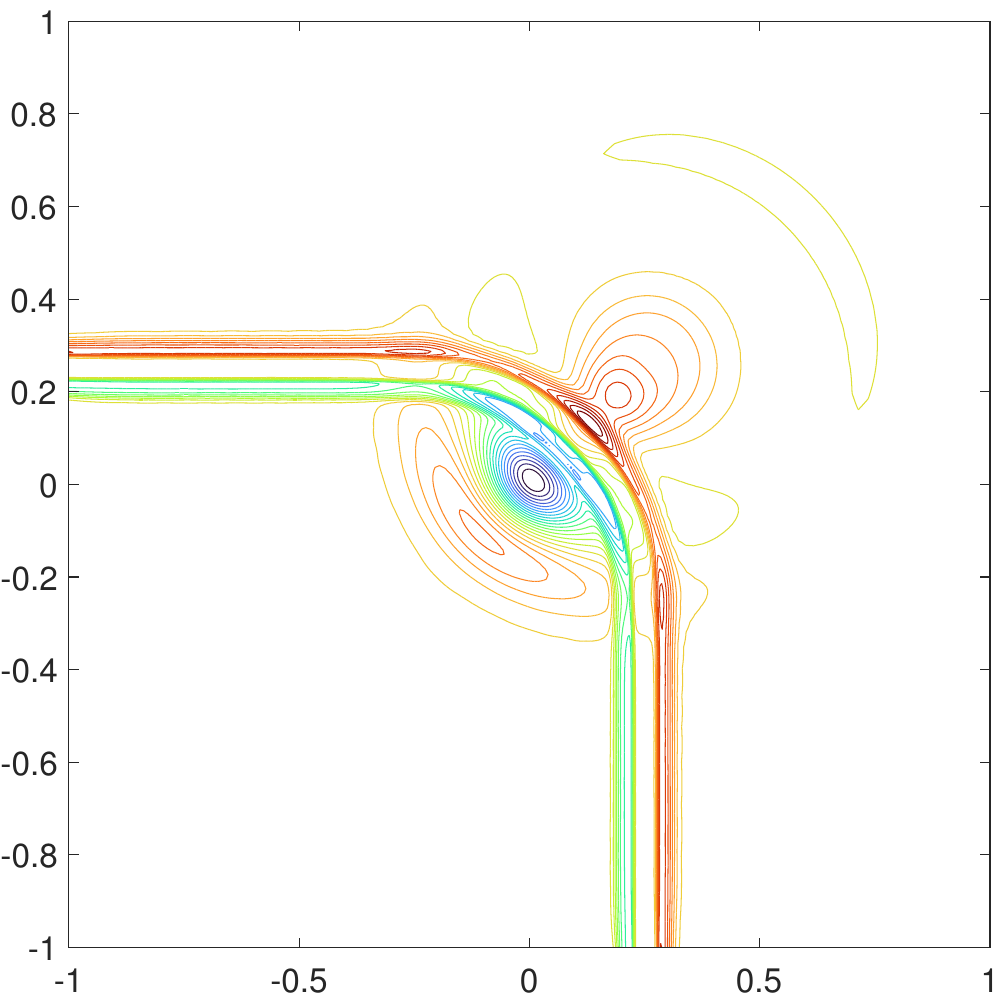}
    \caption{$h_1$}
	\end{subfigure}
 \\
 \begin{subfigure}[b]{0.3\textwidth}
		\centering
		\includegraphics[width=1.0\textwidth]{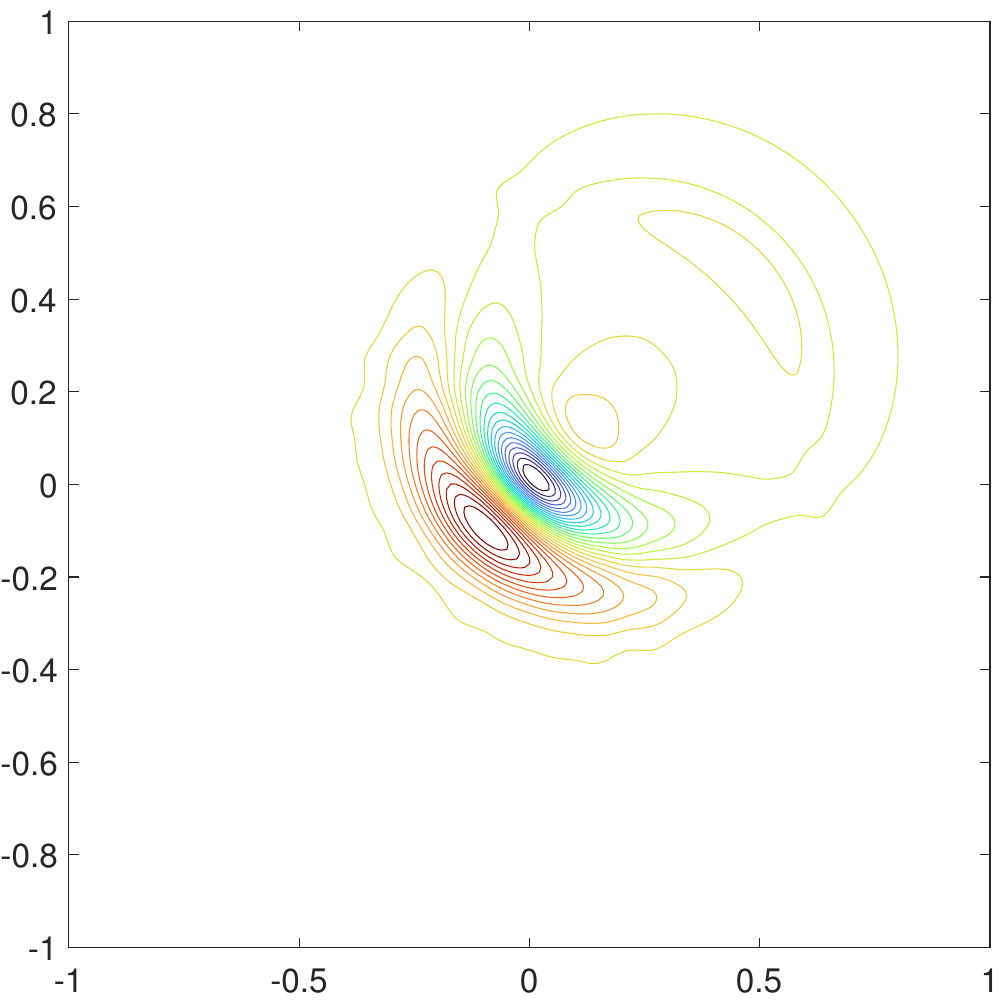}
    \caption{$h_1+h_2+h_3+b$}
	\end{subfigure}
	\begin{subfigure}[b]{0.3\textwidth}
		\centering
		\includegraphics[width=1.0\textwidth]{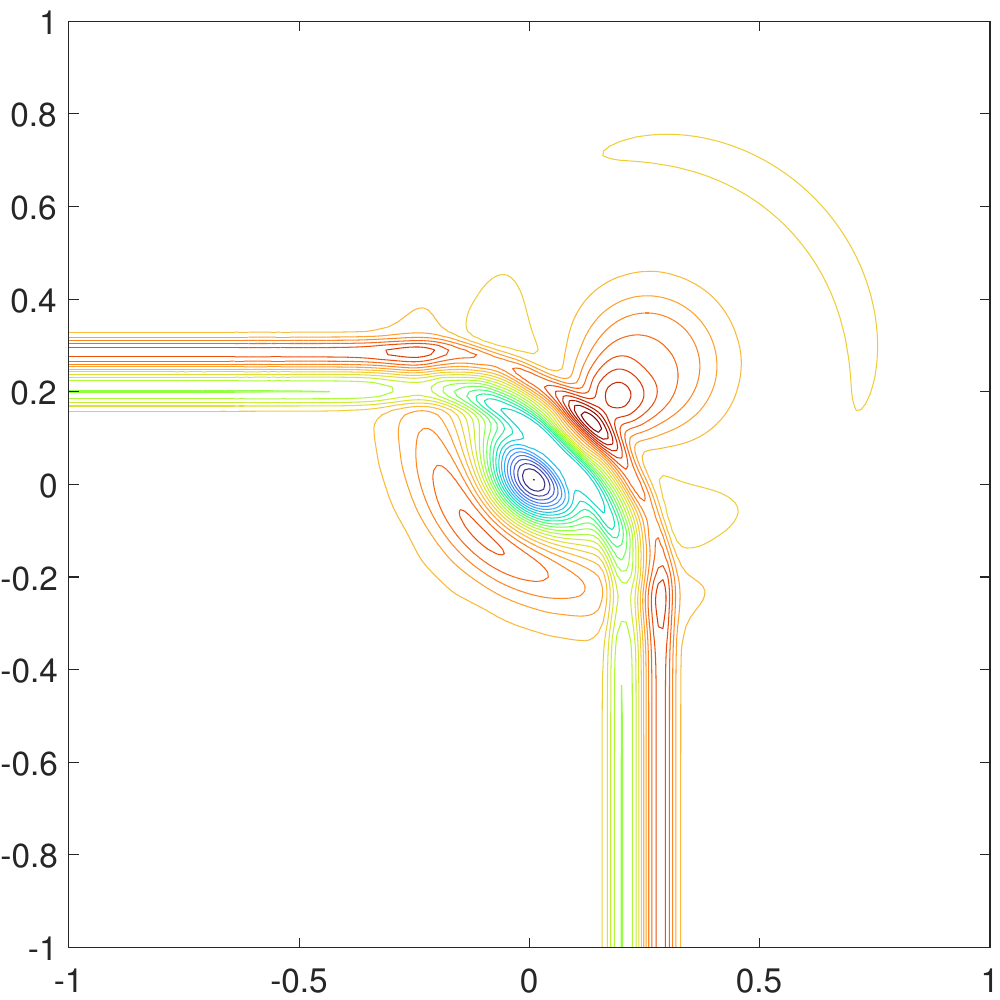}
  \caption{$h_1$}
	\end{subfigure}
 \\
	\begin{subfigure}[b]{0.3\textwidth}
		\centering
		\includegraphics[width=1.0\textwidth]{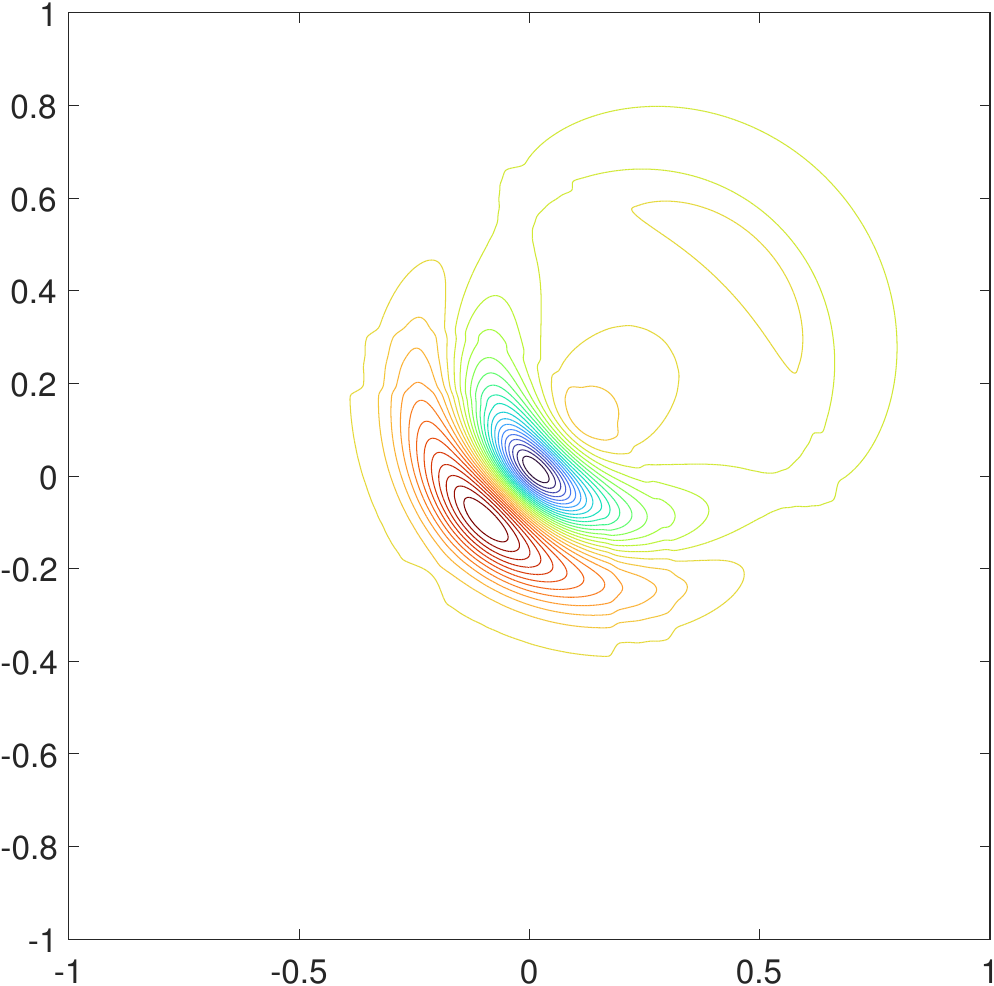}
    \caption{$h_1+h_2+h_3+b$}
	\end{subfigure}
	\begin{subfigure}[b]{0.3\textwidth}
		\centering
		\includegraphics[width=1.0\textwidth]{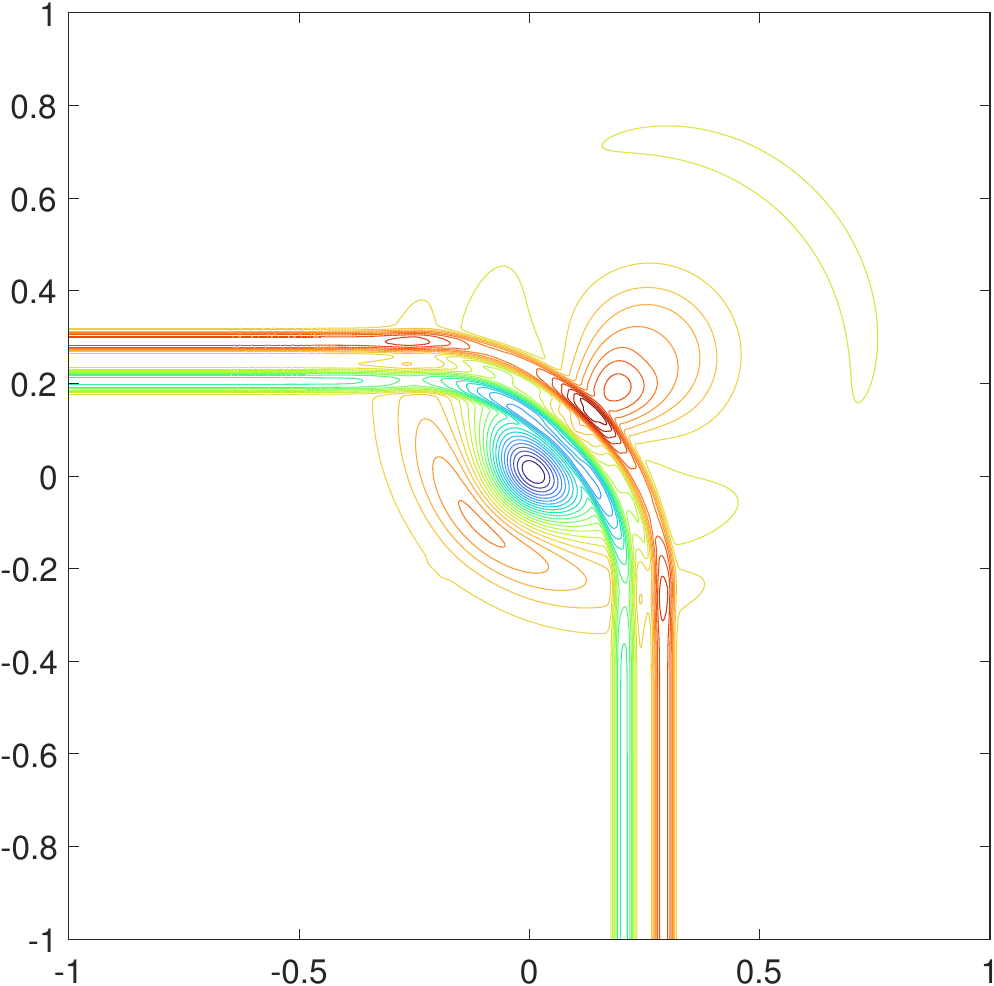}
  \caption{$h_1$}
	\end{subfigure}
	\caption{Example \ref{ex:Pro_NonFlat_Test} for the three-layer case. $30$ equally spaced contours obtained by our schemes. From top to bottom: {\tt MM-ES~}(200$\times$200), {\tt UM-ES~}(200$\times$200), {\tt UM-ES~}(600$\times$600).
	}
	\label{fig:2D_Pro_NonFlat_Test_3L}
\end{figure}
\begin{table}[hbt!]
  \centering
  \begin{tabular}{crrrrr}
    \hline\hline
    & two-layer & three-layer\\ \hline
    {\tt UM-ES}~($200\times200$) & 3m44s   & 21m29s      \\
    {\tt MM-ES}~($200\times200$) & 20m05s    & 63m35s     \\
    {\tt UM-ES}~($600\times600$) & 2h42m   & 17h54m   \\ 
    \hline\hline
  \end{tabular}
  \caption{Example \ref{ex:Pro_NonFlat_Test}. The CPU times for our schemes on the different meshes. The tests are solved up to $t=0.1$.}\label{tb:Time_Compare}
\end{table}
Figure \ref{fig:2D_Pro_NonFlat_Test_Mesh} shows the $200\times200$ adaptive meshes obtained by the \texttt{MM-ES} schemes.
Figure \ref{fig:2D_Pro_NonFlat_Test} plots $30$ equally spaced contours of the water surface levels obtained by the \texttt{UM-ES} scheme on $200\times200$ and $600\times600$ meshes, and obtained by the \texttt{MM-ES} scheme with $200\times200$ mesh at $t = 0.1$.
It can be seen that the results are comparable to those in \cite{Kurganov2009Central},
and our schemes can capture the complicated structure around $(0.2,~0.2)$ without oscillations.
One can observe that the mesh points concentrate near the sharp transitions of the water surface level $h_1$,
which improves the local resolution such that the results on $200\times 200$ adaptive moving mesh are as good as those on $600\times 600$ fixed uniform mesh.
The numerical results for the three-layer case are also shown in Figure \ref{fig:2D_Pro_NonFlat_Test_3L}.
One can also see that the resolution of the numerical results obtained by the \texttt{MM-ES} scheme is better than that obtained by the \texttt{UM-ES} scheme with the same number of mesh points.  

To verify the high efficiency of our schemes on adaptive moving meshes, the CPU times are listed in Table \ref{tb:Time_Compare}.
One observes that the \texttt{MM-ES} scheme takes less CPU time compared with the \texttt{UM-ES} scheme using a finer mesh,
which illustrates the efficiency of our adaptive moving mesh method.
The CPU times are tested with the code programmed by MATLAB R2021b on a laptop with Intel\textsuperscript{\textregistered} Core{\texttrademark} i7-8750H CPU @2.20GHz, 24GB memory.


\begin{example}[Circle dam-break problem]\rm\label{ex:DamBreak}
	This example test is taken from  \cite{Castro2012Central}, which is used to verify the shock-capturing ability and efficiency of our schemes for the two-layer SWEs. The bottom topography is  
	\begin{equation*}
		b(x_1,x_2) = 0.5e^{-5\left(x_1^2+x_2^2\right)}-2.0,
	\end{equation*}
	and the initial data are
	\begin{align*}
		&h_1(x_1,x_2,0) = \begin{cases}
			1.8, & \text{if}\quad x_1^2+x_2^2 \geqslant 1,\\
	0.2, & \text{otherwise},
		\end{cases}\\
	    &h_2(x_1,x_2,0) = \begin{cases}
		0.2-0.5e^{-5\left(x_1^2+x_2^2\right)}, & \text{if}\quad x_1^2+x_2^2 \geqslant 1,\\
		1.8-0.5e^{-5\left(x_1^2+x_2^2\right)}, & \text{otherwise},
		\end{cases}
	\end{align*}
	with zero velocities in the physical domain $[-2, 2] \times [-2, 2]$ with outflow boundary conditions, and $g=9.812$, $r_{12}=0.98$, $\rho_2=1$. The output time is $t = 1.0$. The monitor function is the same as that in Example \ref{ex:2D_WB_Test}, but with $\sigma = h_2+b+5$.
\end{example}
\begin{figure}[!htb]
	\centering
 \begin{subfigure}[b]{0.3\textwidth}
		\centering
		\includegraphics[width=1.0\textwidth]{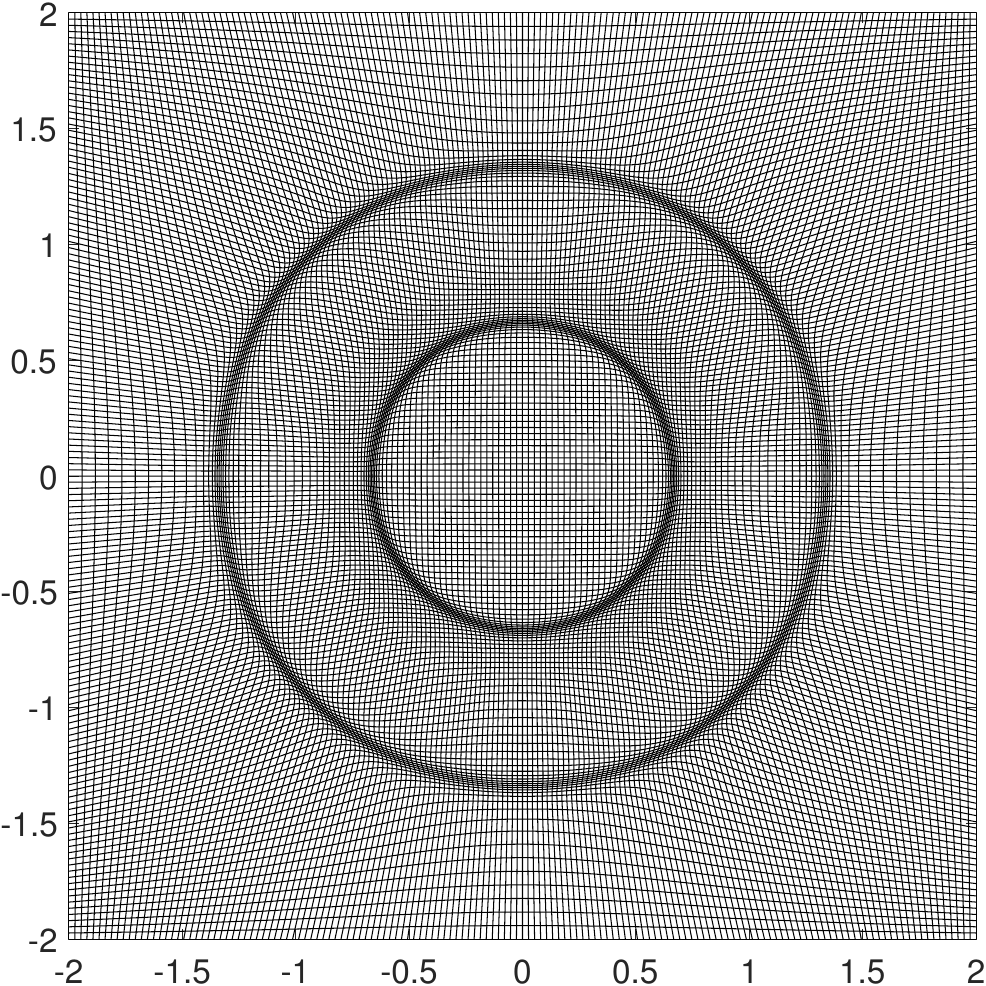}
  \caption{adaptive mesh}
	\end{subfigure}
	\begin{subfigure}[b]{0.3\textwidth}
		\centering
\includegraphics[width=1.0\textwidth]{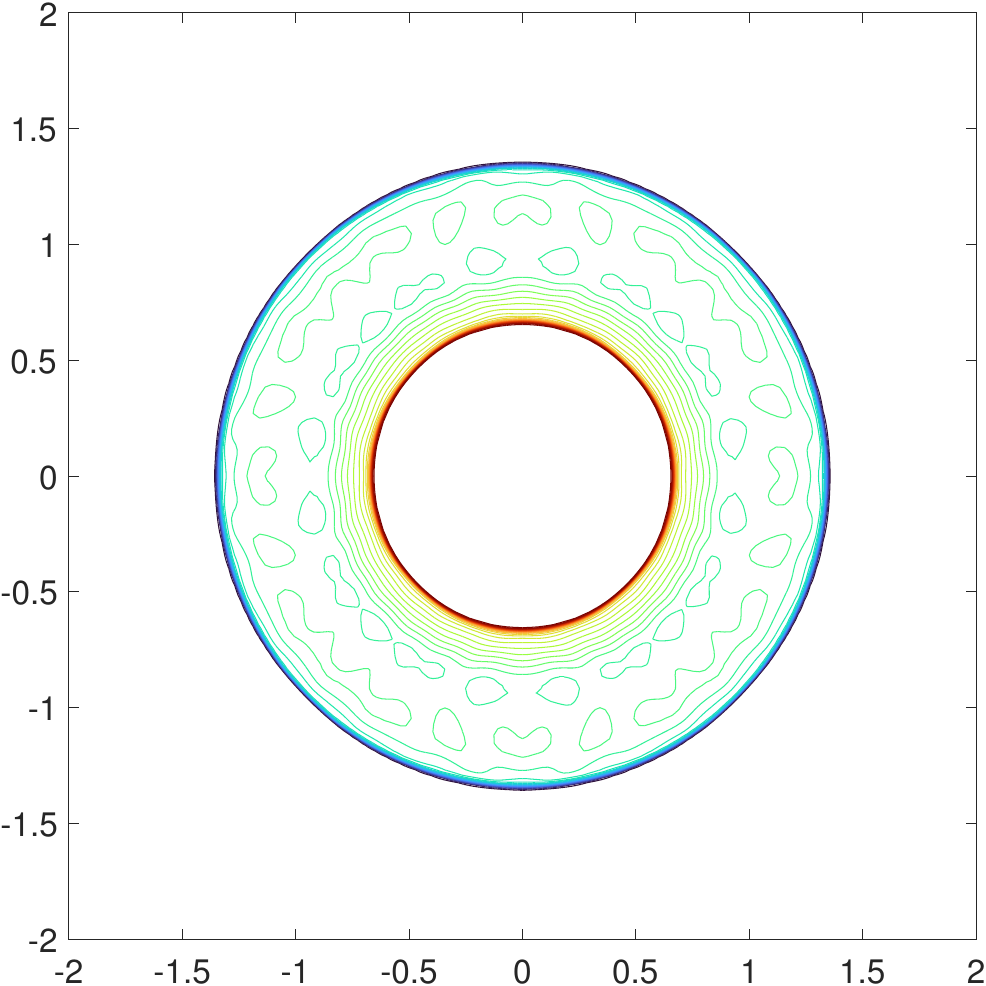}
\caption{$40$ contours of $h_2+b$}
	\end{subfigure}
	\begin{subfigure}[b]{0.3\textwidth}
	\centering
\includegraphics[width=1.0\textwidth]{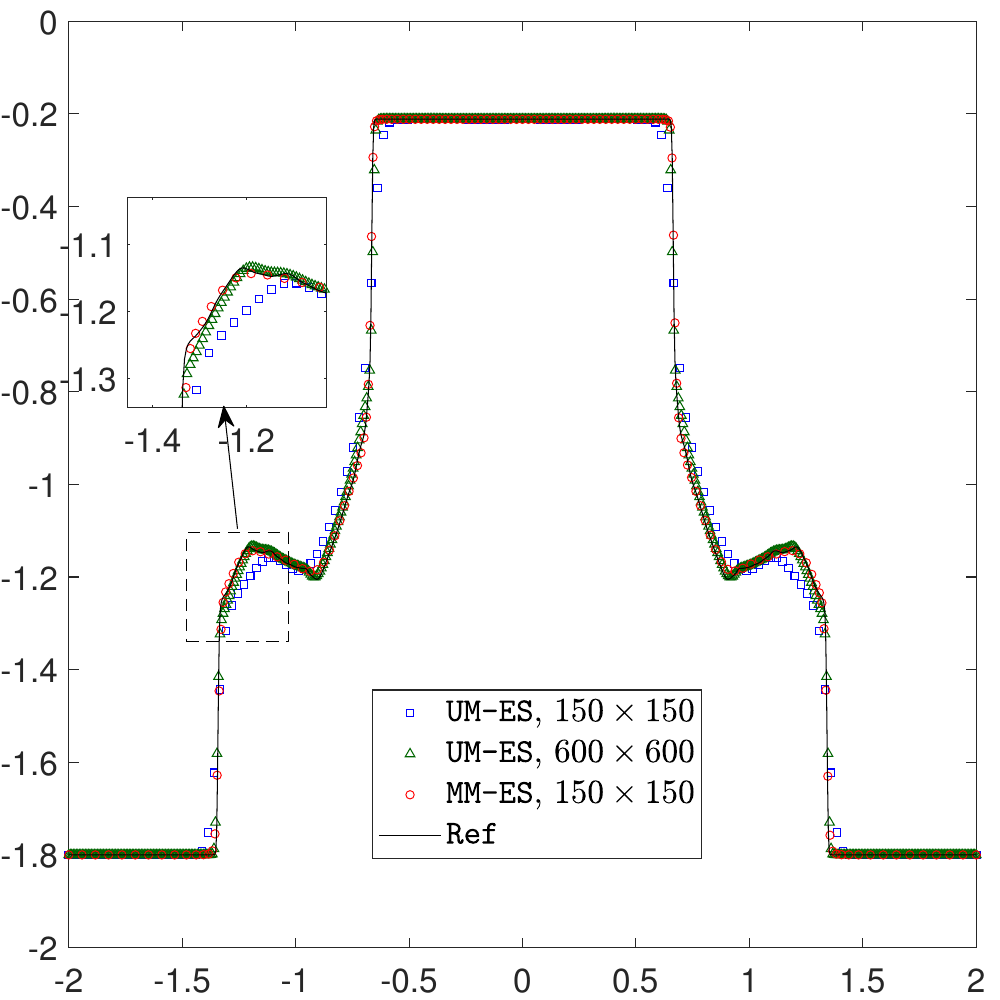}
\caption{cut line along $x_2=0$}
\end{subfigure}
	\caption{Example \ref{ex:DamBreak}.  The numerical results obtained by using the \texttt{MM-ES} scheme with $150\times150$ mesh, and the \texttt{UM-ES} scheme with $150\times150$ and $600\times600$ meshes.
	}
	\label{fig:2D_Dam_Break}
\end{figure}

Figure \ref{fig:2D_Dam_Break} gives the $150\times150$ adaptive mesh, the $40$ equally spaced contours of the water surface level $h_2+b$ obtained by the \texttt{MM-ES} scheme on $150\times150$ mesh and the cut line along $x_2=0$ of $h_2+b$ compared with the \texttt{UM-ES} scheme. The reference solution is obtained by using the \texttt{UM-ES} scheme with $1000\times1000$ mesh. It is seen that our schemes can capture the discontinuities accurately and the \texttt{MM-ES} scheme gives a higher resolution around $x_1\approx\pm1.7$ compared with the \texttt{UM-ES} scheme on a finer mesh.


\begin{example}[The perturbed flow in lake at rest]\label{ex:Pertubation}\rm
 This example tests the ability of our schemes to capture small perturbations over the lake at rest for the two-layer SWEs. The physical domain is $[0,2] \times [0,1]$ with outflow boundary conditions. The bottom topography resembles an ``oval hump'', similar to the example for the single-layer SWEs in \cite{Xing2017Numerical,Xing2005High}, but with extra water depth $h_1$. Initially, 
  \begin{align*}
    & h_2= \begin{cases}1.01-b,~&\text{if}\quad x_1 \in[0.05,0.15], \\
      1-b,~&\text{otherwise},
    \end{cases} \\
    &b=0.8 \exp(-5(x_1-0.9)^2-50(x_2-0.5)^2),\\
    & h_1 = 2-h_2-b,\\
    & u_1=v_1=0,~u_2=v_2=0.
  \end{align*}
 The density ratio is $r_{12} = 0.85$, with $\rho_2 = 1.0$. The monitor function used in this test is the same as Example \ref{ex:2D_WB_Test}, except that the value of $\sigma$ is equal to $h_2+b$.
\end{example}
Figure \ref{fig:2D_UM_Pertubation} shows the $30$ equally spaced contour lines of the water surface levels $h_1+h_2+b$ and $h_2+b$ at $t = 0.9$, $1.3$, $1.7$, $2.0$ obtained by using the \texttt{UM-ES} scheme with $300\times150$ mesh.
The $300\times150$ adaptive mesh obtained by the \texttt{MM-ES} schemes at $t=0.9$ is plotted in Figure \ref{fig:2D_MM_Pertubation}. 
The comparison of the results between the \texttt{UM-ES} and \texttt{MM-ES} schemes is shown in Figure \ref{fig:2D_UM_and_MM_Pertubation}.
One observes that our schemes can capture the wave structures accurately without obvious oscillations, and results of the \texttt{MM-ES} scheme on the $300\times 150$ mesh are comparable to those of the \texttt{UM-ES} scheme on the $600\times 300$ mesh, highlighting the efficiency of the \texttt{MM-ES} schemes.

\begin{figure}[!htb]
  \centering
  \begin{subfigure}[b]{0.4\textwidth}
    \centering
    \includegraphics[width=1.0\textwidth]{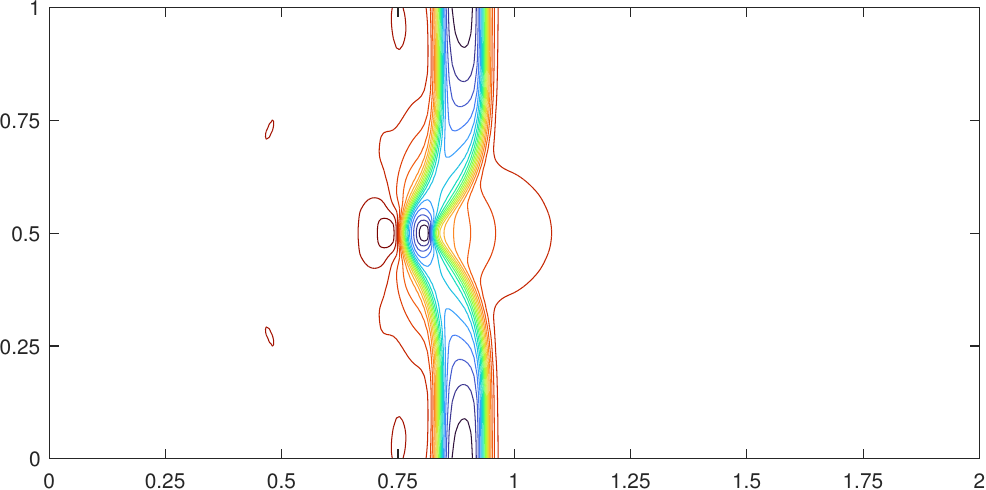}
  \end{subfigure}~
  \begin{subfigure}[b]{0.4\textwidth}
    \centering
    \includegraphics[width=1.0\textwidth]{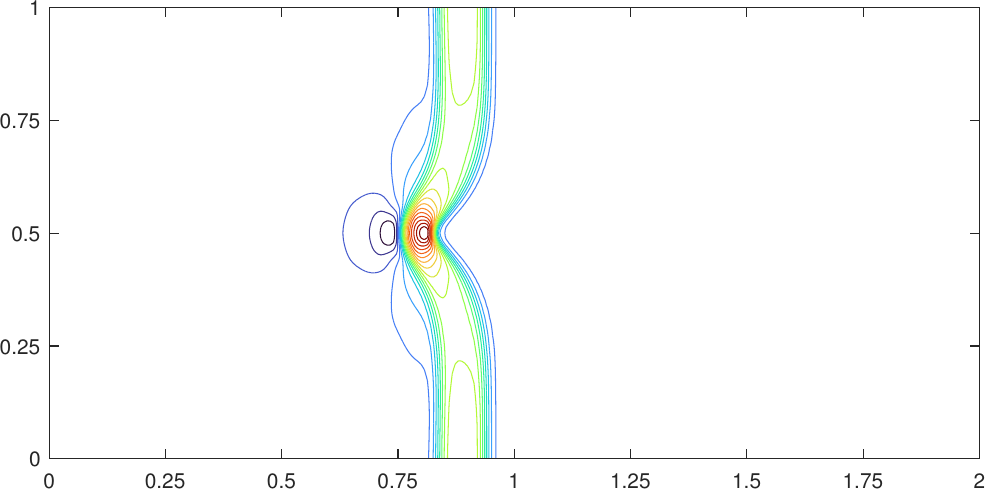}
  \end{subfigure}
  \\
  \begin{subfigure}[b]{0.4\textwidth}
    \centering
    \includegraphics[width=1.0\textwidth]{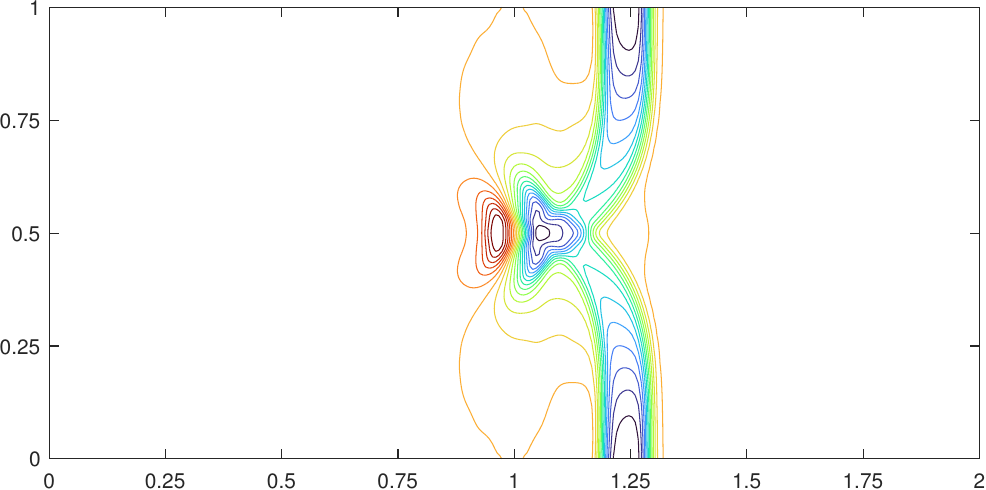}
  \end{subfigure}~
  \begin{subfigure}[b]{0.4\textwidth}
    \centering
    \includegraphics[width=1.0\textwidth]{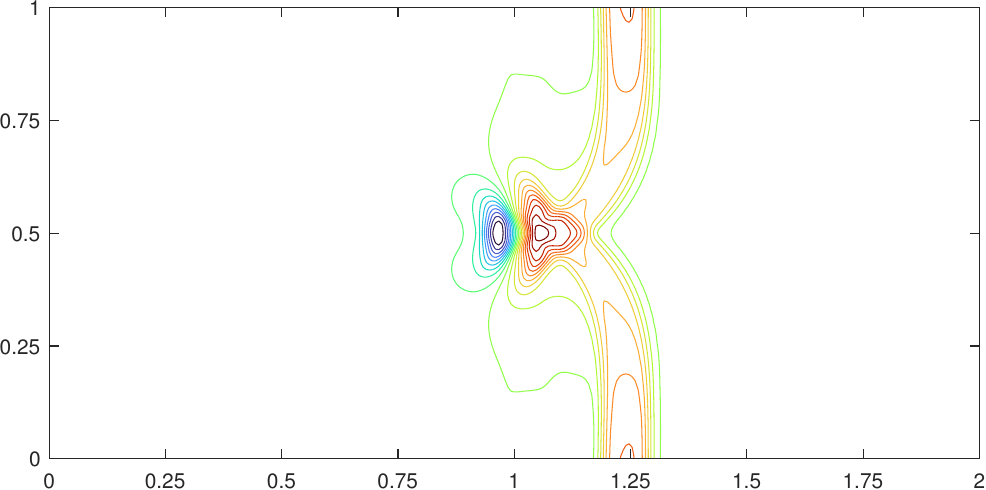}
  \end{subfigure}
  \quad\\
  \begin{subfigure}[b]{0.4\textwidth}
    \centering
    \includegraphics[width=1.0\textwidth]{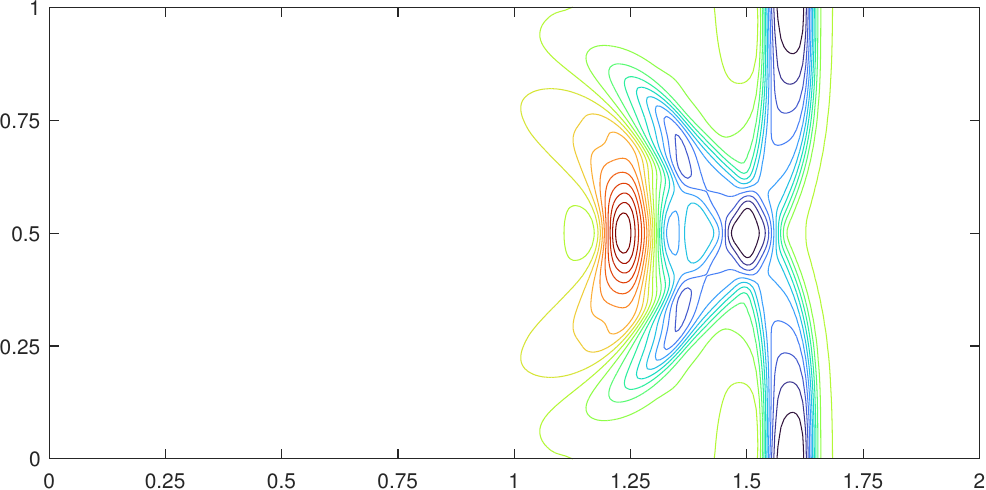}
  \end{subfigure}~
  \begin{subfigure}[b]{0.4\textwidth}
    \centering
    \includegraphics[width=1.0\textwidth]{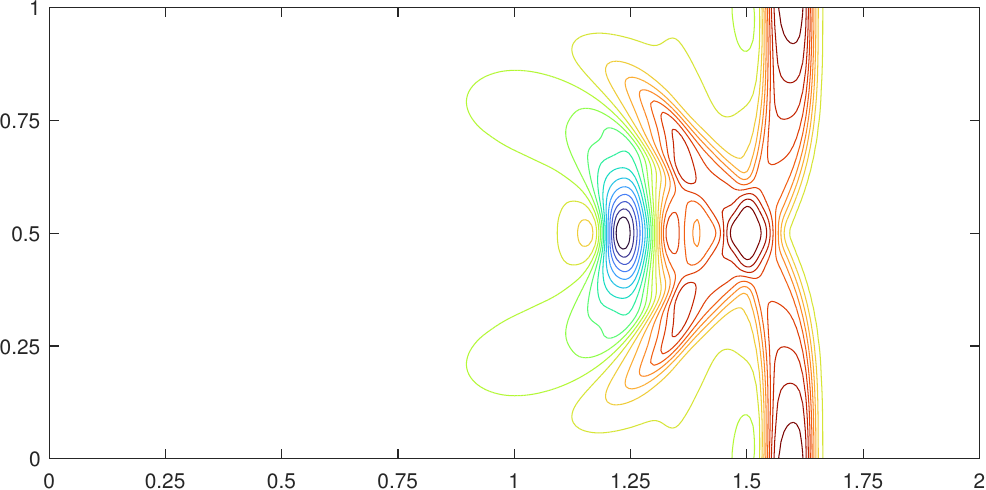}
  \end{subfigure}
  \quad\\
  \begin{subfigure}[b]{0.4\textwidth}
    \centering
    \includegraphics[width=1.0\textwidth]{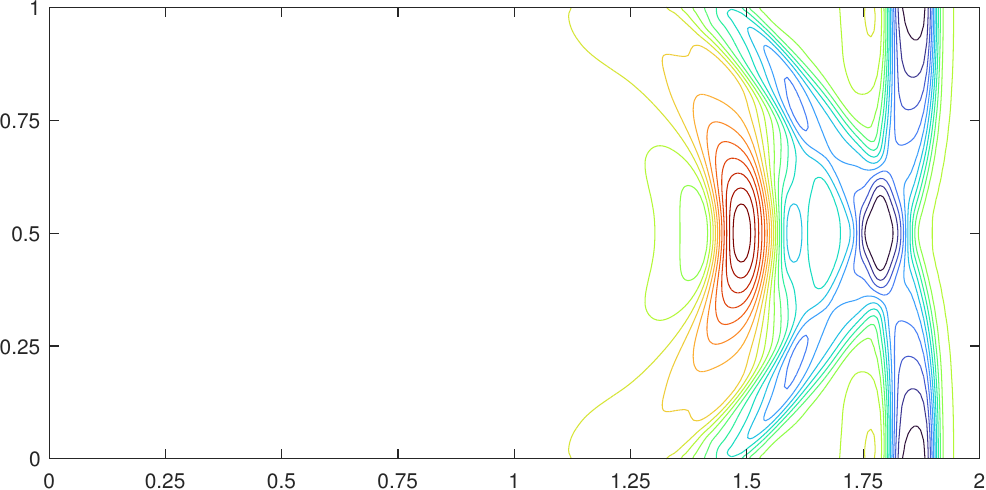}
  \end{subfigure}~
  \begin{subfigure}[b]{0.4\textwidth}
    \centering
    \includegraphics[width=1.0\textwidth]{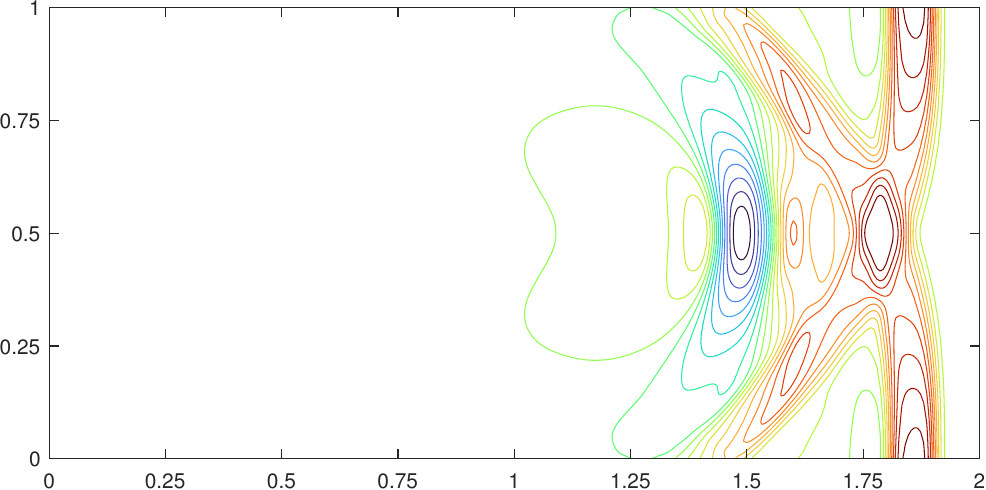}
  \end{subfigure}
  \caption{Example \ref{ex:Pertubation}. $30$ equally spaced contours of the water surface levels obtained by the \texttt{UM-ES} scheme on $300\times150$ mesh. From top to bottom: $t = 0.9$, $1.3$, $1.7$, $2.0$. Left: $h_1+h_2+b$, right: $h_2+b$.}\label{fig:2D_UM_Pertubation}
\end{figure}
\begin{figure}[!htb]
	\centering
	\begin{subfigure}[b]{0.5\textwidth}
		\centering
		\includegraphics[width=1.0\textwidth]{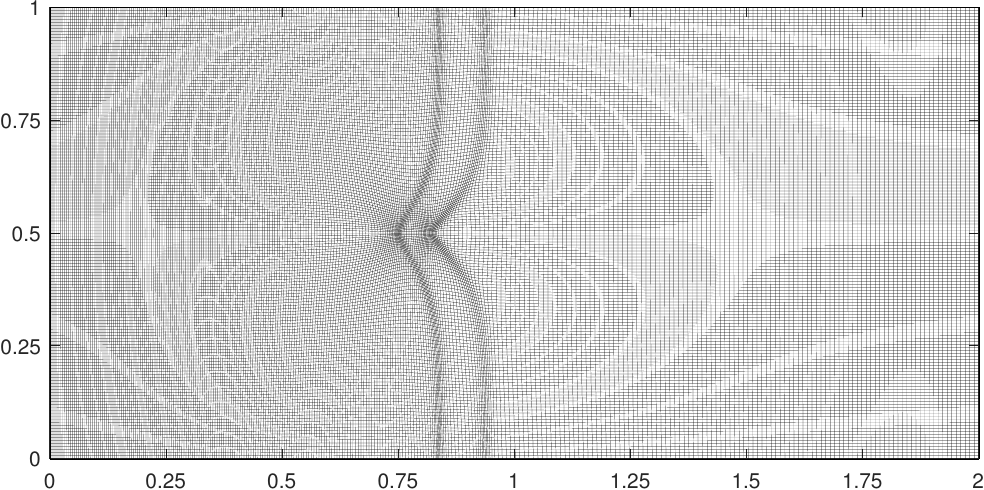}	
\end{subfigure}
	\caption{Example \ref{ex:Pertubation}. The adaptive mesh obtained by the \texttt{MM-ES} scheme with $300\times150$ mesh at $t=0.9$.}
 \label{fig:2D_MM_Pertubation}
\end{figure}

\begin{figure}[!htb]
	\centering
	\begin{subfigure}[b]{0.4\textwidth}
		\centering
		\includegraphics[width=1.0\textwidth]{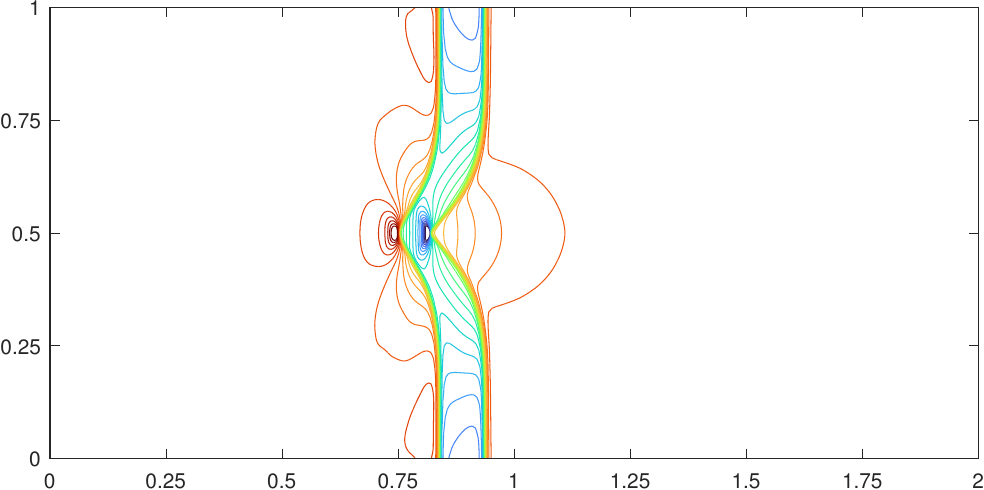}
	\end{subfigure}
	\begin{subfigure}[b]{0.4\textwidth}
		\centering
		\includegraphics[width=1.0\textwidth]{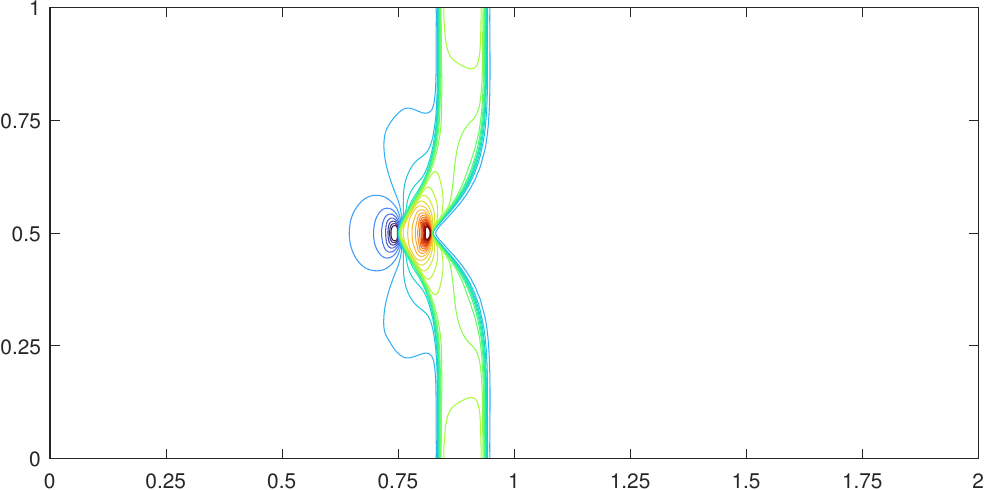}
	\end{subfigure}
  \quad\\
    \vspace{3pt}
	\begin{subfigure}[b]{0.4\textwidth}
		\centering
		\includegraphics[width=1.0\textwidth]{Figures/2DMMPerturbation/2D_UM_301_151_t=0.9_h1h2b.pdf}
	\end{subfigure}
	\begin{subfigure}[b]{0.4\textwidth}
		\centering
		\includegraphics[width=1.0\textwidth]{Figures/2DMMPerturbation/2D_UM_301_151_t=0.9_h2b.pdf}
	\end{subfigure}
  \quad\\
  \vspace{3pt}
	\begin{subfigure}[b]{0.4\textwidth}
		\centering
		\includegraphics[width=1.0\textwidth]{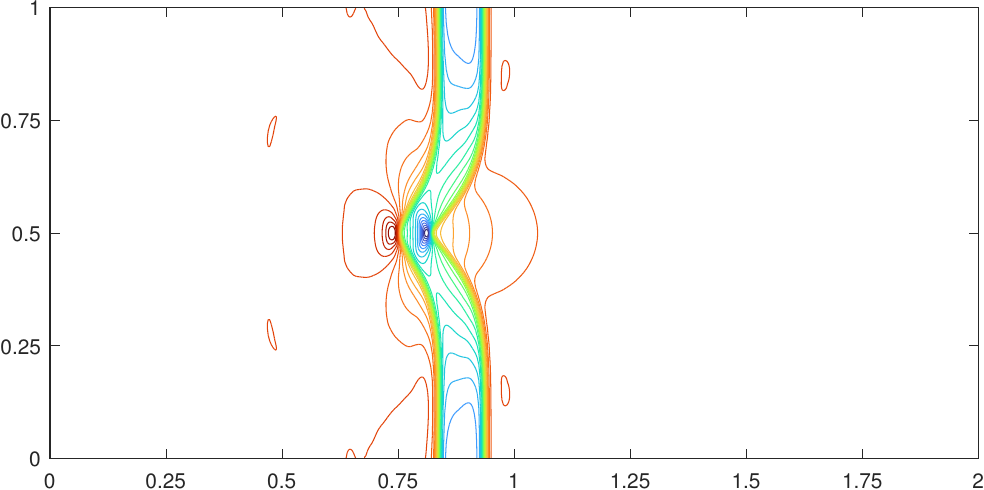}
	\end{subfigure}
	\begin{subfigure}[b]{0.4\textwidth}
		\centering
		\includegraphics[width=1.0\textwidth]{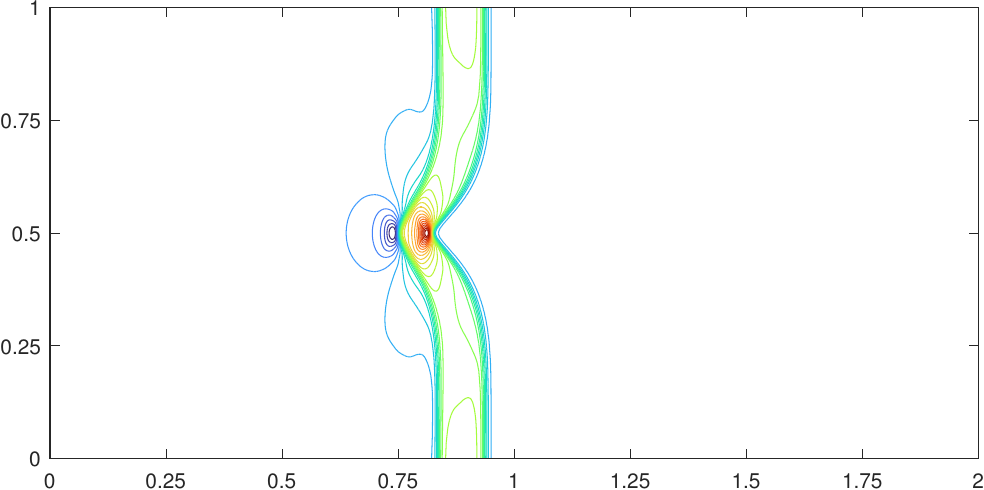}
	\end{subfigure}
	\quad
	\caption{Example \ref{ex:Pertubation}. $30$ equally spaced contours of the water surface levels obtained by the \texttt{UM-ES} and \texttt{MM-ES}  schemes at $t = 0.9$. From top to bottom: \texttt{MM-ES}~(300$\times$150), \texttt{UM-ES}~(300$\times$150), \texttt{UM-ES}~(600$\times$300). Left: $h_1+h_2+b$, right: $h_2+b$.}\label{fig:2D_UM_and_MM_Pertubation}
\end{figure}

\section{Conclusion}\label{section:Conc}
This paper has presented high-order accurate WB ES finite difference schemes for the ML-SWEs (the number of layers $M\geqslant2$) with non-flat bottom topography on fixed and adaptive moving meshes, which extended our previous works \cite{Duan2021_High,Zhang2023High}.
The proof of the convexity of the energy function for an arbitrary $M$ was more challenging than the single-layer case due to the coupling terms resulting from the multiple layers in the energy function, leading to the more complicated Hessian matrix.
The recurrence relations of the leading principal minors or the quadratic forms of the Hessian matrix were found to complete the proof.
To design WB schemes, the sufficient condition for the EC fluxes should include compatible discretizations of the source terms similar to the single-layer case.
The condition was decoupled into $M$ individual equalities for each layer, so that it was convenient to construct the two-point EC flux for the multi-layer system.
The high-order WB EC fluxes were constructed by using the two-point WB EC flux as a building block,
and then the high-order WB ES schemes satisfying semi-discrete energy inequalities were obtained by adding the WENO-based dissipation terms to the high-order WB EC fluxes with compatible discretizations of the source terms.
The high-order explicit SSP-RK3 schemes were used to obtain the fully-discrete schemes, which were proved to preserve the lake at rest.
Moreover, the extension to the adaptive moving meshes was built on the reformulated system and carefully designed dissipation terms, extending the techniques proposed in \cite{Zhang2023High}.
Some two- and three-layer test cases were employed in the numerical experiments to confirm our findings, and due to the lack of an analytical expression for the eigenstructure of the multi-layer system, the proper amount of dissipation was chosen by the estimation of the maximal wave speed, employing the method in \cite{Kurganov2009Central}.

\section*{Acknowledgments}
The authors were partially supported by
the National Key R\&D Program of China (Project Number 2020YFA0712000),
the National Natural Science Foundation of China (No. 12171227 \& 12288101).
J.M. Duan was supported by the Alexander von Humboldt Foundation Research Fellowship CHN-1234352-HFST-P.

\appendix
\section{Proof of the convexity by using quadratic forms} \label{Sec:Proof_of_Remark2.1}

Under the condition $0<\rho_1<\rho_2<\dots<\rho_{\mathcal{N}}<\dots$, the quadratic form of the Hessian matrix $\mathcal{M}_\mathcal{N} = \partial^2{\eta}/\partial {\bm{U}}^2$ is
\begin{equation}\label{eq:Qua_Form}
\begin{aligned}
    Q_\mathcal{N} := 
    \bm{\beta}_{\mathcal{N}} \mathcal{M}_\mathcal{N} \bm{\beta}_{\mathcal{N}}^{\mathrm{T}} = &\ g\rho_1\left(
    \sum_{m=1}^{\mathcal{N}}
    \beta_{3m-2}
    \right)^2
    +g\sum_{m=2}^{\mathcal{N}}(\rho_{m}-\rho_{m-1})
    \left(
    \sum_{l=m}^{\mathcal{N}}
    \beta_{3l-2}
    \right)^2
    \\ & +\sum_{m=1}^{\mathcal{N}}
    \dfrac{\rho_m}{h_m}
    \left(
    \left(
    u_{m}\beta_{3m-2}-\beta_{3m-1}
    \right)^2
    +
    \left(
    v_{m}\beta_{3m-2}-\beta_{3m}
    \right)^2
    \right),
\end{aligned}
\end{equation}
with $\bm{\beta}_{\mathcal{N}} = \left(\beta_1,\dots, \beta_{3\mathcal{N}}\right) \in \mathbb{R}^{1\times{3\mathcal{N}}}$, thus the energy function is convex.

\begin{proof}
The proof is given by induction on $\mathcal{N}$.
When $\mathcal{N}=2$, it is easy to show that the quadratic form of the Hessian matrix $\mathcal{M}_{2}$ \eqref{eq:Hessian_M2} is
\begin{align*}
Q_2 = &\ \dfrac{\rho_1}{h_1}\left(u_1^2\beta_1-2u_1\beta_1\beta_2+\beta_2^2+
v_1^2\beta_1-2v_1\beta_1\beta_3+\beta_3^2\right)
\\ &+
\dfrac{\rho_2}{h_2}\left(u_2^2\beta_4-2u_2\beta_4\beta_5+\beta_5^2+
v_2^2\beta_4-2v_2\beta_4\beta_6+\beta_6^2\right)
\\ &+g\rho_1\beta_1^2+2g\rho_1 \beta_1 \beta_4 +g\rho_2\beta_4^2
\\ = &\ \dfrac{\rho_1}{h_1}\left(
\left(u_1\beta_1-\beta_2
\right)^2+
\left(v_1\beta_1-\beta_3
\right)^2
\right) + \dfrac{\rho_2}{h_2}\left(
\left(u_2\beta_4-\beta_5
\right)^2+
\left(v_2\beta_4-\beta_6
\right)^2
\right)
\\ & +
g\rho_1 \left(
\beta_1+\beta_4
\right)^2+g\left(
\rho_2-\rho_1
\right)\beta_4^2,
\end{align*}
which is positive as $0<\rho_1<\rho_2$ for any $\bm{\beta}_{2} \in \mathbb{R}^{6}\backslash \{\bm{0}_6\}$, thus the matrix $\mathcal{M}_{2}$ is positive-definite.

Now it suffices to show that if the quadratic form of $\mathcal{M}_{\mathcal{N}}$ is $Q_{\mathcal{N}}$ in \eqref{eq:Qua_Form},
then the the quadratic form of $\mathcal{M}_{\mathcal{N}+1}$ is $Q_{\mathcal{N}+1}$.
According to the matrix structure in Lemma \ref{lemma:Hessian}, one can calculate $Q_{\mathcal{N}+1}$ by adding several extra terms to $Q_{\mathcal{N}}$, which can be written as
\begin{align*}
Q_{\mathcal{N}+1} =&\ \bm{\beta}_{\mathcal{N}+1} \mathcal{M}_{\mathcal{N}+1} \bm{\beta}_{\mathcal{N}+1}^{\mathrm{T}} \\ =
&\ Q_{\mathcal{N}} +2\beta_{3\mathcal{N}+1}
    \left(
    \sum_{m=1}^{\mathcal{N}}g\rho_{m}\beta_{3m-2}
    \right)+g\rho_{\mathcal{N}+1}\beta_{3\mathcal{N}+1}^2
    \\&+\dfrac{\rho_{\mathcal{N}+1}}{h_{\mathcal{N}+1}}\left(
    \left(
    u_{\mathcal{N}+1}\beta_{3\mathcal{N}+1} - \beta_{3\mathcal{N}+2}
    \right)^2
    +\left(
    v_{\mathcal{N}+1}\beta_{3\mathcal{N}+1} - \beta_{3\mathcal{N}+3}
    \right)^2
    \right) \\
=&\ g\rho_1\left(
    \sum_{m=1}^{\mathcal{N}}
    \beta_{3m-2}
    \right)^2
    +g\sum_{m=2}^{\mathcal{N}}(\rho_{m}-\rho_{m-1})
    \left(
    \sum_{l=m}^{\mathcal{N}}
    \beta_{3l-2}
    \right)^2
    \\ & +\sum_{m=1}^{\mathcal{N}}
    \dfrac{\rho_m}{h_m}
    \left(
    \left(
    u_{m}\beta_{3m-2}-\beta_{3m-1}
    \right)^2
    +
    \left(
    v_{m}\beta_{3m-2}-\beta_{3m}
    \right)^2
    \right)
    \\
    & +2\beta_{3\mathcal{N}+1}
    \left(
    \sum_{m=1}^{\mathcal{N}}g\rho_{m}\beta_{3m-2}
    \right)+g\rho_{\mathcal{N}+1}\beta_{3\mathcal{N}+1}^2
    \\&+\dfrac{\rho_{\mathcal{N}+1}}{h_{\mathcal{N}+1}}\left(
    \left(
    u_{\mathcal{N}+1}\beta_{3\mathcal{N}+1} - \beta_{3\mathcal{N}+2}
    \right)^2
    +\left(
    v_{\mathcal{N}+1}\beta_{3\mathcal{N}+1} - \beta_{3\mathcal{N}+3}
    \right)^2
    \right)
    \\
    =: &\ \mathcal{H}_{1}+\mathcal{H}_2,
\end{align*}
with 
\begin{align*}
    &\mathcal{H}_1 = \sum_{m=1}^{\mathcal{N}+1}
    \dfrac{\rho_m}{h_m}
    \left(
    \left(
    u_{m}\beta_{3m-2}-\beta_{3m-1}
    \right)^2
    +
    \left(
    v_{m}\beta_{3m-2}-\beta_{3m}
    \right)^2
    \right),\\
    &\mathcal{H}_2 =g\rho_1\left(
    \sum_{m=1}^{\mathcal{N}}
    \beta_{3m-2}
    \right)^2
    +g\sum_{m=2}^{\mathcal{N}}(\rho_{m}-\rho_{m-1})
    \left(
    \sum_{l=m}^{\mathcal{N}}
    \beta_{3l-2}
    \right)^2
    \\&\quad\quad +2\beta_{3\mathcal{N}+1}
    \left(
    \sum_{m=1}^{\mathcal{N}}g\rho_{m}\beta_{3m-2}
    \right)+g\rho_{\mathcal{N}+1}\beta_{3\mathcal{N}+1}^2.
\end{align*}
The term $\mathcal{H}_2$ can be further simplified as 
\begin{align*}
    \mathcal{H}_2 = 
    &\ g\rho_1\left(
    \left(
    \sum_{m=1}^{\mathcal{N}}
    \beta_{3m-2}
    \right)^2
    +
    2 \beta_{3\mathcal{N}+1}
    \left(
    \sum_{m=1}^{\mathcal{N}}
    \beta_{3m-2}
    \right)
    +\beta_{3\mathcal{N}+1}^2
    \right)
    \\
    &+g\sum_{m=2}^{\mathcal{N}}
    \left(\rho_{m}
    -\rho_{m-1}
    \right)
    \left(
    \left(
    \sum_{l=m}^{\mathcal{N}}
    \beta_{3l-2}
    \right)^2
    +2\beta_{3\mathcal{N}+1}
    \left(
    \sum_{l=m}^{\mathcal{N}}
    \beta_{3l-2}
    \right)
    +\beta_{3\mathcal{N}+1}^2
    \right)
   \\
   &+ 2g\beta_{3\mathcal{N}+1}
    \left(
    \sum_{m=1}^{\mathcal{N}}\rho_{m}\beta_{3m-2}
    \right)+g\rho_{\mathcal{N}+1}\beta_{3\mathcal{N}+1}^2
\\&- 2g\beta_{3\mathcal{N}+1}
\left(
\rho_1\left(
    \sum_{m=1}^{\mathcal{N}}
    \beta_{3m-2}
    \right)
    + \sum_{m=2}^{\mathcal{N}}
    \left(\rho_{m}
    -\rho_{m-1}
    \right)\left(
    \sum_{l=m}^{\mathcal{N}}
    \beta_{3l-2}
    \right)
\right)
\\&- g\beta_{3\mathcal{N}+1}^2
\left(
\rho_1
    + \sum_{m=2}^{\mathcal{N}}
    \left(\rho_{m}
    -\rho_{m-1}
    \right)
\right).
\end{align*}
Using
\begin{align*}
    &g\rho_{\mathcal{N}+1}\beta_{3\mathcal{N}+1}^2 - g\beta_{3\mathcal{N}+1}^2\left(\rho_1+\sum_{m=2}^{\mathcal{N}}
    \left(\rho_{m}
    -\rho_{m-1}
    \right)\right) = g\left(\rho_{\mathcal{N}+1}-\rho_{\mathcal{N}}\right)\beta_{3\mathcal{N}+1}^2
    ,\\
    &
\rho_1\left(
    \sum_{m=1}^{\mathcal{N}}
    \beta_{3m-2}
    \right)
    + \sum_{m=2}^{\mathcal{N}}
    \left(\rho_{m}
    -\rho_{m-1}
    \right)\left(
    \sum_{l=m}^{\mathcal{N}}
    \beta_{3l-2}
    \right) = 
    \sum_{m=1}^{\mathcal{N}}\rho_{m}\beta_{3m-2},
\end{align*}
yields 
\begin{align*}
    \mathcal{H}_2 = 
    g\rho_1\left(
    \sum_{m=1}^{\mathcal{N}+1}\beta_{3m-2}\right)^2+g\sum_{m=2}^{\mathcal{N}+1}(\rho_{m}-\rho_{m-1})
    \left(
    \sum_{l=m}^{\mathcal{N}+1}
    \beta_{3l-2}
    \right)^2.
\end{align*}
Thus the quadratic form of $\mathcal{M}_{\mathcal{N}+1}$ is $Q_{\mathcal{N}+1}$.
Moreover, it is easy to verify that $Q_{\mathcal{N}} > 0$, so the energy function is convex.
The proof is completed.
\end{proof}

\section{Explicit expressions of the exact solutions used in Examples \ref{ex:1DSmooth}, \ref{eq:Smooth_2D}}
\label{Sec:Exact_Solution}
In Example \ref{ex:1DSmooth}, the source terms for the two-layer case are $\bm{S} = (0,s_2,0,s_4,0)^{\mathrm{T}}$, with
\begin{align*}
s_2 = &~\pi \cos\left(\pi x\right)\left(\cos\left(\pi t\right)\cos\left(\pi x\right)+6\right)+\pi \cos\left(\pi t\right)\sin\left(\pi x\right)
\\
&~-\frac{3\pi \cos\left(\pi t\right)\sin\left(\pi x\right)\left(\cos\left(\pi t\right)\cos\left(\pi x\right)+6\right)}{2} -\pi \cos\left(\pi t\right)\sin\left(\pi x\right)\left(\frac{\cos\left(\pi t\right)\cos\left(\pi x\right)}{2}+3\right)\\
&~+\frac{\pi \cos\left(\pi t\right){\sin\left(\pi t\right)}^2{\sin\left(\pi x\right)}^3}{{\left(\cos\left(\pi t\right)\cos\left(\pi x\right)+6\right)}^2}+\frac{2\pi \cos\left(\pi x\right){\sin\left(\pi t\right)}^2\sin\left(\pi x\right)}{\cos\left(\pi t\right)\cos\left(\pi x\right)+6},
\\
	s_4 = &~ \pi \cos\left(\pi x\right)\left(\cos\left(\pi t\right)\cos\left(\pi x\right)+4\right)+\pi \cos\left(\pi t\right)\sin\left(\pi x\right)\\
	&-\frac{\pi \cos\left(\pi t\right)\sin\left(\pi x\right)\left(\cos\left(\pi t\right)\cos\left(\pi x\right)+4\right)}{2}
	\\&~
	-\pi \cos\left(\pi t\right)\sin\left(\pi x\right)\left(\frac{\cos\left(\pi t\right)\cos\left(\pi x\right)}{2}+2\right)-\pi r\cos\left(\pi t\right)\sin\left(\pi x\right)\left(\cos\left(\pi t\right)\cos\left(\pi x\right)+4\right)
	\\&~ +\frac{\pi \cos\left(\pi t\right){\sin\left(\pi t\right)}^2{\sin\left(\pi x\right)}^3}{{\left(\cos\left(\pi t\right)\cos\left(\pi x\right)+4\right)}^2}+\frac{2\pi \cos\left(\pi x\right){\sin\left(\pi t\right)}^2\sin\left(\pi x\right)}{\cos\left(\pi t\right)\cos\left(\pi x\right)+4}.
\end{align*}
For the three-layer case, the expressions are omitted.

In Example \ref{eq:Smooth_2D}, the exact solutions for the three-layer case are
	\begin{align*}
		&h_1(x_1,x_2,t) = 
		\cos\left(\pi t\right)\cos\left(\pi x_1\right)+\cos\left(\pi t\right)\cos\left(\pi x_2\right)+8,\\
		&u_1(x_1,x_2,t) = \frac{\sin\left(\pi t\right)\sin\left(\pi x_1\right)}{h_1},~v_1(x_1,x_2,t) = \frac{\sin\left(\pi t\right)\sin\left(\pi x_2\right)}{h_1}, \\
		&h_2(x_1,x_2,t) = 
		\cos\left(\pi t\right)\cos\left(\pi x_1\right)+\cos\left(\pi t\right)\cos\left(\pi x_2\right)+6,\\
		&u_2(x_1,x_2,t) = \frac{\sin\left(\pi t\right)\sin\left(\pi x_1\right)}{h_2},~v_2(x_1,x_2,t) = \frac{\sin\left(\pi t\right)\sin\left(\pi x_2\right)}{h_2}, \\
		&h_3(x_1,x_2,t) = 
		\cos\left(\pi t\right)\cos\left(\pi x_1\right)+\cos\left(\pi t\right)\cos\left(\pi x_2\right)+4,\\
		&u_3(x_1,x_2,t) = \frac{\sin\left(\pi t\right)\sin\left(\pi x_1\right)}{h_3},~v_2(x_1,x_2,t) = \frac{\sin\left(\pi t\right)\sin\left(\pi x_2\right)}{h_3}, \\
		&b(x_1,x_2) = \sin\left(\pi x_1\right)+\sin\left(\pi x_1\right)+\frac{3}{2},
 	\end{align*}
and the extra source terms for the two-layer and three-layer cases can be obtained by some algebraic manipulations similar to Example \ref{ex:1DSmooth}.

\end{document}